\documentclass[10pt,aps,prb, preprint,reqno]{amsart}
\usepackage{amsmath}
\usepackage{amssymb}
\usepackage{yfonts}
\usepackage{hyperref}
\usepackage{mathrsfs}
\usepackage {latexsym}
\usepackage{graphics}
\usepackage{txfonts}
\usepackage{enumitem}
\allowdisplaybreaks

\DeclareMathSymbol{:}{\mathord}{operators}{"3A}

\newtheorem{theorem}{Theorem}[section]
\newtheorem{remark}{Remark}[section]
\newtheorem{lemma}[theorem]{Lemma}

\newtheorem{proposition}[theorem]{Proposition}

\newtheorem{define}{Definition}[section]

\allowdisplaybreaks 

\begin{document}
\title[Non-uniqueness in law for Boussinesq system]{Non-uniqueness in law for Boussinesq system forced by random noise}
 
\author{Kazuo Yamazaki}  
\address{Texas Tech University, Department of Mathematics and Statistics, Lubbock, TX, 79409-1042, U.S.A.; Phone: 806-834-6112; Fax: 806-742-1112; E-mail: (kyamazak@ttu.edu)}
\date{}
\maketitle

\begin{abstract}
Non-uniqueness in law for three-dimensional Navier-Stokes equations forced by random noise was established recently in Hofmanov$\acute{\mathrm{a}}$ et al. (2019, arXiv:1912.11841 [math.PR]). The purpose of this work is to prove non-uniqueness in law for the Boussinesq system forced by random noise. Diffusion within the equation of its temperature scalar field has a full Laplacian and the temperature scalar field can be initially smooth. 
\vspace{5mm}

\end{abstract}
\footnote{2010MSC : 35A02; 35Q30; 35R60}

\section{Introduction} 

\subsection{Motivation from physics and applications}\label{Motivation from physics and applications}
Ocean circulation is turbulent as motions on a wide range of scales from a few centimeters up to thousands of kilometers interact continuously with one another while atmospheric turbulence refers to small-scale irregular air motions that vary in both speed and direction due to wind. At certain scales in the atmosphere and oceans, fluid dynamics can be seen to be governed by the interaction of gravity and rotation of the earth with density variations about a reference state (e.g., \cite[Cha. 1.1]{M03}), and Boussinesq system that couples the Navier-Stokes (NS) equations with another equation of temperature scalar field is said to be the most appropriate model on these scales (e.g., \cite[Sec. 14.2]{T88}). Subtracting pure conduction solution from the temperature scalar field transforms the Boussinesq system to B$\acute{\mathrm{e}}$nard problem of thermohydraulics (e.g., \cite[p. 133]{T97}) while two-dimensional (2D) Boussinesq system is also famous for its correspondence with three-dimensional (3D) axisymmetric swirling flows (e.g., \cite[Sec. 5.4.1]{MB02}). Hereafter, we denote ``$n$-dimensional'' by $n$D for $n\in\mathbb{N}$. These partial differential equations (PDEs) under random force have also been studied for more than half a century as an effective approach to investigate turbulence (e.g., \cite{N65}). Encouraged by numerical conjectures (e.g., \cite{ES94, PS92}), rigorous proofs of well- or ill-posedness of PDEs in fluid dynamics have received special attention for many decades in both deterministic and stochastic cases, some of which we review next.  

\subsection{Previous results concerning uniqueness}
We take $x \in \mathbb{T}^{n}$ for $n \in \mathbb{N}$ although much of subsequent discussions apply to the case $x \in \mathbb{R}^{n}$. We represent velocity, pressure, and temperature fields respectively by $u: \hspace{0.5mm} \mathbb{R}_{+} \times \mathbb{T}^{n} \mapsto \mathbb{R}^{n}$, $\pi: \hspace{0.5mm}  \mathbb{R}_{+} \times \mathbb{T}^{n} \mapsto \mathbb{R}$, and $\theta:  \hspace{0.5mm}  \mathbb{R}_{+} \times \mathbb{T}^{n} \mapsto \mathbb{R}$, viscous and thermal diffusivity respectively by $\nu \geq 0$ and $\kappa \geq 0$, and $j$-th component of standard basis of $\mathbb{R}^{n}$ by $e^{j}$. Hereafter, we denote a $k$-th component of any vector $v$ by $v^{k}$ and $\partial_{t} \triangleq \frac{\partial}{\partial t}$. Under such notations, a system of our main concern may be written as 
\begin{subequations}\label{1}
\begin{align}
& \partial_{t} u + (u\cdot\nabla) u + \nabla \pi + \nu (-\Delta)^{m} u =  \theta e^{n}, \hspace{1mm} \nabla\cdot u = 0, \hspace{3mm} t > 0, \hspace{3mm} u(x,0) = u^{\text{in}} (x), \label{1a}\\
& \partial_{t} \theta + (u\cdot\nabla) \theta + \kappa (-\Delta)^{l} \theta = 0,  \hspace{31mm} t > 0, \hspace{3mm} \theta(x,0) = \theta^{\text{in}}(x), \label{1b}
\end{align}
\end{subequations} 
where $m, l > 0,$ and $(-\Delta)^{\alpha}$ for general $\alpha \in \mathbb{R}$ is defined by 
\begin{equation}\label{2}
(-\Delta)^{\alpha} f (x) \triangleq \sum_{k \in \mathbb{Z}^{n}}  \lvert k \rvert^{2\alpha} \hat{f}(k) e^{ik\cdot x}  
\end{equation} 
(e.g., \cite{CC04}). The case $\theta \equiv 0$ and $m = 1$ reduces to the NS equations and additionally considering $\nu = 0$ leads to Euler equations. We recall that $v \in C_{\text{weak}}^{0} ([0,T]; L_{x}^{2}) \cap L^{2}([0,T]; \dot{H}_{x}^{1})$ is called a Leray-Hopf weak solution of the NS equations if $v(t, \cdot)$ is weakly divergence-free, mean-zero, satisfies both  \eqref{1a} with $\theta \equiv 0$ distributionally and an energy inequality of $\lVert v(t) \rVert_{L_{x}^{2}}^{2} + 2 \nu \lVert v \rVert_{L_{t}^{2}\dot{H}_{x}^{1}}^{2} \leq \lVert v(0) \rVert_{L_{x}^{2}}^{2}$ for any $t \in [0,T]$. On the other hand, $v \in C_{t}^{0} L_{x}^{2}$ is called a weak solution of the NS equations if $v(t, \cdot)$ is weakly divergence-free, mean-zero, and satisfies \eqref{1a} with $\theta \equiv 0$ distributionally for any $t \in [0,T]$ (see \cite[Def. 3.5 and 3.6]{BV19b}). 

In case $\theta \equiv 0$, \eqref{1a} was introduced in \cite[Rem. 8.1]{L59} by Lions who subsequently claimed the uniqueness of its Leray-Hopf weak solution when $m \geq \frac{1}{2} + \frac{n}{4}$ (\cite[Equ. (6.164)]{L69}). Hereafter, we  refer to \eqref{1a} with $\theta \equiv 0$ as the generalized NS (GNS) equations while \eqref{1} as the generalized Boussinesq system. It is well-known that the GNS equations have a rescaling property such that if $u(t,x)$ is its solution, then so is $u_{\lambda}(t,x) \triangleq \lambda^{2m-1} u(\lambda^{2m} t, \lambda x)$ for any $\lambda \in \mathbb{R}_{+}$ that satisfies $\lVert u_{\lambda}  (t) \rVert_{L^{2}(\mathbb{R}^{n})}^{2} = \lambda^{4m-2-n} \lVert u(\lambda^{2m} t) \rVert_{L^{2}(\mathbb{R}^{n})}^{2}$; considering that $4m - 2 - n = 0$ when $m = \frac{1}{2} + \frac{n}{4}$, we say that the GNS equations is $L_{x}^{2}$-norm subcritical, critical, and supercritical when $m > \frac{1}{2} + \frac{n}{4}, m = \frac{1}{2} + \frac{n}{4},$ and $m < \frac{1}{2} + \frac{n}{4}$, respectively. Such a classification clarifies that the GNS equations of which Leray-Hopf weak solutions were shown to be unique by Lions were $L_{x}^{2}$-norm subcritical or critical; to this day, only a logarithmic improvement by Tao \cite{T09} has been made in the supercritical regime. In case $n = 2$, one can show that $\nabla \times u (t) \in L^{p} (\mathbb{R}^{2})$ for all $p \in [1,\infty]$ and $t \geq 0$ if $\nabla \times u^{\text{in}} \in L^{1}(\mathbb{R}^{2}) \cap L^{\infty} (\mathbb{R}^{2})$ and prove global well-posedness of the 2D Euler equations (\cite{Y63}). However, this phenomenon for the 2D Euler equations is no longer valid if its initial data is rougher, e.g., $u^{\text{in}} \in L_{x}^{2}$. 

Analogous classification for the generalized Boussinesq system is more subtle. Only when $m = l$, it has a rescaling property with $u_{\lambda} (t,x) \triangleq \lambda^{2m-1} u(\lambda^{2m} t, \lambda x,)$ and $\theta_{\lambda}(t,x) \triangleq \lambda^{4m-1} \theta(\lambda^{2m} t, \lambda x)$. Starting from smooth initial data has led to much developments in the 2D case (e.g., \cite{C06, HK09, HL05, JMWZ14}); in particular, \cite{HKR10, HKR10b} established global well-posedness when $\nu > 0, m = \frac{1}{2}$, $\kappa = 0$ and $\nu = 0, \kappa > 0$, $l = \frac{1}{2}$, respectively. In the 3D case, \cite{Y15, Y15a} showed that \eqref{1} is globally well-posed starting from smooth initial data if $\nu > 0, m = \frac{1}{2} + \frac{n}{4}$ and $\kappa = 0$, extending Lions' result with zero thermal diffusion.  

Concerning uniqueness in the stochastic case when \eqref{1a} with $\theta \equiv 0$ is forced by noise of form $F(u) dB$ where $F$ is a certain operator and $B$ is a Brownian motion, we recall that uniqueness in law holds if for any solution $(u, B)$ and $(\tilde{u}, \tilde{B})$ with same initial distributions, the law of $u$ coincides with that of $\tilde{u}$ while path-wise uniqueness holds if for any solutions $(u, B)$ and $(\tilde{u}, B)$ with common initial data defined on same probability space, $u(t) = \tilde{u}(t)$ for all $t$ with probability one. Classical Yamada-Watanabe theorem states that path-wise uniqueness implies the uniqueness in law while its converse is false (see \cite[Exa. 3.5 in Sec. 5.3]{KS91} for a counterexample). Global existence of Leray-Hopf type weak solutions to the stochastic NS equations was shown by Flandoli and Gatarek \cite{FG95} via a stochastic analogue of Galerkin approximation; we chose to call their solution ``Leray-Hopf type'' because their solution lies in $L_{T}^{2} \dot{H}_{x}^{1}$ (see \cite[Def. 3.1]{FG95}).  Stochastic Boussinesq system has also caught much attention: well-posedness \cite{BM14, XB11, Y16}; ergodicity \cite{F97, LW04}; large deviation \cite{CM10, DM09}. Up to this point, the general consensus was that path-wise uniqueness for 3D stochastic NS equations, which is not difficult to prove in the 2D case (e.g., \cite{CM10}), seems to be as difficult as the deterministic case while verifying its uniqueness in law may be feasible (e.g., \cite[p. 878--879]{DD03}). Next, we review developments of convex integration technique that has played the role of a game-changer concerning the non-uniqueness of these PDEs. 

In 1954, Nash \cite{N54} proved a breakthrough $C^{1}$ isometric embedding theorem in differential geometry. Gromov considered such a result as a primary example of $h$-principle and initiated convex integration technique \cite[Par. 2.4]{G86}. M$\ddot{\mathrm{u}}$ller and $\check{\mathrm{S}}$ver$\acute{\mathrm{a}}$k extended this technique to Lipschitz mappings and obtained unexpected solutions to some Euler-Lagrange equations \cite{MS98, MS03}. Motivated by these works, De Lellis and Sz$\acute{\mathrm{e}}$kelyhidi Jr. \cite{DS09} wrote $n$D Euler equations as a differential inclusion and proved   existence of its weak solution $u \in L^{\infty} (\mathbb{R}_{+} \times \mathbb{R}^{n})$ with compact support, extending previous works by Scheffer \cite{S93} and Shnirelman \cite{S97} which required $n = 2$ and that $u \in L_{t,x}^{2}$. These developments attracted much attention toward the resolution of Onsager's conjecture \cite{O49} as well, specifically that every weak solution $u \in C_{x}^{\alpha}$ to the Euler equations conserves energy if $\alpha > \frac{1}{3}$ while  if $\alpha \leq \frac{1}{3}$, then there exists a weak solution $u \in C_{x}^{\alpha}$ that does not conserve energy. The case $\alpha > \frac{1}{3}$ was settled relatively earlier in \cite{CET94, E94} while the case $\alpha < \frac{1}{3}$ required many extensions and new ideas beyond the technique from \cite{DS09} (e.g., \cite{BDIS15, DS10, DS13}); eventually, Isett \cite{I18} using Mikado flows settled the case $\alpha < \frac{1}{3}$ if $n \geq 3$. 

An important extension of convex integration applicability from the Euler equations to the NS equations was made by Buckmaster and Vicol \cite{BV19a} who proved non-uniqueness of weak solutions to the 3D NS equations, solving an open problem from \cite[p. 88]{S63} whether a non-constant solution to the 3D NS equations can come to rest in finite time. While non-uniqueness of Leray-Hopf weak solutions remains unknown (\cite{GS17} for numerical conjecture), various extensions of \cite{BV19a} followed: non-uniqueness of weak solutions to the 3D GNS equations with $m \in [1, \frac{5}{4})$ \cite{LT20}; the set of singular times of the solutions to the 3D GNS equations has Hausdorff dimension strictly less than one \cite{BCV18}; non-uniqueness of weak solutions to the 2D GNS equations with $m \in [0, 1)$ \cite{LQ20}; non-uniqueness of weak solutions to the 2D generalized Boussinesq system with $m \in [0,1), l = 1$ \cite{LTZ20}. These developments on the deterministic NS equations incited new results in the stochastic case as well; in particular, Hofmanov$\acute{\mathrm{a}}$ et al. \cite{HZZ19} proved non-uniqueness in law of the 3D stochastic NS equations (see cases $n = 3$, $m \in (\frac{13}{20}, \frac{5}{4})$ and $n = 2$, $m \in (0, 1)$ respectively in \cite{Y20a, Y20c}). We also refer to \cite{BFH20, CFF19, HZZ20} for further applications of convex integration in the stochastic case. 

\section{Statement of main results}
In the deterministic case, taking $\theta \equiv 0$ reduces \eqref{2} to the NS equations; thus, non-uniqueness for the NS equations actually implies that of \eqref{2}. The stochastic case is interesting because even if we take $\theta^{\text{in}} \equiv 0$ on $\mathbb{T}^{n}$, in sharp contrast to the deterministic case, a zero temperature field would not be a solution to the stochastic Boussinesq system due to its random force. Thus, to claim non-uniqueness in law for the stochastic Boussinesq system from analogous results on the stochastic GNS equations in \cite{HZZ19, Y20a, Y20c}, we must not only take $\theta^{\text{in}} \equiv 0$ on $\mathbb{T}^{n}$ but also consider zero noise on the equation of the temperature field, and then rely on the zero temperature solution. Attaining the same result with general data $\theta^{\text{in}}$ and non-zero random force is far from trivial. In order to investigate such a case, we study the following stochastic Boussinesq system with general data $\theta^{\text{in}}$: 
\begin{subequations}\label{3}
\begin{align}
&du + [(-\Delta)^{m} u + \text{div} (u\otimes u) + \nabla \pi - \theta e^{n}]dt = F_{1}(u)dB_{1}, \hspace{3mm} \nabla\cdot u =0, \hspace{3mm} t > 0, \label{3a} \\
&d\theta + [- \Delta \theta + \text{div} (u\theta) ]dt = F_{2}(\theta)dB_{2}, \hspace{45mm} t > 0, \label{3b} 
\end{align}
\end{subequations} 
\begin{equation}\label{4}
\text{where } \hspace{1mm} m \in (0,1) \text{ if } n = 2 \hspace{1mm} \text{ while } \hspace{1mm} m \in (\frac{13}{20}, \frac{5}{4}) \text{ if } n = 3. 
\end{equation} 
We let $(\mathcal{F}_{t})_{t\geq 0}$ be the canonical filtration of $(B_{1}, B_{2})$ augmented by all the $\textbf{P}$-negligible sets.

\begin{theorem}\label{Theorem 2.1}
Suppose that \eqref{4} holds, $F_{k} \equiv 1$, $B_{k}$ is a $G_{k}G_{k}^{\ast}$-Wiener process for both $k \in \{1,2\}$, 
and 
\begin{equation}\label{5}
\text{Tr} ( (-\Delta)^{\max \{\frac{n}{2} + 2 \sigma, \frac{n+2}{2} - m + 2 \sigma \}} G_{1} G_{1}^{\ast}) < \infty \hspace{2mm} \text{ and } \hspace{2mm} \text{Tr} ((-\Delta)^{\frac{n}{2}+ 2\sigma} G_{2}G_{2}^{\ast} ) < \infty 
\end{equation} 
for some $\sigma > 0$. Then, given $T> 0, K > 1$, and $\kappa \in (0,1)$, there exist $\gamma \in (0,1)$ and a $\textbf{P}$-almost surely (a.s.) strictly positive stopping time $\mathfrak{t}$ such that 
\begin{equation}\label{6}
\textbf{P} (\{ \mathfrak{t} \geq T \} ) > \kappa 
\end{equation} 
and the following is additionally satisfied. There exist $(\mathcal{F}_{t})_{t\geq 0}$-adapted processes $(u,\theta)$ that is a weak solution of \eqref{3} starting from a deterministic initial condition$(u^{\text{in}}, \theta^{\text{in}})$, satisfies for all $p \in [1, \infty)$, 
\begin{equation}\label{estimate 17}
\text{esssup}_{\omega \in \Omega} \lVert u(\omega) \rVert_{C_{\mathfrak{t}}\dot{H}_{x}^{\gamma}} < \infty, \hspace{5mm}  
\mathbb{E}^{\textbf{P}}[ \lVert \theta \rVert_{C_{\mathfrak{t}} L_{x}^{p}}^{p} + \lVert \theta \rVert_{L_{\mathfrak{t}}^{2} \dot{H}_{x}^{1}}^{p} ] < \infty,  
\end{equation} 
\begin{equation}
\mathbb{E}^{\textbf{P}} [ \lVert \theta(t\wedge \mathfrak{t} ) \rVert_{L_{x}^{2}}^{2} + 2 \int_{0}^{t \wedge \mathfrak{t}} \lVert \theta \rVert_{\dot{H}_{x}^{1}}^{2} dr] \leq \lVert \theta^{\text{in}} \rVert_{L_{x}^{2}}^{2} + \mathbb{E}^{\textbf{P}}[(t\wedge \mathfrak{t}) \text{Tr} (G_{2}G_{2}^{\ast})], 
\end{equation} 
and on a set $\{\mathfrak{t} \geq T \}$,  
\begin{equation}\label{[Equ. (4), Y20c]}
\lVert u(T) \rVert_{L_{x}^{2}} > K e^{\frac{T}{2}} ( \lVert u^{\text{in}} \rVert_{L_{x}^{2}} + \lVert \theta^{\text{in}} \rVert_{L_{x}^{2}} + \sum_{l=1}^{2} \sqrt{  \text{Tr} (G_{l}G_{l}^{\ast})}).
\end{equation} 
\end{theorem} 

\begin{theorem}\label{Theorem 2.2} 
Suppose that \eqref{4} holds, $F_{k} \equiv 1$, $B_{k}$ is a $G_{k}G_{k}^{\ast}$-Wiener process for both $k \in \{1,2\}$, and \eqref{5} holds for some $\sigma > 0$. Then non-uniqueness in law holds for \eqref{3} on $[0,\infty)$. Moreover, for all $T > 0$ fixed, non-uniqueness in law holds for \eqref{3} on $[0,T]$. 
\end{theorem} 

\begin{theorem}\label{Theorem 2.3}
Suppose that \eqref{4} holds, $F_{1}(u) = u, F_{2}(\theta) = \theta$, and $B_{k}$ is an $\mathbb{R}$-valued Wiener process on $(\Omega, \mathcal{F}, \textbf{P})$ for both $k \in \{1,2\}$. Then, given $T > 0, K > 1$, and $\kappa \in (0,1)$, there exist $\gamma \in (0,1)$ and a $\textbf{P}$-a.s. strictly positive stopping time $\mathfrak{t}$ such that \eqref{6} holds and the following is additionally satisfied. There exist $(\mathcal{F}_{t})_{t\geq 0}$-adapted processes $(u,\theta)$ that is a weak solution to \eqref{3} starting from a deterministic initial condition $(u^{\text{in}}, \theta^{\text{in}})$, satisfies for all $p \in [1,\infty)$, 
\begin{equation}\label{estimate 94}
\text{esssup}_{\omega \in \Omega} \lVert u(\omega) \rVert_{C_{\mathfrak{t}} \dot{H}_{x}^{\gamma}} < \infty, \hspace{3mm} \text{esssup}_{\omega \in \Omega} [\lVert \theta (\omega) \rVert_{C_{\mathfrak{t}} L_{x}^{p}} + \lVert \theta(\omega) \rVert_{L_{\mathfrak{t}}^{2} \dot{H}_{x}^{1}}] < \infty, 
\end{equation}
\begin{equation}
\mathbb{E}^{\textbf{P}}[ \lVert \theta(t \wedge \mathfrak{t}) \rVert_{L_{x}^{2}}^{2} + 2 \int_{0}^{t \wedge \mathfrak{t}} \lVert \theta \rVert_{\dot{H}_{x}^{1}}^{2} dr] \leq  \lVert \theta^{\text{in}} \rVert_{L_{x}^{2}}^{2} + \mathbb{E}^{\textbf{P}}[ \int_{0}^{t \wedge \mathfrak{t}} \lVert \theta \rVert_{L_{x}^{2}}^{2} dr], 
\end{equation} 
and on a set $\{\mathfrak{t} \geq T \}$, 
\begin{equation}\label{[Equ. (6), Y20c]}
\lVert u(T) \rVert_{L_{x}^{2}} > K e^{T} [ \lVert u^{\text{in}} \rVert_{L_{x}^{2}} + \lVert \theta^{\text{in}} \rVert_{L_{x}^{2}}]. 
\end{equation} 
\end{theorem}

\begin{theorem}\label{Theorem 2.4} 
Suppose that \eqref{4} holds, $F_{1}(u) = u, F_{2}(\theta) = \theta$, and $B_{k}$ is an $\mathbb{R}$-valued Wiener process on $(\Omega, \mathcal{F}, \textbf{P})$ for both $k \in \{1,2\}$. Then non-uniqueness in law holds for \eqref{3} on $[0,\infty)$. Moreover, for all $T > 0$ fixed, non-uniqueness in law holds for \eqref{3} on $[0,T]$. 
\end{theorem} 

\begin{remark}\label{Remark 2.2}
To the best of the author's knowledge, this is the first instance of non-uniqueness in law for a system of stochastic PDEs. It will be an interesting future work to try to extend Theorems \ref{Theorem 2.1}-\ref{Theorem 2.4} to the thermal diffusion of the form $(-\Delta)^{l} \theta$ with $l \in (0, 1)$. In fact, \cite[Lem. 2.5]{CC04} in case $n = 2$ (and \cite[Lem. 2.1]{CCGO09} in case $n = 3$) gives a positivity of a fractional Laplacian; i.e., $\int_{\mathbb{T}^{n}} (-\Delta)^{l} \theta \lvert \theta \rvert^{p-2} \theta dx \geq 0$ for any $p \geq 1$ and $l \in [0,1]$. Thus, we can certainly extend the $L^{p}$-estimate of $\theta$ with $(-\Delta)^{l} \theta$ in \eqref{estimate 114}, \eqref{estimate 113}, \eqref{estimate 70}, and \eqref{estimate 81}; however, the proofs of Cauchy property in $L_{x}^{2}$-norm in \eqref{estimate 38}, \eqref{estimate 0}, and \eqref{estimate 23} seem to require the full Laplacian. 
\end{remark} 

\begin{remark}
The proofs are inspired by \cite{HZZ19, LTZ20, Y20a, Y20c}. We emphasize one particular difference. E.g., in \cite[p. 3740]{LTZ20}, for the initial step of convex integration technique on 2D deterministic Boussinesq system, the authors take zero velocity, zero temperature, and zero Reynolds stress. In our proof, we choose specific forms instead (see \eqref{[Equ. (50), Y20c]}-\eqref{estimate 106}, \eqref{[Equ. (41), Y20a]}-\eqref{estimate 107}, \eqref{estimate 108}-\eqref{estimate 110}, \eqref{estimate 111}-\eqref{estimate 112}). This is due to technicality that arises in deriving \eqref{[Equ. (4), Y20c]} and \eqref{[Equ. (6), Y20c]}. Some of the major new challenges include Cauchy estimates of the temperature equation (e.g., \eqref{estimate 27}, \eqref{estimate 0}, \eqref{estimate 80}) and the additional estimates due to $\theta e^{n}$ in the Reynolds stress estimate (e.g., \eqref{estimate 63}, \eqref{estimate 68}, \eqref{estimate 77}, \eqref{estimate 357}) which will be elaborated in Remarks \ref{Remark 4.2}, \ref{another remark}, \ref{Remark 4.4}. 
\end{remark}
In what follows, we describe notations and preliminaries, and thereafter prove Theorems \ref{Theorem 2.1}-\ref{Theorem 2.4}; we intend to make these proofs as complete and self-contained as possible. 

\section{Notations and preliminaries}\label{Section 3}
For convenience, we denote $\mathbb{N}_{0} \triangleq \mathbb{N} \cup \{0\}$ and write $A \overset{(\cdot)}{\lesssim}_{a,b} B$ and $A \overset{(\cdot)}{\approx}_{a,b} B$ to indicate respectively the existence of a constant $C = C(a,b) \geq 0$ such that $A \leq CB$ and $A = CB$ due to $(\cdot)$. We denote by $\mathring{\otimes}$ a trace-free tensor product. While we reserve $\mathbb{P}$ for the Leray projection operator, we define $\mathbb{P}_{<r}$ to be a Fourier operator with a Fourier symbol $1_{\{ \lvert \xi \rvert < r \}} (\xi)$ and $\mathbb{P}_{\geq r} \triangleq \text{Id} - \mathbb{P}_{<r}$.  We write for $p \in [1,\infty]$, 
\begin{equation}\label{C-t,x}
\lVert g \rVert_{L^{p}} \triangleq \lVert g \rVert_{L_{t}^{\infty} L_{x}^{p}}, \hspace{2mm} \lVert g \rVert_{C^{N}} \triangleq \lVert g \rVert_{L_{t}^{\infty} C_{x}^{N}} \triangleq \sum_{0\leq \lvert \alpha \rvert \leq N} \lVert D^{\alpha} g \rVert_{L^{\infty}}, \hspace{2mm}  \lVert g \rVert_{C_{t,x}^{N}} \triangleq \sum_{0\leq k + \lvert \alpha \rvert \leq N} \lVert \partial_{t}^{k} D^{\alpha} g\rVert_{L^{\infty}}.
\end{equation} 
Next, we define 
\begin{align*}
&\mathcal{V}_{1} \triangleq \{v \in C^{\infty} (\mathbb{T}^{n}): \hspace{0.5mm}  v \text{ is } \mathbb{R}^{n}\text{-valued, periodic}, \int_{\mathbb{T}^{n}} v dx = 0, \text{ and } \nabla\cdot v = 0 \}, \\
&\mathcal{V}_{2} \triangleq \{\Theta \in C^{\infty} (\mathbb{T}^{n}): \hspace{0.5mm}  \Theta \text{ is } \mathbb{R}\text{-valued, periodic, and } \int_{\mathbb{T}^{n}} \Theta dx = 0 \}, 
\end{align*} 
$L_{\sigma}^{2}$ and $\mathring{L}^{2}$ respectively to be the closures of $\mathcal{V}_{1}$ and $\mathcal{V}_{2}$ in $L^{2}(\mathbb{T}^{n})$. For any Polish space $H$, we write $\mathcal{B}(H)$ to denote the $\sigma$-algebra of Borel sets in $H$. We denote an expectation with respect to (w.r.t.) any probability measure $P$ by $\mathbb{E}^{P}$ and law of a random variable $X$ by $\mathcal{L}(X)$. We denote by $\langle \cdot, \cdot \rangle$ an $L^{2}(\mathbb{T}^{n})$-inner product, $\langle \langle A, B \rangle \rangle$ a quadratic variation of $A$ and $B$, while $\langle \langle A \rangle \rangle \triangleq \langle \langle A, A \rangle \rangle$. We let
\begin{equation}\label{estimate 4}
\Omega_{0} \triangleq C([0,\infty); H^{-3} (\mathbb{T}^{n})) \cap L_{\text{loc}}^{\infty} ([0,\infty); L_{\sigma}^{2}) \times C([0,\infty); H^{-n} (\mathbb{T}^{n})) \cap L_{\text{loc}}^{\infty} ([0,\infty); \mathring{L}^{2}).
\end{equation} 
We define $\xi \triangleq (\xi_{1}, \xi_{2}): \hspace{0.5mm}  \Omega_{0} \mapsto H^{-3} (\mathbb{T}^{n}) \times H^{-n}(\mathbb{T}^{n})$ the canonical process by $\xi_{t} (\omega) \triangleq \omega(t)$. We also denote by $\mathcal{P} (\Omega_{0})$ the set of all probability measures on $(\Omega_{0}, \mathcal{B})$ where $\mathcal{B}$ is the Borel $\sigma$-algebra of $\Omega_{0}$ from the topology of locally uniform convergence on $\Omega_{0}$. Similarly, for any $t \geq 0$, we define 
\begin{equation}
\Omega_{t} \triangleq C([t,\infty); H^{-3} (\mathbb{T}^{n})) \cap L_{\text{loc}}^{\infty} ([t,\infty); L_{\sigma}^{2}) \times C([t,\infty); H^{-n} (\mathbb{T}^{n})) \cap L_{\text{loc}}^{\infty} ([t,\infty); \mathring{L}^{2}),  
\end{equation} 
equipped with Borel $\sigma$-algebra $\mathcal{B}^{t} \triangleq \sigma \{ \xi(s) : \hspace{0.5mm} s \geq t \}$. Furthermore, we define $\mathcal{B}_{t}^{0} \triangleq \sigma \{ \xi(s): \hspace{0.5mm} s \leq t \}$ and $\mathcal{B}_{t} \triangleq \cap_{s > t} \mathcal{B}_{s}^{0}$ for $t \geq 0$. For any Hilbert spaces $U_{1}$ and $U_{2}$, we denote by $L_{2}(U_{1}, L_{\sigma}^{2})$ and $L_{2}(U_{2}, \mathring{L}^{2})$ the spaces of all Hilbert-Schmidt operators from $U_{1}$ to $L_{\sigma}^{2}$ and from $U_{2}$ to $\mathring{L}^{2}$ with norms $\lVert \cdot \rVert_{L_{2}(U_{1}, L_{\sigma}^{2})}$ and $\lVert \cdot \rVert_{L_{2}(U_{2}, \mathring{L}^{2})}$, respectively. We impose on $G_{1}: \hspace{0.5mm} L_{\sigma}^{2} \mapsto L_{2}(U_{1}, L_{\sigma}^{2})$ and $G_{2}: \hspace{0.5mm}  \mathring{L}^{2} \mapsto L_{2}(U_{2}, \mathring{L}^{2})$ to be $\mathcal{B}(L_{\sigma}^{2})/\mathcal{B}(L_{2}(U_{1}, L_{\sigma}^{2}))$-measurable and $\mathcal{B}(\mathring{L}^{2})/ \mathcal{B} (L_{2}(U_{2}, \mathring{L}^{2}))$-measurable, respectively. They must also satisfy 
\begin{subequations}
\begin{align}
&\lVert G_{1} (\phi) \rVert_{L_{2} (U_{1}, L_{\sigma}^{2})} \leq C(1+ \lVert \phi \rVert_{L_{x}^{2}}), \hspace{5mm} 
 \lVert G_{2} (\phi) \rVert_{L_{2} (U_{2}, \mathring{L}^{2})} \leq C(1+ \lVert \phi \rVert_{L_{x}^{2}}), \label{[Equ. (11), Y20a]}\\
&\lim_{l\to\infty} \lVert G_{1} (\psi_{l})^{\ast} \phi - G_{1} (\psi)^{\ast} \phi \rVert_{U_{1}} = 0, \hspace{3mm} \lim_{l\to\infty} \lVert G_{2} (\psi_{l})^{\ast} \phi - G_{2} (\psi)^{\ast} \phi \rVert_{U_{2}} = 0, \label{[Equ. (12), Y20a]} 
\end{align} 
\end{subequations} 
for all $\phi, \psi_{l}, \psi \in C^{\infty} (\mathbb{T}^{n}) \cap L_{\sigma}^{2}$ in case of $G_{1}$ while $C^{\infty} (\mathbb{T}^{n}) \cap \mathring{L}^{2}$ in case of $G_{2}$ such that $\lim_{l\to\infty}\lVert \psi_{l} - \psi \rVert_{L_{x}^{2}} = 0$. We also assume the existence of Hilbert spaces $\tilde{U}_{1}, \tilde{U}_{2}$ such that the embeddings $U_{1} \hookrightarrow \tilde{U}_{1}$ and $U_{2}\hookrightarrow \tilde{U}_{2}$ are Hilbert-Schmidt. We also define 
\begin{align}
\bar{\Omega} \triangleq& C([0, \infty), H^{-3} (\mathbb{T}^{n}) \times \tilde{U}_{1}) \cap L_{\text{loc}}^{\infty} ([0,\infty); L_{\sigma}^{2} \times \tilde{U}_{1}) \nonumber \\
 \times& C([0,\infty); H^{-n} (\mathbb{T}^{n}) \times \tilde{U}_{2}) \cap L_{\text{loc}}^{\infty} ([0,\infty); \mathring{L}^{2} \times \tilde{U}_{2}) 
\end{align} 
and $\mathcal{P} (\bar{\Omega})$ to be the set of all probability measures on $(\bar{\Omega}, \bar{\mathcal{B}})$ where $\bar{\mathcal{B}}$ is the Borel $\sigma$-algebra of $\bar{\Omega}$. Analogously, we define the canonical process on $\bar{\Omega}$ by $(\xi, \zeta): \hspace{0.5mm} \bar{\Omega}:\hspace{0.5mm}  \mapsto H^{-3} (\mathbb{T}^{n}) \times \tilde{U}_{1} \times H^{-n} (\mathbb{T}^{n}) \times \tilde{U}_{2}$ by $(\xi_{t} (\omega), \zeta_{t}(\omega)) \triangleq \omega(t)$. Finally, we define 
\begin{equation}
\bar{\mathcal{B}}^{t} \triangleq \sigma  \{ (\xi(s), \zeta(s)): \hspace{0.5mm} s \geq t \}, \hspace{1mm} \bar{\mathcal{B}}_{t}^{0} \triangleq \sigma  \{ (\xi(s), \zeta(s)): \hspace{0.5mm} s \leq t \},  \hspace{1mm} \text{ and } \hspace{1mm} \bar{\mathcal{B}}_{t} \triangleq \cap_{s > t} \bar{\mathcal{B}}_{s}^{0} \text{ for } t \geq 0. 
\end{equation} 

\section{Proofs of Theorems \ref{Theorem 2.1}-\ref{Theorem 2.2}}
We start with formal definition of solutions to \eqref{3}. 
\begin{define}\label{Definition 4.1}
Let $s \geq 0$ and $\xi^{\text{in}} = (\xi_{1}^{\text{in}}, \xi_{2}^{\text{in}}) \in L_{\sigma}^{2} \times \mathring{L}^{2}$. Then $P \in \mathcal{P} (\Omega_{0})$ is a martingale solution to \eqref{3} with initial condition $\xi^{\text{in}}$ at initial time $s$ if 
\begin{enumerate}
\item [](M1) $P ( \{ \xi(t) = \xi^{\text{in}} \hspace{1mm} \forall \hspace{1mm} t \in [0,s] \}) = 1$ and for all $l \in \mathbb{N}$, 
\begin{equation}
P ( \{ \xi \in \Omega_{0}: \hspace{0.5mm} \int_{0}^{l} \lVert G_{1} (\xi_{1} (r)) \rVert_{L_{2} (U_{1}, L_{\sigma}^{2})}^{2} + \lVert G_{2} (\xi_{2} (r)) \rVert_{L_{2} (U_{2}, \mathring{L}^{2} )}^{2} dr < \infty \} ) = 1, 
\end{equation} 
\item [](M2) for every $\psi_{i} = (\psi_{i}^{1}, \psi_{i}^{2}) \in C^{\infty} (\mathbb{T}^{n}) \cap L_{\sigma}^{2} \times C^{\infty} (\mathbb{T}^{n}) \cap \mathring{L}^{2}$ and $t \geq s$, the processes 
\begin{subequations}\label{[Equ. (14), Y20a]}
\begin{align}
M_{1,t,s}^{i} \triangleq& \langle \xi_{1}(t) - \xi_{1}(s), \psi_{i}^{1} \rangle + \int_{s}^{t} \langle \text{div} (\xi_{1}(r) \otimes \xi_{1}(r)) + (-\Delta)^{m} \xi_{1}(r) - \xi_{2}(r) e^{n}, \psi_{i}^{1} \rangle dr, \\
M_{2,t,s}^{i} \triangleq& \langle \xi_{2}(t) - \xi_{2}(s), \psi_{i}^{2} \rangle + \int_{s}^{t} \langle \text{div} (\xi_{1}(r)\xi_{2}(r)) -\Delta \xi_{2}(r), \psi_{i}^{2} \rangle dr,  
\end{align} 
\end{subequations}
are continuous, square-integrable $(\mathcal{B}_{t})_{t\geq s}$-martingales under $P$ such that  
\begin{equation}\label{estimate 302}
\langle \langle M_{k,t,s}^{i} \rangle \rangle = \int_{s}^{t} \lVert G_{k} (\xi_{k} (r))^{\ast} \psi_{i}^{k} \rVert_{U_{k}}^{2} dr, \hspace{5mm} k \in \{1,2\},
\end{equation} 
\item [](M3) for any $q \in \mathbb{N}$, there exists a function $t \mapsto C_{t,q} \in \mathbb{R}_{+}$ such that for all $t \geq s$, 
\begin{align}
&\mathbb{E}^{P} [ \sup_{r \in [0,t]} \lVert \xi_{1}(r) \rVert_{L_{x}^{2}}^{2q} + \int_{s}^{t} \lVert \xi_{1}(r) \rVert_{\dot{H}_{x}^{\gamma}}^{2} dr \nonumber \\
& \hspace{2mm} + \sup_{r \in [0,t]} \lVert \xi_{2}(r) \rVert_{L_{x}^{2}}^{2q} + \int_{s}^{t} \lVert \xi_{2} (r) \rVert_{\dot{H}_{x}^{1}}^{2} dr] \leq C_{t,q} (1+ \lVert \xi_{1}^{\text{in}} \rVert_{L_{x}^{2}}^{2q} + \lVert \xi_{2}^{\text{in}} \rVert_{L_{x}^{2}}^{2q}). \label{estimate 24}
\end{align}  
\end{enumerate} 
The set of all such martingale solutions with the same constant $C_{t,q}$ in \eqref{estimate 24} for every $q \in \mathbb{N}$ and $t \geq s$ will be denoted by $\mathcal{C} ( s, \xi^{\text{in}}, \{C_{t,q} \}_{q \in \mathbb{N}, t \geq s } )$. 
\end{define}
 
If $\{ \psi_{j}^{1}\}_{j=1}^{\infty}$ and $\{\psi_{j}^{2} \}_{j=1}^{\infty}$ are complete orthonormal systems of $L_{\sigma}^{2}$ and $\mathring{L}^{2}$ that consist of eigenvectors of $G_{1}G_{1}^{\ast}$ and $G_{2}G_{2}^{\ast}$, then 
\begin{equation}\label{estimate 19}
M_{k,t,s} \triangleq \sum_{j=1}^{\infty} M_{k,t,s}^{j} \psi_{j}^{k}, \hspace{5mm} k \in \{1,2\}, 
\end{equation}
becomes a $G_{k}G_{k}^{\ast}$-Wiener process starting from initial time $s$ w.r.t. $(\mathcal{B}_{t})_{t\geq s}$ under $P$, respectively. In order to define a martingale solution up to a stopping time $\tau: \hspace{0.5mm}  \Omega_{0} \mapsto [0,\infty]$, we define the space of trajectories stopped at time $\tau$ by 
\begin{equation}\label{estimate 305}
\Omega_{0,\tau} \triangleq \{ \omega (\cdot \wedge \tau(\omega)): \hspace{0.5mm} \omega \in \Omega_{0} \} = \{ \omega \in \Omega_{0}:\hspace{0.5mm}  \xi(t,\omega) = \xi(t\wedge \tau(\omega), \omega) \hspace{1mm} \forall \hspace{1mm} t \geq 0 \}. 
\end{equation} 

\begin{define}\label{Definition 4.2}
Let $s \geq 0, \xi^{\text{in}}= (\xi_{1}^{\text{in}}, \xi_{2}^{\text{in}}) \in L_{\sigma}^{2} \times \mathring{L}^{2}$ and $\tau \geq s$ be a stopping time of $(\mathcal{B}_{t})_{t\geq s}$. Then $P \in \mathcal{P} (\Omega_{0, \tau})$ is a martingale solution to \eqref{3} on $[s, \tau]$ with initial condition $\xi^{\text{in}}$ at initial time $s$ if 
\begin{enumerate}
\item [](M1) $P(\{ \xi(t) = \xi^{\text{in}} \hspace{1mm} \forall \hspace{1mm} t \in [0, s] \}) = 1$ and for all $l \in \mathbb{N}$, 
\begin{equation}
P( \{ \xi \in \Omega_{0}:\hspace{0.5mm}  \int_{0}^{l \wedge \tau} \lVert G_{1} (\xi_{1} (r)) \rVert_{L_{2} (U_{1}, L_{\sigma}^{2})}^{2} + \lVert G_{2} (\xi_{2} (r)) \rVert_{L_{2} (U_{2}, \mathring{L}^{2})}^{2} dr < \infty \}) = 1, 
\end{equation} 
\item [](M2) for every $\psi_{i} = (\psi_{i}^{1}, \psi_{i}^{2}) \in C^{\infty} (\mathbb{T}^{n}) \cap L_{\sigma}^{2} \times C^{\infty} (\mathbb{T}^{n}) \cap \mathring{L}^{2}$ and $t \geq s$, the processes 
\begin{align*}
M_{1, t\wedge \tau, s}^{i} \triangleq& \langle \xi_{1}(t\wedge \tau) - \xi_{1}^{\text{in}}, \psi_{i}^{1} \rangle + \int_{s}^{t \wedge \tau} \langle \text{div} (\xi_{1} (r) \otimes \xi_{1}(r)) + (-\Delta)^{m} \xi_{1}(r) - \xi_{2}(r) e^{n}, \psi_{i}^{1} \rangle dr, \\
M_{2, t \wedge \tau, s}^{i} \triangleq& \langle \xi_{2}(t\wedge \tau) - \xi_{2}^{\text{in}}, \psi_{i}^{2} \rangle + \int_{s}^{t\wedge \tau} \langle \text{div} (\xi_{1}(r) \xi_{2}(r)) - \Delta \xi_{2}(r), \psi_{i}^{2} \rangle dr,  
\end{align*} 
are continuous, square-integrable $(\mathcal{B}_{t})_{t\geq s}$-martingales under $P$ such that  
\begin{equation}
\langle \langle M_{k,t\wedge \tau,s}^{i} \rangle \rangle = \int_{s}^{t\wedge \tau} \lVert G_{k} (\xi_{k} (r))^{\ast} \psi_{i}^{k} \rVert_{U_{k}}^{2} dr, \hspace{5mm} k \in \{1,2\}, 
\end{equation} 
\item [](M3) for any $q \in \mathbb{N}$, there exists a function $t \mapsto C_{t,q} \in \mathbb{R}_{+}$ such that for all $t \geq s$ 
\begin{align}
&\mathbb{E}^{P} [ \sup_{r \in [0,t\wedge \tau]} \lVert \xi_{1}(r) \rVert_{L_{x}^{2}}^{2q} + \int_{s}^{t\wedge \tau} \lVert \xi_{1}(r) \rVert_{\dot{H}_{x}^{\gamma}}^{2} dr \nonumber \\
& \hspace{2mm} + \sup_{r \in [0,t\wedge \tau]} \lVert \xi_{2}(r) \rVert_{L_{x}^{2}}^{2q} + \int_{s}^{t\wedge \tau} \lVert \xi_{2} (r) \rVert_{\dot{H}_{x}^{1}}^{2} dr] \leq C_{t,q} (1+ \lVert \xi_{1}^{\text{in}} \rVert_{L_{x}^{2}}^{2q} + \lVert \xi_{2}^{\text{in}} \rVert_{L_{x}^{2}}^{2q}). 
\end{align}  
\end{enumerate} 
\end{define} 

First result concerns existence and stability of martingale solutions according to Definition \ref{Definition 4.1}. 
\begin{proposition}\label{Proposition 4.1}
For any $(s, \xi^{\text{in}}) \in [0,\infty) \times L_{\sigma}^{2} \times \mathring{L}^{2}$, there exists a martingale solution $P \in \mathcal{P} (\Omega_{0})$ to \eqref{3} with initial condition $\xi^{\text{in}}$ at initial time $s$ that satisfies Definition \ref{Definition 4.1}. Moreover, if there exists a family $\{ (s_{l}, \xi_{l}) \}_{l\in\mathbb{N}} \subset [0,\infty) \times L_{\sigma}^{2} \times \mathring{L}^{2}$ such that $\lim_{l\to\infty} \lVert (s_{l}, \xi_{l}) - (s, \xi^{\text{in}}) \rVert_{\mathbb{R} \times L_{\sigma}^{2}\times \mathring{L}^{2}} = 0$ and $P_{l} \in \mathcal{C} (s_{l}, \xi_{l}, \{C_{t,q} \}_{q\in\mathbb{N}, t \geq s_{l}})$ is the martingale solution corresponding to $(s_{l}, \xi_{l})$, then there exists a subsequence $\{P_{l_{k}} \}_{k \in \mathbb{N}}$ that converges weakly to some $P \in \mathcal{C} ( s,\xi^{\text{in}}, \{C_{t,q}\}_{q \in \mathbb{N}, t \geq s })$.  
\end{proposition} 
 
\begin{proof}[Proof of Proposition \ref{Proposition 4.1}]
For completeness, we sketch its proof in the Subsection \ref{Subsection 6.1} of Appendix. 
\end{proof} 
 
Proposition \ref{Proposition 4.1} leads to the following two results, which are only slight modifications of \cite[Pro. 3.2 and 3.4]{HZZ19} to which we refer interested readers for details. 
 
\begin{lemma}\label{Lemma 4.2}
(cf. \cite[Pro. 3.2]{HZZ19}) Let $\tau$ be a bounded stopping time of $(\mathcal{B}_{t})_{t\geq 0}$. Then, for every $\omega \in \Omega_{0}$, there exists $Q_{\omega} \triangleq \delta_{\omega} \otimes_{\tau(\omega)} R_{\tau(\omega), \xi(\tau(\omega), \omega)} \in \mathcal{P} (\Omega_{0})$ where $\delta_{\omega}$ is a point-mass at $\omega$ such that 
\begin{subequations}
\begin{align} 
& Q_{\omega} ( \{ \omega' \in \Omega_{0}:\hspace{0.5mm}  \xi(t, \omega') = \omega(t) \hspace{1mm} \forall \hspace{1mm} t \in [0, \tau(\omega) ] \}) = 1, \label{[Equ. (20a), Y20a]} \\
& Q_{\omega}(A) = R_{\tau (\omega), \xi(\tau(\omega), \omega)} (A) \hspace{1mm} \forall \hspace{1mm} A \in \mathcal{B}^{\tau(\omega)}, \label{[Equ. (20b), Y20a]}
\end{align}
\end{subequations} 
where $R_{\tau(\omega), \xi(\tau(\omega), \omega)} \in \mathcal{P}(\Omega_{0})$ is a martingale solution to \eqref{3} with initial condition $\xi(\tau(\omega), \omega)$ at initial time $\tau(\omega)$, and the mapping $\omega \mapsto Q_{\omega}(B)$ is $\mathcal{B}_{\tau}$-measurable for every $B \in \mathcal{B}$. 
\end{lemma} 

\begin{lemma}\label{Lemma 4.3}
(cf. \cite[Pro. 3.4]{HZZ19}) Let $\tau$ be a bounded stopping time of $(\mathcal{B}_{t})_{t\geq 0}$, $\xi^{\text{in}} \in L_{\sigma}^{2} \times \mathring{L}^{2}$, and $P$ be a martingale solution to \eqref{3} on $[0,\tau]$ with initial condition $\xi^{\text{in}}$ at initial time 0 that satisfies Definition \ref{Definition 4.2}. Suppose that there exists a Borel set $\mathcal{N} \subset \Omega_{0,\tau}$ such that $P(\mathcal{N}) = 0$ and $Q_{\omega}$ from Lemma \ref{Lemma 4.2} satisfies for every $\omega \in \Omega_{0} \setminus \mathcal{N}$ 
\begin{equation}\label{[Equ. (22), Y20a]}
Q_{\omega} (\{\omega' \in \Omega_{0}:\hspace{0.5mm}  \tau(\omega') = \tau(\omega) \}) = 1. 
\end{equation} 
Then the probability measure $P \otimes_{\tau}R \in \mathcal{P}(\Omega_{0})$ defined by 
\begin{equation}\label{[Equ. (23), Y20a]} 
P\otimes_{\tau} R (\cdot) \triangleq \int_{\Omega_{0}} Q_{\omega} (\cdot) P(d\omega) 
\end{equation} 
satisfies $P \otimes_{\tau}R \rvert_{\Omega_{0,\tau}} = P \rvert_{\Omega_{0,\tau}}$ and it is a martingale solution to \eqref{3} on $[0,\infty)$ with initial condition $\xi^{\text{in}}$ at initial time 0. 
\end{lemma} 

Now we let $\mathcal{B}_{\tau}$ represent the $\sigma$-algebra associated to the stopping time $\tau$ and consider 
\begin{subequations}
\begin{align}
& dz_{1} + (-\Delta)^{m} z_{1} dt + \nabla \pi_{1} dt = dB_{1}, \hspace{2mm} \nabla\cdot z_{1} = 0, \hspace{2mm} z_{1}(0, x) \equiv 0, \label{estimate 12}\\
& dz_{2} - \Delta z_{2} dt = dB_{2} \hspace{37mm}  \hspace{2mm} z_{2}(0,x) \equiv 0, \label{estimate 13}
\end{align}
\end{subequations} 
and 
\begin{subequations}\label{estimate 21}
\begin{align}
& \partial_{t} v + (-\Delta)^{m} v + \text{div} ((v+ z_{1}) \otimes (v + z_{1})) + \nabla \pi_{2} = (\Theta + z_{2}) e^{n}, \hspace{3mm} \nabla\cdot v = 0,  \\
& \partial_{t} \Theta - \Delta \Theta + \text{div} ((v+ z_{1}) (\Theta + z_{2})) = 0, 
\end{align}
\end{subequations}
so that $(u,\theta) = (v+ z_{1}, \Theta + z_{2})$ solves \eqref{3} with $\pi = \pi_{1} + \pi_{2}$. We fix $G_{k}G_{k}^{\ast}$-Wiener processes $B_{k}$ on $(\Omega, \mathcal{F}, \textbf{P})$ for both $k \in \{1,2\}$ with $(\mathcal{F}_{t})_{t\geq 0}$ as the canonical filtration of $(B_{1}, B_{2})$ augmented by all the $\textbf{P}$-negligible sets. We see that 
\begin{equation}\label{estimate 20}
z_{1}(t) = \int_{0}^{t} e^{- (t-r) (-\Delta)^{m}} \mathbb{P} dB_{1}(r), \hspace{3mm} z_{2}(t) = \int_{0}^{t} e^{(t-r) \Delta} dB_{2}(r), 
\end{equation} 
where $e^{- (-\Delta)^{m} t}$ and $e^{-\Delta t}$ are semigroups generated by $- (-\Delta)^{m}$ and $- \Delta$, respectively and we recall from Section \ref{Section 3} that $\mathbb{P}$ is a Leray projection operator.  Next, let us prove the following Proposition \ref{Proposition 4.4} for the full range of $m \in (0, \frac{5}{4})$ in case $n = 3$. 

\begin{proposition}\label{Proposition 4.4}
Consider 
\begin{equation}\label{estimate 138}
dz + (-\Delta)^{\lambda} z dt + \nabla \pi dt = dB, \hspace{1mm} \nabla\cdot z = 0 \text{ or } \hspace{1mm} dz + (-\Delta)^{\lambda} z dt = dB 
\end{equation} 
where $B$ is a $GG^{\ast}$-Wiener process and $\lambda \in (0, \frac{1}{2} + \frac{n}{4})$ for $n \in \{2,3\}$. Suppose that 
\begin{equation}\label{estimate 15}
\text{Tr} ((-\Delta)^{\max \{ \frac{n}{2} + 2 \sigma, \frac{n+2}{2} - \lambda + 2 \sigma \}} GG^{\ast}) < \infty
\end{equation}
for some $\sigma > 0$ where 
\begin{equation*}
\max\{ \frac{n}{2} + 2 \sigma, \frac{n+2}{2} - \lambda + 2 \sigma \} = 
\begin{cases}
\frac{n}{2} + 2 \sigma & \text{ if } \lambda \geq 1, \\
\frac{n+2}{2} - \lambda + 2 \sigma & \text{ if } \lambda < 1.
\end{cases}
\end{equation*}
Then, for all $\delta \in (0, \frac{1}{2})$, $T> 0$, and $l \in \mathbb{N}$,  
\begin{equation}\label{estimate 16} 
\mathbb{E}^{\textbf{P}}[ \lVert z \rVert_{C_{T} \dot{H}_{x}^{\frac{2+n+ \sigma}{2}}}^{l} + \lVert z \rVert_{C_{T}^{\frac{1}{2} - \delta} \dot{H}_{x}^{\frac{n+ \sigma}{2}}}^{l} ] <\infty. 
\end{equation} 
Consequently, if \eqref{4}-\eqref{5} hold, then $z_{1}$ and $z_{2}$ that solve respectively \eqref{estimate 12}-\eqref{estimate 13} satisfy \eqref{estimate 16}. 
\end{proposition}

\begin{proof}[Proof of Proposition \ref{Proposition 4.4}]
The consequence is clear for $z_{1}$ considering \eqref{4}-\eqref{5}; for $z_{2}$, \eqref{estimate 15} with $\lambda = 1$ is precisely \eqref{5} so that its claim also follows. The proof of \eqref{estimate 16} follows from that of \cite[Pro. 34]{D13}. In short, one can take $\alpha \in (0, \frac{3\sigma}{4\lambda})$ if $1 \geq \lambda $ and $\alpha \in (0, \frac{3\sigma}{4\lambda} + \frac{1}{2} - \frac{1}{2\lambda})$ otherwise, and then define 
\begin{equation}\label{estimate 314}
Y(s) \triangleq 
\begin{cases}
\frac{\sin(\pi \alpha)}{\pi} \int_{0}^{s} e^{- (-\Delta)^{\lambda} (s-r)} (s-r)^{-\alpha} \mathbb{P} d B(r) & \text{ if } \exists  \hspace{1mm} \nabla \pi \text{ in } \eqref{estimate 138}, \\
\frac{\sin(\pi \alpha)}{\pi} \int_{0}^{s} e^{- (-\Delta)^{\lambda} (s-r)} (s-r)^{-\alpha} d B(r) & \text{ if } \not\exists \hspace{1mm} \nabla \pi \text{ in } \eqref{estimate 138}
\end{cases} 
\end{equation}
(cf. \cite[Equ. (5.15)]{DZ14}). Then one can show that $\int_{0}^{t} (t-s)^{\alpha- 1} e^{- (-\Delta)^{\lambda}(t-s)} Y(s) ds= z (t)$ using \eqref{estimate 314} and our choice of $\alpha$ and $\mathbb{E}^{\textbf{P}} [ \lVert (-\Delta)^{\frac{2+ n + \sigma}{4}} Y(s) \rVert_{L_{x}^{2}}^{2l}] \lesssim_{l}  1$ using \eqref{estimate 15}. Integrating this inequality over $[0,T]$ and relying on Fubini's theorem give $\mathbb{E}^{\textbf{P}}[ \int_{0}^{T} \lVert (-\Delta)^{\frac{2+ n + \sigma}{4}} Y(s) \rVert_{L_{x}^{2}}^{2l} ds] \lesssim_{l} 1$, from which we can deduce $\mathbb{E}^{\textbf{P}}[\lVert (-\Delta)^{\frac{2+ n + \sigma}{4}} z(t) \rVert_{C_{T}L_{x}^{2}}^{2l}] \lesssim_{l} 1$. The second inequality in \eqref{estimate 16} can be proven similarly, and we refer to \cite[Pro. 34]{D13} and \cite[Pro. 4.4]{Y20c} for details. 
\end{proof} 
Next, for every $\omega = (\omega_{1}, \omega_{2}) \in \Omega_{0}$ we define 
\begin{subequations}\label{[Equ. (40a), Y20c]}
\begin{align}
& M_{1, t, 0}^{\omega} \triangleq \omega_{1}(t) - \omega_{1}(0) + \int_{0}^{t} \mathbb{P} \text{div} (\omega_{1}(r) \otimes \omega_{1}(r)) + (-\Delta)^{m} \omega_{1}(r) - \mathbb{P} \omega_{2}(r) e^{n} dr, \label{estimate 141} \\
& M_{2, t,0}^{\omega} \triangleq \omega_{2}(t) - \omega_{2}(0) + \int_{0}^{t} \text{div} (\omega_{1}(r) \omega_{2}(r)) - \Delta \omega_{2}(r) dr, 
\end{align}
\end{subequations}
and 
\begin{subequations}\label{[Equ. (40b), Y20c]}
\begin{align}
& Z_{1}^{\omega}(t) \triangleq M_{1, t,0}^{\omega} - \int_{0}^{t} \mathbb{P} (-\Delta)^{m} e^{- (t-r) (-\Delta)^{m}} M_{1, r,0}^{\omega} dr, \\
& Z_{2}^{\omega} (t) \triangleq M_{2,t,0}^{\omega}+\int_{0}^{t} \Delta e^{(t-r) \Delta } M_{2,r,0}^{\omega} dr. 
\end{align}
\end{subequations}
If $P$ is a martingale solution to \eqref{3}, then the mapping $\omega \mapsto M_{k,t,0}^{\omega}$ for both $k \in \{1,2\}$ is a $G_{k}G_{k}^{\ast}$-Wiener processes under $P$ and 
we can show 
\begin{equation}\label{[Equ. (30), Y20a]}
Z_{1}(t) = \int_{0}^{t} e^{-(t-r) (-\Delta)^{m}}\mathbb{P}  dM_{1,r,0} \hspace{1mm} \text{ and } \hspace{1mm} Z_{2}(t) =  \int_{0}^{t} e^{(t-r) \Delta} dM_{2,r,0}. 
\end{equation}
As $M_{k,t,0}^{\omega}$ is a $G_{k}G_{k}^{\ast}$-Wiener process under $P$ for both $k \in \{1,2\}$, Proposition \ref{Proposition 4.4} gives 
\begin{equation}
Z_{k} \in C_{T} \dot{H}_{x}^{\frac{2+ n + \sigma}{2}} \cap C_{T}^{\frac{1}{2} - \delta} \dot{H}_{x}^{\frac{n+\sigma}{2}}
\end{equation} 
$P$-a.s. for any $\delta \in (0, \frac{1}{2})$ and any $T > 0$. Now we define 
\begin{align}
\tau_{L}^{\lambda}(\omega) &\triangleq \inf\{ t \geq 0: \hspace{0.5mm} C_{S}\max_{k=1,2} \lVert Z_{k}^{\omega}(t) \rVert_{\dot{H}_{x}^{\frac{2+ n + \sigma}{2}}} > (L- \frac{1}{\lambda})^{\frac{1}{4}} \} \nonumber \\
& \wedge \inf\{t \geq 0: \hspace{0.5mm} C_{S} \max_{k= 1,2} \lVert Z_{k}^{\omega} \rVert_{C_{t}^{\frac{1}{2} - \delta} \dot{H}_{x}^{\frac{n+\sigma}{2}}} > (L - \frac{1}{\lambda})^{\frac{1}{2}} \} \wedge L, \hspace{5mm} \tau_{L}(\omega) \triangleq \lim_{\lambda \to \infty} \tau_{L}^{\lambda}(\omega) \label{[Equ. (31) and (32), Y20a]}
\end{align}
where $C_{S} > 0$ is a Sobolev constant such that $\lVert f \rVert_{L^{\infty} (\mathbb{T}^{n})} \leq C_{S} \lVert f \rVert_{\dot{H}^{\frac{n+\sigma}{2}}(\mathbb{T}^{n})}$ for all $f \in \dot{H}^{\frac{n+\sigma}{2}}(\mathbb{T}^{n})$ that is mean-zero so that $(\tau_{L}^{\lambda})_{\lambda \in \mathbb{N}}$ is non-decreasing in $\lambda$. It follows that $\tau_{L}$ is a $(\mathcal{B}_{t})_{t\geq 0}$-stopping time (see \cite[Lem. 3.5]{HZZ19}). Next, we assume Theorem \ref{Theorem 2.1} on a probability space $(\Omega, \mathcal{F}, (\mathcal{F}_{t})_{t\geq 0}, \textbf{P})$ and denote by $P$ the law of the solution $(u,\theta)$ constructed from Theorem \ref{Theorem 2.1}. 

\begin{proposition}\label{Proposition 4.5}
Let $\tau_{L}$ be defined by \eqref{[Equ. (31) and (32), Y20a]}. Then $P$, the law of $(u,\theta)$, is a martingale solution on $[0, \tau_{L}]$ according to Definition \ref{Definition 4.2}. 
\end{proposition}

\begin{proof}[Proof of Proposition \ref{Proposition 4.5}]
For completeness, we sketch the proof in the Subsection \ref{Subsection 6.4} of Appendix referring to \cite[Pro. 3.7]{HZZ19} for details. 
\end{proof} 

\begin{proposition}\label{Proposition 4.6}
Let $\tau_{L}$ be defined by \eqref{[Equ. (31) and (32), Y20a]} and $P = \mathcal{L} ((u,\theta))$ constructed from Theorem \ref{Theorem 2.1}. Then, $P \otimes_{\tau_{L}}R$ in \eqref{[Equ. (23), Y20a]} is a martingale solution on $[0,\infty)$ according to Definition \ref{Definition 4.1}. 
\end{proposition}
 
\begin{proof}[Proof of Proposition \ref{Proposition 4.6}]
For completeness, we sketch the proof in the Subsection \ref{Subsection 6.5} referring to \cite[Pro. 3.8]{HZZ19} for details. 
\end{proof}  
 
\begin{proof}[Proof of Theorem \ref{Theorem 2.2} assuming Theorem \ref{Theorem 2.1}]
We fix $T > 0$ arbitrary, $\kappa \in (0,1)$ and $K > 1$ such that $\kappa K^{2} \geq 1$, rely on Theorem \ref{Theorem 2.1} and Proposition \ref{Proposition 4.6} to deduce the existence of $L > 1$ and a martingale solution  $P \otimes_{\tau_{L}}R$ to \eqref{3} on $[0,\infty)$ such that $P \otimes_{\tau_{L}} R = P$ on $[0, \tau_{L}]$ where $P= \mathcal{L} (u,\theta)$ for a solution constructed in Theorem \ref{Theorem 2.1}. Hence, $P \otimes_{\tau_{L}} R$ starts with a deterministic initial condition $\xi^{\text{in}} = (u^{\text{in}}, \theta^{\text{in}})$ from the proof of Theorem \ref{Theorem 2.1} and satisfies 
\begin{equation}\label{estimate 139}
P \otimes_{\tau_{L}} R (\{ \tau_{L} \geq T \}) \overset{\eqref{[Equ. (23), Y20a]} \eqref{[Equ. (221), Y20d]}\eqref{[Equ. (212), Y20d]}}{=} \textbf{P} ( \{ T_{L} \geq T \})   \overset{\eqref{6}}{>} \kappa 
\end{equation} 
so that 
\begin{equation}
\mathbb{E}^{P \otimes_{\tau_{L}} R} [ \lVert \xi(T) \rVert_{L_{x}^{2}}^{2}] \overset{\eqref{estimate 139} \eqref{[Equ. (4), Y20c]}}{>} \kappa K^{2} e^{T} [\lVert \xi_{1}^{\text{in}} \rVert_{L_{x}^{2}}^{2} + \lVert \xi_{2}^{\text{in}} \rVert_{L_{x}^{2}}^{2} +  \text{Tr} (G_{1}G_{1}^{\ast}) +  \text{Tr} (G_{2}G_{2}^{\ast})]. 
\end{equation} 
It is well-known that a classical Galerkin approximation (e.g., \cite{FR08} and \cite[The. 4.2.4]{Z12} in case of a fractional Laplacian) gives another martingale solution $\mathcal{Q}$ such that 
\begin{equation*}
\mathbb{E}^{\mathcal{Q}} [ \lVert \xi(T) \rVert_{L_{x}^{2}}^{2}] \leq e^{T} [ \lVert \xi_{1}^{\text{in}} \rVert_{L_{x}^{2}}^{2} + \lVert \xi_{2}^{\text{in}} \rVert_{L_{x}^{2}}^{2} +  \text{Tr} ( G_{1}G_{1}^{\ast}) + \text{Tr} (G_{2}G_{2}^{\ast})],
\end{equation*} 
which implies non-uniqueness in law for \eqref{3}, completing the proof of Theorem \ref{Theorem 2.2}.
\end{proof} 
 
Considering  \eqref{estimate 21}, for $q  \in \mathbb{N}_{0}$ we aim to construct a solution $(v_{q}, \theta_{q}, \mathring{R}_{q})$ to 
\begin{subequations}\label{[Equ. (46), Y20c]}
\begin{align}
& \partial_{t} v_{q} + (-\Delta)^{m} v_{q} + \text{div} ((v_{q} + z_{1}) \otimes (v_{q} + z_{1})) + \nabla \pi_{q} = \theta_{q} e^{n} + \text{div} \mathring{R}_{q}, \hspace{3mm} \nabla\cdot v_{q} = 0,  \label{estimate 340}\\
& d\theta_{q} + [-\Delta \theta_{q} + \text{div} ((v_{q} + z_{1} ) \theta_{q} )] dt = dB_{2}, \label{estimate 323}
\end{align}
\end{subequations} 
where $\mathring{R}_{q}$ will be a trace-free symmetric matrix. For any $a > 0, b \in \mathbb{N}, \beta \in (0,1)$,  and $L \geq 1$ to be specified subsequently, we define 
\begin{equation}\label{[Equ. (47), Y20c]}
\lambda_{q} \triangleq a^{b^{q}}, \hspace{3mm} \delta_{q} \triangleq \lambda_{q}^{-2\beta}, \hspace{3mm} \text{ and } \hspace{3mm} M_{0}(t) \triangleq L^{4} e^{4Lt}. 
\end{equation} 
We see from \eqref{[Equ. (33), Y20a]} that for any $\delta \in (0, \frac{1}{12})$, $t \in [0, T_{L}]$, and both $k \in \{1, 2 \}$, 
\begin{equation}\label{[Equ. (38), Y20a]}
\lVert z_{k} (t) \rVert_{L_{x}^{\infty}} \leq L^{\frac{1}{4}}, \hspace{3mm} \lVert z_{k} (t) \rVert_{\dot{W}_{x}^{1,\infty}} \leq L^{\frac{1}{4}}, \hspace{3mm} \text{ and } \hspace{3mm} \lVert z_{k} \rVert_{C_{t}^{\frac{1}{2} - 2  \delta} L_{x}^{\infty}} \leq L^{\frac{1}{2}} 
\end{equation} 
by definition of $C_{S}$. Now if $a^{\beta b} > 3$ and $b\geq 2$, then $\sum_{1 \leq \iota \leq q} \delta_{\iota}^{\frac{1}{2}} < \frac{1}{2}$ for any $q \in \mathbb{N}$. Let us set a convention that $\sum_{1\leq \iota \leq 0} \triangleq 0$, denote by $c_{R} > 0$ a universal small constant to be described subsequently from the proof of Proposition \ref{Proposition 4.8 for n=2} in case $n = 2$ and Proposition \ref{Proposition 4.8 for n=3} in case $n=3$ (e.g., \eqref{[Equ. (83), Y20c]}-\eqref{[Equ. (84), Y20c]}, and \eqref{[Equ. (73), Y20a]}) and assume the following bounds over $t \in [0, T_{L}]$ inductively:  
\begin{subequations}\label{[Equ. (49), Y20c]}
\begin{align}
& \lVert v_{q} \rVert_{C_{t}L_{x}^{2}} \leq M_{0}(t)^{\frac{1}{2}} (1 + \sum_{1 \leq \iota \leq q} \delta_{\iota}^{\frac{1}{2}}) \leq 2 M_{0}(t)^{\frac{1}{2}}, \hspace{5mm} \lVert v_{q} \rVert_{C_{t,x}^{1}} \leq M_{0}(t)^{\frac{1}{2}} \lambda_{q}^{4}, \label{[Equ.  (49a) and (49b), Y20c]}\\
& \lVert \mathring{R}_{q} \rVert_{C_{t}L_{x}^{1}} \leq c_{R}M_{0}(t)  \delta_{q+1}, \label{[Equ. (49c), Y20c]}\\
& \mathbb{E}^{\textbf{P}} [ \lVert \theta_{q}(t \wedge T_{L} ) \rVert_{L_{x}^{2}}^{2} + 2 \int_{0}^{t\wedge T_{L} } \lVert \theta_{q} \rVert_{\dot{H}_{x}^{1}}^{2} dr] \leq   \lVert \theta_{q}(0) \rVert_{L_{x}^{2}}^{2} + \mathbb{E}^{\textbf{P}}[ (t\wedge T_{L} ) \text{Tr} ( G_{2}G_{2}^{\ast})]. \label{estimate 26}
\end{align} 
\end{subequations} 
We note that for all $q \in \mathbb{N}_{0}$ fixed, $v_{q}$ will be shown to have at least the regularity of $C_{t,x}^{1}$ as in \eqref{[Equ. (49a) and (49b), Y20c]} and hence $\theta_{q}$, considering \eqref{estimate 323}, will have sufficient regularity to apply It$\hat{\mathrm{o}}$'s formula (e.g., \cite{K10}), as we will in order to verify \eqref{estimate 26}. 

In what follows, we consider the case $n = 2$ first so that $m \in (0,1)$ by \eqref{4}, followed by the case $n = 3$. The main reason why we need to separate these two cases is because their settings of convex integration differ significantly. While we employ Mikado flows in case $n =3$, it is inapplicable in case $n =2$; in fact, a lack of suitable replacement of Mikado flows in the 2D case is precisely the reason why Isett's resolution of Onsager's conjecture was only for $n \geq 2$ (see \cite[p. 877]{I18}). In case $n=2$, we employ 2D intermittent stationary flows from \cite{CDS12, LQ20}. For notations and preliminaries hereafter in case $n=2$, we refer to Subsection \ref{Preliminaries needed for convex integration in 2D case and more}. We impose that $a \in 10 \mathbb{N}$ so that $\lambda_{q+1} \in 10 \mathbb{N} \subset 5 \mathbb{N}$ as required in \eqref{[Equ. (18), Y20c]}. 

\begin{proposition}\label{Proposition 4.7 for n=2}
Fix $\theta^{\text{in}} \in H^{2} (\mathbb{T}^{2})$ that is deterministic and mean-zero. Let 
\begin{equation}\label{[Equ. (50), Y20c]}
v_{0}(t,x) \triangleq \frac{L^{2} e^{2Lt}}{2\pi} 
\begin{pmatrix}
\sin(x^{2}) & 0 
\end{pmatrix}^{T}. 
\end{equation} 
Then there exists a unique solution $\theta_{0} \in L_{\omega}^{\infty} L_{t}^{\infty} H_{x}^{2}$ to the following linear stochastic PDE with additive noise: 
\begin{equation}\label{estimate 25}
d\theta_{0} + [-\Delta \theta_{0} + \text{div} ((v_{0} +z_{1})\theta_{0})] dt = dB_{2} \hspace{1mm} \text{ for } \hspace{1mm}  t > 0 \text{ and } \theta_{0}(0,x) = \theta^{\text{in}}(x)
\end{equation} 
where $z_{1}$ solves \eqref{estimate 12}. It follows that together with 
\begin{align}
\mathring{R}_{0} (t,x) \triangleq& \frac{L^{3} e^{2Lt}}{\pi} 
\begin{pmatrix}
0 & - \cos(x^{2}) \\
-\cos(x^{2}) & 0 
\end{pmatrix}  \nonumber \\
& + ( \mathcal{R} (-\Delta)^{m} v_{0} - \mathcal{R} \theta_{0} e^{2} + v_{0} \mathring{\otimes}z_{1} + z_{1} \mathring{\otimes} v_{0} + z_{1} \mathring{\otimes} z_{1})(t,x), \label{estimate 106}
\end{align} 
$(v_{0}, \theta_{0})$ satisfy \eqref{[Equ. (46), Y20c]} at level $q= 0$, where $\mathcal{R}$ is defined in Lemma \ref{[Def. 9, Lem. 10, CDS12]}. Moreover, \eqref{[Equ. (49), Y20c]} is satisfied at level $q= 0$ provided 
\begin{equation}\label{[Equ. (52), Y20c]}
\max \{ (18 \pi^{-1} \lVert \theta^{\text{in}} \rVert_{L_{x}^{2}})^{\frac{1}{3}}, (72 \pi)^{\frac{4}{7}} \} < L, \hspace{3mm} (51 \pi^{2}) 9 < (51 \pi^{2}) a^{2 \beta b} \leq c_{R} L \leq c_{R} (a^{4} \pi - 1),  
\end{equation} 
where the inequality of $9 < a^{2\beta b}$ is assumed for the justification of the second inequality of \eqref{[Equ.  (49a) and (49b), Y20c]}. Furthermore, $v_{0}(0,x)$ and $\mathring{R}_{0}(0,x)$ are both deterministic. 
\end{proposition}

\begin{proof}[Proof of Proposition \ref{Proposition 4.7 for n=2}]
First, we can immediately deduce 
\begin{equation}\label{[Equ. (53), Y20c]} 
\lVert v_{0} (t) \rVert_{L_{x}^{2}} \overset{\eqref{[Equ. (47), Y20c]} \eqref{[Equ. (50), Y20c]}}{=}  \frac{M_{0}(t)^{\frac{1}{2}}}{\sqrt{2}} \leq M_{0}(t)^{\frac{1}{2}}, \hspace{3mm} \lVert v_{0} \rVert_{C_{t,x}^{1}}   \overset{\eqref{[Equ. (47), Y20c]} \eqref{[Equ. (52), Y20c]}}{\leq} M_{0}(t)^{\frac{1}{2}} \lambda_{0}^{4}
\end{equation} 
(see  \cite[Equ. (53)]{Y20c}), which verifies \eqref{[Equ.  (49a) and (49b), Y20c]} at level $q=0$. The existence and uniqueness of solution $\theta_{0}$ to \eqref{estimate 25} is standard (e.g., \cite[Cha. 5]{DZ14}). Concerning the verification of \eqref{[Equ. (49c), Y20c]} at level $q = 0$, we will need to estimate $\lVert \mathcal{R} \theta_{0} e^{2}\rVert_{L_{x}^{1}}$ and $L_{\omega}^{p}$-estimate for $p < \infty$ will not suffice. Let us compute the following for general $q$ for subsequent convenience; the case $q = 0$ is our current case. We define $\Theta_{q} \triangleq \theta_{q} - z_{2}$ so that from \eqref{estimate 13} and \eqref{estimate 323}, we deduce 
\begin{equation}\label{estimate 40}
\partial_{t} \Theta_{q} - \Delta \Theta_{q} + (v_{q} + z_{1}) \cdot \nabla (\Theta_{q} + z_{2}) = 0, 
\end{equation}  
from which it can be computed by using divergence-free property of $v_{q}$ and $z_{1}$ that 
\begin{equation}\label{estimate 42}
\frac{1}{2} \partial_{t} \lVert \Theta_{q} \rVert_{L_{x}^{2}}^{2} + \lVert \Theta_{q} \rVert_{\dot{H}_{x}^{1}}^{2} = - \int_{\mathbb{T}^{2}}(v_{q} + z_{1}) \cdot \nabla z_{2} \Theta_{q} dx 
\leq ( \lVert v_{q} \rVert_{L_{x}^{2}} + \lVert z_{1} \rVert_{L_{x}^{2}}) \lVert z_{2} \rVert_{\dot{W}_{x}^{1,\infty}} \lVert \Theta_{q} \rVert_{L_{x}^{2}}.  
\end{equation} 
As $\frac{1}{2} \partial_{t} \lVert \Theta_{0} \rVert_{L_{x}^{2}}^{2} = \lVert \Theta_{0} \rVert_{L_{x}^{2}} \partial_{t} \lVert \Theta_{0} \rVert_{L_{x}^{2}}$, we obtain
\begin{equation}\label{estimate 43} 
\partial_{t} \lVert \Theta_{0} \rVert_{L_{x}^{2}} \leq (\lVert v_{0} \rVert_{L_{x}^{2}} + \lVert z_{1} \rVert_{L_{x}^{2}}) \lVert z_{2} \rVert_{\dot{W}_{x}^{1,\infty}} 
\overset{\eqref{[Equ. (53), Y20c]} \eqref{[Equ. (38), Y20a]} }{\leq} L^{\frac{9}{4}} e^{2Lt} 4\pi^{2}. 
\end{equation}
We integrate over $[0,t]$ to obtain 
\begin{equation}\label{estimate 315}
\lVert \Theta_{0}(t) \rVert_{L_{x}^{2}} \leq \lVert \theta^{\text{in}} \rVert_{L_{x}^{2}} + L^{\frac{5}{4}} 2 \pi^{2} e^{2Lt}.  
\end{equation} 
Applying \eqref{estimate 315} to \eqref{estimate 42} at level $q= 0$ shows that $\Theta_{0} \in L_{\omega}^{\infty} L_{t}^{\infty} L_{x}^{2} \cap L_{\omega}^{\infty} L_{t}^{2} H_{x}^{1}$ and a straight-forward bootstrap estimate using \eqref{[Equ. (33), Y20a]} and  \eqref{[Equ. (53), Y20c]}, as well as the fact that $\Delta v_{0} = - v_{0}$  show that $\Theta_{0} \in L_{\omega}^{\infty} L_{t}^{\infty} H_{x}^{2} \cap L_{\omega}^{\infty} L_{t}^{2} H_{x}^{3}$. Considering \eqref{[Equ. (33), Y20a]}, this implies that $\theta_{0} \in L_{\omega}^{\infty} L_{t}^{\infty} H_{x}^{2}$ as claimed. Next, the fact that \eqref{estimate 340} at level $q  =0$ with $\pi = - (v_{0} \cdot z_{1} + \frac{1}{2} \lvert z_{1} \rvert^{2})$   is satisfied can be proven immediately. Next, we realize that clearly $v_{0}$ is divergence-free and mean-zero so that $(-\Delta)^{m} v_{0}$ is also mean-zero while $\theta_{0}(t)$ is mean-zero for all $t \geq 0$; consequently, $\mathcal{R}(-\Delta)^{m} v_{0}$ and $\mathcal{R}\theta_{0}e^{2}$ are both trace-free and symmetric due to Lemma \ref{[Def. 9, Lem. 10, CDS12]}, leading to $\mathring{R}_{0}$ being trace-free and symmetric. Next, by Lemma \ref{[Def. 9, Lem. 10, CDS12]},  
\begin{equation}\label{estimate 335}
\lVert \mathcal{R} \theta_{0} e^{2} \rVert_{C_{t}L_{x}^{1}} \overset{\eqref{estimate 315}}{\leq} 6 \pi ( \lVert \theta^{\text{in}} \rVert_{L_{x}^{2}} + L^{\frac{5}{4}} e^{2Lt} 2 \pi^{2} + 2 \pi \lVert z_{2} \rVert_{C_{t}L_{x}^{\infty}})  
\overset{ \eqref{[Equ. (38), Y20a]}\eqref{[Equ. (52), Y20c]}}{\leq} \pi^{2} M_{0}(t) L^{-1}. 
\end{equation} 
On the other hand, we know from \cite[Equ. (56)]{Y20c} that 
\begin{equation}\label{estimate 322}
 \lVert \mathring{R}_{0} + \mathcal{R} \theta_{0}e^{2} \rVert_{C_{t}L_{x}^{1}} 
\leq 16 L M_{0}(t)^{\frac{1}{2}} + 8 \pi M_{0}(t)^{\frac{1}{2}} + 2 0 \pi M_{0}(t)^{\frac{1}{2}} L^{\frac{1}{4}} + 5 (2\pi)^{2} L^{\frac{1}{2}}. 
\end{equation} 
Directly due to \eqref{[Equ. (47), Y20c]} we can bound \eqref{estimate 322} furthermore by $50 \pi^{2} M_{0}(t) L^{-1}$ so that 
\begin{align*}
\lVert \mathring{R}_{0}(t,x) \rVert_{C_{t}L_{x}^{1}} \overset{\eqref{estimate 322}}{\leq}& 50 \pi^{2} M_{0}(t) L^{-1} + \lVert \mathcal{R} \theta_{0} e^{2} \rVert_{C_{t}L_{x}^{1}} \overset{\eqref{estimate 335}}{\leq} 51 \pi^{2} M_{0}(t) L^{-1}  \overset{\eqref{[Equ. (52), Y20c]}}{\leq} c_{R}M_{0}(t) \delta_{1}, 
\end{align*} 
which verifies \eqref{[Equ. (49c), Y20c]} at level $q = 0$. On the other hand, as $v_{0}$ and $z_{1}$ are both bounded, it is clear from \eqref{estimate 25} that \eqref{estimate 26} holds at level $q = 0$. Indeed, computing for general $q \in \mathbb{N}_{0}$ for subsequent convenience, standard computations on \eqref{estimate 323} give  
\begin{align}
&\lVert \theta_{q}(t \wedge T_{L} ) \rVert_{L_{x}^{2}}^{2} + 2 \int_{0}^{t \wedge T_{L} } \lVert \theta_{q} \rVert_{\dot{H}_{x}^{1}}^{2} dr \nonumber\\
\leq& \lVert \theta^{\text{in}} \rVert_{L_{x}^{2}}^{2} + \int_{0}^{t\wedge T_{L} } \int_{\mathbb{T}^{2}} 2 \theta_{q}  d B_{2}(s)dx  + (t\wedge T_{L} ) \text{Tr} (G_{2}G_{2}^{\ast});\label{estimate 30}
\end{align}
thus, taking expectation $\mathbb{E}^{\textbf{P}}$ leads to \eqref{estimate 26} at level $q = 0$. Finally, it is clear that $v_{0}(0,x)$ is deterministic. As $\theta^{\text{in}}$ is deterministic and $z_{1}(0,x) \equiv 0$ by \eqref{estimate 12}, so is $\mathring{R}_{0}(0,x)$. 
\end{proof} 

\begin{proposition}\label{Proposition 4.8 for n=2}
Fix $\theta^{\text{in}} \in H^{2}(\mathbb{T}^{2})$ that is deterministic and mean-zero from the hypothesis of Proposition \ref{Proposition 4.7 for n=2}. Let $L$ satisfy 
\begin{equation}\label{estimate 318}
L > \max\{ ( 18 \pi^{-1} \lVert \theta^{\text{in}} \rVert_{L_{x}^{2}})^{\frac{1}{3}}, (72 \pi)^{\frac{4}{7}}, c_{R}^{-1}459 \pi^{2} \} 
\end{equation}
and suppose that $(v_{q}, \theta_{q}, \mathring{R}_{q})$ are $(\mathcal{F}_{t})_{t\geq 0}$-adapted processes that solve \eqref{[Equ. (46), Y20c]} and satisfy \eqref{[Equ. (49), Y20c]}. Then there exist a choice of parameters $a, b,$ and $\beta$ such that \eqref{[Equ. (52), Y20c]} is fulfilled and $(\mathcal{F}_{t})_{t\geq 0}$-adapted processes $(v_{q+1}, \theta_{q+1}, \mathring{R}_{q+1})$ that satisfy \eqref{[Equ. (46), Y20c]}, \eqref{[Equ. (49), Y20c]} at level $q+1$, and for all $t \in [0, T_{L}]$ and $p \in [1, \infty)$, 
\begin{subequations}
\begin{align}
& \lVert v_{q+1} (t) - v_{q}(t) \rVert_{L_{x}^{2}} \leq M_{0}(t)^{\frac{1}{2}} \delta_{q+1}^{\frac{1}{2}}, \label{[Equ. (57), Y20c]}\\
& \mathbb{E}^{\textbf{P}} [ \lVert \theta_{q+1} - \theta_{q} \rVert_{C_{t  \wedge T_{L} }L_{x}^{2}}^{2p} + ( \int_{0}^{t \wedge T_{L}} \lVert \theta_{q+1} - \theta_{q} \rVert_{\dot{H}_{x}^{1}}^{2} dr)^{p} ] \lesssim_{p, \lVert \theta^{\text{in}} \rVert_{L_{x}^{2}}, \text{Tr} (G_{2}G_{2}^{\ast}), L} \lambda_{q+1}^{- \frac{ 8 \beta p}{8+ \beta}}. \label{estimate 22}
\end{align}
\end{subequations} 
Finally, if $v_{q}(0,x)$ and $\mathring{R}_{q}(0,x)$ are deterministic, then so are $v_{q+1}(0,x)$ and $\mathring{R}_{q+1}(0,x)$. 
\end{proposition}

\begin{proof}[Proof of Theorem \ref{Theorem 2.1} if $n = 2$ assuming Proposition \ref{Proposition 4.8 for n=2}]
Fix $\theta^{\text{in}} \in H^{2}(\mathbb{T}^{2})$ that is deterministic and mean-zero from the hypothesis of Proposition \ref{Proposition 4.7 for n=2}. Given $T > 0, K > 1$, and $\kappa \in (0,1)$, starting from $(v_{0}, \theta_{0}, \mathring{R}_{0})$ in Proposition \ref{Proposition 4.7 for n=2}, Proposition \ref{Proposition 4.8 for n=2} gives us $(v_{q}, \theta_{q}, \mathring{R}_{q})$ for all $q \geq 1$ that satisfies \eqref{[Equ. (46), Y20c]} and \eqref{[Equ. (49), Y20c]}. Then, for all $\gamma \in (0, \frac{\beta}{4+ \beta})$ and $t \in [0, T_{L}]$, by Gagliardo-Nirenberg's inequality, we can deduce 
\begin{equation}\label{estimate 29}
\sum_{q \geq 0} \lVert v_{q+1}(t) - v_{q}(t) \rVert_{\dot{H}_{x}^{\gamma}} \overset{ \eqref{[Equ. (57), Y20c]}}{\lesssim} \sum_{q \geq 0} (M_{0}(t)^{\frac{1}{2}} \delta_{q+1}^{\frac{1}{2}})^{1-\gamma} ( \lVert v_{q+1} \rVert_{C_{t,x}^{1}} + \lVert v_{q} \rVert_{C_{t,x}^{1}})^{\gamma} 
\overset{ \eqref{[Equ. (49a) and (49b), Y20c]}}{\lesssim} M_{0}(t)^{\frac{1}{2}}.
\end{equation} 
Thus, $\{v_{q}\}_{q=0}^{\infty}$ is Cauchy in $C_{T_{L}} \dot{H}^{\gamma}(\mathbb{T}^{2})$ and we deduce $\lim_{q\to \infty} v_{q} = v \in C([0, T_{L}]; \dot{H}^{\gamma}(\mathbb{T}^{2}))$ for which there exists a deterministic constant $C_{L,1} > 0$ such that 
\begin{equation}\label{[Equ. (59), Y20c]}
\sup_{t \in [0, T_{L}]} \lVert v(t) \rVert_{\dot{H}_{x}^{\gamma}} \leq C_{L,1}. 
\end{equation} 
As each $v_{q}$ is $(\mathcal{F}_{t})_{t\geq 0}$-adapted, $v$ is also $(\mathcal{F}_{t})_{t\geq 0}$-adapted. Next, a standard computation on \eqref{estimate 323} using the fact that $v_{q} + z_{1}$ is divergence-free so that $\int_{\mathbb{T}^{n}} (v_{q} + z_{1}) \cdot \nabla \theta_{q} \lvert \theta_{q} \rvert^{p-2} \theta_{q} dx = 0$ and Burkholder-Davis-Gundy inequality (e.g., \cite[p. 166]{KS91}), shows that for all $q \in \mathbb{N}_{0}$, 
\begin{align}
\mathbb{E}^{\textbf{P}} [ \lVert \theta_{q} \rVert_{C_{t \wedge T_{L} }L_{x}^{p}}^{p}] \lesssim& \lVert \theta^{\text{in}} \rVert_{L_{x}^{p}}^{p} + \sqrt{L} \text{Tr} ((-\Delta)^{\frac{n}{2} + 2 \sigma} G_{2}G_{2}^{\ast})^{\frac{1}{2}} \mathbb{E}^{\textbf{P}} [ \lVert \theta_{q} \rVert_{C_{t\wedge T_{L}}L_{x}^{p-1}}^{p-1}] \nonumber \\
&+ p (p-1) \text{Tr} ((-\Delta)^{\frac{n}{2} + 2 \sigma} G_{2}G_{2}^{\ast}) L \mathbb{E}^{\textbf{P}} [ \lVert \theta_{q} \rVert_{C_{t\wedge T_{L}}L_{x}^{p-2}}^{p-2} ];   \label{estimate 114}
\end{align} 
we chose to state this for general $n \in \{2,3\}$ for subsequent convenience with our current case being $n = 2$.  Continuing from \eqref{estimate 42}, we can show 
\begin{align}
\lVert \Theta_{q} (t) \rVert_{L_{x}^{2}} \leq& \lVert \theta^{\text{in}} \rVert_{L_{x}^{2}} + \int_{0}^{t} (\lVert v_{q} \rVert_{L_{x}^{2}} + 2\pi \lVert z_{1} \rVert_{L_{x}^{\infty}}) \lVert z_{2} \rVert_{\dot{W}_{x}^{1,\infty}} dr \nonumber\\
\overset{\eqref{[Equ. (38), Y20a]} \eqref{[Equ.  (49a) and (49b), Y20c]}}{\leq}& \lVert \theta^{\text{in}} \rVert_{L_{x}^{2}} + t [ 2 M_{0}(t)^{\frac{1}{2}} + 2 \pi L^{\frac{1}{4}}]L^{\frac{1}{4}}. \label{estimate 319}
\end{align} 
Along with $\lVert z_{2} \rVert_{C_{t}L_{x}^{\infty}} \leq 2\pi L^{\frac{1}{4}}$ due to \eqref{[Equ. (38), Y20a]}, we deduce for all $q \in \mathbb{N}_{0}$ 
\begin{equation}\label{estimate 317}
\lVert \theta_{q} \rVert_{C_{t}L_{x}^{2}} \leq \lVert \Theta_{q} \rVert_{C_{t}L_{x}^{2}} + \lVert z_{2} \rVert_{C_{t}L_{x}^{2}} \leq    \lVert \theta^{\text{in}} \rVert_{L_{x}^{2}} + t [ 2 M_{0}(t)^{\frac{1}{2}} + 2 \pi L^{\frac{1}{4}}]L^{\frac{1}{4}} + 2 \pi L^{\frac{1}{4}}. 
\end{equation}
Starting from \eqref{estimate 317} and using \eqref{estimate 114}, inductively we can now conclude that for all $q \in \mathbb{N}_{0}$, $\theta_{q} \in L_{\omega}^{p} C_{t}L_{x}^{p}$, $p \in [1, \infty)$. This allows one to interpolate and use \eqref{estimate 22} so that for any $p \in [1, \infty)$ fixed, 
\begin{align}  
\mathbb{E}^{\textbf{P}} [ \lVert \theta_{q+1} - \theta_{q} \rVert_{C_{t \wedge T_{L}} L_{x}^{p}}^{p}] \leq&   ( \mathbb{E}^{\textbf{P}} [ \sup_{r \in [0,t \wedge T_{L}]} \lVert \theta_{q+1}(r) - \theta_{q}(r) \rVert_{L_{x}^{2p-2}}^{2p-2} ])^{\frac{1}{2}} (\mathbb{E}^{\textbf{P}} [ \sup_{r \in [0,t \wedge T_{L}]} \lVert \theta_{q+1}(r) - \theta_{q}(r) \rVert_{L_{x}^{2}}^{2}])^{\frac{1}{2}} \nonumber\\
&\overset{\eqref{estimate 22}}{\lesssim}_{p, \lVert \theta^{\text{in}} \rVert_{H_{x}^{2}}, \text{Tr} (G_{2}G_{2}^{\ast}), \text{Tr} ((-\Delta)^{1+ 2 \sigma} G_{2}G_{2}^{\ast}), L} \lambda_{q+1}^{- \frac{4\beta}{8+\beta}}. \label{estimate 338} 
\end{align} 
Thus, we conclude from \eqref{estimate 22} that $\{\theta_{q} \}_{q =0}^{\infty}$ is Cauchy in $\cap_{p \in [1,\infty)} L_{\omega}^{p} C_{T_{L}} L_{x}^{p} \cap L_{\omega}^{p} L_{T_{L}}^{2} \dot{H}_{x}^{1}$ so that we have $\lim_{q\to\infty} \theta_{q} \triangleq \theta \in \cap_{p\in [1,\infty)} L_{\omega}^{p} C_{T_{L}} L_{x}^{p} \cap L_{\omega}^{p} L_{T_{L}}^{2} \dot{H}_{x}^{1}$ for which there exists a deterministic constant $C_{L,2} = C_{L,2} (p) > 0$ for $p \in [1,\infty)$ such that 
\begin{equation}\label{estimate 320}
\mathbb{E}^{\textbf{P}} [ \lVert \theta \rVert_{C_{T_{L}} L_{x}^{p}}^{p} + \lVert \theta \rVert_{L_{T_{L}}^{2} \dot{H}_{x}^{1}}^{p}] \leq C_{L,2}, 
\end{equation} 
which verifies the second inequality of \eqref{estimate 17}.  As each $\theta_{q}$ is $(\mathcal{F}_{t})_{t\geq 0}$-adapted, so is $\theta$. Finally, for all $t \in [0,T_{L}]$, $\lVert \mathring{R}_{q} \rVert_{C_{t} L_{x}^{1}}  \overset{\eqref{[Equ. (49c), Y20c]}}{\leq} c_{R} M_{0}(t) \delta_{q+1} \to 0$ as $q\to \infty$. Thus, $u = v + z_{1}$ and $\theta$ solve \eqref{3}. Now for $c_{R} > 0$ that is determined from the proof of Proposition \ref{Proposition 4.8 for n=2}, we choose $L  = L(T, K, c_{R}, G_{1}, G_{2}, \lVert \theta^{\text{in}} \rVert_{L_{x}^{2}}, \lVert u^{\text{in}} \rVert_{L_{x}^{2}})$ that satisfies \eqref{estimate 318} to be larger if necessary to satisfy 
\begin{subequations}\label{[Equ. (60), Y20c]}
\begin{align} 
&\frac{3}{2} + \frac{1}{L} < (\frac{1}{\sqrt{2}} - \frac{1}{2}) e^{LT} \text{ or equivalently } (\frac{3}{2} M_{0}(0)^{\frac{1}{2}} + L) e^{LT} < (\frac{1}{\sqrt{2}} - \frac{1}{2}) M_{0}(T)^{\frac{1}{2}}, \label{estimate 332} \\
&L^{\frac{1}{4}} 2 \pi + K e^{\frac{T}{2}} (\lVert \theta^{\text{in}} \rVert_{L_{x}^{2}}+ \sum_{l=1}^{2} \sqrt{ \text{Tr} (G_{l}G_{l}^{\ast})}) \leq (e^{LT} - K e^{\frac{T}{2}}) \lVert u^{\text{in}} \rVert_{L_{x}^{2}} + L e^{LT}, \label{estimate 333}  
\end{align} 
\end{subequations} 
where $u^{\text{in}}(x) = v(0,x)$ as $z_{1}(0,x)  \overset{\eqref{estimate 12}}{\equiv} 0$. Because $\lim_{L\to\infty} T_{L} = + \infty$ $\textbf{P}$-a.s. due to \eqref{estimate 16} and \eqref{[Equ. (33), Y20a]}, for the fixed $T >0$ and $\kappa > 0$, increasing $L$ sufficiently larger if necessary gives \eqref{6}. Next, as $z_{1}(t)$ is clearly $(\mathcal{F}_{t})_{t\geq 0}$-adapted, we see that $u$ is $(\mathcal{F}_{t})_{t\geq 0}$-adapted. Moreover, \eqref{[Equ. (38), Y20a]} and \eqref{[Equ. (59), Y20c]} imply the first inequality of \eqref{estimate 17}. Next, we can compute for all $t \leq T_{L}$ 
\begin{equation}\label{[Equ. (61), Y20c]}
\lVert v(t) - v_{0}(t) \rVert_{L_{x}^{2}}  \overset{\eqref{[Equ. (57), Y20c]}}{\leq} M_{0}(t)^{\frac{1}{2}} \sum_{q \geq 0} \delta_{q+1}^{\frac{1}{2}} \overset{\eqref{[Equ. (47), Y20c]} \eqref{[Equ. (52), Y20c]}}{<} M_{0}(t)^{\frac{1}{2}} (\frac{1}{2})
\end{equation} 
(see \cite[Equ. (61)]{Y20c}). We can also deduce $(\lVert v(0) \rVert_{L_{x}^{2}} + L) e^{LT} < \lVert v(T) \rVert_{L_{x}^{2}}$ due to \eqref{[Equ. (53), Y20c]}, \eqref{[Equ. (60), Y20c]}-\eqref{[Equ. (61), Y20c]} (see \cite[Equ. (62)]{Y20c}).  Therefore, on $\{T_{L} \geq T \}$, 
\begin{align}
\lVert u(T) \rVert_{L_{x}^{2}} \geq& \lVert v(T) \rVert_{L_{x}^{2}} - \lVert z_{1}(T) \rVert_{L_{x}^{2}} > (\lVert v(0) \rVert_{L_{x}^{2}} + L) e^{LT} - \lVert z_{1}(T) \rVert_{L_{x}^{\infty}} 2\pi \label{estimate 142}\\
\overset{\eqref{estimate 12} \eqref{[Equ. (38), Y20a]}}{\geq}& ( \lVert u^{\text{in}} \rVert_{L_{x}^{2}} + L) e^{LT} - L^{\frac{1}{4}} 2 \pi \overset{\eqref{[Equ. (60), Y20c]}}{\geq} K e^{\frac{T}{2}} (\lVert u^{\text{in}} \rVert_{L_{x}^{2}} + \lVert \theta^{\text{in}} \rVert_{L_{x}^{2}} + \sum_{l=1}^{2} \sqrt{ \text{Tr} (G_{l}G_{l}^{\ast})}), \nonumber 
\end{align} 
which verifies \eqref{[Equ. (4), Y20c]}. At last, because $v_{0}(0,x)$ from Proposition \ref{Proposition 4.7 for n=2} is deterministic, Proposition \ref{Proposition 4.8 for n=2} implies that $v(0,x)$ remains deterministic; as $z_{1}(0,x) \equiv 0$ by \eqref{estimate 12}, we conclude that $u^{\text{in}}$ is deterministic. 
\end{proof} 

\subsection{Convex integration to prove Proposition \ref{Proposition 4.8 for n=2}}

\subsubsection{Choice of parameters}
We fix 
\begin{equation}\label{[Equ. (64), Y20c]}
m^{\ast} \triangleq 2 m- 1 \text{ if } m \in (\frac{1}{2}, 1) \text{ while } m^{\ast} \triangleq 0 \text{ if } m \in (0, \frac{1}{2}] 
\end{equation}
so that $m^{\ast} \in [0, 1)$. Furthermore, we fix $L$ that satisfies \eqref{estimate 318}, 
\begin{equation}\label{[Equ. (65), Y20c]}
\eta \in \mathbb{Q}_{+} \cap ( \frac{1-m^{\ast}}{16}, \frac{1-m^{\ast}}{8}] 
\end{equation} 
from which we see that $\eta \in (0, \frac{1}{8}]$, and 
\begin{equation}\label{[Equ. (66), Y20c]}
\alpha \triangleq \frac{1-m}{400}. 
\end{equation} 
We set 
\begin{equation}\label{[Equ. (67), Y20c]}
r \triangleq \lambda_{q+1}^{1- 6 \eta}, \hspace{2mm} \mu \triangleq \lambda_{q+1}^{1- 4 \eta}, \hspace{2mm} \text{ and } \hspace{2mm} \sigma \triangleq \lambda_{q+1}^{2\eta - 1} 
\end{equation} 
so that the condition of $1 \ll r \ll \mu \ll \sigma^{-1} \ll \lambda_{q+1}$ from \eqref{[Equ. (18), Y20c]} is satisfied as $\eta \leq \frac{1}{8}$. Moreover, for the $\alpha$ fixed in \eqref{[Equ. (66), Y20c]}, we can choose 
\begin{equation}\label{estimate 62}
b \in \{ \iota \in \mathbb{N}: \hspace{0.5mm} \iota > \frac{16}{\alpha} \} 
\end{equation}
such that $r \in \mathbb{N}$ and $\lambda_{q+1} \sigma \in 10 \mathbb{N}$ so that the conditions of $r \in \mathbb{N}$ and $\lambda_{q+1}, \lambda_{q+1} \sigma \in 5 \mathbb{N}$ from \eqref{[Equ. (18), Y20c]} are fulfilled. For the $\alpha$ in \eqref{[Equ. (66), Y20c]} and $b$ in \eqref{estimate 62} fixed, we can take $\beta$ sufficiently small so that 
\begin{equation}\label{[Equ. (68), Y20c]} 
\alpha > 96 \beta b. 
\end{equation} 
We also choose 
\begin{equation}\label{[Equ. (69), Y20c]} 
l \triangleq \lambda_{q+1}^{- \frac{3\alpha}{2}} \lambda_{q}^{-2} 
\end{equation} 
that has an immediate consequence of 
\begin{equation}\label{[Equ. (70), Y20c]}
l \lambda_{q}^{4} \overset{\eqref{estimate 62}}{\leq} \lambda_{q+1}^{-\alpha} \text{ and } l^{-1} \overset{\eqref{estimate 62}}{\leq} \lambda_{q+1}^{2\alpha}
\end{equation} 
by taking $a \in 10 \mathbb{N}$ sufficiently large. Concerning \eqref{[Equ. (52), Y20c]}, by \eqref{estimate 318} we have $L > (18 \pi^{-1} \lVert \theta^{\text{in}} \rVert_{L_{x}^{2}})^{\frac{1}{3}}$ and $L > (72 \pi)^{\frac{4}{7}}$ and choosing $a \in 10 \mathbb{N}$ sufficiently large gives $c_{R} L \leq c_{R} (a^{4} \pi - 1)$ while $\beta > 0$ sufficiently small gives  $(51 \pi^{2}) 9 < 51 \pi^{2} a^{2 \beta b} \leq c_{R} L$. Thus, we consider such $m^{\ast}, \eta, \alpha, b$, and $l$ fixed, preserving our freedom to take $a \in 10 \mathbb{N}$ larger and $\beta > 0$ smaller as needed. 

\subsubsection{Mollification}
We let $\{ \phi_{\epsilon} \}_{\epsilon > 0}$ and $\{ \varphi_{\epsilon}\}_{\epsilon > 0}$, specifically $\phi_{\epsilon} (\cdot) \triangleq  \frac{1}{\epsilon^{2}} \phi(\frac{\cdot}{\epsilon})$ and $\varphi_{\epsilon}(\cdot) \triangleq \frac{1}{\epsilon} \varphi(\frac{\cdot}{\epsilon})$ respectively, be families of standard mollifiers on $\mathbb{R}^{2}$ and $\mathbb{R}$ with mass one where the latter is compactly supported on $\mathbb{R}_{+}$. Then we mollify $v_{q}, \theta_{q}, \mathring{R}_{q}$, and $z_{k}$ in space and time to obtain for both $k \in \{1,2\}$, 
\begin{equation}\label{[Equ. (71), Y20c]}
v_{l} \triangleq (v_{q} \ast_{x} \phi_{l}) \ast_{t} \varphi_{l}, \hspace{1mm} \theta_{l} \triangleq (\theta_{q} \ast_{x} \phi_{l}) \ast_{t} \varphi_{l}, \hspace{1mm} \mathring{R}_{l} \triangleq (\mathring{R}_{q} \ast_{x} \phi_{l}) \ast_{t} \varphi_{l}, \hspace{1mm} z_{k,l} \triangleq (z_{k} \ast_{x} \phi_{l}) \ast_{t} \varphi_{l}. 
\end{equation} 
Then we see that 
\begin{align}\label{[Equ. (72), Y20c]}
\partial_{t} v_{l} + (-\Delta)^{m} v_{l} + \text{div} ((v_{l} + z_{1,l}) \otimes (v_{l} + z_{1,l}) ) + \nabla \pi_{l}= \theta_{l} e^{2} + \text{div} (\mathring{R}_{l} + R_{\text{com1}})
\end{align} 
if we define 
\begin{subequations}\label{[Equ. (73), Y20c]}
\begin{align}
\pi_{l} \triangleq& (\pi_{q} \ast_{x} \phi_{l}) \ast_{t} \varphi_{l} - \frac{1}{2} ( \lvert v_{l} + z_{1,l} \rvert^{2} - (\lvert v_{q} + z_{1} \rvert^{2} \ast_{x} \phi_{l}) \ast_{t} \varphi_{l}), \label{[Equ. (73a), Y20c]}\\
R_{\text{com1}} \triangleq& R_{\text{commutator1}} \triangleq (v_{l} + z_{1,l}) \mathring{\otimes} (v_{l} + z_{1,l}) - (((v_{q} +z_{1}) \mathring{\otimes} (v_{q} + z_{1}))\ast_{x} \phi_{l})\ast_{t} \varphi_{l}. \label{[Equ. (73b), Y20c]}
\end{align} 
\end{subequations} 
For all $t \in [0, T_{L}]$ and $N \geq 1$, using \eqref{[Equ. (68), Y20c]} and taking $a \in 10 \mathbb{N}$ sufficiently large we have  
\begin{subequations}
\begin{align}
& \lVert v_{q} - v_{l} \rVert_{C_{t}L_{x}^{2}} \overset{\eqref{[Equ. (49a) and (49b), Y20c]} \eqref{[Equ. (70), Y20c]}}{\leq} \frac{1}{4} M_{0}(t)^{\frac{1}{2}} \delta_{q+1}^{\frac{1}{2}},  \label{[Equ. (74a), Y20c]}\\
&\lVert v_{l} \rVert_{C_{t}L_{x}^{2}} \overset{\eqref{[Equ. (49a) and (49b), Y20c]}}{\leq} M_{0}(t)^{\frac{1}{2}} (1+ \sum_{1 \leq \iota \leq q} \delta_{\iota}^{\frac{1}{2}}), \hspace{10mm}  \lVert v_{l} \rVert_{C_{t,x}^{N}} \overset{\eqref{[Equ. (49a) and (49b), Y20c]} \eqref{estimate 62} \eqref{[Equ. (69), Y20c]} }{\leq} l^{-N }M_{0}(t)^{\frac{1}{2}} \lambda_{q+1}^{-\alpha}  \label{[Equ. (74b) and (74c), Y20c]}
\end{align}
\end{subequations} 
(see  \cite[Equ. (74)]{Y20c}). 

\subsubsection{Perturbation}
We let $\chi$ be a smooth function such that 
\begin{equation}\label{[Equ. (75), Y20c]}
\chi(z) \triangleq 1 \text{ if } z \in [0,1], \hspace{1mm}  \hspace{1mm} z \leq 2 \chi(z) \leq 4z \text{ for } z \in (1,2), \hspace{1mm} \text{ and } \hspace{1mm} \chi(z) = z \text{ if } z \in [2,\infty). 
\end{equation} 
We define for $t \in [0, T_{L}]$ and $\omega \in \Omega$, 
\begin{equation}\label{[Equ. (76), Y20c]}
\rho(\omega, t, x) \triangleq 4c_{R} \delta_{q+1} M_{0}(t) \chi((c_{R} \delta_{q+1} M_{0}(t))^{-1} \lvert \mathring{R}_{l} (\omega, t, x) \rvert ) 
\end{equation} 
for which it follows that 
\begin{equation}\label{[Equ. (77), Y20c]} 
\lvert \frac{ \mathring{R}_{l} (\omega, t, x) }{\rho(\omega, t ,x) } \rvert \overset{\eqref{[Equ. (75), Y20c]} \eqref{[Equ. (76), Y20c]}}{\leq} \frac{1}{2}, 
\end{equation} 
which is useful in deriving \eqref{estimate 150}. For any $p \in [1, \infty],$ $t \in [0, T_{L}]$, and $N \geq 0$, we have 
\begin{equation}\label{[Equ. (78) and (79), Y20c]}
\lVert \rho \rVert_{C_{t}L_{x}^{p}} \overset{\eqref{[Equ. (75), Y20c]}}{\leq} 12 (( 4 \pi^{2})^{\frac{1}{p}} c_{R} \delta_{q+1} M_{0}(t) + \lVert \mathring{R}_{l}  \rVert_{C_{t}L_{x}^{p}}), \hspace{1mm} \lVert \mathring{R}_{l} \rVert_{C_{t,x}^{N}} \overset{\eqref{[Equ. (49c), Y20c]}}{\lesssim} l^{-N-3}  M_{0}(t) c_{R} \delta_{q+1} 
\end{equation} 
(see \cite[Equ. (78)-(79)]{Y20c}). Moreover, for any $N \geq 0$ and $t \in [0, T_{L}]$, we can estimate 
\begin{equation}\label{[Equ. (80), Y20c]}
\lVert \rho \rVert_{C_{t}C_{x}^{N}} \overset{\eqref{[Equ. (78) and (79), Y20c]}}{\lesssim} c_{R} \delta_{q+1} M_{0}(t) l^{-3-N} \text{ and } \lVert \rho \rVert_{C_{t}^{1}C_{x}^{k}} \overset{\eqref{[Equ. (78) and (79), Y20c]}}{\lesssim} c_{R} \delta_{q+1} M_{0}(t) l^{-4 (k+1)}, \hspace{1mm} k \in \{0,1,2\}
\end{equation} 
(see \cite[Equ. (80)]{Y20c}). Next, we recall $\gamma_{\zeta}$ from Lemma \ref{[Lemma 4.1, LQ20]} and define the amplitude function 
\begin{equation}\label{[Equ. (81), Y20c]} 
a_{\zeta} (\omega, t, x) \triangleq a_{\zeta, q+1} (\omega, t, x) \triangleq \rho(\omega, t, x)^{\frac{1}{2}} \gamma_{\zeta} (\frac{ \mathring{R}_{l} (\omega, t, x)}{\rho(\omega, t, x)}) 
\end{equation} 
that admits the following estimates: for all $t \in [0, T_{L}]$, $N \in \mathbb{N}_{0}$, and $k \in \{0,1,2\}$, with $C_{\Lambda}$ and $M$ from \eqref{[Equ. (15), Y20c]} by requiring $c_{R}^{\frac{1}{4}} \ll \frac{1}{M}$, 
\begin{subequations}\label{estimate 150}
\begin{align}
&\lVert a_{\zeta} \rVert_{C_{t}L_{x}^{2}} \overset{\eqref{[Equ. (77), Y20c]} \eqref{[Equ. (81), Y20c]}}{\leq} \lVert \rho \rVert_{C_{t} L_{x}^{1}}^{\frac{1}{2}} \lVert \gamma_{\zeta} \rVert_{C(B_{\frac{1}{2}} (0))} \overset{\eqref{[Equ. (15), Y20c]} \eqref{[Equ. (49c), Y20c]}  \eqref{[Equ. (78) and (79), Y20c]}  }{\leq} \frac{c_{R}^{\frac{1}{4}} M_{0}(t)^{\frac{1}{2}} \delta_{q+1}^{\frac{1}{2}}}{2 \lvert \Lambda \rvert}, \label{[Equ. (83), Y20c]}\\
& \lVert a_{\zeta} \rVert_{C_{t}C_{x}^{N}} \overset{\eqref{[Equ. (77), Y20c]}-\eqref{[Equ. (81), Y20c]} }{\leq} c_{R}^{\frac{1}{4}} \delta_{q+1}^{\frac{1}{2}} M_{0}(t)^{\frac{1}{2}} l^{- \frac{3}{2} - 4N}, \hspace{1mm} \lVert a_{\zeta} \rVert_{C_{t}^{1}C_{x}^{k}} \overset{ \eqref{[Equ. (77), Y20c]}-\eqref{[Equ. (81), Y20c]}}{\leq} c_{R}^{\frac{1}{4}} \delta_{q+1}^{\frac{1}{2}} M_{0}(t)^{\frac{1}{2}} l^{-(k+1) 4} \label{[Equ. (84), Y20c]}
\end{align} 
\end{subequations} 
(see \cite[Equ. (83)-(84)]{Y20c}). Next, we recall $\psi_{\zeta}, \eta_{\zeta}, \mathbb{W}_{\zeta}$, $\Lambda^{+}$, $\Lambda^{-}$, and $\Lambda$, respectively from \eqref{[Equ. (12), Y20c]}, \eqref{[Equ. (19), Y20c]}, \eqref{[Equ. (21), Y20c]}, and \eqref{[Equ. (10), Y20c]} and define the perturbation as 
\begin{equation}\label{[Equ. (85), Y20c]}
w_{q+1} \triangleq w_{q+1}^{(p)} + w_{q+1}^{(c)} + w_{q+1}^{(t)} \text{ and } v_{q+1} \triangleq v_{l} + w_{q+1} 
\end{equation} 
where 
\begin{equation}\label{[Equ. (86), Y20c]}
w_{q+1}^{(p)} \triangleq \sum_{\zeta \in \Lambda} a_{\zeta} \mathbb{W}_{\zeta},  w_{q+1}^{(c)} \triangleq \sum_{\zeta \in \Lambda} \nabla^{\bot} (a_{\zeta} \eta_{\zeta}) \psi_{\zeta},  w_{q+1}^{(t)} \triangleq \mu^{-1} ( \sum_{\zeta \in \Lambda^{+}} - \sum_{\zeta \in \Lambda^{-}}) \mathbb{P} \mathbb{P}_{\neq 0} (a_{\zeta}^{2} \mathbb{P}_{\neq 0} \eta_{\zeta}^{2} \zeta). 
\end{equation} 
It follows that $w_{q+1}$ is both divergence-free and mean-zero (see \cite[Equ. (87)]{Y20c} for details). For all $t \in [0, T_{L}]$ and $p \in (1,\infty)$, by relying on \cite[Lem. 6.2]{LQ20} we have the estimates of 
\begin{subequations}
\begin{align}
& \lVert w_{q+1}^{(p)} \rVert_{C_{t}L_{x}^{2}} \overset{ \eqref{[Equ. (24a) and (24b), Y20c]} \eqref{[Equ. (70), Y20c]} \eqref{[Equ. (83), Y20c]}}{\lesssim} c_{R}^{\frac{1}{4}} \delta_{q+1}^{\frac{1}{2}} M_{0}(t)^{\frac{1}{2}}, \hspace{4mm}  \lVert w_{q+1}^{(p)}\rVert_{C_{t}L_{x}^{p}} \overset{\eqref{[Equ. (24a) and (24b), Y20c]} \eqref{[Equ. (84), Y20c]} }{\lesssim} \delta_{q+1}^{\frac{1}{2}} M_{0}(t)^{\frac{1}{2}} l^{-\frac{3}{2}} r^{1- \frac{2}{p}}, \label{[Equ. (88) and (89a), Y20c]} \\
& \lVert w_{q+1}^{(c)} \rVert_{C_{t}L_{x}^{p}} \overset{ \eqref{[Equ. (13b), Y20c]} \eqref{[Equ. (86), Y20c]}}{\lesssim} \delta_{q+1}^{\frac{1}{2}} M_{0}(t)^{\frac{1}{2}} l^{- \frac{11}{2}} \sigma r^{2 - \frac{2}{p}}, \hspace{1mm} \lVert w_{q+1}^{(t)} \rVert_{C_{t}L_{x}^{p}} \overset{\eqref{[Equ. (24a) and (24b), Y20c]} }{\lesssim} \mu^{-1} \delta_{q+1} M_{0}(t) l^{-3} r^{2 - \frac{2}{p}} \label{[Equ. (89b) and (89c), Y20c]}
\end{align}
\end{subequations} 
(see \cite[Equ. (88)-(89)]{Y20c}). In turn, these estimates in \eqref{[Equ. (88) and (89a), Y20c]}-\eqref{[Equ. (89b) and (89c), Y20c]} lead to, for all $t \in [0, T_{L}]$ and $p \in (1, \infty)$, 
\begin{equation}\label{[Equ. (90) and (91), Y20c]}
\lVert w_{q+1}^{(c)} \rVert_{C_{t}L_{x}^{p}} + \lVert w_{q+1}^{(t)} \rVert_{C_{t}L_{x}^{p}} \overset{\eqref{[Equ. (70), Y20c]} }{\lesssim} \delta_{q+1} M_{0}(t) l^{-3} r^{2- \frac{2}{p}} \lambda_{q+1}^{4 \eta -1}, \hspace{2mm} \lVert w_{q+1} \rVert_{C_{t}L_{x}^{2}} \overset{\eqref{[Equ. (70), Y20c]}}{\leq} \frac{3}{4} \delta_{q+1}^{\frac{1}{2}} M_{0}(t)^{\frac{1}{2}}
\end{equation}
(see \cite[Equ. (90)-(91)]{Y20c}). By applying \eqref{[Equ. (85), Y20c]}, \eqref{[Equ. (74a), Y20c]}-\eqref{[Equ. (74b) and (74c), Y20c]}, and \eqref{[Equ. (90) and (91), Y20c]}, we are able to deduce both the first inequality of \eqref{[Equ. (49a) and (49b), Y20c]} at level $q+1$ and \eqref{[Equ. (57), Y20c]} (see \cite[Equ. (92)]{Y20c}).  For norms of higher order, we can compute   
\begin{subequations}\label{estimate 151}
\begin{align}
& \lVert w_{q+1}^{(p)} \rVert_{C_{t,x}^{1}} \overset{  \eqref{[Equ. (84), Y20c]} \eqref{[Equ. (70), Y20c]}}{\lesssim} \delta_{q+1}^{\frac{1}{2}} M_{0}(t)^{\frac{1}{2}} \lambda_{q+1}^{3- 14 \eta} l^{-\frac{3}{2}},  \lVert w_{q+1}^{(c)} \rVert_{C_{t,x}^{1}} \overset{\eqref{[Equ. (13b), Y20c]} \eqref{[Equ. (84), Y20c]}}{\lesssim} \delta_{q+1}^{\frac{1}{2}} M_{0}(t)^{\frac{1}{2}} \lambda_{q+1}^{3- 18 \eta} l^{-\frac{3}{2}}, \label{[Equ. (93a) and (93b), Y20c]} \\
& \lVert w_{q+1}^{(t)} \rVert_{C_{t,x}^{1}} \overset{\eqref{[Equ. (86), Y20c]} \eqref{[Equ. (70), Y20c]} \eqref{[Equ. (84), Y20c]}}{\lesssim} \lambda_{q+1}^{3- 16 \eta + \alpha} \delta_{q+1} M_{0}(t) l^{-3} \label{[Equ. (94), Y20c]}
\end{align}
\end{subequations} 
(see \cite[Equ. (93)-(94)]{Y20c}). Taking advantage of \eqref{[Equ. (65), Y20c]}, \eqref{[Equ. (66), Y20c]}, \eqref{[Equ. (70), Y20c]}, \eqref{[Equ. (74b) and (74c), Y20c]}, \eqref{[Equ. (85), Y20c]}, \eqref{estimate 151} allows us to conclude that the second inequality of \eqref{[Equ. (49a) and (49b), Y20c]} at level $q+1$ holds (see \cite[Equ. (95)]{Y20c}). Next, with $\theta^{\text{in}} \in H^{2}(\mathbb{T}^{2})$ from hypothesis fixed, and $v_{q+1} = v_{l}+ w_{q+1}$ already constructed from \eqref{[Equ. (85), Y20c]}, we deduce the unique solution $\theta_{q+1}$ to the linear transport-diffusion equation with additive noise \eqref{estimate 323} starting from $\theta_{q+1}(0,x) = \theta^{\text{in}}(x)$, which can be shown to satisfy \eqref{estimate 26} identically to \eqref{estimate 30}. Concerning \eqref{estimate 22}, we see that $\theta_{q+1} - \theta_{q}$ satisfies 
\begin{equation}\label{estimate 27}
\partial_{t} (\theta_{q+1} - \theta_{q}) - \Delta (\theta_{q+1} - \theta_{q}) + (v_{q+1} + z_{1}) \cdot \nabla (\theta_{q+1} - \theta_{q}) + (v_{q+1} - v_{q}) \cdot \nabla \theta_{q} = 0; 
\end{equation} 
fortunately, the noise canceled out because it is only additive. Therefore, we obtain 
\begin{equation}\label{estimate 28}
\frac{1}{2} \partial_{t} \lVert \theta_{q+1} - \theta_{q} \rVert_{L_{x}^{2}}^{2} + \lVert \theta_{q+1} - \theta_{q} \rVert_{\dot{H}_{x}^{1}}^{2}  = \int_{\mathbb{T}^{2}} (v_{q+1} - v_{q}) \cdot \nabla (\theta_{q+1} - \theta_{q}) \theta_{q}dx. 
\end{equation} 
\begin{remark}\label{Remark 4.2}
From \eqref{estimate 28} we want to deduce the Cauchy type bound of \eqref{estimate 22} by taking advantage of \eqref{[Equ. (57), Y20c]} that we already proved. Thus, we need a bound that consists of $\lVert v_{q+1} - v_{q} \rVert_{L_{x}^{2}}$. However, we point out that if we bound the right hand side of \eqref{estimate 28} by 
\begin{equation*}
\frac{1}{2} \lVert \theta_{q+1} - \theta_{q} \rVert_{\dot{H}_{x}^{1}}^{2}+ \frac{1}{2} \lVert v_{q+1} - v_{q} \rVert_{L_{x}^{2}}^{2} \lVert\theta_{q} \rVert_{L_{x}^{\infty}}^{2}, 
\end{equation*} 
then it becomes hopeless to obtain the Cauchy type bound of \eqref{estimate 22} because we cannot handle $\lVert \theta_{q} \rVert_{L_{x}^{\infty}}$. We do have a bound of $L_{\omega}^{2}L_{t}^{2} \dot{H}_{x}^{1}$-bound by inductive hypothesis; yet, $\dot{H}^{1}(\mathbb{T}^{2}) \hookrightarrow L^{\infty}(\mathbb{T}^{2})$ is also false. The break here is to take advantage of interpolation inequality similarly to \eqref{estimate 29}; i.e., we will give up on $\lVert v_{q+1} - v_{q} \rVert_{L_{x}^{2}}$ and compromise to $\lVert v_{q+1} - v_{q} \rVert_{L_{x}^{2}}^{\epsilon} \lVert v_{q+1} - v_{q} \rVert_{\dot{H}_{x}^{1}}^{1-\epsilon}$ for some $\epsilon > 0$ where the difficult term $\lVert v_{q+1} - v_{q} \rVert_{\dot{H}_{x}^{1}}^{1-\epsilon}$ must be handled by the second inequality of \eqref{[Equ. (49a) and (49b), Y20c]} at the level of $q+1$, which has fortunately for us already been proven. This will allow us to handle $\theta_{q}$ by bounding it with $\lVert \theta_{q} \rVert_{L_{x}^{p}}$ for arbitrarily large but finite $p$ and then relying on $H^{1}(\mathbb{T}^{2}) \hookrightarrow L^{p}(\mathbb{T}^{2})$ which holds for every $p \in [2, \infty)$. The new difficulty however is that the Cauchy bound from $\lVert v_{q+1} - v_{q} \rVert_{L_{x}^{2}}^{\epsilon}$  must be preserved despite $(M_{0}(t)^{\frac{1}{2}} \lambda_{q+1}^{4})^{1-\epsilon}$ that we expect from $\lVert v_{q+1} - v_{q} \rVert_{\dot{H}_{x}^{1}}^{1-\epsilon}$ where $\lambda_{q+1}$ is, of course, dangerously large; thus, it will be crucial to carefully choose such $\epsilon$. We will now choose $p =\frac{16 + 2\beta}{\beta}$ and $\epsilon = \frac{8}{8+\beta}$ and proceed.  

Before we do so, let us comment that this is certainly possible only because $n= 2$ as $\dot{H}^{1}(\mathbb{T}^{3}) \hookrightarrow L^{p}(\mathbb{T}^{3})$ is false for $p > 6$; one may be tempted to consider $(-\Delta)^{l} \theta$ for $l = \frac{3}{2}$ in \eqref{3b} instead of $-\Delta \theta$ and try to extend this argument in the 3D case; however, subsequently in \eqref{estimate 37}  we will need to consider an $L^{p^{\ast}}$-estimate of $\theta$ for $p^{\ast} \in (1,2)$ in which the positivity of $\int_{\mathbb{T}^{n}} (-\Delta)^{l} \theta \lvert \theta \rvert^{p^{\ast} -2} \theta \geq 0$ will be crucial (recall Remark \ref{Remark 2.2}). 
\end{remark}
We now compute from \eqref{estimate 28}
\begin{align}
\frac{1}{2} \partial_{t} \lVert \theta_{q+1} - \theta_{q} \rVert_{L_{x}^{2}}^{2} + \lVert \theta_{q+1} - \theta_{q} \rVert_{\dot{H}_{x}^{1}}^{2} \lesssim& \lVert v_{q+1} - v_{q} \rVert_{L_{x}^{2}}^{\frac{8}{8+\beta}} \lVert v_{q+1} - v_{q} \rVert_{\dot{H}_{x}^{1}}^{\frac{\beta}{8+\beta}} \lVert \theta_{q+1} - \theta_{q} \rVert_{\dot{H}_{x}^{1}} \lVert \theta_{q}\rVert_{\dot{H}_{x}^{1}} \nonumber\\
\overset{\eqref{[Equ. (57), Y20c]}\eqref{[Equ. (49a) and (49b), Y20c]} }{\leq}& \frac{1}{2} \lVert \theta_{q+1} - \theta_{q} \rVert_{\dot{H}_{x}^{1}}^{2} + C M_{0}(t) \lambda_{q+1}^{ - \frac{8\beta}{8+\beta}} \lVert \theta_{q}\rVert_{\dot{H}_{x}^{1}}^{2} \label{estimate 38}  
\end{align}  
by the embedding $H^{1}(\mathbb{T}^{2}) \hookrightarrow L^{\frac{16 + 2\beta}{\beta}} (\mathbb{T}^{2})$, Gagliardo-Nirenberg's and Young's inequalities. Integrating over $[0,t]$, taking supremum over $[0,t]$ and raising to the power $p \in [1, \infty)$ give
\begin{equation}\label{estimate 33}
\lVert \theta_{q+1} - \theta_{q} \rVert_{C_{t}L_{x}^{2}}^{2p} + (\int_{0}^{t} \lVert \theta_{q+1} - \theta_{q} \rVert_{\dot{H}_{x}^{1}}^{2} dr)^{p}  \lesssim \lambda_{q+1}^{- \frac{8 \beta p}{8+ \beta}} M_{0}(t)^{p} (\int_{0}^{t} \lVert \theta_{q} \rVert_{\dot{H}_{x}^{1}}^{2} ds)^{p}. 
\end{equation} 
We return to \eqref{estimate 30}, take supremum over $[0,t]$ on the right and then left sides, and then raise to the power of $p \in [1, \infty)$ to obtain for all $t \in [0, T_{L}]$,  
\begin{equation}\label{estimate 31}
\lVert \theta_{q} \rVert_{C_{t}L_{x}^{2}}^{2p} + ( \int_{0}^{t}  \lVert \theta_{q} \rVert_{\dot{H}_{x}^{1}}^{2} dr)^{p} 
\lesssim_{p}  \lVert \theta^{\text{in}} \rVert_{L_{x}^{2}}^{2p} + \sup_{r \in [0, t]}  \lvert \int_{0}^{r} \int_{\mathbb{T}^{2}} \theta_{q} dx dB_{2}(s) \rvert^{p} + t^{p} \text{Tr} (G_{2}G_{2}^{\ast})^{p}.
\end{equation} 
After taking expectation $\mathbb{E}^{\textbf{P}}$, standard applications of Burkholder-Davis-Gundy, H$\ddot{\mathrm{o}}$lder's and Young's inequalities lead us to  
\begin{equation}\label{estimate 32}
\mathbb{E}^{\textbf{P}} [ \sup_{r \in [0,t \wedge T_{L}]} \lvert \int_{0}^{r} \int_{\mathbb{T}^{2}} \theta_{q} dB_{s} \rvert^{p}]  \leq \frac{1}{2} \mathbb{E}^{\textbf{P}} [ \lVert \theta_{q} \rVert_{C_{t \wedge T_{L}}L_{x}^{2}}^{2p}] + C L^{p} \text{Tr} (G_{2}G_{2}^{\ast})^{p}. 
\end{equation} 
Applying \eqref{estimate 32} to \eqref{estimate 31} after taking expectation $\mathbb{E}^{\textbf{P}}$ and then subtracting $\frac{1}{2} \mathbb{E}^{\textbf{P}} [ \lVert \theta_{q} \rVert_{C_{t}L_{x}^{2}}^{2p}]$ give us for all $p \in [1, \infty)$ and $t \in [0, T_{L}]$,  
\begin{equation}\label{estimate 34}
\mathbb{E}^{\textbf{P}} [ \lVert \theta_{q} \rVert_{C_{t \wedge T_{L}}L_{x}^{2}}^{2p} ] +\mathbb{E}^{\textbf{P}} [ (\int_{0}^{t \wedge T_{L}} \lVert \theta_{q} \rVert_{\dot{H}_{x}^{1}}^{2} dr)^{p}] \lesssim_{p, \lVert \theta^{\text{in}} \rVert_{L_{x}^{2}}, \text{Tr} (G_{2}G_{2}^{\ast}), L} 1. 
\end{equation} 
At last, taking expectation $\mathbb{E}^{\textbf{P}}$ in \eqref{estimate 33} and applying \eqref{estimate 34} gives us  \eqref{estimate 22} as desired.

\subsubsection{Reynolds stress} 
We have due to \eqref{[Equ. (46), Y20c]}, \eqref{[Equ. (85), Y20c]}, and \eqref{[Equ. (72), Y20c]}, 
\begin{align}
& \text{div} \mathring{R}_{q+1} - \nabla \pi_{q+1} \label{[Equ. (98), Y20c]}\\
=& \underbrace{(-\Delta)^{m} w_{q+1} + \partial_{t} (w_{q+1}^{(p)} + w_{q+1}^{(c)}) + \text{div} ((v_{l} + z_{1, l}) \otimes w_{q+1} + w_{q+1} \otimes (v_{l} + z_{1, l})) + (\theta_{l} - \theta_{q+1}) e^{2} }_{\text{div} (R_{\text{lin}} ) + \nabla \pi_{\text{lin}}} \nonumber\\
&+ \underbrace{\text{div} (( w_{q+1}^{(c)} + w_{q+1}^{(t)}) \otimes w_{q+1} + w_{q+1}^{(p)} \otimes (w_{q+1}^{(c)} + w_{q+1}^{(t)}))}_{\text{div}(R_{\text{cor}}) + \nabla \pi_{\text{cor}}} + \underbrace{\text{div}(w_{q+1}^{(p)} \otimes w_{q+1}^{(p)} + \mathring{R}_{l}) + \partial_{t} w_{q+1}^{(t)}}_{\text{div}(R_{\text{osc}}) + \nabla \pi_{\text{osc}}} \nonumber \\
&+ \underbrace{\text{div} ( v_{q+1} \otimes z_{1} - v_{q+1} \otimes z_{1, l} + z_{1} \otimes v_{q+1} - z_{1,l} \otimes v_{q+1} + z_{1} \otimes z_{1} - z_{1,l} \otimes z_{1,l})}_{\text{div} (R_{\text{com2}} ) + \nabla \pi_{\text{com2}}} + \text{div} R_{\text{com1}} - \nabla \pi_{l} \nonumber 
\end{align} 
within which we specify 
\begin{subequations}\label{[Equ. (99), Y20c]}
\begin{align}
R_{\text{lin}} \triangleq& R_{\text{linear}} \triangleq \mathcal{R} ( -\Delta)^{m} w_{q+1}+ \mathcal{R} \partial_{t} (w_{q+1}^{(p)} + w_{q+1}^{(c)}) \nonumber\\
& \hspace{21mm} + (v_{l} + z_{1, l}) \mathring{\otimes} w_{q+1} + w_{q+1} \mathring{\otimes} (v_{l} + z_{1, l}) + \mathcal{R} (( \theta_{l} - \theta_{q+1}) e^{2}), \label{[Equ. (99a), Y20c]}\\
\pi_{\text{lin}} \triangleq& \pi_{\text{linear}} \triangleq (v_{l} + z_{1, l}) \cdot w_{q+1}, \label{[Equ. (99b), Y20c]}\\
 R_{\text{cor}} \triangleq& R_{\text{corrector}} \triangleq (w_{q+1}^{(c)} + w_{q+1}^{(t)} ) \mathring{\otimes} w_{q+1} + w_{q+1}^{(p)} \mathring{\otimes} (w_{q+1}^{(c)} + w_{q+1}^{(t)}), \label{[Equ. (99c), Y20c]}\\
 \pi_{\text{cor}} \triangleq & \pi_{\text{corrector}} \triangleq \frac{1}{2} [ (w_{q+1}^{(c)} + w_{q+1}^{(t)}) \cdot w_{q+1} + w_{q+1}^{(p)} \cdot (w_{q+1}^{(c)} + w_{q+1}^{(t)}) ], \label{[Equ. (99d), Y20c]}\\
R_{\text{com2}} \triangleq&  R_{\text{commutator2}} \triangleq v_{q+1} \mathring{\otimes} (z_{1} - z_{1,l}) + (z_{1} - z_{1,l}) \mathring{\otimes} v_{q+1} + z_{1} \mathring{\otimes} z_{1} - z_{1,l} \mathring{\otimes} z_{1,l},  \label{[Equ. (99e), Y20c]} \\
\pi_{\text{com2}} \triangleq&  \pi_{\text{commutator2}} \triangleq v_{q+1} \cdot (z_{1}- z_{1,l}) + \frac{1}{2} \lvert z_{1} \rvert^{2} - \frac{1}{2} \lvert z_{1,l} \rvert^{2}.  \label{[Equ. (99f), Y20c]}
\end{align}
\end{subequations}  
Concerning the explicit forms of $R_{\text{osc}}$ and $\pi_{\text{osc}}$ within \eqref{[Equ. (98), Y20c]}, we refer to \cite[Equ. (114)]{Y20c}. We set, along with $R_{\text{com1}}$ and $\pi_{l}$ from \eqref{[Equ. (73), Y20c]}
\begin{equation}\label{estimate 58}
\mathring{R}_{q+1} \triangleq R_{\text{lin}} + R_{\text{cor}} + R_{\text{osc}} + R_{\text{com2}} + R_{\text{com1}},  \hspace{2mm} \pi_{q+1} \triangleq \pi_{l} - \pi_{\text{lin}} - \pi_{\text{cor}} - \pi_{\text{osc}} - \pi_{\text{com2}}, 
\end{equation} 
and choose 
\begin{equation}\label{[Equ. (116), Y20c]}
p^{\ast} \triangleq \frac{16(1-6\eta)}{300 \alpha + 16 (1-7\eta)} \overset{\eqref{[Equ. (64), Y20c]}-\eqref{[Equ. (66), Y20c]}}{\in} (1,2). 
\end{equation} 
Within $R_{\text{lin}}$ we estimate 
\begin{equation}\label{estimate 63}
\lVert \mathcal{R} ((\theta_{l} - \theta_{q+1}) e^{2}) \rVert_{C_{t}L_{x}^{p^{\ast}}} \leq I + II 
\end{equation}
where 
\begin{equation}\label{estimate 36}
I \triangleq \lVert \mathcal{R} ((\theta_{q} - \theta_{q+1}) e^{2}) \rVert_{C_{t}L_{x}^{p^{\ast}}}, \text{ and } II \triangleq\lVert \mathcal{R} ((\theta_{l} - \theta_{q}) e^{2}) \rVert_{C_{t}L_{x}^{p^{\ast}}}.
\end{equation} 
First, we return to \eqref{estimate 27} and compute 
\begin{align}
&\frac{1}{p^{\ast}} \partial_{t} \lVert \theta_{q+1} - \theta_{q} \rVert_{L_{x}^{p^{\ast}}}^{p^{\ast}} +  (p^{\ast} - 1) \int_{\mathbb{T}^{2}} \lvert \nabla (\theta_{q+1} - \theta_{q}) \rvert^{2} \lvert \theta_{q+1} - \theta_{q} \rvert^{p^{\ast} - 2} dx  \nonumber \\
=& - \int_{\mathbb{T}^{2}} (v_{q+1} - v_{q}) \cdot \nabla \theta_{q} \lvert \theta_{q+1} - \theta_{q} \rvert^{p^{\ast} - 2} (\theta_{q+1} - \theta_{q}) dx. \label{estimate 37}  
\end{align} 
Let us comment on some difficulties in the following remark. 
\begin{remark}\label{another remark}
To make use of the diffusion, first natural idea to estimate \eqref{estimate 37} by 
\begin{align}
& - \int_{\mathbb{T}^{2}} (v_{q+1} - v_{q}) \cdot \nabla \theta_{q} \lvert \theta_{q+1} - \theta_{q} \rvert^{p^{\ast} -2} (\theta_{q+1} - \theta_{q}) dx   \label{estimate 140}\\
\leq& \frac{p^{\ast} -1}{2} \int_{\mathbb{T}^{2}} \lvert \nabla (\theta_{q+1} - \theta_{q}) \rvert^{2} \lvert \theta_{q+1} - \theta_{q} \rvert^{p^{\ast} -2} dx + \frac{p^{\ast} -1}{2} \int_{\mathbb{T}^{2}} \lvert v_{q+1} - v_{q} \rvert^{2} \lvert \theta_{q+1} - \theta_{q} \rvert^{p^{\ast} -2} \lvert \theta_{q} \rvert^{2} dx \nonumber 
\end{align} 
and hope to bound the second term by 
\begin{equation*}
\frac{p^{\ast} -1}{2} \lVert v_{q+1} - v_{q} \rVert_{L_{x}^{p^{\ast}}}^{2} \lVert \theta_{q+1} - \theta_{q} \rVert_{L_{x}^{p^{\ast}}}^{p^{\ast} -2} \lVert \theta_{q} \rVert_{L_{x}^{\infty}}^{2}. 
\end{equation*} 
However, we have two problems; we would not be able to handle $\lVert \theta_{q} \rVert_{L_{x}^{\infty}}^{2}$; more importantly, such H$\ddot{\mathrm{o}}$lder's inequality is not even allowed because $p^{\ast} - 2< 0$. The difficulty being that $p^{\ast} -2 < 0$, a second natural idea would be to estimate 
\begin{equation*}
\lVert \theta_{q+1} - \theta_{q} \rVert_{L_{x}^{p^{\ast}}} \leq (2\pi)^{2-p^{\ast}} \lVert \theta_{q+1} - \theta_{q} \rVert_{L_{x}^{2}} 
\end{equation*} 
by H$\ddot{\mathrm{o}}$lder's inequality and estimate following \eqref{estimate 37}-\eqref{estimate 140} to obtain  
\begin{align}
\frac{1}{2} \partial_{t} \lVert \theta_{q+1} - \theta_{q} \rVert_{L_{x}^{2}}^{2} + \lVert \nabla (\theta_{q+1} - \theta_{q}) \rVert_{L_{x}^{2}}^{2} \leq \frac{1}{2} \lVert \nabla (\theta_{q+1} - \theta_{q}) \rVert_{L_{x}^{2}}^{2} + \frac{1}{2} \lVert v_{q+1} - v_{q} \rVert_{L_{x}^{2}}^{2} \lVert \theta_{q} \rVert_{L_{x}^{\infty}}^{2}. 
\end{align} 
While this application of H$\ddot{\mathrm{o}}$lder's inequality is allowed, the difficulty of $\lVert \theta_{q} \rVert_{L_{x}^{\infty}}^{2}$ remains; more importantly, from $\lVert v_{q+1} - v_{q} \rVert_{L_{x}^{2}}^{2}$, we can only expect a bound of  $M_{0}(t) \delta_{q+1}$ due to \eqref{[Equ. (57), Y20c]}; however, \eqref{[Equ. (49c), Y20c]} at level of $q+1$ a bound by $c_{R}M_{0}(t) \delta_{q+2}$ where $\delta_{q+2} \ll \delta_{q+1}$. Third natural idea would be the interpolation similarly to \eqref{estimate 38}; i.e., we can try to estimate 
\begin{align*}
\int_{\mathbb{T}^{2}} (v_{q+1} -v_{q}) \cdot \nabla (\theta_{q+1} - \theta_{q}) \theta_{q}dx \leq& \frac{1}{2} \lVert \nabla (\theta_{q+1} - \theta_{q}) \rVert_{L_{x}^{2}}^{2} + C \lVert v_{q+1} - v_{q} \rVert_{L_{x}^{p_{1}}}^{2} \lVert \theta_{q} \rVert_{L_{x}^{p_{2}}} 
\end{align*}
where $\frac{1}{p_{1}} + \frac{1}{p_{2}} = \frac{1}{2}$, $p_{2} < \infty$, and rely on Sobolev embedding $H^{1}(\mathbb{T}^{2}) \hookrightarrow L^{p_{2}} (\mathbb{T}^{2})$, and interpolate on $\lVert v_{q+1} - v_{q} \rVert_{L_{x}^{p_{1}}}^{2}$ between $L^{r}$-norm for $r <2$ (for which we need to obtain a bound akin to $\delta_{q+2}$) and $C_{t,x}^{1}$-norm (for which we need to rely on the second inequality of \eqref{[Equ. (49a) and (49b), Y20c]} at the level $q+1$). We attempted this approach but failed to close this argument. 
\end{remark} 

Our last approach is to give up on taking advantage of the diffusion and go ahead with $L^{p^{\ast}}$-norm bound instead of $L^{2}$-norm bound from \eqref{estimate 37} to obtain 
\begin{equation}\label{estimate 113}
\frac{1}{p^{\ast}} \partial_{t} \lVert \theta_{q+1} - \theta_{q} \rVert_{L_{x}^{p^{\ast}}}^{p^{\ast}}  \leq \lVert v_{q+1} - v_{q} \rVert_{L_{x}^{p^{\ast}}} \lVert \nabla \theta_{q} \rVert_{L_{x}^{\infty}} \lVert \theta_{q+1} - \theta_{q} \rVert_{L_{x}^{p^{\ast}}}^{p^{\ast} -1}. 
\end{equation} 
Using the fact that $\frac{1}{p^{\ast}} \partial_{t} \lVert \theta_{q+1} - \theta_{q} \rVert_{L_{x}^{p^{\ast}}}^{p^{\ast}} = \lVert \theta_{q+1} - \theta_{q} \rVert_{L_{x}^{p^{\ast}}}^{p^{\ast} -1} \partial_{t} \lVert \theta_{q+1} - \theta_{q} \rVert_{L_{x}^{p^{\ast}}}$, integrating over $[0,t]$, and taking supremum over $[0,t]$ on the right and then left hand sides give for all $t \in [0, T_{L}]$  
\begin{equation}\label{estimate 49}
\lVert \theta_{q+1}- \theta_{q} \rVert_{C_{t}L_{x}^{p^{\ast}}} \leq \lVert v_{q+1} - v_{q} \rVert_{C_{t}L_{x}^{p^{\ast}}} \int_{0}^{t} \lVert  \theta_{q} \rVert_{\dot{W}_{x}^{1,\infty}} dr. 
\end{equation} 
Now we need to estimate $\lVert v_{q+1} - v_{q} \rVert_{C_{t}L_{x}^{p^{\ast}}}$ and $\int_{0}^{t} \lVert \theta_{q} \rVert_{\dot{W}_{x}^{1,\infty}} dr$. The need to obtain a new estimate on  $\int_{0}^{t} \lVert  \theta_{q} \rVert_{\dot{W}_{x}^{1,\infty}} dr$ seems a daunting task; however, it becomes fortunately possible as follows. The idea is that we will go ahead and employ $H^{2}(\mathbb{T}^{2})$-estimate on $\Theta_{q}$ from \eqref{estimate 40} so that the diffusive term gives us $L_{t}^{2} \dot{H}_{x}^{3}$ which, together with \eqref{[Equ. (38), Y20a]}, is more than enough to bound $L_{t}^{1}\dot{W}_{x}^{1,\infty}$-norm of $\theta_{q}$. Let us make the following remark. 

\begin{remark}\label{Remark 4.4}
In an $H_{x}^{2}$-estimate of $\Theta_{q}$ in \eqref{estimate 40}, classical examples (e.g., \cite{BKM84}) warn us that it will depend on $L_{t}^{1}\dot{W}_{x}^{1,\infty}$-bound of $(v_{q} + z_{1})$. While \eqref{[Equ. (38), Y20a]} takes care of the $L_{t}^{1} \dot{W}_{x}^{1,\infty}$-bound on $z_{1}$, such a bound on $v_{q}$ still seems too large at first sight. The break here is that it is a bound on $v_{q}$, and not $v_{q+1}$. Indeed, in \eqref{estimate 27}, from which this estimate started, we could have naively written 
\begin{equation}
(v_{q+1} + z_{1}) \cdot \nabla \theta_{q+1} - (v_{q} + z_{1}) \cdot \nabla \theta_{q} = (v_{q+1} - v_{q}) \cdot \nabla \theta_{q+1} + (v_{q} + z_{1}) \cdot \nabla (\theta_{q+1} - \theta_{q})
\end{equation} 
instead of $(v_{q+1} + z_{1}) \cdot \nabla (\theta_{q+1} - \theta_{q}) + (v_{q+1} - v_{q}) \cdot \nabla \theta_{q}$. If we did, then we would have $\lVert \nabla \theta_{q+1} \rVert_{L_{x}^{\infty}}$ instead of $\lVert \nabla \theta_{q} \rVert_{L_{x}^{\infty}}$ in \eqref{estimate 113}-\eqref{estimate 49} which would have translated to a necessary bound on $L_{T}^{1}\dot{W}_{x}^{1,\infty}$-norm of $v_{q+1}$ rather than $v_{q}$; considering the second inequality of \eqref{[Equ. (49a) and (49b), Y20c]}, we realize that such a bound is too large in case of $v_{q+1}$ but not necessarily  for $v_{q}$ because $\lambda_{q} \ll \lambda_{q+1}$. Let us now make these ideas precise. 
\end{remark} 
From \eqref{estimate 40} we can integrate by parts several times and estimate 
\begin{align}
 \frac{1}{2} \partial_{t} \lVert \Theta_{q}(t) \rVert_{\dot{H}_{x}^{2}}^{2} +& \lVert  \Theta_{q}(t) \rVert_{\dot{H}_{x}^{3}}^{2} 
= \int_{\mathbb{T}^{2}}[ \nabla (v_{q} + z_{1}) \cdot \nabla \Theta_{q} \cdot \nabla \Delta \Theta_{q} - \nabla (v_{q} + z_{1}) \cdot \nabla \nabla \Theta_{q}  \Delta \Theta_{q}  \nonumber\\
& \hspace{14mm} + \nabla (v_{q} + z_{1}) \cdot \nabla z_{2} \cdot \nabla \Delta \Theta_{q}+ (v_{q} + z_{1}) \cdot \nabla \nabla z_{2} \cdot \nabla \Delta \Theta_{q}](t) dx \nonumber\\
\overset{\eqref{[Equ. (33), Y20a]} \eqref{[Equ. (38), Y20a]}}{\lesssim}& ( \lVert v_{q} \rVert_{C_{t,x}^{1}} + \lVert z_{1} \rVert_{L_{x}^{\infty}} + \lVert z_{1} \rVert_{\dot{W}_{x}^{1,\infty}})(t) ( \lVert \nabla \Theta_{q} \rVert_{L_{x}^{2}} \lVert \nabla \Delta \Theta_{q} \rVert_{L_{x}^{2}} + L^{\frac{1}{4}} \lVert \nabla \Delta \Theta_{q} \rVert_{L_{x}^{2}})(t)  \nonumber\\
\overset{\eqref{[Equ. (38), Y20a]} \eqref{[Equ. (49a) and (49b), Y20c]}}{\lesssim}& M_{0}(t)^{\frac{1}{2}} \lambda_{q}^{4}   ( \lVert \Theta_{q} \rVert_{\dot{H}_{x}^{1}} \lVert \Theta_{q} \rVert_{\dot{H}_{x}^{3}}+ L^{\frac{1}{4}} \lVert  \Theta_{q} \rVert_{\dot{H}_{x}^{3}} )(t)  \nonumber\\
\leq&  \frac{1}{2} \lVert \Theta_{q}(t) \rVert_{\dot{H}_{x}^{3}}^{2} + C M_{0}(t) \lambda_{q}^{8} (\lVert \Theta_{q}(t) \rVert_{\dot{H}_{x}^{1}}^{2} + L^{\frac{1}{2}}), \label{estimate 324} 
\end{align} 
where we relied directly on \eqref{[Equ. (33), Y20a]} rather than \eqref{[Equ. (38), Y20a]} to handle $\lVert z_{2} \rVert_{\dot{H}_{x}^{2}}$. We also point out that it is crucial to have no $\lVert \Theta_{q} \rVert_{\dot{H}_{x}^{2}}$ on the right hand side here as an exponential growth such as $e^{ \lambda_{q}^{8} \int_{0}^{t}M_{0}(s) ds}$ will be far too large to handle; indeed, it is crucial to attain a linear growth from \eqref{estimate 324}. Subtracting $\frac{1}{2} \lVert \Theta_{q} (t)\rVert_{\dot{H}_{x}^{3}}^{2}$ from both sides and integrating over $[0,t]$ give us
\begin{equation}\label{estimate 47}
\lVert \Theta_{q}(t) \rVert_{\dot{H}_{x}^{2}}^{2} + \int_{0}^{t} \lVert \Theta_{q} \rVert_{\dot{H}_{x}^{3}}^{2} dr \leq \lVert \theta^{\text{in}} \rVert_{\dot{H}_{x}^{2}}^{2} + C M_{0}(t) \lambda_{q}^{8} [ \int_{0}^{t} \lVert \Theta_{q} \rVert_{\dot{H}_{x}^{1}}^{2 } dr + L^{\frac{1}{2}}]. 
\end{equation} 
In order to handle $ \int_{0}^{t} \lVert \Theta_{q} \rVert_{\dot{H}_{x}^{1}}^{2 } dr$, we apply \eqref{estimate 319} to \eqref{estimate 42} to deduce for all $t \in [0, T_{L}]$ 
\begin{equation}\label{estimate 46}
\frac{1}{2} \lVert \Theta_{q}(t) \rVert_{L_{x}^{2}}^{2} + \int_{0}^{t} \lVert \Theta_{q} \rVert_{\dot{H}_{x}^{1}}^{2} dr \overset{\eqref{[Equ.  (49a) and (49b), Y20c]} \eqref{[Equ. (38), Y20a]} }{\leq} \frac{1}{2} \lVert \theta^{\text{in}} \rVert_{L_{x}^{2}}^{2} + 3 M_{0}(t) [ \lVert \theta^{\text{in}} \rVert_{L_{x}^{2}} + 3 M_{0}(t) t] t. 
\end{equation} 
Applying \eqref{estimate 46} to \eqref{estimate 47} gives for all $t \in [0, T_{L}]$ 
\begin{equation}\label{estimate 48} 
\lVert \Theta_{q}(t) \rVert_{\dot{H}_{x}^{2}}^{2} + \int_{0}^{t} \lVert \Theta_{q} \rVert_{\dot{H}_{x}^{3}}^{2} dr \lesssim M_{0}(t) \lambda_{q}^{8} [ \lVert \theta^{\text{in}} \rVert_{H_{x}^{2}}^{2} + M_{0}(t)^{2} t^{2} + L^{\frac{1}{2}}]. 
\end{equation} 
Applying \eqref{estimate 48} to \eqref{estimate 49} now gives for all $t \in [0, T_{L}]$  
\begin{align}
\lVert \theta_{q+1} - \theta_{q} \rVert_{C_{t}L_{x}^{p^{\ast}}}& \overset{\eqref{estimate 49}}{\lesssim}  \lVert v_{q+1} - v_{q} \rVert_{C_{t}L_{x}^{p^{\ast}}} \sqrt{T_{L}} (\int_{0}^{t} \lVert \Theta_{q} \rVert_{\dot{H}_{x}^{3}}^{2} +  \lVert z_{2} \rVert_{\dot{W}_{x}^{1,\infty}}^{2} dr)^{\frac{1}{2}}  \nonumber \\
\overset{\eqref{[Equ. (38), Y20a]} \eqref{estimate 48}}{\lesssim}& \lVert v_{q+1} - v_{q} \rVert_{C_{t}L_{x}^{p^{\ast}}} \sqrt{T_{L}} M_{0}(t)^{\frac{1}{2}} \lambda_{q}^{4} [ \lVert \theta^{\text{in}} \rVert_{H_{x}^{2}} + M_{0}(t) T_{L} + L^{\frac{1}{4}} + \sqrt{T_{L}} L^{\frac{1}{4}}].\label{estimate 50}
\end{align}  
Our next task is to estimate $\lVert v_{q+1} - v_{q} \rVert_{C_{t}L_{x}^{p^{\ast}}}$ within \eqref{estimate 50}. We write for all $t \in [0, T_{L}]$  
\begin{equation}\label{estimate 51}
\lVert v_{q+1} - v_{q} \rVert_{C_{t}L_{x}^{p^{\ast}}} \leq  I_{1} + I_{2} \text{ where } I_{1} \triangleq \lVert v_{l} - v_{q} \rVert_{C_{t}L_{x}^{p^{\ast}}} \text{ and } I_{2} \triangleq \lVert v_{q+1} - v_{l} \rVert_{C_{t}L_{x}^{p^{\ast}}},  
\end{equation} 
where 
\begin{equation}\label{estimate 52}
I_{1} \lesssim \lVert v_{l} - v_{q} \rVert_{C_{t}L_{x}^{\infty}} \lesssim l \lVert v_{q} \rVert_{C_{t,x}^{1}}  \overset{\eqref{[Equ. (49a) and (49b), Y20c]}}{\lesssim} l M_{0}(t)^{\frac{1}{2}} \lambda_{q}^{4} \overset{\eqref{[Equ. (70), Y20c]}}{\lesssim} \lambda_{q+1}^{-\alpha} M_{0}(t)^{\frac{1}{2}}. 
\end{equation} 
The estimate on $I_{2}$ is more subtle. We estimate from \eqref{estimate 51} for $a \in 10 \mathbb{N}$ sufficiently large 
\begin{align}
I_{2} \overset{\eqref{[Equ. (85), Y20c]}}{\leq}& \lVert w_{q+1}^{(p)} \rVert_{C_{t}L_{x}^{p^{\ast}}} + \lVert w_{q+1}^{(c)} \rVert_{C_{t}L_{x}^{p^{\ast}}} + \lVert w_{q+1}^{(t)} \rVert_{C_{t}L_{x}^{p^{\ast}}} \label{estimate 53}\\
\overset{\eqref{[Equ. (88) and (89a), Y20c]} \eqref{[Equ. (90) and (91), Y20c]}}{\lesssim}& \delta_{q+1}^{\frac{1}{2}} M_{0}(t)^{\frac{1}{2}} l^{-\frac{3}{2}} r^{1- \frac{2}{p^{\ast}}}   + \delta_{q+1} M_{0}(t) l^{-3} r^{2 - \frac{2}{p^{\ast}}} \lambda_{q+1}^{4 \eta -1}    \nonumber\\
\overset{ \eqref{[Equ. (67), Y20c]} \eqref{[Equ. (70), Y20c]}}{\lesssim}& \delta_{q+1}^{\frac{1}{2}} M_{0}(t)^{\frac{1}{2}} (\lambda_{q+1}^{1- 6 \eta})^{1- \frac{2}{p^{\ast}}} \lambda_{q+1}^{3 \alpha} [ \lambda_{q+1}^{3 \alpha - 2\eta}M_{0}(t)^{\frac{1}{2}} + 1] 
\lesssim \delta_{q+1}^{\frac{1}{2}} M_{0}(t)^{\frac{1}{2}} \lambda_{q+1}^{(1- 6 \eta) (1- \frac{2}{p^{\ast}})+ 3 \alpha} \nonumber
\end{align}  
where the last inequality used the fact that $3\alpha < 2 \eta$ which can be readily verified by \eqref{[Equ. (64), Y20c]}-\eqref{[Equ. (66), Y20c]}. Applying \eqref{estimate 52}-\eqref{estimate 53} to \eqref{estimate 51} and the resulting bound to \eqref{estimate 50} gives us 
\begin{equation}\label{estimate 56}
\lVert \theta_{q+1} - \theta_{q} \rVert_{C_{t}L_{x}^{p^{\ast}}} \lesssim I_{3} + I_{4},  
\end{equation} 
where 
\begin{subequations}
\begin{align}
& I_{3} \triangleq  \lambda_{q+1}^{-\alpha} M_{0}(t) \sqrt{T_{L}} \lambda_{q}^{4} (\lVert \theta^{\text{in}} \rVert_{H_{x}^{2}} + M_{0}(t) T_{L} + L^{\frac{1}{4}} + \sqrt{T_{L}} L^{\frac{1}{4}}), \label{estimate 325}\\
& I_{4}\triangleq \delta_{q+1}^{\frac{1}{2}} M_{0}(t) \lambda_{q+1}^{(1- 6 \eta) (1- \frac{2}{p^{\ast}}) + 3 \alpha} \sqrt{T_{L}} \lambda_{q}^{4} (\lVert \theta^{\text{in}} \rVert_{H_{x}^{2}} + M_{0}(t) T_{L} + L^{\frac{1}{4}} + \sqrt{T_{L}} L^{\frac{1}{4}}).\label{estimate 326}
\end{align}
\end{subequations} 
Because 
\begin{equation}\label{estimate 88}
2 \beta b - \alpha + \frac{4}{b} \overset{\eqref{[Equ. (68), Y20c]}}{<} \frac{\alpha}{48} - \alpha + \frac{4}{b}  \overset{\eqref{estimate 62}}{<} \frac{\alpha}{48} - \alpha + \frac{\alpha}{4} = - \frac{35\alpha}{48}, 
\end{equation}
we can immediately see that taking $a\in 10 \mathbb{N}$ sufficiently large gives us 
\begin{equation}\label{estimate 54}
I_{3} \approx_{T_{L}} c_{R} M_{0}(t) \delta_{q+2} a^{b^{q+1} [2\beta b - \alpha + \frac{4}{b}]} (\lVert \theta^{\text{in}} \rVert_{H_{x}^{2}} + M_{0}(t) T_{L} + L^{\frac{1}{4}} + \sqrt{T_{L}} L^{\frac{1}{4}}) \ll c_{R}M_{0}(t) \delta_{q+2}.  
\end{equation} 
On the other hand, because 
\begin{equation}
(1- 6 \eta) (1- \frac{2}{p^{\ast}}) + 3 \alpha + \frac{4}{b} \overset{\eqref{[Equ. (116), Y20c]}}{=} -1 + 8 \eta - \frac{69\alpha}{2} + \frac{4}{b}, 
\end{equation} 
taking $a \in 10 \mathbb{N}$ sufficiently large gives us 
\begin{align}
I_{4} \overset{\eqref{estimate 326}}{\approx}& c_{R} M_{0}(t) \delta_{q+2} [ a^{b^{q+1} [ -1 + 8 \eta - \frac{69\alpha}{2} + \frac{4}{b} + 2 \beta b]} \delta_{q+1}^{\frac{1}{2}} \sqrt{T_{L}}] (\lVert \theta^{\text{in}} \rVert_{H_{x}^{2}} + M_{0}(t) T_{L} + L^{\frac{1}{4}} + \sqrt{T_{L}} L^{\frac{1}{4}}) \nonumber\\
\lesssim_{T_{L}}&  
c_{R} M_{0}(t) \delta_{q+2}  a^{b^{q+1} [- \frac{1643\alpha}{48}]}  (\lVert \theta^{\text{in}} \rVert_{H_{x}^{2}} + M_{0}(t) T_{L} + L^{\frac{1}{4}} + \sqrt{T_{L}} L^{\frac{1}{4}}) 
 \ll c_{R} M_{0}(t) \delta_{q+2}, \label{estimate 55}
\end{align} 
where we also used that 
\begin{equation}
-1 + 8 \eta - \frac{69\alpha}{2} + \frac{4}{b}+  2\beta b \overset{ \eqref{[Equ. (68), Y20c]} \eqref{estimate 62}}{<} -1 + 8 \eta - \frac{1643\alpha}{48} \overset{\eqref{[Equ. (65), Y20c]}}{\leq} -\frac{1643\alpha}{48}. 
\end{equation}
Applying \eqref{estimate 54} and \eqref{estimate 55} to \eqref{estimate 56} finally gives us $\lVert \theta_{q+1} - \theta_{q} \rVert_{C_{t}L_{x}^{p^{\ast}}} \ll c_{R} M_{0}(t) \delta_{q+2}$. Therefore, due to Lemma \ref{[Def. 9, Lem. 10, CDS12]}, for any $t \in [0, T_{L}]$, we conclude 
\begin{equation}\label{estimate 64}
I \overset{\eqref{estimate 36}}{=} \lVert \mathcal{R} ((\theta_{q} - \theta_{q+1} )e^{2}) \rVert_{C_{t} L_{x}^{p^{\ast}}} \ll c_{R} M_{0}(t) \delta_{q+2}.
\end{equation}
Next, we estimate $II$ from \eqref{estimate 36}. For subsequent convenience, we compute the following for general $n \in \{2,3\}$, with the current case being $n = 2$. First, we can compute for any $\epsilon \in (0, \min \{1 + \frac{\sigma}{2}, 1+ 2 \delta \})$ and $t \in [0, T_{L}]$  
\begin{equation}\label{estimate 348}
II \lesssim \lVert \mathcal{R} (\theta_{l} - \theta_{q}) e^{n} \rVert_{C_{t,x}} 
\lesssim l^{\frac{1}{2} - 2 \delta} ( \lVert \theta_{q} \rVert_{C_{t}^{\frac{1}{2} - 2 \delta} \dot{H}_{x}^{\frac{n}{2} - 1 + \epsilon}} + \lVert \theta_{q} \rVert_{C_{t}\dot{H}_{x}^{\frac{n}{2} + \epsilon - \frac{1}{2} - 2\delta}}). 
\end{equation} 
Now we can apply $\nabla$ on \eqref{estimate 40} so that 
\begin{equation}
\partial_{t} \nabla \Theta_{q} = \Delta \nabla \Theta_{q} - \nabla ((v_{q} + z_{1}) \cdot \nabla ( \Theta_{q} + z_{2}))
\end{equation} 
and estimate for any $t \in [0, T_{L}]$ 
\begin{align}
 \int_{0}^{t} \lVert \partial_{t} \nabla \Theta_{q} \rVert_{L_{x}^{2}}^{2} dr \lesssim& \int_{0}^{t}  \lVert \Delta \nabla \Theta_{q} \rVert_{L_{x}^{2}}^{2} + ( \lVert v_{q} \rVert_{\dot{W}_{x}^{1,\infty}}^{2} + \lVert  z \rVert_{\dot{W}_{x}^{1,\infty}}^{2})(\lVert \Theta_{q} \rVert_{\dot{H}_{x}^{2}}^{2} + \lVert z_{2} \rVert_{\dot{H}_{x}^{2}}^{2}) dr \nonumber\\
& \hspace{10mm} \overset{\eqref{estimate 48}\eqref{[Equ.  (49a) and (49b), Y20c]}\eqref{[Equ. (33), Y20a]} }{\lesssim}   M_{0}(t)^{4} \lambda_{q}^{16} [ \lVert \theta^{\text{in}} \rVert_{H_{x}^{2}}^{2}  + M_{0}(t)^{2} t^{2} + L^{\frac{1}{2}}].   
\end{align}
This allows us to compute due to \eqref{[Equ. (33), Y20a]} and \eqref{estimate 46} for any $t \in [0, T_{L}]$ 
\begin{align}
\lVert \theta_{q} \rVert_{C_{t}^{\frac{1}{2} - 2\delta} \dot{H}_{x}^{\frac{n}{2} - 1 + \epsilon}} 
\lesssim&  \lVert \Theta_{q} \rVert_{C_{t}^{\frac{1}{2} - 2\delta} \dot{H}_{x}^{\frac{n}{2} - 1 + \epsilon}} +  \lVert z_{2} \rVert_{C_{t}^{\frac{1}{2} - 2\delta} \dot{H}_{x}^{\frac{n}{2} - 1 + \epsilon}} \nonumber\\
\lesssim&  M_{0}(t)^{2} \lambda_{q}^{8} [ \lVert \theta^{\text{in}} \rVert_{H_{x}^{2}}(1+ \sqrt{t})  + M_{0}(t) t + L^{\frac{1}{4}}]. \label{estimate 346}
\end{align}  
On the other hand, due to our choice of $\epsilon \in (0, \min\{ 1 + \frac{\sigma}{2}, 1 + 2\delta \})$, for any $t \in [0, T_{L}]$  
\begin{equation}\label{estimate 347}
\lVert \theta_{q} \rVert_{C_{t} \dot{H}_{x}^{\frac{n}{2} - \frac{1}{2} + \epsilon - 2 \delta}} 
\lesssim  \lVert \Theta_{q} \rVert_{C_{t} \dot{H}_{x}^{2}} + \lVert z_{2} \rVert_{C_{t} \dot{H}_{x}^{\frac{n + 2 + \sigma}{2}}} 
\overset{\eqref{estimate 48} \eqref{[Equ. (33), Y20a]} }{\lesssim} M_{0}(t)^{\frac{1}{2}} \lambda_{q}^{4} [ \lVert \theta^{\text{in}} \rVert_{H_{x}^{2}} + M_{0}(t) t + L^{\frac{1}{4}}].  
\end{equation} 
We apply \eqref{estimate 346}-\eqref{estimate 347} to \eqref{estimate 348} to deduce for $\delta \in (0, \frac{1}{12})$  
\begin{align}
II \overset{\eqref{estimate 348} \eqref{estimate 346}\eqref{estimate 347}}{\lesssim}& l^{\frac{1}{2} - 2 \delta }  M_{0}(t)^{2} \lambda_{q}^{8} [ \lVert \theta^{\text{in}} \rVert_{H_{x}^{2}}(1+ \sqrt{t})  + M_{0}(t) t + L^{\frac{1}{4}}] \label{estimate 349}\\
\overset{\eqref{estimate 62}-\eqref{[Equ. (69), Y20c]}}{\lesssim}& c_{R} \delta_{q+2} M_{0}(t)  a^{b^{q+1} [ -\frac{\alpha}{48}]} M_{0}(t) ( \lVert \theta^{\text{in}} \rVert_{H_{x}^{2}}(1+ \sqrt{t}) + + M_{0}(t) t + L^{\frac{1}{4}}) 
\ll c_{R} \delta_{q+2} M_{0}(t) \nonumber
\end{align} 
by taking $a \in 10 \mathbb{N}$ sufficiently large, where we also used the fact that 
\begin{align*}
2 \beta b - \frac{\alpha}{2} + \frac{22}{3b} \overset{\eqref{[Equ. (68), Y20c]} \eqref{estimate 62} }{<} \frac{\alpha}{48} - \frac{\alpha}{2} + (\frac{22}{3}) (\frac{\alpha}{16}) = \frac{\alpha}{48} - \frac{24\alpha}{48} + \frac{22\alpha}{48} =  - \frac{\alpha}{48}. 
\end{align*} 
Therefore, applying \eqref{estimate 64} and \eqref{estimate 349} to \eqref{estimate 63} allows us to conclude 
\begin{equation}\label{estimate 327}
\lVert \mathcal{R} ((\theta_{l} - \theta_{q+1})e^{2} ) \rVert_{C_{t}L_{x}^{p^{\ast}}} \ll c_{R} M_{0}(t) \delta_{q+2}. 
\end{equation} 
It can be clearly seen from the computations in \cite[Equ. (99a), (120)-(123)]{Y20c} that 
\begin{equation}\label{estimate 328}
\lVert R_{\text{lin}} - \mathcal{R} ((\theta_{l} - \theta_{q+1} ) e^{2}) \rVert_{C_{t}L_{x}^{p^{\ast}}} \ll c_{R}M_{0}(t) \delta_{q+2}.
\end{equation} 
Therefore, we conclude by \eqref{estimate 327}-\eqref{estimate 328} that 
\begin{equation}\label{estimate 329} 
\lVert R_{\text{lin}} \rVert_{C_{t}L_{x}^{p^{\ast}}} \leq (2\pi)^{-2 (\frac{p^{\ast} -1}{p^{\ast}})} \frac{c_{R} M_{0}(t) \delta_{q+2}}{5}. 
\end{equation} 
By \cite[Equ. (125), (131)-(133)]{Y20c} we obtain 
\begin{align}
& \lVert R_{\text{cor}} \rVert_{C_{t}L_{x}^{p^{\ast}}} +\lVert R_{\text{osc}} \rVert_{C_{t}L_{x}^{p^{\ast}}} + \lVert R_{\text{com1}} \rVert_{C_{t}L_{x}^{1}} + \lVert R_{\text{com2}} \rVert_{C_{t}L_{x}^{1}}  \nonumber\\
\leq& (2\pi)^{-2 (\frac{p^{\ast} -1}{p^{\ast}})} \frac{ 2 c_{R} M_{0}(t) \delta_{q+2}}{5} + \frac{2c_{R} M_{0}(t) \delta_{q+2}}{5},   \label{estimate 132}
\end{align}
from which, along with \eqref{estimate 329}, we can now conclude by H$\ddot{\mathrm{o}}$lder's inequality that 
\begin{align}
\lVert \mathring{R}_{q+1} \rVert_{C_{t}L_{x}^{1}} \overset{\eqref{estimate 58}}{\leq}& (2\pi)^{2 (\frac{p^{\ast} -1}{p^{\ast}})} [ \lVert R_{\text{lin}} \rVert_{C_{t}L_{x}^{p^{\ast}}} + \lVert R_{\text{cor}} \rVert_{C_{t}L_{x}^{p^{\ast}}} + \lVert R_{\text{osc}} \rVert_{C_{t}L_{x}^{p^{\ast}}}] \nonumber \\
&+ \lVert R_{\text{com1}} \rVert_{C_{t}L_{x}^{1}} + \lVert R_{\text{com2}} \rVert_{C_{t}L_{x}^{1}} \leq c_{R} M_{0}(t) \delta_{q+2}. 
\end{align} 
This verifies \eqref{[Equ. (49c), Y20c]} at level $q+1$. Lastly, the argument to verify that $v_{q+1}$ is $(\mathcal{F}_{t})$-adapted is identical to \cite{HZZ19, Y20a, Y20c}. As $v_{q+1}$ is $(\mathcal{F}_{t})$-adapted, so is $\theta_{q+1}$. It follows from \eqref{[Equ. (99), Y20c]}-\eqref{estimate 58} that $\mathring{R}_{q+1}$ is also $(\mathcal{F}_{t})_{t\geq 0}$-adapted. Finally, the argument of $v_{q+1}(0,x)$ being deterministic is also identical to those of previous works \cite{HZZ19, Y20a, Y20c}. As $\theta^{\text{in}}$ is deterministic, $R_{\text{lin}}(0,x)$, $R_{\text{cor}}(0,x), R_{\text{com2}}(0,x), R_{\text{osc}}(0,x)$, and $R_{\text{com1}}(0,x)$ are all deterministic; hence, so is $\mathring{R}_{q+1}(0,x)$. 

Next, we consider the case $n = 3$ so that $m \in (\frac{13}{20}, \frac{5}{4})$ by \eqref{4}. For notations and preliminaries hereafter in case $n = 3$, we refer to Subsection \ref{Preliminaries needed for convex integration in 3D case and more}. 
\begin{proposition}\label{Proposition 4.7 for n=3}
Fix $\theta^{\text{in}} \in H^{2}(\mathbb{T}^{3})$ that is deterministic and mean-zero. Let 
\begin{equation}\label{[Equ. (41), Y20a]}
v_{0}(t,x) \triangleq \frac{L^{2} e^{2Lt}}{(2\pi)^{\frac{3}{2}}} 
\begin{pmatrix}
\sin(x^{3}) & 0 & 0 
\end{pmatrix}^{T}. 
\end{equation} 
Then there exists a unique solution $\theta_{0} \in L_{\omega}^{\infty} L_{t}^{\infty} H_{x}^{2}$ to \eqref{estimate 25} again. It follows that together with 
\begin{align}
\mathring{R}_{0}(t,x) \triangleq& \frac{2L^{3} e^{2Lt}}{(2\pi)^{\frac{3}{2}}} 
\begin{pmatrix}
0 & 0 & - \cos(x^{3}) \\
 0 & 0 & 0 \\
 -\cos(x^{3}) & 0 & 0 
\end{pmatrix} 
\nonumber\\
&+ 
(\mathcal{R} (-\Delta)^{m} v_{0} - \mathcal{R} \theta_{0} e^{3} + v_{0} \mathring{\otimes} z_{1} + z_{1} \mathring{\otimes} v_{0} + z_{1} \mathring{\otimes} z_{1})(t,x), \label{estimate 107}
\end{align} 
$(v_{0}, \theta_{0})$ satisfy \eqref{[Equ. (46), Y20c]} at level $q= 0$ where $\mathcal{R}$ is defined in Lemma \ref{divergence inverse operator}. Moreover, \eqref{[Equ. (49), Y20c]} is satisfied at level $q= 0$ provided 
\begin{equation}\label{[Equ. (43), Y20a]}
\max \{ (18 \lVert \theta^{\text{in}} \rVert_{L_{x}^{2}})^{\frac{1}{3}}, 16 \} < L, (18) (2\pi)^{\frac{3}{2}} 9 < (18) (2\pi)^{\frac{3}{2}} a^{2 \beta b} \leq c_{R} L \leq c_{R} ( \frac{(2\pi)^{\frac{3}{2}} a^{4} - 2}{2}), 
\end{equation} 
where the inequality $9 < a^{2\beta b}$ is assumed for the justification of the second inequality of \eqref{[Equ.  (49a) and (49b), Y20c]}. Furthermore, $v_{0}(0,x)$ and $\mathring{R}_{0}(0,x)$ are both deterministic.  
\end{proposition}

\begin{proof}[Proof of Proposition \ref{Proposition 4.7 for n=3}]
The proof of Proposition \ref{Proposition 4.7 for n=3} is similar to that of Proposition \ref{Proposition 4.7 for n=2}. We can readily verify the same bounds as \eqref{[Equ. (53), Y20c]} (see \cite[Equ. (44)]{Y20a}). Concerning $\lVert \mathring{R}_{0}(t,x) \rVert_{C_{t}L_{x}^{1}}$, we can compute from \eqref{estimate 40} identically to \eqref{estimate 42}-\eqref{estimate 43} to obtain 
\begin{equation}\label{estimate 60}
\lVert \Theta_{0} \rVert_{C_{t}L_{x}^{2}}  \leq \lVert  \theta^{\text{in}} \rVert_{L_{x}^{2}} + L^{\frac{5}{4}} e^{2L t} (2\pi)^{\frac{3}{2}}. 
\end{equation} 
Taking advantage of \eqref{[Equ. (53), Y20c]} and \eqref{estimate 60}, we can verify that $\theta_{0} \in L_{\omega}^{\infty} L_{t}^{\infty} H_{x}^{2}$ as we did in the proof of Proposition \ref{Proposition 4.7 for n=2}. Next, setting $\pi_{0} \triangleq - \frac{1}{3} (2 v_{0} \cdot z_{1} + \lvert z_{1} \rvert^{2})$ shows that  $(v_{0}, \theta_{0})$ satisfy \eqref{estimate 340} at level $q= 0$. Similarly to the proof of Proposition \ref{Proposition 4.7 for n=2}, $\nabla\cdot v_{0} = 0$ while the hypothesis that $\theta^{\text{in}}$ is mean-zero implies that $\theta_{0}(t)$ is mean-zero for all $t \geq 0$ and hence $\mathcal{R}(\theta_{0} e^{3})(t,x)$ is trace-free and symmetric for all $t$ and $x$; it follows that $\mathring{R}_{0}$ is trace-free and symmetric. Next, 
\begin{equation}\label{estimate 330} 
\lVert \mathcal{R} \theta_{0} e^{3} \rVert_{C_{t}L_{x}^{1}} \overset{\eqref{estimate 336}}{\leq} (2\pi)^{\frac{3}{2}} 6 [ \lVert \Theta_{0} \rVert_{C_{t}L_{x}^{2}} + (2\pi)^{\frac{3}{2}} \lVert z_{2} \rVert_{C_{t}L_{x}^{\infty}}] \overset{\eqref{estimate 60}  \eqref{[Equ. (38), Y20a]} \eqref{[Equ. (43), Y20a]}}{\leq} (2\pi)^{\frac{3}{2}} L^{-1} M_{0}(t). 
\end{equation} 
From the proof of \cite[Pro. 4.7]{Y20a}, using the fact that $L > 16$ we see that 
\begin{equation}\label{estimate 337}
 \Vert \mathring{R}_{0} (t) + \mathcal{R} \theta_{0} e^{3} \rVert_{L_{x}^{1}} \leq 17 (2\pi)^{\frac{3}{2}} M_{0}(t) L^{-1}. 
\end{equation} 
Therefore, 
\begin{align*}
\lVert \mathring{R}_{0} (t) \rVert_{L_{x}^{1}}  \overset{\eqref{estimate 337} \eqref{estimate 330}}{\leq}& 17 (2\pi)^{\frac{3}{2}} M_{0}(t) L^{-1}  + (2\pi)^{\frac{3}{2}} L^{-1} M_{0}(t)  
\overset{\eqref{[Equ. (43), Y20a]}}{\leq}c_{R} \delta_{1}M_{0}(t). 
\end{align*} 
Finally, the verification that $v_{0}(0,x)$ and $\mathring{R}_{0}(0,x)$ are both deterministic is identical to that of Proposition \ref{Proposition 4.7 for n=2}. 
\end{proof}

\begin{proposition}\label{Proposition 4.8 for n=3}
Fix $\theta^{\text{in}}  \in H^{2} (\mathbb{T}^{3})$ that is deterministic and mean-zero from the hypothesis of Proposition \ref{Proposition 4.7 for n=3}. Let $L$ satisfy 
\begin{equation}\label{estimate 339}
L > \max\{ (18 \lVert \theta^{\text{in}} \rVert_{L_{x}^{2}})^{\frac{1}{3}}, 16, c_{R}^{-1} 162 (2\pi)^{\frac{3}{2}} \},  
\end{equation} 
and suppose that $(v_{q}, \theta_{q}, \mathring{R}_{q})$ are $(\mathcal{F}_{t})_{t\geq 0}$-adapted processes that solve \eqref{[Equ. (46), Y20c]} and satisfy \eqref{[Equ. (49), Y20c]}. Then there exist a choice of parameters $a, b,$ and $\beta$ such that \eqref{[Equ. (43), Y20a]} is fulfilled and $(\mathcal{F}_{t})_{t\geq 0}$-adapted processes $(v_{q+1}, \theta_{q+1}, \mathring{R}_{q+1})$ that satisfy \eqref{[Equ. (46), Y20c]}, \eqref{[Equ. (49), Y20c]} at level $q+1$, and for all $t \in [0, T_{L}]$ and $p \in [1, \infty)$, 
\begin{subequations}
\begin{align}
& \lVert v_{q+1} (t) - v_{q}(t) \rVert_{L_{x}^{2}} \leq M_{0}(t)^{\frac{1}{2}} \delta_{q+1}^{\frac{1}{2}},  \label{[Equ. (45), Y20a]}\\
& \mathbb{E}^{\textbf{P}} [ \lVert \theta_{q+1} - \theta_{q} \rVert_{C_{t \wedge T_{L}}L_{x}^{2}}^{2p} + ( \int_{0}^{t \wedge T_{L}} \lVert \theta_{q+1} - \theta_{q} \rVert_{\dot{H}_{x}^{1}}^{2} dr)^{p} ] \nonumber\\
& \hspace{20mm} \lesssim_{p, \lVert \theta^{\text{in}} \rVert_{H_{x}^{2}}, \text{Tr} ((-\Delta)^{\frac{3}{2} + 2 \sigma} G_{2}G_{2}^{\ast}), L} \lambda_{q+1}^{-  \beta p (\frac{21 + \beta}{12 + 2 \beta})}. \label{estimate 59}
\end{align}
\end{subequations} 
Finally, if $v_{q}(0,x)$ and $\mathring{R}_{q}(0,x)$ are deterministic, then so are $v_{q+1}(0,x)$ and $\mathring{R}_{q+1}(0,x)$. 
\end{proposition}

\begin{proof}[Proof of Theorem \ref{Theorem 2.1} if $n = 3$ assuming Proposition \ref{Proposition 4.8 for n=3}]
The existence of $(v_{q}, \theta_{q}, \mathring{R}_{q})$ for all $q \geq 1$ that satisfies \eqref{[Equ. (46), Y20c]} and \eqref{[Equ. (49), Y20c]}, and $v \in C([0,T_{L}]; \dot{H}^{\gamma}(\mathbb{T}^{3}))$ that is $(\mathcal{F}_{t})_{t \geq 0}$-adapted and satisfies \eqref{[Equ. (59), Y20c]} follow identically from the proof of Theorem \ref{Theorem 2.1} in case $n =2$. Next, similarly to \eqref{estimate 60} we can show 
\begin{align*}
\partial_{t} \lVert \Theta_{q}(t) \rVert_{L_{x}^{2}} \leq \lVert v_{q} (t) \rVert_{L_{x}^{2}} \lVert z_{2}(t) \rVert_{\dot{W}_{x}^{1,\infty}} + (2\pi)^{\frac{3}{2}} \lVert z_{1}(t) \rVert_{L_{x}^{\infty}} \lVert z_{2}(t) \rVert_{\dot{W}_{x}^{1,\infty}}  \overset{\eqref{[Equ. (38), Y20a]} \eqref{[Equ. (49a) and (49b), Y20c]}}{\leq} 2M_{0}(t)^{\frac{1}{2}} L^{\frac{1}{4}} + (2\pi)^{\frac{3}{2}} L^{\frac{1}{2}} 
\end{align*} 
and hence integrating over $[0,t]$ and taking supremum over $[0,t]$ on the left hand side give 
\begin{equation}\label{estimate 359}
\lVert \Theta_{q} \rVert_{C_{t}L_{x}^{2}} \leq \lVert \theta^{\text{in}} \rVert_{L_{x}^{2}} + t [ 2M_{0}(t)^{\frac{1}{2}} + (2\pi)^{\frac{3}{2}} L^{\frac{1}{4}}] L^{\frac{1}{4}}.
\end{equation} 
Along with $\lVert z_{2} \rVert_{C_{t}L_{x}^{2}} \leq (2\pi)^{\frac{3}{2}} L^{\frac{1}{4}}$ due to \eqref{[Equ. (38), Y20a]}, we deduce for all $q \in \mathbb{N}_{0}$ 
\begin{equation}\label{estimate 331}
\lVert \theta_{q} \rVert_{C_{t}L_{x}^{2}} \leq \lVert \Theta_{q} \rVert_{C_{t}L_{x}^{2}} + \lVert z_{2} \rVert_{C_{t}L_{x}^{2}}  \leq \lVert \theta^{\text{in}} \rVert_{L_{x}^{2}} + t[ 2 M_{0}(t)^{\frac{1}{2}} + (2\pi)^{\frac{3}{2}} L^{\frac{1}{4}}]L^{\frac{1}{4}} + (2\pi)^{\frac{3}{2}} L^{\frac{1}{4}}.  
\end{equation}
Inductively from \eqref{estimate 114} and \eqref{estimate 331}, we can now conclude that for all $q \in \mathbb{N}_{0}$ $\theta_{q} \in L_{\omega}^{p} C_{t}L_{x}^{p}$ for all $p \in [1, \infty)$ with a bound by a constant $C(\lVert \theta^{\text{in}} \rVert_{H_{x}^{2}}, p, \text{Tr} ( (-\Delta)^{\frac{3}{2} + 2 \sigma} G_{2}G_{2}^{\ast}), L)$. This allows us to conclude from \eqref{estimate 59} that $\{\theta_{q}\}_{q=0}^{\infty}$ is Cauchy in not only $\cap_{p\in[1,\infty)} L_{\omega}^{p} L_{T_{L}}^{2} \dot{H}_{x}^{1}$ but also interpolate identically to \eqref{estimate 338} to conclude that it is Cauchy also in $\cap_{p\in [1,\infty)} L_{\omega}^{p} C_{T_{L}} L_{x}^{p}$: 
\begin{equation}  
 \mathbb{E}^{\textbf{P}} [ \lVert \theta_{q+1} - \theta_{q} \rVert_{C_{t \wedge T_{L}} L_{x}^{p}}^{p}] \overset{\eqref{estimate 59}}{\lesssim}_{\lVert \theta^{\text{in}} \rVert_{H_{x}^{2}}, p, \text{Tr} ((-\Delta)^{\frac{3}{2} + 2 \sigma} G_{2}G_{2}^{\ast} ), L} \lambda_{q+1}^{- \frac{\beta}{2} ( \frac{21+ \beta}{12 + 2 \beta})}.
\end{equation}  
Therefore, we deduce the existence of $\lim_{q\to\infty} \theta_{q} \triangleq \theta \in \cap_{p \in [1,\infty)} L_{\omega}^{p} C_{T_{L}} L_{x}^{p} \cap L_{\omega}^{p} L_{T_{L}}^{2} \dot{H}_{x}^{1}$ for which there exists a deterministic constant $C_{L,2} = C_{L,2}(p) > 0$ for $p \in [1,\infty)$ such that \eqref{estimate 320} holds and hence verify the second inequality of \eqref{estimate 17}. As each $\theta_{q}$ is $(\mathcal{F}_{t})_{t\geq 0}$-adapted, so is $\theta$. Finally, for all $t \in [0, T_{L}]$, $\lVert \mathring{R}_{q} \rVert_{C_{t}L_{x}^{1}} \to 0$ as $q\to \infty$ so that $(u,\theta)$ with $u = v + z_{1}$ solve \eqref{3}. The rest of the proof is identical to that of Theorem \ref{Theorem 2.1} in case $n = 2$ with the only exception that we must replace \eqref{estimate 333} by 
\begin{equation}\label{estimate 334}
L^{\frac{1}{4}} (2 \pi)^{\frac{3}{2}}  + K e^{\frac{T}{2}} (\lVert \theta^{\text{in}} \rVert_{L_{x}^{2}}+ \sum_{l=1}^{2} \sqrt{ \text{Tr} (G_{l}G_{l}^{\ast})}) \leq (e^{LT} - K e^{\frac{T}{2}}) \lVert u^{\text{in}} \rVert_{L_{x}^{2}} + L e^{LT}
\end{equation}
so that analogous computations in \eqref{estimate 142} go through. 
\end{proof} 

\subsection{Convex integration to prove Proposition \ref{Proposition 4.8 for n=3}}

\subsubsection{Choice of parameters}
For $L $ that satisfies \eqref{estimate 339}, we choose 
\begin{equation}\label{[Equ. (52), Y20a]}
\alpha \triangleq \frac{5-4m}{480} 
\end{equation} 
instead of \eqref{[Equ. (66), Y20c]} while we choose $b$ identically to \eqref{estimate 62} and be sure to take $\beta$ sufficiently small subsequently to guarantee \eqref{[Equ. (68), Y20c]} . We take $l$ identically to \eqref{[Equ. (69), Y20c]} so that \eqref{[Equ. (70), Y20c]} continues to hold. The last inequality of $c_{R} L \leq c_{R} ( \frac{ (2\pi)^{\frac{3}{2}} a^{4} - 2}{2})$ in \eqref{[Equ. (43), Y20a]} holds by taking $a$ sufficiently large while the inequalities of $162 (2\pi)^{\frac{3}{2}}  < (18) (2\pi)^{\frac{3}{2}} a^{2\beta b} \leq c_{R} L$ in \eqref{[Equ. (43), Y20a]} also holds by taking $\beta > 0$ sufficiently small because we chose $L > c_{R}^{-1} 162 (2\pi)^{\frac{3}{2}}$ in \eqref{estimate 339}. Therefore, we hereafter consider such $\alpha, L$, and $b$ fixed, while take $a > 0$ such that $a^{\frac{25- 20m}{24}} \in \mathbb{N}$, which will be crucial in \eqref{[Equ. (62), Y20a]}, as large and $\beta > 0$ as small as necessary. 

\subsubsection{Mollification}
We mollify identically to \eqref{[Equ. (71), Y20c]} with the only exception of $\phi_{\epsilon} (\cdot) \triangleq \frac{1}{\epsilon^{3}}  \phi(\frac{\cdot}{\epsilon})$ being a mollifier in $\mathbb{R}^{3}$. It follows that 
\begin{equation}\label{estimate 341}
\partial_{t} v_{l} + (-\Delta)^{m} v_{l} + \text{div} ((v_{l} + z_{1,l}) \otimes (v_{l} + z_{1,l}) ) + \nabla \pi_{l}= \theta_{l} e^{3} + \text{div} (\mathring{R}_{l} + R_{\text{com1}})
\end{equation} 
with $R_{\text{com1}}$ identical to that in \eqref{[Equ. (73b), Y20c]} while 
\begin{equation}\label{estimate 145}
\pi_{l} \triangleq (\pi_{q} \ast_{x} \phi_{l}) \ast_{t} \varphi_{l} -\frac{1}{3} ( \lvert v_{l} + z_{1,l} \rvert^{2} - ( \lvert v_{q} + z_{1} \rvert^{2} \ast_{x} \phi_{l}) \ast_{t} \varphi_{l}). 
\end{equation} 
Due to \eqref{[Equ. (70), Y20c]} and \eqref{[Equ.  (49a) and (49b), Y20c]} we have the identical estimates to \eqref{[Equ. (74a), Y20c]}-\eqref{[Equ. (74b) and (74c), Y20c]}. Next, we define 
\begin{equation}\label{[Equ. (62), Y20a]}
r_{\lVert} \triangleq \lambda_{q+1}^{\frac{13 - 20m}{12}}, \hspace{3mm} r_{\bot}\triangleq \lambda_{q+1}^{\frac{1-20m}{24}}, \hspace{3mm} \mu \triangleq \frac{ \lambda_{q+1}^{2m-1} r_{\lVert}}{r_{\bot}} = \lambda_{q+1}^{2m-1} \lambda_{q+1}^{\frac{25 - 20m}{24}},  
\end{equation} 
which satisfies $r_{\bot} \ll r_{\lVert} \ll 1$ and $r_{\bot}^{-1} \ll \lambda_{q+1}$ for $a> 0$ sufficiently large as $m \in (\frac{13}{20}, \frac{5}{4})$. Our choice of $a^{\frac{25 - 20m}{24}} \in \mathbb{N}$ and $b \in \mathbb{N}$ guarantee that $\lambda_{q+1} r_{\bot} = a^{b^{q+1} (\frac{25 -20m}{24})} \in \mathbb{N}$ which is needed to ensure the appropriate periodicity of $W_{\zeta}, V_{\zeta}, \Phi_{\zeta}, \phi_{\zeta}$ and $\psi_{\zeta}$ in \eqref{[Equ. (168), Y20a]}, \eqref{[Equ. (173), Y20a]}, and \eqref{[Equ. (167), Y20a]}. 

\subsubsection{Perturbation}

We define $\chi(z)$ and $\rho(\omega, t, x)$ identically to \eqref{[Equ. (75), Y20c]}-\eqref{[Equ. (76), Y20c]}, from which the bound of \eqref{[Equ. (77), Y20c]} follows. For any $p \in [1, \infty], N \geq 0$, and $t \in [0, T_{L}]$, we have 
\begin{equation}\label{estimate 342}
\lVert \rho \rVert_{C_{t}L_{x}^{p}} \overset{\eqref{[Equ. (75), Y20c]}}{\leq} 12 ((8 \pi^{3} )^{\frac{1}{p}} c_{R} \delta_{q+1} M_{0}(t) + \lVert \mathring{R}_{l} \rVert_{C_{t}L_{x}^{p}}), \hspace{3mm} \lVert \mathring{R}_{l} \rVert_{C_{t,x}^{N}} \overset{\eqref{[Equ. (49c), Y20c]}}{\lesssim} l^{-4-N } M_{0}(t) c_{R} \delta_{q+1}; 
\end{equation} 
it follows that for all $k \in \{0,1,2\}$, 
\begin{equation}\label{estimate 343}
\lVert \rho \rVert_{C_{t}C_{x}^{N}} \overset{\eqref{estimate 342} \eqref{[Equ. (49c), Y20c]} }{\lesssim} c_{R} \delta_{q+1}M_{0}(t) l^{-4-N}, \hspace{3mm} \lVert \rho \rVert_{C_{t}^{1}C_{x}^{k}} \overset{\eqref{[Equ. (75), Y20c]}\eqref{estimate 342} }{\lesssim} c_{R} \delta_{q+1} M_{0}(t) l^{-5 (k+1)} 
\end{equation} 
(see \cite[Equ. (67)-(69)]{Y20a}). We recall $\gamma_{\zeta}$ from Lemma \ref{[Lem. B.1, HZZ19]} and define an amplitude function slightly differently from \eqref{[Equ. (81), Y20c]} as follows:
\begin{equation}\label{estimate 344}
a_{\zeta} (\omega, t, x) \triangleq a_{\zeta, q+1} (\omega, t, x) \triangleq \rho(\omega, t, x)^{\frac{1}{2}} \gamma_{\zeta} ( \text{Id} - \frac{\mathring{R}_{l}(\omega, t, x)}{\rho(\omega, t, x)}) (2\pi)^{-\frac{3}{4}}.
\end{equation} 
It follows that the same estimate in \eqref{[Equ. (83), Y20c]} holds and additionally for all $t \in [0, T_{L}], N \geq 0$, $k \in \{0,1,2\}$, and $C_{\Lambda}$ and $M$ from \eqref{[Equ. (163), Y20a]}, 
\begin{equation}\label{[Equ. (73), Y20a]}
\lVert a_{\zeta} \rVert_{C_{t}C_{x}^{N}} \overset{\eqref{estimate 343} \eqref{estimate 344}}{\leq} c_{R}^{\frac{1}{4}} \delta_{q+1}^{\frac{1}{2}} M_{0}(t)^{\frac{1}{2}} l^{- 2 - 5N}, \hspace{5mm} \lVert a_{\zeta} \rVert_{C_{t}^{1}C_{x}^{k}} \overset{ \eqref{estimate 342} \eqref{estimate 343} }{\leq} c_{R}^{\frac{1}{4}} \delta_{q+1}^{\frac{1}{2}} M_{0}(t)^{\frac{1}{2}} l^{-(k+1) 5}
\end{equation}
(see \cite[Equ. (72)-(73)]{Y20a}). Next, we define 
\begin{subequations}
\begin{align}
&w_{q+1}^{(p)} \triangleq \sum_{\zeta \in \Lambda} a_{\zeta} W_{\zeta},  \hspace{3mm}  w_{q+1}^{(c)} \triangleq \sum_{\zeta \in \Lambda} \text{curl} (\nabla a_{\zeta} \times V_{\zeta}) + \nabla a_{\zeta} \times \text{curl} V_{\zeta} + a_{\zeta} W_{\zeta}^{(c)}, \label{[Equ. (74a) and (74b), Y20a]} \\
& w_{q+1}^{(t)} \triangleq - \mu^{-1} \sum_{\zeta \in \Lambda} \mathbb{P} \mathbb{P}_{\neq 0} (a_{\zeta}^{2} \phi_{\zeta}^{2} \psi_{\zeta}^{2} \zeta), \label{[Equ. (74c), Y20a]}
\end{align}
\end{subequations} 
where $\phi_{\zeta}, \psi_{\zeta}, W_{\zeta}, V_{\zeta}, W_{\zeta}^{(c)}$ are defined in \eqref{[Equ. (167), Y20a]}, \eqref{[Equ. (168), Y20a]}, and \eqref{[Equ. (173), Y20a]}, and we define $w_{q+1}$ and $v_{q+1}$ identically as \eqref{[Equ. (85), Y20c]}. We can estimate for all $p \in (1, \infty)$ and $t \in [0, T_{L}]$ 
\begin{subequations}
\begin{align}
& \lVert w_{q+1}^{(p)} \rVert_{C_{t}L_{x}^{2}} \overset{\eqref{[Equ. (74a) and (74b), Y20a]}}{\leq}  \frac{M_{0}(t)^{\frac{1}{2}}}{2} \delta_{q+1}^{\frac{1}{2}},  \hspace{10mm}  \lVert w_{q+1}^{(p)} \rVert_{C_{t}L_{x}^{p}}  \overset{\eqref{[Equ. (73), Y20a]} \eqref{[Equ. (174), Y20a]}}{\lesssim}  M_{0}(t)^{\frac{1}{2}} \delta_{q+1}^{\frac{1}{2}} l^{-2} r_{\bot}^{\frac{2}{p} -1} r_{\lVert}^{\frac{1}{p} - \frac{1}{2}}, \label{[Equ. (77) and (78a), Y20a]}\\
&  \lVert w_{q+1}^{(c)} \rVert_{C_{t}L_{x}^{p}}  \overset{\eqref{[Equ. (174), Y20a]}}{\lesssim} M_{0}(t)^{\frac{1}{2}} \delta_{q+1}^{\frac{1}{2}} l^{-12} r_{\bot}^{\frac{2}{p}} r_{\lVert}^{\frac{1}{p} - \frac{3}{2}},   \lVert w_{q+1}^{(t)} \rVert_{C_{t}L_{x}^{p}} \lesssim  \delta_{q+1} M_{0}(t) l^{-4} r_{\bot}^{\frac{2}{p} -1} r_{\lVert}^{\frac{1}{p} - 2} \lambda_{q+1}^{1-2m} \label{[Equ. (78b) and (78c), Y20a]}
\end{align} 
\end{subequations}
(see \cite[Equ. (77)-(78)]{Y20a}). These lead us furthermore to, for all $t \in [0, T_{L}]$ and $p \in (1,\infty)$ 
\begin{equation}\label{[Equ. (80) and (81), Y20a]}
 \lVert w_{q+1}^{(c)} \rVert_{C_{t}L_{x}^{p}} + \lVert w_{q+1}^{(t)} \rVert_{C_{t}L_{x}^{p}} \overset{\eqref{[Equ. (78b) and (78c), Y20a]}}{\lesssim} M_{0}(t)^{\frac{1}{2}} \delta_{q+1}^{\frac{1}{2}} l^{-2} r_{\bot}^{\frac{2}{p} -1} r_{\lVert}^{\frac{1}{p} - \frac{1}{2}}, \hspace{1mm} \lVert w_{q+1} \rVert_{C_{t}L_{x}^{2}} \overset{\eqref{[Equ. (85), Y20c]}}{\leq} \frac{3}{4} M_{0}(t)^{\frac{1}{2}} \delta_{q+1}^{\frac{1}{2}}
\end{equation}
(see \cite[Equ. (80)-(81)]{Y20a}). We may now verify the first inequality of \eqref{[Equ. (49a) and (49b), Y20c]} at level $q+1$ via \eqref{[Equ. (74b) and (74c), Y20c]}, \eqref{[Equ. (85), Y20c]} and \eqref{[Equ. (80) and (81), Y20a]} while \eqref{[Equ. (45), Y20a]} by \eqref{[Equ. (74a), Y20c]}, \eqref{[Equ. (85), Y20c]}, and \eqref{[Equ. (80) and (81), Y20a]}. Moreover, for all $t \in [0, T_{L}]$ we can compute 
\begin{subequations}\label{estimate 146}
\begin{align}
&  \lVert w_{q+1}^{(p)} \rVert_{C_{t,x}^{1}} \overset{\eqref{[Equ. (73), Y20a]} \eqref{[Equ. (174c), Y20a]}}{\lesssim} M_{0}(t)^{\frac{1}{2}} l^{-7} r_{\bot}^{-1} r_{\lVert}^{-\frac{1}{2}} \lambda_{q+1}^{2m},   \lVert w_{q+1}^{(c)} \rVert_{C_{t,x}^{1}} \overset{\eqref{[Equ. (73), Y20a]} \eqref{[Equ. (174c), Y20a]}}{\lesssim} M_{0}(t)^{\frac{1}{2}} l^{-17} r_{\lVert}^{-\frac{3}{2}} \lambda_{q+1}^{2m}, \label{[Equ. (82a) and (82b), Y20a]}\\
& \lVert w_{q+1}^{(t)} \rVert_{C_{t,x}^{1}} \overset{\eqref{[Equ. (73), Y20a]} \eqref{[Equ. (174a) and (174b), Y20a]}}{\lesssim}  M_{0}(t) l^{-9} r_{\bot}^{-1} r_{\lVert}^{-2} \lambda_{q+1}^{-2m + 2 + \alpha} (\frac{r_{\bot} \mu}{r_{\lVert}}) \label{[Equ. (85), Y20a]}
\end{align}
\end{subequations}
(see \cite[Equ. (82) and (85)]{Y20a}). Now, applying \eqref{[Equ. (74b) and (74c), Y20c]}, \eqref{[Equ. (85), Y20c]} and \eqref{estimate 146} allows us to verify the second inequality of \eqref{[Equ.  (49a) and (49b), Y20c]} at level $q+1$ (see \cite[Equ. (86)-(87)]{Y20a}). Moreover, with $\theta^{\text{in}}\in H^{2}(\mathbb{T}^{3})$ from the hypothesis and $v_{q+1}$ already constructed via \eqref{[Equ. (85), Y20c]}, identically to the 2D case, we deduce the unique solution $\theta_{q+1}$ to \eqref{estimate 323} starting from $\theta_{q+1}(0,x)  = \theta^{\text{in}}(x)$  that satisfies \eqref{estimate 26}. Concerning the Cauchy estimate \eqref{estimate 59}, as we discussed in Remark \ref{Remark 4.2}, due to $\dot{H}^{1}(\mathbb{T}^{3})\hookrightarrow L^{p}(\mathbb{T}^{3})$ being false for $p > 6$, we cannot rely on $L_{\omega}^{2}L_{t}^{2} \dot{H}_{x}^{1}$-bound of $\theta_{q}$. The break here is that we can rely on the fact that $\theta_{q} \in L_{\omega}^{p} C_{T_{L}} L_{x}^{p}$ for every $p \in [1,\infty)$ with a bound by a constant $C(\lVert \theta^{\text{in}} \rVert_{H_{x}^{2}}, p, \text{Tr} ((-\Delta)^{\frac{3}{2} + 2 \sigma} G_{2}G_{2}^{\ast}), L)$, which we verified (recall \eqref{estimate 114} and \eqref{estimate 331}). To make this argument precise, we start from \eqref{estimate 27} and compute 
\begin{align}
\frac{1}{2}\partial_{t} \lVert \theta_{q+1} - \theta_{q} \rVert_{L_{x}^{2}}^{2} + \lVert \theta_{q+1} - \theta_{q} \rVert_{\dot{H}_{x}^{1}}^{2} 
\lesssim& \lVert v_{q+1} - v_{q} \rVert_{L_{x}^{2}}^{\frac{24 + \beta}{2 (12 + 2 \beta)}} \lVert v_{q+1} - v_{q} \rVert_{\dot{H}_{x}^{1}}^{\frac{3\beta}{2 (12 + 2\beta)}} \lVert \theta_{q+1} - \theta_{q} \rVert_{\dot{H}_{x}^{1}} \lVert \theta_{q}\rVert_{L_{x}^{\frac{2 (12 + 2\beta)}{\beta}}} \nonumber \\
\overset{\eqref{[Equ. (45), Y20a]}\eqref{[Equ.  (49a) and (49b), Y20c]} }{\leq}&  \frac{1}{2} \lVert \theta_{q+1} - \theta_{q} \rVert_{\dot{H}_{x}^{1}}^{2} + C M_{0}(t)\lambda_{q+1}^{\frac{ - 21 \beta - \beta^{2}}{12 + 2\beta}} \lVert \theta_{q} \rVert_{L_{x}^{\frac{2(12 + 2\beta)}{\beta}}}^{2} \label{estimate 0}
\end{align} 
where we relied on H$\ddot{\mathrm{o}}$lder's, Gagliardo-Nirenberg's and Young's inequalities (cf. \eqref{estimate 38}). Subtracting $\frac{1}{2} \lVert \theta_{q+1} - \theta_{q} \rVert_{\dot{H}_{x}^{1}}^{2}$ from both sides, integrating over time $[0,t]$, taking supremum over $[0,t]$ on the right and then left hand sides lead us to 
\begin{equation*}
\lVert \theta_{q+1} - \theta_{q} \rVert_{C_{t}L_{x}^{2}}^{2} + \lVert \theta_{q+1} - \theta_{q} \rVert_{L_{t}^{2} \dot{H}_{x}^{1}}^{2} \lesssim M_{0}(t) \lambda_{q+1}^{\frac{- 21 \beta - \beta^{2}}{12 + 2 \beta}} t \lVert \theta_{q} \rVert_{C_{t} L_{x}^{\frac{ 2(12 + 2 \beta)}{\beta}}}^{2}. 
\end{equation*} 
Raising to the power of $p\in [1, \infty)$, taking expectation $\mathbb{E}^{\textbf{P}}$, and relying on the fact that $\theta_{q} \in L_{\omega}^{p} C_{T_{L}} L_{x}^{p}$ for all $p \in [1, \infty)$ give us \eqref{estimate 59}. 

\subsubsection{Reynolds stress} 
Identically to \eqref{[Equ. (98), Y20c]}-\eqref{estimate 58} due to \eqref{[Equ. (46), Y20c]}, \eqref{[Equ. (85), Y20c]}, and \eqref{estimate 341}, we can define $\mathring{R}_{q+1}$ and $\pi_{q+1}$ with 
\begin{subequations}\label{[Equ. (91), Y20a]}
\begin{align}
 R_{\text{lin}} \triangleq& \mathcal{R} ( -\Delta)^{m} w_{q+1} + \mathcal{R}\partial_{t} (w_{q+1}^{(p)} + w_{q+1}^{(c)}) \nonumber\\
 & + (v_{l} + z_{1, l}) \mathring{\otimes} w_{q+1} + w_{q+1} \mathring{\otimes} (v_{l} + z_{1, l}) + \mathcal{R} ((\theta_{l} - \theta_{q+1})e^{3}), \label{[Equ. (91a), Y20a]}\\
 \pi_{\text{lin}} \triangleq& (\frac{2}{3})(v_{l} + z_{1,l}) \cdot w_{q+1}, \label{[Equ. (91b), Y20a]}\\
 R_{\text{cor}} \triangleq& (w_{q+1}^{(c)} + w_{q+1}^{(t)}) \mathring{\otimes} w_{q+1} + w_{q+1}^{(p)} \mathring{\otimes} (w_{q+1}^{(c)} + w_{q+1}^{(t)}), \label{[Equ. (91c), Y20a]}\\
 \pi_{\text{cor}} \triangleq& \frac{1}{3} [ (w_{q+1}^{(c)} + w_{q+1}^{(t)}) \cdot w_{q+1} + w_{q+1}^{(p)} \cdot (w_{q+1}^{(c)} + w_{q+1}^{(t)}) ], \label{[Equ. (91d), Y20a]}\\
R_{\text{com2}} \triangleq& v_{q+1} \mathring{\otimes} (z_{1}-z_{1,l}) + (z_{1}-z_{1,l}) \mathring{\otimes} v_{q+1} + z_{1} \mathring{\otimes} z_{1} - z_{1,l} \mathring{\otimes} z_{1,l}, \label{[Equ. (91e), Y20a]}\\
\pi_{\text{com2}} \triangleq& \frac{1}{3} [2 v_{q+1} \cdot (z_{1}- z_{1,l}) + \lvert z_{1} \rvert^{2} - \lvert z_{1,l} \rvert^{2}], \\
R_{\text{osc}} \triangleq&  \sum_{\zeta \in \Lambda} \mathcal{R}(\nabla a_{\zeta}^{2} \mathbb{P}_{\neq 0} (W_{\zeta} \otimes W_{\zeta})) - \mu^{-1} \sum_{\zeta \in \Lambda} \mathbb{P}_{\neq 0} (\partial_{t} a_{\zeta}^{2} (\phi_{\zeta}^{2} \psi_{\zeta}^{2} \zeta )), \\
\pi_{\text{osc}} \triangleq& \rho + \Delta^{-1} \text{div} \mu^{-1} \sum_{\zeta \in \Lambda} \mathbb{P}_{\neq 0} \partial_{t} (a_{\zeta}^{2} \phi_{\zeta}^{2}\psi_{\zeta}^{2} \zeta). \label{[Equ. (91g), Y20a]}
\end{align}
\end{subequations}
Differently from \eqref{[Equ. (116), Y20c]} we set 
\begin{equation}\label{[Equ. (93), Y20a]}
p^{\ast} \triangleq \frac{40m - 14}{170 \alpha - 19 + 44m} \overset{\eqref{[Equ. (52), Y20a]}}{\in} (1,2). 
\end{equation}
Identically to \eqref{estimate 63}-\eqref{estimate 36}, we see that for all $t \in [0, T_{L}]$ 
\begin{equation}\label{estimate 68}
\lVert \mathcal{R}(( \theta_{l} - \theta_{q+1}) e^{3}) \rVert_{C_{t}L_{x}^{p^{\ast}}} \leq III + IV  
\end{equation} 
where 
\begin{equation}
III \triangleq \lVert \mathcal{R} ((\theta_{q} - \theta_{q+1}) e^{3}) \rVert_{C_{t}L_{x}^{p^{\ast}}} \text{ and } IV \triangleq \lVert \mathcal{R} ((\theta_{l} - \theta_{q}) e^{3} ) \rVert_{C_{t}L_{x}^{p^{\ast}}}. 
\end{equation} 
To handle $III$, identically to the 2D case, we can deduce \eqref{estimate 49}.To handle the time integral of $\lVert \theta_{q} \rVert_{W_{x}^{1,\infty}}$ within \eqref{estimate 49}, identically to the 2D case, we can deduce \eqref{estimate 47}. Then, to handle the integral of $\lVert \Theta_{q} \rVert_{\dot{H}_{x}^{1}}^{2}$ within \eqref{estimate 47}, we apply  \eqref{estimate 359} and an observation that $2 M_{0}(t)^{\frac{1}{2}} L^{\frac{1}{4}} + (2\pi)^{\frac{3}{2}} L^{\frac{1}{2}} \leq (2\pi)^{\frac{3}{2}} M_{0}(t)$ by \eqref{estimate 339} to \eqref{estimate 42} and deduce  
\begin{equation*}
\lVert \Theta_{q} \rVert_{C_{t}L_{x}^{2}} \overset{\eqref{estimate 359}}{\leq} \lVert \theta^{\text{in}} \rVert_{L_{x}^{2}} + t [2M_{0}(t)^{\frac{1}{2}} + (2\pi)^{\frac{3}{2}} L^{\frac{1}{4}}]L^{\frac{1}{4}} \leq \lVert \theta^{\text{in}} \rVert_{L_{x}^{2}} + (2\pi)^{\frac{3}{2}} M_{0}(t) t 
\end{equation*} 
so that 
\begin{equation}\label{estimate 152}
\int_{0}^{t} \lVert \Theta_{q} \rVert_{\dot{H}_{x}^{1}}^{2} dr \overset{\eqref{[Equ.  (49a) and (49b), Y20c]} \eqref{[Equ. (38), Y20a]}}{\leq} \frac{1}{2} \lVert \theta^{\text{in}} \rVert_{L_{x}^{2}}^{2} + (2\pi)^{\frac{3}{2}} M_{0}(t) [ \lVert \theta^{\text{in}} \rVert_{L_{x}^{2}} + (2\pi)^{\frac{3}{2}} M_{0}(t) t] t.
\end{equation} 
We apply \eqref{estimate 152} to \eqref{estimate 47} and deduce for all $t \in [0, T_{L}]$ 
\begin{equation}\label{estimate 67}
\int_{0}^{t} \lVert \Theta_{q}\rVert_{\dot{H}_{x}^{3}}^{2} dr 
\lesssim M_{0}(t) \lambda_{q}^{8} [ \lVert \theta^{\text{in}} \rVert_{H_{x}^{2}}^{2} + M_{0}(t)^{2} t^{2} + L^{\frac{1}{2}}].  
\end{equation} 
Similarly to \eqref{estimate 50}, applying \eqref{estimate 67} to \eqref{estimate 49} gives for all $t \in [0, T_{L}]$ 
\begin{equation}\label{estimate 124}
\lVert \theta_{q+1} - \theta_{q} \rVert_{C_{t}L_{x}^{p^{\ast}}} \overset{\eqref{[Equ. (38), Y20a]}}{\lesssim} \lVert v_{q+1} - v_{q} \rVert_{C_{t}L_{x}^{p^{\ast}}} \sqrt{T_{L}} (M_{0}(t)^{\frac{1}{2}} \lambda_{q}^{4} ( \lVert \theta^{\text{in}} \rVert_{H_{x}^{2}} + M_{0}(t)T_{L} +  L^{\frac{1}{4}}) + \sqrt{T_{L}} L^{\frac{1}{4}}).
\end{equation} 
We split $\lVert v_{q+1} - v_{q} \rVert_{C_{t}L_{x}^{p^{\ast}}}$ identically to \eqref{estimate 51}: 
\begin{equation}\label{estimate 360}
\lVert v_{q+1} - v_{q} \rVert_{C_{t}L_{x}^{p^{\ast}}} \leq  III_{1} + III_{2} \text{ where } III_{1} \triangleq \lVert v_{l} - v_{q} \rVert_{C_{t}L_{x}^{p^{\ast}}} \text{ and } III_{2} \triangleq \lVert v_{q+1} - v_{l} \rVert_{C_{t}L_{x}^{p^{\ast}}} 
\end{equation}
where $III_{1} \lesssim \lambda_{q+1}^{-\alpha} M_{0}(t)^{\frac{1}{2}}$ by identical estimates in \eqref{estimate 52}. On the other hand, we estimate $III_{2}$ carefully as follows:
\begin{align}
III_{2}\overset{\eqref{[Equ. (85), Y20c]}}{\leq}& \lVert w_{q+1}^{(p)} \rVert_{C_{t}L_{x}^{p^{\ast}}} + \lVert w_{q+1}^{(c)} \rVert_{C_{t}L_{x}^{p^{\ast}}} + \lVert w_{q+1}^{(t)} \rVert_{C_{t}L_{x}^{p^{\ast}}} \label{estimate 123} \\
\overset{\eqref{[Equ. (77) and (78a), Y20a]} \eqref{[Equ. (80) and (81), Y20a]}}{\lesssim}&   M_{0}(t)^{\frac{1}{2}} \delta_{q+1}^{\frac{1}{2}} l^{-2} r_{\bot}^{\frac{2}{p^{\ast}} - 1} r_{\lVert}^{\frac{1}{p^{\ast}} - \frac{1}{2}} 
\overset{\eqref{[Equ. (69), Y20c]} \eqref{[Equ. (62), Y20a]}}{\lesssim}  M_{0}(t)^{\frac{1}{2}} \delta_{q+1}^{\frac{1}{2}} \lambda_{q+1}^{\frac{ - 67 \alpha + 6 - 12m}{6}} \lambda_{q}^{4} \overset{\eqref{estimate 62} }{\lesssim} M_{0}(t)^{\frac{1}{2}} \delta_{q+1}^{\frac{1}{2}} \lambda_{q+1}^{ \frac{ - 65 \alpha + 6 - 12m}{6}}  \nonumber
\end{align} 
where we used that 
\begin{align*}
3 \alpha + ( \frac{1- 20m}{24}) (\frac{2}{p^{\ast}} - 1) + (\frac{13 - 20m}{12}) (\frac{1}{p^{\ast}} - \frac{1}{2})\overset{\eqref{[Equ. (93), Y20a]}}{=}\frac{- 67 \alpha + 6 - 12m}{6} 
\end{align*} 
and that $\frac{4}{b} < \frac{\alpha}{3}$ due to \eqref{estimate 62}. By applying \eqref{estimate 52} and \eqref{estimate 123} to \eqref{estimate 360} and then to \eqref{estimate 124} and relying on Lemma \ref{divergence inverse operator}, we deduce similarly to \eqref{estimate 56}-\eqref{estimate 64} that for $a > 0$ sufficiently large 
\begin{align}
III \overset{\eqref{estimate 124} \eqref{estimate 360}}{\lesssim}&  (III_{1} + III_{2}) \sqrt{T_{L}} (M_{0}(t)^{\frac{1}{2}} \lambda_{q}^{4} ( \lVert \theta^{\text{in}} \rVert_{H_{x}^{2}} + M_{0}(t)T_{L} + L^{\frac{1}{4}}) + \sqrt{T_{L}} L^{\frac{1}{4}}) \nonumber\\
\overset{\eqref{estimate 52} \eqref{estimate 123}}{\lesssim}& c_{R} \delta_{q+2} M_{0}(t) [ a^{b^{q+1} [ 2 \beta b - \alpha + \frac{4}{b}]} + a^{b^{q+1} [ 2 \beta b + \frac{- 65 \alpha + 6 - 12m}{6} + \frac{4}{b} ]}] \nonumber\\
& \hspace{5mm}  \times \sqrt{T_{L}} [ \lVert \theta^{\text{in}} \rVert_{H_{x}^{2}} + M_{0}(t) T_{L} + L^{\frac{1}{4}} + \sqrt{T_{L}} L^{\frac{1}{4}}]  \ll c_{R} \delta_{q+2} M_{0}(t)  \label{estimate 87} 
\end{align} 
where we used that 
\begin{subequations}
\begin{align}
&2 \beta b - \alpha + \frac{4}{b} \overset{\eqref{[Equ. (68), Y20c]} \eqref{estimate 62} }{<} -\frac{35\alpha}{48}, \label{estimate 130}\\
& 2 \beta b + \frac{ - 65 \alpha + 6 - 12m}{6} + \frac{4}{b} \overset{\eqref{estimate 62}}{<} - \frac{507 \alpha}{48} + 1 - 2m. \label{estimate 131}
\end{align}
\end{subequations}
On the other hand, it is clear that we can handle $IV$ in \eqref{estimate 68} by identical arguments that led to \eqref{estimate 349}. Thus, we conclude by applying \eqref{estimate 87} and \eqref{estimate 349} to \eqref{estimate 68}  that for all $t \in [0,T_{L}]$ 
\begin{equation}\label{estimate 361}
\lVert \mathcal{R} (( \theta_{l} - \theta_{q+1}) e^{3} ) \rVert_{C_{t}L_{x}^{p^{\ast}}} \overset{\eqref{estimate 68}}{\leq} III + IV \overset{\eqref{estimate 87}}{\ll} c_{R} M_{0}(t) \delta_{q+2}.
\end{equation} 
It follows from \cite[Equ. (91a), (94)-(98)]{Y20a} that 
\begin{equation*}
\lVert R_{\text{lin}} - \mathcal{R} ((\theta_{l} - \theta_{q+1}) e^{3}) \rVert_{L_{x}^{p^{\ast}}} \leq (2\pi)^{-3 ( \frac{p^{\ast} -1}{p^{\ast}})} \frac{c_{R}M_{0}(t) \delta_{q+2}}{10}
\end{equation*} 
so that together with \eqref{estimate 361} we obtain 
\begin{equation}\label{estimate 380}
\lVert R_{\text{lin}} \rVert_{L_{x}^{p^{\ast}}} \leq (2\pi)^{-3 ( \frac{p^{\ast} - 1}{p^{\ast}})} \frac{c_{R}M_{0}(t) \delta_{q+2}}{5}. 
\end{equation} 
It follows from \cite[Equ. (91g), (100), (103)-(106)]{Y20a} that 
\begin{align}
& \lVert R_{\text{cor}} \rVert_{C_{t}L_{x}^{p^{\ast}}} + \lVert R_{\text{osc}} \rVert_{C_{t}L_{x}^{p^{\ast}}} + \lVert R_{\text{com1}} \rVert_{C_{t}L_{x}^{1}} + \lVert R_{\text{com2}} \rVert_{C_{t}L_{x}^{1}} \nonumber\\
\leq& \frac{ 2 ( 2\pi)^{-3 ( \frac{p^{\ast} - 1}{p^{\ast}})} c_{R} M_{0}(t) \delta_{q+2}}{5} + \frac{ 2 c_{R} M_{0}(t) \delta_{q+2}}{5}; \label{estimate 381} 
\end{align} 
hence, together with \eqref{estimate 380}, we now have $\lVert \mathring{R}_{q+1} \rVert_{C_{t}L_{x}^{1}} \leq c_{R} M_{0}(t) \delta_{q+2}$, which is \eqref{[Equ. (49c), Y20c]} at level $q+1$ as desired. The rest of the arguments are similar to those of proof of Proposition \ref{Proposition 4.8 for n=2}. 

\section{Proofs of Theorems \ref{Theorem 2.3}-\ref{Theorem 2.4}}

We recall the definitions of $U_{1}, U_{2}, \tilde{U}_{1}, \tilde{U}_{2}, \bar{\Omega}$ and $\bar{\mathcal{B}}_{t}$ from Section \ref{Section 3} and define the following: 
\begin{define}\label{Definition 5.1}
Let $s \geq 0$, $\xi^{\text{in}} = (\xi_{1}^{\text{in}}, \xi_{2}^{\text{in}}) \in L_{\sigma}^{2} \times \mathring{L}^{2}$, and $\zeta^{\text{in}} = (\zeta_{1}^{\text{in}}, \zeta_{2}^{\text{in}}) \in \tilde{U}_{1} \times \tilde{U}_{2}$. A probability measure $P \in \mathcal{P} (\bar{\Omega})$ is a probabilistically weak solution to \eqref{3} with initial condition $(\xi^{\text{in}}, \zeta^{\text{in}})$ at initial time $s$ if 
\begin{enumerate}
\item [] (M1) $P(\{ \xi(t) = \xi^{\text{in}}, \zeta(t) = \zeta^{\text{in}} \hspace{1mm} \forall \hspace{1mm} t \in [0,s] \}) = 1$ and for all $l \in \mathbb{N}$, 
\begin{equation}
P ( \{ (\xi, \zeta) \in\bar{\Omega}: \hspace{0.5mm} \int_{0}^{l} \lVert F_{1} (\xi_{1}(r)) \rVert_{L_{2}(U_{1}, L_{\sigma}^{2})}^{2} + \lVert F_{2}(\xi_{2}(r)) \rVert_{L_{2} (U_{2}, \mathring{L}^{2})}^{2} dr < \infty \}) = 1, 
\end{equation} 
\item [] (M2) under $P$, $\zeta = (\zeta_{1}, \zeta_{2})$ are cylindrical $(\bar{\mathcal{B}}_{t})_{t\geq s}$-Wiener processes on $U_{1} \times U_{2}$ starting from initial condition $\zeta^{\text{in}}$ at initial time $s$ and for every $\psi_{i} = (\psi_{i}^{1}, \psi_{i}^{2}) \in C^{\infty} (\mathbb{T}^{n}) \cap L_{\sigma}^{2} \times C^{\infty} (\mathbb{T}^{n}) \times \mathring{L}^{2}$ and $t \geq s$, 
\begin{subequations}
\begin{align}
& \langle \xi_{1}(t) - \xi_{1}(s), \psi_{i}^{1} \rangle + \int_{s}^{t} \langle \text{div} (\xi_{1}(r) \otimes \xi_{1}(r)) + (-\Delta)^{m} \xi_{1}(r) - \xi_{2}(r)e^{n}, \psi_{i}^{1} \rangle dr  \nonumber\\
& \hspace{50mm} = \int_{s}^{t} \langle \psi_{i}^{1}, F_{1}(\xi_{1}(r)) d\zeta_{1}(r) \rangle, \\
& \langle \xi_{2}(t)- \xi_{2}(s), \psi_{i}^{2} \rangle + \int_{s}^{t} \langle \text{div} (\xi_{1}(r) \xi_{2}(r)) - \Delta \xi_{2}(r), \psi_{i}^{2} \rangle dr \nonumber\\
& \hspace{50mm} = \int_{s}^{t} \langle \psi_{i}^{2}, F_{2}(\xi_{2}(r)) d \zeta_{2}(r) \rangle, 
\end{align}
\end{subequations} 
\item [] (M3) for any $q \in \mathbb{N}$ there exists a function $t \mapsto C_{t,q} \in \mathbb{R}_{+}$ such that for all $t \geq s$, 
\begin{align}
&\mathbb{E}^{P} [ \sup_{r \in [0,t]} \lVert \xi_{1}(r) \rVert_{L_{x}^{2}}^{2q} + \int_{s}^{t} \lVert \xi_{1}(r) \rVert_{\dot{H}_{x}^{\gamma}}^{2} dr \nonumber \\
& \hspace{2mm} + \sup_{r \in [0,t]} \lVert \xi_{2}(r) \rVert_{L_{x}^{2}}^{2q} + \int_{s}^{t} \lVert \xi_{2} (r) \rVert_{\dot{H}_{x}^{1}}^{2} dr] \leq C_{t,q} (1+ \lVert \xi_{1}^{\text{in}} \rVert_{L_{x}^{2}}^{2q} + \lVert \xi_{2}^{\text{in}} \rVert_{L_{x}^{2}}^{2q}). \label{estimate 125}
\end{align}  
\end{enumerate} 
The set of all such probabilistically weak solutions with the same constant $C_{t,q}$ in \eqref{estimate 125} for every $q \in \mathbb{N}$ and $t \geq s$ will be denoted by $\mathcal{W} ( s, \xi^{\text{in}}, \zeta^{\text{in}}, \{C_{t,q} \}_{q \in \mathbb{N}, t \geq s } )$. 
\end{define}

For any stopping time $\tau$, we set 
\begin{equation}\label{estimate 126}
\bar{\Omega}_{\tau} \triangleq \{ \omega( \cdot \wedge \tau(\omega)): \hspace{0.5mm} \omega \in \bar{\Omega} \} = \{ \omega \in \bar{\Omega}: \hspace{0.5mm} (\xi, \zeta)(t,\omega) = (\xi, \zeta)(t\wedge \tau(\omega), \omega) \} 
\end{equation} 
and denote the $\sigma$-algebra associated to $\tau$ by $(\bar{\mathcal{B}}_{\tau})$. 

\begin{define}\label{Definition 5.2}
Let $s \geq 0, \xi^{\text{in}} = ( \xi_{1}^{\text{in}}, \xi_{2}^{\text{in}}) \in L_{\sigma}^{2} \times \mathring{L}^{2}$, and $\zeta^{\text{in}} = (\zeta_{1}^{\text{in}}, \zeta_{2}^{\text{in}}) \in \tilde{U}_{1} \times \tilde{U}_{2}$. Let $\tau \geq s$ be a stopping time of $(\bar{\mathcal{B}}_{t})_{t\geq s}$. A probability measure $P \in \mathcal{P} (\bar{\Omega}_{\tau})$ is a probabilistically weak solution to \eqref{3} on $[s, \tau]$ with initial condition $(\xi^{\text{in}}, \zeta^{\text{in}})$ at initial time $s$ if 
\begin{enumerate}
\item [] (M1) $P( \{ \xi(t) = \xi^{\text{in}}, \zeta(t) = \zeta^{\text{in}} \hspace{1mm} \forall \hspace{1mm} t \in [0,s] \}) = 1$ and for all $l \in \mathbb{N}$, 
\begin{equation}
P(\{ ( \xi, \zeta) \in \bar{\Omega}: \hspace{0.5mm} \int_{0}^{l \wedge \tau} \lVert F_{1}(\xi_{1}(r)) \rVert_{L_{2}(U_{1}, L_{\sigma}^{2})}^{2} + \lVert F_{2} (\xi_{2} (r)) \rVert_{L_{2}(U_{2}, \mathring{L}^{2})}^{2} dr < \infty \}) = 1, 
\end{equation} 
\item [] (M2) under $P$, for both $k \in \{1,2\}$, $\langle \zeta_{k}(\cdot \wedge \tau), l_{i}^{k} \rangle_{U_{k}}$, where $\{l_{i}^{k}\}_{i \in \mathbb{N}}$ is an orthonormal basis of $U_{k}$, is a continuous, square-integrable $(\bar{\mathcal{B}}_{t})_{t\geq s}$-martingale with initial condition $\langle \zeta_{k}^{\text{in}}, l_{i}^{k} \rangle_{U_{k}}$ at initial time $s$ with its quadratic variation process given by $(t\wedge \tau - s) \lVert l_{i} ^{k}\rVert_{U_{k}}^{2}$ and for every $\psi_{i} = (\psi_{i}^{1}, \psi_{i}^{2}) \in C^{\infty} (\mathbb{T}^{n}) \cap L_{\sigma}^{2} \times C^{\infty} (\mathbb{T}^{n}) \cap \mathring{L}^{2}$ and $t \geq s$, 
\begin{subequations}\label{[Equ. (136), Y20c]}
\begin{align}
& \langle \xi_{1}(t\wedge \tau) - \xi_{1}(s), \psi_{i}^{1} \rangle + \int_{s}^{t\wedge \tau} \langle \text{div} (\xi_{1}(r) \otimes \xi_{1}(r)) + (-\Delta)^{m} \xi_{1}(r) - \xi_{2}(r)e^{n}, \psi_{i}^{1} \rangle dr  \nonumber\\
& \hspace{50mm} = \int_{s}^{t\wedge \tau} \langle \psi_{i}^{1}, F_{1}(\xi_{1}(r)) d\zeta_{1}(r) \rangle, \\
& \langle \xi_{2}(t\wedge \tau)- \xi_{2}(s), \psi_{i}^{2} \rangle + \int_{s}^{t\wedge \tau} \langle \text{div} (\xi_{1}(r) \xi_{2}(r)) - \Delta \xi_{2}(r), \psi_{i}^{2} \rangle dr \nonumber \\
& \hspace{50mm} = \int_{s}^{t\wedge \tau} \langle \psi_{i}^{2}, F_{2}(\xi_{2}(r)) d \zeta_{2}(r) \rangle, 
\end{align}
\end{subequations} 
\item [] (M3) for any $q \in \mathbb{N}$ there exists a function $t \mapsto C_{t,q} \in \mathbb{R}_{+}$ such that for all $t \geq s$,   
\begin{align}
&\mathbb{E}^{P} [ \sup_{r \in [0,t\wedge \tau]} \lVert \xi_{1}(r) \rVert_{L_{x}^{2}}^{2q} + \int_{s}^{t\wedge \tau} \lVert \xi_{1}(r) \rVert_{\dot{H}_{x}^{\gamma}}^{2} dr \nonumber \\
& \hspace{2mm} + \sup_{r \in [0,t\wedge \tau]} \lVert \xi_{2}(r) \rVert_{L_{x}^{2}}^{2q} + \int_{s}^{t\wedge \tau} \lVert \xi_{2} (r) \rVert_{\dot{H}_{x}^{1}}^{2} dr] \leq C_{t,q} (1+ \lVert \xi_{1}^{\text{in}} \rVert_{L_{x}^{2}}^{2q} + \lVert \xi_{2}^{\text{in}} \rVert_{L_{x}^{2}}^{2q}). 
\end{align}  
\end{enumerate} 
\end{define}

\begin{proposition}\label{Proposition 5.1}
For every $(s, \xi^{\text{in}}, \zeta^{\text{in}}) \in [0,\infty) \times L_{\sigma}^{2} \times \mathring{L}^{2} \times \tilde{U}_{1} \times \tilde{U}_{2}$, there exists a probabilistically weak solution $P \in \mathcal{P} (\bar{\Omega})$ to \eqref{3} with initial condition $(\xi^{\text{in}}, \zeta^{\text{in}})$ at initial time $s$ according to Definition \ref{Definition 5.1}. Moreover, if there exists a family $\{(s_{l}, \xi_{l}, \zeta_{l})\}_{l\in\mathbb{N}} \subset [0,\infty) \times L_{\sigma}^{2} \times \mathring{L}^{2} \times \tilde{U}_{1} \times \tilde{U}_{2}$ such that $\lim_{l\to\infty} \lVert (s_{l}, \xi_{l}, \zeta_{l}) - (s, \xi^{\text{in}}, \zeta^{\text{in}}) \rVert_{\mathbb{R} \times L_{\sigma}^{2} \times \mathring{L}^{2} \times \tilde{U}_{1}\times \tilde{U}_{2}} = 0$ and $P_{l} \in \mathcal{W} (s_{l}, \xi_{l}, \zeta_{l}, \{C_{t,q} \}_{q \in \mathbb{N}, t \geq s_{l}})$, then there exists a subsequence $\{P_{l_{k}} \}_{k\in\mathbb{N}}$ that converges weakly to some $P \in \mathcal{W} (s, \xi^{\text{in}}, \zeta^{\text{in}}, \{C_{t,q}\}_{q \in \mathbb{N}, t \geq s})$. 
\end{proposition} 

The proof of Proposition \ref{Proposition 5.1} follows from Proposition \ref{Proposition 4.1} identically to previous works (see \cite[The. 5.1]{HZZ19}, \cite[Pro. 5.1]{Y20a}, \cite[Pro. 5.1]{Y20c}). Additionally, the following two results also follow from \cite[Pro. 5.2 and 5.3]{HZZ19}, to which we refer interested readers for details. 

\begin{lemma}\label{Lemma 5.2}
\rm{(\cite[Pro. 5.2]{HZZ19})} Let $\tau$ be a bounded $(\bar{\mathcal{B}}_{t})_{t \geq 0}$-stopping time. Then, for every $\omega \in \bar{\Omega}$, there exists $Q_{\omega} \in \mathcal{P}(\bar{\Omega})$ such that 
\begin{subequations}
\begin{align}
& Q_{\omega} ( \{ \omega' \in \bar{\Omega}: \hspace{0.5mm}  ( \xi, \zeta) (t, \omega') = (\xi, \zeta) (t,\omega) \hspace{1mm} \forall \hspace{1mm} t \in [0, \tau(\omega)] \}) = 1, \\
& Q_{\omega} (A) = R_{\tau(\omega), \xi(\tau(\omega), \omega), \zeta(\tau(\omega), \omega)} (A) \hspace{1mm} \forall \hspace{1mm} A \in \bar{\mathcal{B}}^{\tau(\omega)}, 
\end{align} 
\end{subequations}
where $R_{\tau(\omega), \xi(\tau(\omega), \omega), \zeta(\tau(\omega), \omega)} \in \mathcal{P} (\bar{\Omega})$ is a probabilistically weak solution to \eqref{3} with initial condition $(\xi(\tau(\omega), \omega), \zeta(\tau(\omega), \omega))$ at initial time $\tau(\omega)$. Moreover, for every $A \in \bar{\mathcal{B}}$, the mapping $\omega \mapsto Q_{\omega}(A)$ is $\bar{\mathcal{B}}_{\tau}$-measurable.  
\end{lemma} 

\begin{lemma}\label{Lemma 5.3}
\rm{(\cite[Pro. 5.3]{HZZ19})} Let $\tau$ be a bounded $(\bar{\mathcal{B}}_{t})_{t\geq 0}$-stopping time, $\xi^{\text{in}} = (\xi_{1}^{\text{in}}, \xi_{2}^{\text{in}}) \in L_{\sigma}^{2} \times \mathring{L}^{2}$, and $P$ be a probabilistically weak solution to \eqref{3} on $[0,\tau]$ with initial condition $(\xi^{\text{in}}, 0)$ at initial time $0$ according to Definition \ref{Definition 5.2}. Suppose that there exists a Borel set $\mathcal{N} \subset \bar{\Omega}_{\tau}$ such that $P(\mathcal{N}) = 0$ and $Q_{\omega}$ from Lemma \ref{Lemma 5.2} satisfies for every $\omega \in \bar{\Omega}_{\tau} \setminus \mathcal{N}$ 
\begin{equation}\label{[Equ. (143), Y20c]}
Q_{\omega} (\{ \omega' \in \bar{\Omega}: \hspace{0.5mm}  \tau(\omega') = \tau(\omega) \}) = 1. 
\end{equation} 
Then the probability measure $P\otimes_{\tau}R \in \mathcal{P} (\bar{\Omega})$ defined by 
\begin{equation}\label{[Equ. (144), Y20c]}
P \otimes_{\tau} R (\cdot) \triangleq \int_{\bar{\Omega}} Q_{\omega} (\cdot) P(d \omega)
\end{equation} 
satisfies $P\otimes_{\tau}R \rvert_{\bar{\Omega}_{\tau}} = P \rvert_{\bar{\Omega}_{\tau}}$ and it is a probabilistically weak solution to \eqref{3} on $[0,\infty)$ with initial condition $(\xi^{\text{in}}, 0)$ at initial time $0$. 
\end{lemma} 

Now we fix $\mathbb{R}$-valued Wiener processes $B_{1}$ and $B_{2}$ on $(\Omega, \mathcal{F}, \textbf{P})$ with $(\mathcal{F}_{t})_{t\geq 0}$ as its normal filtration. For $l \in \mathbb{N}, L > 1$, and $\delta \in (0, \frac{1}{12})$, we define 
\begin{align}
\tau_{L}^{l} (\omega) \triangleq& \inf\{t \geq 0: \hspace{0.5mm} \max_{k=1,2} \lvert \zeta_{k}(t,\omega) \rvert > (L - \frac{1}{l})^{\frac{1}{4}}\} \nonumber\\
& \wedge \inf\{t \geq 0: \hspace{0.5mm} \max_{k=1,2} \lVert \zeta_{k}(\omega) \rVert_{C_{t}^{\frac{1}{2} - 2 \delta}} > (L - \frac{1}{l})^{\frac{1}{2}} \wedge L, \hspace{5mm} \tau_{L}(\omega) \triangleq& \lim_{l\to\infty} \tau_{L}^{l}(\omega). \label{[Equ. (145a) and (145b), Y20c]}
\end{align}
Comparing \eqref{3} and \eqref{[Equ. (136), Y20c]} we see that $F_{k} (\xi_{k}(r)) = \xi_{k}(r), \zeta_{k} = B_{k}$ for $k \in \{1,2\}$; as Brownian path is locally H$\ddot{\mathrm{o}}$lder continuous with exponent $\alpha \in (0,\frac{1}{2})$, it follows that $\tau_{L}$ is a stopping time of $(\bar{\mathcal{B}}_{t})_{t\geq 0}$. For the fixed $(\Omega, \mathcal{F}, \textbf{P})$, we assume Theorem \ref{Theorem 2.3} and denote by $(u,\theta)$ the solution constructed by Theorem \ref{Theorem 2.3} on $[0,\mathfrak{t}]$ where $\mathfrak{t} = T_{L}$ for $L$ sufficiently large and differently from \eqref{[Equ. (33), Y20a]} 
\begin{equation}\label{[Equ. (146), Y20c]}
T_{L} \triangleq \inf\{t> 0: \hspace{0.5mm} \max_{k=1,2} \lvert B_{k}(t) \rvert \geq L^{\frac{1}{4}} \}  \wedge \inf\{t > 0:  \hspace{0.5mm} \max_{k=1,2} \lVert B_{k} \rVert_{C_{t}^{\frac{1}{2} - 2\delta}} \geq L^{\frac{1}{2}} \} \wedge L.  
\end{equation}
With $P$ representing the law of $(u,\theta, B_{1}, B_{2})$, the following two results also follows immediately from previous works (\cite{HZZ19, Y20a, Y20c}) making use of the fact that 
\begin{equation}\label{[Equ. (147), Y20c]}
\zeta(t, (u,\theta, B_{1}, B_{2})) = (B_{1}, B_{2})(t) \hspace{5mm} \forall \hspace{1mm} t \in [0, T_{L}] \hspace{1mm} \textbf{P}\text{-almost surely}. 
\end{equation}
\begin{proposition}\label{Proposition 5.4}
\rm{(cf. \cite[Pro. 5.4]{HZZ19}, \cite[Pro. 5.4]{Y20a}, \cite[Pro. 5.4]{Y20c})} Let $\tau_{L}$ be defined by \eqref{[Equ. (145a) and (145b), Y20c]}. Then $P$, the law of $(u,\theta, B_{1}, B_{2})$, is a probabilistically weak solution to \eqref{3} on $[0, \tau_{L}]$ that satisfies Definition \ref{Definition 5.2}. 
\end{proposition}

\begin{proposition}\label{Proposition 5.5}
\rm{(cf. \cite[Pro. 5.5]{HZZ19}, \cite[Pro. 5.5]{Y20a}, \cite[Pro. 5.5]{Y20c})}  Let $\tau_{L}$ be defined by \eqref{[Equ. (145a) and (145b), Y20c]}. Then $P\otimes_{\tau_{L}}R$ defined in \eqref{[Equ. (144), Y20c]} is a probabilistically weak solution to \eqref{3} on $[0,\infty)$ that satisfies Definition \ref{Definition 5.1}. 
\end{proposition}

\begin{proof}[Proof of Theorem \ref{Theorem 2.4} assuming Theorem \ref{Theorem 2.3}]
This proof is similar to the proof of Theorem \ref{Theorem 2.2} assuming Theorem \ref{Theorem 2.1}. We fix $T > 0$ arbitrarily, $\kappa \in (0,1)$, and $K > 1$ such that $\kappa K^{2} \geq 1$. The probability measure $P \otimes_{\tau_{L}}R$ from Proposition \ref{Proposition 5.5} satisfies $P \otimes_{\tau_{L}} R ( \{ \tau_{L} \geq T \}) > \kappa$ due to Lemma \ref{Lemma 5.3} and \eqref{6} which implies 
\begin{equation*}
\mathbb{E}^{P \otimes_{\tau_{L}}R}[ \lVert \xi(T) \rVert_{L_{x}^{2}}^{2}] > \kappa K^{2} e^{2T} [ \lVert \xi_{1}^{\text{in}} \rVert_{L_{x}^{2}}^{2} + \lVert \xi_{2}^{\text{in}} \rVert_{L_{x}^{2}}^{2}] 
\end{equation*}
where $\xi^{\text{in}} = (\xi_{1}^{\text{in}}, \xi_{2}^{\text{in}})$ is the deterministic initial condition in Theorem \ref{Theorem 2.3}. On the other hand, it is well-known that via a Galerkin approximation, one can readily construct a probabilistically weak solution $\mathcal{Q}$ to \eqref{3} starting also from $\xi^{\text{in}}$ such that $\mathbb{E}^{\mathcal{Q}}[ \lVert \xi(T) \rVert_{L_{x}^{2}}^{2}] \leq e^{2T} [\lVert \xi_{1}^{\text{in}} \rVert_{L_{x}^{2}}^{2} + \lVert \xi_{2}^{\text{in}} \rVert_{L_{x}^{2}}^{2}]$. This implies a lack of joint uniqueness in law and consequently non-uniqueness in law for \eqref{3} due to Cherny's law (\cite[Lem. C.1]{HZZ19}). 
\end{proof}  

Now we define 
\begin{equation}\label{estimate 76}
\Upsilon_{1}(t) \triangleq e^{B_{1}(t)},\hspace{3mm} \Upsilon_{2}(t) \triangleq e^{B_{2}(t)}, \hspace{3mm} v \triangleq \Upsilon_{1}^{-1} u, \hspace{3mm} \Theta \triangleq \Upsilon_{2}^{-1} \theta. 
\end{equation} 
By It$\hat{\mathrm{o}}$'s product formula we see that they satisfy 
\begin{subequations}\label{[Equ. (148), Y20c]}
\begin{align}
& \partial_{t} v + \frac{1}{2} v + (-\Delta)^{m} v + \Upsilon_{1} \text{div} (v \otimes v) + \Upsilon_{1}^{-1} \nabla \pi = \Upsilon_{1}^{-1} \Upsilon_{2} \Theta e^{n}, \hspace{3mm} \nabla \cdot v = 0, \\
& \partial_{t} \Theta + \frac{1}{2} \Theta - \Delta \Theta + \Upsilon_{1} \text{div} (v \Theta) = 0. 
\end{align}
\end{subequations} 
Considering this, for every $q \in \mathbb{N}_{0}$ we will construct $(v_{q}, \Theta_{q}, \mathring{R}_{q})$ that solves 
\begin{subequations}\label{[Equ. (149), Y20c]}
\begin{align}
& \partial_{t} v_{q} + \frac{1}{2} v_{q} + (-\Delta)^{m} v_{q} + \Upsilon_{1} \text{div} (v_{q} \otimes v_{q}) + \nabla p_{q}  = \Upsilon_{1}^{-1} \Upsilon_{2} \Theta_{q} e^{n} + \text{div} \mathring{R}_{q}, \nabla\cdot v_{q} = 0,  \label{estimate 71}\\
& \partial_{t} \Theta_{q} + \frac{1}{2} \Theta_{q} - \Delta \Theta_{q} + \Upsilon_{1} \text{div} (v_{q} \Theta_{q}) = 0, \label{estimate 72} 
\end{align}
\end{subequations} 
where $\mathring{R}_{q}$ is a trace-free, symmetric matrix. Similarly to the proof of Theorems \ref{Theorem 2.1}-\ref{Theorem 2.2}, we define $\lambda_{q}$ and $\delta_{q}$ identically to \eqref{[Equ. (47), Y20c]}, while differently we define 
\begin{equation}\label{[Equ. (150), Y20c]}
M_{0}(t) \triangleq e^{4L t + 2L} \text{ and } m_{L} \triangleq \sqrt{3} L^{\frac{1}{4}} e^{\frac{1}{2} L^{\frac{1}{4}}}. 
\end{equation} 
We have for $L > 1, \delta \in (0, \frac{1}{12}), t \in [0, T_{L}]$, and $k \in \{1,2\}$, 
\begin{equation}\label{[Equ. (151) and (152), Y20c]}
\lvert B_{k}(t) \rvert \overset{\eqref{[Equ. (146), Y20c]}}{\leq} L^{\frac{1}{4}}, \lVert B_{k} \rVert_{C_{t}^{\frac{1}{2} - 2 \delta}} \overset{\eqref{[Equ. (146), Y20c]}}{\leq} L^{\frac{1}{2}}; \text{consequently}, \lVert \Upsilon_{k} \rVert_{C_{t}^{\frac{1}{2} - 2\delta}} + \lvert \Upsilon_{k}(t) \rvert + \lvert \Upsilon_{k}^{-1} (t) \rvert \leq m_{L}^{2}. 
\end{equation} 
For inductive estimates, we assume for all $t \in [0, T_{L}]$ and $p \in [1,\infty]$, 
\begin{subequations}\label{[Equ. (153), Y20c]}
\begin{align}
& \lVert v_{q} \rVert_{C_{t}L_{x}^{2}} \leq m_{L} M_{0}(t)^{\frac{1}{2}}(1+ \sum_{1 \leq \iota \leq q} \delta_{\iota}^{\frac{1}{2}}) \leq 2 m_{L} M_{0}(t)^{\frac{1}{2}}, \hspace{3mm} \lVert v_{q} \rVert_{C_{t,x}^{1}} \leq m_{L} M_{0}(t)^{\frac{1}{2}} \lambda_{q}^{4}, \label{[Equ. (153a) and (153b), Y20c]} \\
& \lVert \mathring{R}_{q} \rVert_{C_{t}L_{x}^{1}} \leq c_{R} M_{0}(t) \delta_{q+1}, \label{[Equ. (153c), Y20c]}\\
& \mathbb{E}^{\textbf{P}} [ \lVert \theta_{q} (t\wedge T_{L}) \rVert_{L_{x}^{2}}^{2} + 2 \int_{0}^{t\wedge T_{L}} \lVert \theta_{q} \rVert_{\dot{H}_{x}^{1}}^{2} dr] \leq \lVert \theta_{q}(0) \rVert_{L_{x}^{2}}^{2} + \mathbb{E}^{\textbf{P}} [ \int_{0}^{t\wedge T_{L}} \lVert \theta_{q} \rVert_{L_{x}^{2}}^{2} dr], \label{estimate 86}\\
& \lVert \Theta_{q}(t) \rVert_{L_{x}^{2}}^{2} + \int_{0}^{t} \lVert \Theta_{q} (r) \rVert_{L_{x}^{2}}^{2} + 2 \lVert \Theta_{q}(r) \rVert_{\dot{H}_{x}^{1}}^{2} dr = \lVert \Theta_{q}(0) \rVert_{L_{x}^{2}}^{2},  \lVert \Theta_{q}(t) \rVert_{L_{x}^{p}} \leq e^{- \frac{t}{2}} \lVert \Theta_{q}(0) \rVert_{L_{x}^{p}}, \label{estimate 70}
\end{align}
\end{subequations} 
where we assumed again that $a^{\beta b} > 3$, to be formally stated in \eqref{[Equ. (156), Y20c]} to justify $\sum_{1 \leq \iota} \delta_{\iota}^{\frac{1}{2}} < \frac{1}{2}$. 

Now we first consider the case $n = 2$ so that $m \in (0,1)$ by \eqref{4}. For notations and preliminaries hereafter in case $n=2$, we refer again to Subsection \ref{Preliminaries needed for convex integration in 2D case and more}.  We impose again that $a \in 10 \mathbb{N}$ and $b \in \mathbb{N}$ so that $\lambda_{q+1} \in 10 \mathbb{N} \subset 5 \mathbb{N}$ as required in \eqref{[Equ. (18), Y20c]}.  

\begin{proposition}\label{Proposition 5.6 for n=2}
Fix $\theta^{\text{in}} \in H^{2}(\mathbb{T}^{2})$ that is deterministic and mean-zero. Let 
\begin{equation}\label{estimate 108}
v_{0}(t,x) \triangleq \frac{m_{L} e^{2L t + L}}{2\pi} 
\begin{pmatrix}
\sin(x^{2}) & 0 
\end{pmatrix}^{T}. 
\end{equation}  
Then there exists a unique solution $\Theta_{0} \in L_{\omega}^{\infty} L_{t}^{\infty} H_{x}^{2} \cap L_{\omega}^{\infty} L_{t}^{2} H_{x}^{3}$ to the following linear random PDE: 
\begin{equation}\label{estimate 109}
\partial_{t} \Theta_{0} + \frac{1}{2} \Theta_{0} - \Delta \Theta_{0} + \Upsilon_{1} \text{div} (v_{0} \Theta_{0}) = 0, \hspace{3mm} \Theta_{0} (0,x) = \theta^{\text{in}} (x). 
\end{equation} 
It follows that together with 
\begin{align}
\mathring{R}_{0}(t,x) \triangleq& \frac{m_{L} (2L + \frac{1}{2}) e^{2L t + L}}{2\pi} 
\begin{pmatrix}
0 & - \cos(x^{2})\\
-\cos(x^{2}) & 0 
\end{pmatrix}  \nonumber\\
& \hspace{15mm} + \mathcal{R} (-\Delta)^{m} v_{0}(t,x) - \mathcal{R} (\Upsilon_{1}^{-1} \Upsilon_{2} \Theta_{0}(t,x) e^{2}), \label{estimate 110}
\end{align}
$(v_{0}, \Theta_{0})$ satisfy \eqref{[Equ. (149), Y20c]} at level $q = 0$. Moreover, \eqref{[Equ. (153), Y20c]} at level $q = 0$ is satisfied provided 
\begin{equation}\label{[Equ. (156), Y20c]}
72 \sqrt{3} < 8 \sqrt{3} a^{2 \beta b} \leq \frac{c_{R} e^{L - \frac{1}{2} L^{\frac{1}{4}}}}{L^{\frac{1}{4}} [ 2 L + \frac{1}{2} + \pi + \pi \lVert \theta^{\text{in}} \rVert_{L_{x}^{2}} ]}, \hspace{5mm}  (\frac{3}{2})^{\frac{4}{3}} \leq L \leq a^{4} \pi - 1,  
\end{equation} 
where the inequality $9 < a^{2 \beta b}$ is assumed for the justification of second inequality of \eqref{[Equ. (153a) and (153b), Y20c]}. Furthermore, $v_{0}(0,x)$ and $\mathring{R}_{0}(0,x)$ are both deterministic. 
\end{proposition} 

\begin{proof}[Proof of Proposition \ref{Proposition 5.6 for n=2}]
This proof is similar to that of Proposition \ref{Proposition 4.7 for n=2}. We only sketch the main steps. It can be shown immediately that for all $t \in [0, T_{L}]$
\begin{equation}\label{[Equ. (157), Y20c]}
\lVert v_{0}(t) \rVert_{L_{x}^{2}} = \frac{m_{L} M_{0}(t)^{\frac{1}{2}}}{\sqrt{2}} \leq m_{L} M_{0}(t)^{\frac{1}{2}}, \hspace{2mm} \lVert v_{0} \rVert_{C_{t,x}^{1}} \overset{\eqref{[Equ. (156), Y20c]}}{\leq} m_{L} M_{0}(t)^{\frac{1}{2}} \lambda_{0}^{4} 
\end{equation} 
(see \cite[Equ. (157)]{Y20c}). The estimates in  \eqref{estimate 70} are clear from a standard $L^{p}(\mathbb{T}^{2})$-estimate as multiplying \eqref{estimate 109} by $\lvert \Theta_{0} \rvert^{p-2} \Theta_{0}$ for any $p \in [1,\infty)$ and integrating over $\mathbb{T}^{2}$ give 
\begin{equation*}
\frac{1}{p} \partial_{t} \lVert \Theta_{0} \rVert_{L_{x}^{p}}^{p} + \frac{1}{2} \lVert \Theta_{0} \rVert_{L_{x}^{p}}^{p} - \int_{\mathbb{T}^{2}} \Delta \Theta_{0} \lvert \Theta_{0} \rvert^{p-2} \Theta_{0} dx = 0 
\end{equation*} 
due to $v_{0}$ being divergence-free. Using \eqref{estimate 70}, \eqref{[Equ. (157), Y20c]}, and \eqref{[Equ. (151) and (152), Y20c]}, as well as the fact that $\Delta v_{0} = -v_{0}$, one can immediately bootstrap from \eqref{estimate 70} to verify that $\Theta_{0} \in L_{\omega}^{\infty} L_{t}^{\infty} H_{x}^{2} \cap L_{\omega}^{\infty} L_{t}^{2} H_{x}^{3}$. Moreover, \eqref{estimate 86} also follows from the equation of $\theta_{0} = \Upsilon_{2} \Theta_{0}$. Because $\theta^{\text{in}}$ is mean-zero by hypothesis, so is $\Theta_{0}(t)$ for all $t \geq 0$; consequently, $\mathcal{R} (\Upsilon_{1}^{-1} \Upsilon_{2} \Theta_{0}e^{2})$ is trace-free and symmetric. As $v_{0}$ is also mean-zero, $\mathcal{R}(-\Delta)^{m} v_{0}$ is also trace-free and symmetric, leading to $\mathring{R}_{0}$ also being trace-free and symmetric. Moreover, \eqref{estimate 71} at level $q = 0$ holds with $p_{0} \equiv 0$. Finally, by H$\ddot{\mathrm{o}}$lder's inequality and Lemma \ref{[Def. 9, Lem. 10, CDS12]}, for all $t \in [0, T_{L}]$ 
\begin{equation}\label{estimate 74}
\lVert \mathcal{R} (\Upsilon_{1}^{-1} \Upsilon_{2} \Theta_{0} e^{2}) \rVert_{L_{x}^{1}} \overset{\eqref{[Equ. (151) and (152), Y20c]}}{\leq} 6 \pi e^{2L^{\frac{1}{4}}} \lVert \Theta_{0} e^{2} \rVert_{L_{x}^{2}} \overset{\eqref{[Equ. (150), Y20c]} \eqref{[Equ. (156), Y20c]}  }{\leq} 8 \pi m_{L} M_{0}(t)^{\frac{1}{2}} \lVert \theta^{\text{in}} \rVert_{L_{x}^{2}}.  
\end{equation}
Moreover, 
\begin{align*}
\lVert \mathring{R}_{0}(t) + \mathcal{R} ( \Upsilon_{1}^{-1} \Upsilon_{2} \Theta_{0} e^{2}) \rVert_{L_{x}^{1}} 
\leq& m_{L} (2 L + \frac{1}{2}) M_{0}(t)^{\frac{1}{2}}8 + 8 \pi \lVert v_{0}(t) \rVert_{L_{x}^{2}} \\
\overset{\eqref{[Equ. (157), Y20c]}}{\leq}& m_{L} 8 M_{0}(t)^{\frac{1}{2}} [2L + \frac{1}{2} + \pi]
\end{align*} 
where the first inequality is due to \cite[Equ. (158)]{Y20c}. This, along with \eqref{estimate 74}, imply 
\begin{equation}
\lVert \mathring{R}_{0} (t) \rVert_{L_{x}^{1}} \leq m_{L} 8 M_{0}(t)^{\frac{1}{2}}[ 2L + \frac{1}{2} + \pi + \pi \lVert \theta^{\text{in}} \rVert_{L_{x}^{2}}] \overset{\eqref{[Equ. (156), Y20c]}}{\leq} c_{R} M_{0}(t) \delta_{1}. 
\end{equation} 
The rest of the arguments are identical to those of Proposition \ref{Proposition 4.7 for n=2}. 
\end{proof}

\begin{proposition}\label{Proposition 5.7 for n=2}
Fix $\theta^{\text{in}} \in H^{2}(\mathbb{T}^{2})$ that is deterministic and mean-zero from the hypothesis of Proposition \ref{Proposition 5.6 for n=2}. Let $L$ satisfy 
\begin{equation}\label{[Equ. (159), Y20c]}
L > (\frac{3}{2})^{\frac{4}{3}}, \hspace{5mm} 72 \sqrt{3} < \frac{ c_{R} e^{L - \frac{1}{2} L^{\frac{1}{4}}}}{L^{\frac{1}{4}} [ 2L + \frac{1}{2} + \pi + \pi \lVert \theta^{\text{in}} \rVert_{L_{x}^{2}}]}, 
\end{equation} 
and suppose that $(v_{q}, \Theta_{q}, \mathring{R}_{q})$ are $(\mathcal{F}_{t})_{t\geq 0}$-adapted processes that solve \eqref{[Equ. (149), Y20c]} and satisfy \eqref{[Equ. (153), Y20c]}. Then there exist a choice of parameters $a, b,$ and $\beta$ such that \eqref{[Equ. (156), Y20c]} is fulfilled and $(\mathcal{F}_{t})_{t\geq 0}$-adapted processes $(v_{q+1}, \Theta_{q+1}, \mathring{R}_{q+1})$ that satisfy \eqref{[Equ. (149), Y20c]}, \eqref{[Equ. (153), Y20c]} at level $q+1$, and for all $t \in [0, T_{L}]$, 
\begin{subequations}
\begin{align}
& \lVert v_{q+1}(t) - v_{q}(t) \rVert_{L_{x}^{2}} \leq m_{L} M_{0}(t)^{\frac{1}{2}} \delta_{q+1}^{\frac{1}{2}}, \label{[Equ. (160), Y20c]}\\
& \lVert \Theta_{q+1} - \Theta_{q} \rVert_{C_{t}L_{x}^{2}}^{2} + \int_{0}^{t} \lVert \Theta_{q+1} - \Theta_{q} \rVert_{\dot{H}_{x}^{1}}^{2} dr \leq e^{2 L^{\frac{1}{4}}}  m_{L}^{2} M_{0}(t) \delta_{q+1} \lVert \theta^{\text{in}} \rVert_{L_{x}^{\infty}}^{2}. \label{estimate 75}
\end{align}
\end{subequations} 
Finally, if $v_{q}(0,x)$ and $\mathring{R}_{q}(0,x)$ are deterministic, then so are $v_{q+1}(0,x)$ and $\mathring{R}_{q+1}(0,x)$. 
\end{proposition}  

\begin{proof}[Proof of Theorem \ref{Theorem 2.3} if $n=2$ assuming Proposition \ref{Proposition 5.7 for n=2}]
Fix $\theta^{\text{in}} \in H^{2} (\mathbb{T}^{2})$ that is deterministic and mean-zero from the hypothesis of Proposition \ref{Proposition 5.6 for n=2}, any $T > 0, K > 1$, and $\kappa \in (0,1)$. Then we take $L$ that satisfies \eqref{[Equ. (159), Y20c]} and enlarge it if necessary to satisfy 
\begin{equation}\label{[Equ. (208), Y20c]} 
( \frac{1}{\sqrt{2}} - \frac{1}{2}) e^{2L T} > e^{2L^{\frac{1}{2}}} [ (\frac{1}{2} + \frac{1}{\sqrt{2}} ) + 3^{-\frac{1}{2}} L^{- \frac{1}{4}} e^{- \frac{1}{2} L^{\frac{1}{4}}} e^{-L} \lVert \theta^{\text{in}} \rVert_{L_{x}^{2}}] \text{ and } L > [ \ln (Ke^{T})]^{2}. 
\end{equation} 
We can start from $(v_{0}, \Theta_{0}, \mathring{R}_{0})$ in Proposition \ref{Proposition 5.6 for n=2}, and via Proposition \ref{Proposition 5.7 for n=2} inductively obtain a sequence $(v_{q}, \Theta_{q}, \mathring{R}_{q})$ that satisfies \eqref{[Equ. (149), Y20c]}, \eqref{[Equ. (153), Y20c]}, and \eqref{[Equ. (160), Y20c]}-\eqref{estimate 75}. For any $\gamma \in (0, \frac{\beta}{4+ \beta})$ and any $t \in [0,T_{L}]$, we can show $\sum_{q\geq 0} \lVert v_{q+1}(t) - v_{q}(t) \rVert_{\dot{H}_{x}^{\gamma}} \lesssim m_{L} M_{0}(t)^{\frac{1}{2}}$ similarly to \eqref{estimate 29}. Thus,  $\{v_{q} \}_{q=0}^{\infty}$ is Cauchy in $C_{T_{L}} \dot{H}_{x}^{\gamma}$ and hence we deduce the existence of $\lim_{q\to\infty} v_{q} \triangleq v \in C([0, T_{L}]; \dot{H}^{\gamma} (\mathbb{T}^{2}))$. On the other hand, \eqref{estimate 75}, \eqref{estimate 70} and interpolation show that $\{\Theta_{q}\}_{q=0}^{\infty}$ is Cauchy in $\cap_{p\in [1, \infty)} C_{T_{L}} L_{x}^{p} \cap L_{T_{L}}^{2} \dot{H}_{x}^{1}$. Therefore, we can deduce $\lim_{q\to\infty} \Theta_{q} \triangleq \Theta \in \cap_{p\in [1,\infty)} C_{T_{L}} L_{x}^{p} \cap L_{T_{L}}^{2} \dot{H}_{x}^{1}$. Finally, clearly \eqref{[Equ. (153c), Y20c]} implies that $\lim_{q\to\infty} \lVert \mathring{R}_{q} \rVert_{C_{T_{L}}L_{x}^{1}} = 0$ and hence $(v, \Theta)$ solves \eqref{[Equ. (148), Y20c]}. Because $u = e^{B_{1} (t)} v$ where $\lvert e^{B_{1}} \rvert \leq e^{L^{\frac{1}{4}}}$ for all $t \in [0, T_{L}]$, we deduce $esssup_{\omega \in \Omega} \sup_{s \in [0, \mathfrak{t}]} \lVert u(s) \rVert_{\dot{H}_{x}^{\gamma}} < \infty$; similarly considering that $\theta = e^{B_{2}(t)} \Theta$ where $\lvert e^{B_{2}} \rvert \leq e^{L^{\frac{1}{4}}}$ for all $t \in [0,T_{L}]$ shows that \eqref{estimate 94} is satisfied. Next, for all $t \in [0, T_{L}]$, we can show similarly to \eqref{[Equ. (61), Y20c]} that $\lVert v(t) - v_{0}(t) \rVert_{L_{x}^{2}} \leq  \frac{m_{L}}{2} M_{0}(t)^{\frac{1}{2}}$ by \eqref{[Equ. (160), Y20c]} and \eqref{[Equ. (156), Y20c]}, which in turn implies 
\begin{equation}\label{[Equ. (209), Y20c]} 
\lVert v(0) \rVert_{L_{x}^{2}} \leq  \lVert v(0) - v_{0}(0) \rVert_{L_{x}^{2}} + \lVert v_{0} (0) \rVert_{L_{x}^{2}}   \overset{\eqref{[Equ. (157), Y20c]}}{\leq} (\frac{1}{2} + \frac{1}{\sqrt{2}}) m_{L} M_{0}(0)^{\frac{1}{2}}. 
\end{equation} 
These lead us to, on a set $\{T_{L} \geq T\}$, 
\begin{equation}\label{[Equ. (210), Y20c]}
\lVert v(T) \rVert_{L_{x}^{2}} \overset{\eqref{[Equ. (157), Y20c]}}{\geq}  (\frac{1}{\sqrt{2}} - \frac{1}{2}) m_{L} M_{0}(T)^{\frac{1}{2}} \\
\overset{\eqref{[Equ. (208), Y20c]} \eqref{[Equ. (209), Y20c]}}{>}  e^{2L^{\frac{1}{2}}} ( \lVert v(0) \rVert_{L_{x}^{2}} + \lVert \theta^{\text{in}} \rVert_{L_{x}^{2}}). 
\end{equation}  
For the fixed $T > 0, \kappa \in (0,1)$, taking $L$ even larger gives us \eqref{6} because $\lim_{L\to\infty} T_{L} = + \infty$ $\textbf{P}$-a.s. by \eqref{[Equ. (146), Y20c]}. We also see that $u^{\text{in}} (x)  = v(0,x)$ which is deterministic because $v_{q}(0,x)$ is deterministic for all $q \in \mathbb{N}_{0}$ due to Propositions \ref{Proposition 5.6 for n=2} and \ref{Proposition 5.7 for n=2}. Clearly from \eqref{estimate 76}, $(u,\theta) = (\Upsilon_{1} v, \Upsilon_{2} \Theta)$ is an $(\mathcal{F}_{t})_{t\geq 0}$-adapted solution of \eqref{3}. Finally, due to the fact that $\lvert \Upsilon_{1}(T) \rvert \geq e^{-L^{\frac{1}{4}}}$ by \eqref{[Equ. (151) and (152), Y20c]}, \eqref{[Equ. (210), Y20c]}, and \eqref{[Equ. (208), Y20c]}, we see that 
\begin{equation*}
\lVert u(T) \rVert_{L_{x}^{2}} \overset{\eqref{[Equ. (210), Y20c]}}{>} e^{-L^{\frac{1}{4}}} [ e^{2L^{\frac{1}{2}}} (\lVert v(0) \rVert_{L_{x}^{2}} + \lVert \theta^{\text{in}} \rVert_{L_{x}^{2}})] \overset{\eqref{[Equ. (208), Y20c]}}{\geq}  K e^{T} ( \lVert u^{\text{in}} \rVert_{L_{x}^{2}}+ \lVert \theta^{\text{in}} \rVert_{L_{x}^{2}}).  
\end{equation*}  
\end{proof} 

\subsection{Convex integration to prove Proposition \ref{Proposition 5.7 for n=2}}

\subsubsection{Choice of parameters}
We fix $L$ sufficiently large that satisfies \eqref{[Equ. (159), Y20c]}. We take the same choices of $m^{\ast}, \eta, \alpha, r, \mu$, $\sigma$, and $b$ in \eqref{[Equ. (64), Y20c]}-\eqref{estimate 62},  such that $r \in \mathbb{N}$ and $\lambda_{q+1} \sigma \in 10 \mathbb{N}$ so that $r \in \mathbb{N}$ and $\lambda_{q+1}\sigma \in 5 \mathbb{N}$ from \eqref{[Equ. (18), Y20c]} are satisfied. Then we make sure to take $\beta > 0$ sufficiently small to satisfy \eqref{[Equ. (68), Y20c]} and then $l$ by \eqref{[Equ. (69), Y20c]} so that \eqref{[Equ. (70), Y20c]} remains valid. We take $a \in 10 \mathbb{N}$ larger if necessary to satisfy $a^{26} \geq \sqrt{3} L^{\frac{1}{4}} e^{\frac{1}{2} L^{\frac{1}{4}}}$ which implies 
\begin{equation}\label{[Equ. (161), Y20c]}
m_{L} \overset{\eqref{estimate 62}}{\leq} a^{\frac{3 \alpha b}{2} + 2} \overset{\eqref{[Equ. (69), Y20c]}}{\leq} l^{-1} \text{ and } m_{L} \overset{\eqref{[Equ. (159), Y20c]}}{\leq} c_{R} e^{L} \overset{\eqref{[Equ. (150), Y20c]}}{\leq} M_{0}(t)^{\frac{1}{2}}. 
\end{equation} 
Taking $a \in 10 \mathbb{N}$ even larger guarantees $L \leq a^{4} \pi -1$ in \eqref{[Equ. (156), Y20c]}, while taking $\beta > 0$ even smaller if necessary gives us the other inequalities in \eqref{[Equ. (156), Y20c]} due to \eqref{[Equ. (159), Y20c]}. 

\subsubsection{Mollification}
We mollify $v_{q}, \theta_{q}, \mathring{R}_{q}$ identically to \eqref{[Equ. (71), Y20c]} while for $k \in \{1,2\}$, 
\begin{equation}\label{estimate 96}
\Upsilon_{k,l} \triangleq \Upsilon_{k} \ast_{t} \varphi_{l}; 
\end{equation} 
it follows from \eqref{estimate 71} that they satisfy 
\begin{align}
& \partial_{t}v_{l} + \frac{1}{2} v_{l} + (-\Delta)^{m} v_{l} + \Upsilon_{1,l} \text{div} (v_{l} \otimes v_{l}) + \nabla p_{l} \nonumber \\
& \hspace{10mm} = \text{div} (\mathring{R}_{l} + R_{\text{com1}}) + ((\Upsilon_{1}^{-1} \Upsilon_{2} \Theta_{q} e^{2}) \ast_{x} \phi_{l}) \ast_{t} \varphi_{l}, \label{estimate 127}
\end{align} 
where 
\begin{subequations}\label{estimate 128}
\begin{align}
p_{l} \triangleq& (p_{q} \ast_{x} \phi_{l})\ast_{t} \varphi_{l} - \frac{1}{2} (\Upsilon_{1,l} \lvert v_{l} \rvert^{2} - ((\Upsilon_{1} \lvert v_{q} \rvert^{2}) \ast_{x}\phi_{l})\ast_{t} \varphi_{l}),\\
R_{\text{com1}} \triangleq& - ((\Upsilon_{1}(v_{q} \mathring{\otimes} v_{q})) \ast_{x} \phi_{l}) \ast_{t} \varphi_{l} + \Upsilon_{1,l} (v_{l} \mathring{\otimes} v_{l}). \label{estimate 147}
\end{align}
\end{subequations} 
We can compute for all $t \in [0, T_{L}]$ and $N \geq 1$, due to \eqref{[Equ. (153a) and (153b), Y20c]}
\begin{subequations}
\begin{align}
& \lVert v_{q} - v_{l} \rVert_{C_{t}L_{x}^{2}} \leq  (\frac{m_{L}}{4}) M_{0}(t)^{\frac{1}{2}} \delta_{q+1}^{\frac{1}{2}}, 
\lVert v_{l} \rVert_{C_{t}L_{x}^{2}}  \leq m_{L} M_{0}(t)^{\frac{1}{2}} (1+ \sum_{1 \leq \iota \leq q} \delta_{\iota}^{\frac{1}{2}}) \leq 2m_{L} M_{0}(t)^{\frac{1}{2}}, \label{[Equ. (165a) and (165b), Y20c]}\\
& \lVert v_{l} \rVert_{C_{t,x}^{N}} \overset{ \eqref{[Equ. (69), Y20c]}}{\leq} l^{-N} m_{L} M_{0}(t)^{\frac{1}{2}} \lambda_{q+1}^{-\alpha} \label{[Equ. (165c), Y20c]}
\end{align}
\end{subequations} 
(see \cite[Equ. (165)]{Y20c}). 

\subsubsection{Perturbation}
We can continue to define $\chi$ and $\rho$ identically as we did in \eqref{[Equ. (75), Y20c]}-\eqref{[Equ. (76), Y20c]} except $M_{0}(t)$ being defined now by \eqref{[Equ. (150), Y20c]} instead of \eqref{[Equ. (47), Y20c]}. As the only difference thus far is the definition of $M_{0}(t)$, one can verify that \eqref{[Equ. (77), Y20c]}, \eqref{[Equ. (78) and (79), Y20c]}, and \eqref{[Equ. (80), Y20c]} all remain valid. Next, we define a modified amplitude function by 
\begin{equation}\label{[Equ. (166), Y20c]}
\bar{a}_{\zeta} (\omega, t,x) \triangleq \bar{a}_{\zeta, q+1} (\omega, t,x) \triangleq \Upsilon_{1,l}^{-\frac{1}{2}} a_{\zeta} (\omega, t,x), 
\end{equation} 
where $a_{\zeta} (\omega, t,x)$ is identical to that defined in \eqref{[Equ. (81), Y20c]}. Making use of $\lVert \Upsilon_{k,l}^{-\frac{1}{2}} \rVert_{C_{t}} \leq m_{L}$ for both $k \in \{1,2\}$, we can estimate for all $t \in [0, T_{L}]$, 
\begin{equation}\label{[Equ. (168), Y20c]}
\lVert \bar{a}_{\zeta} \rVert_{C_{t}L_{x}^{2}} \overset{\eqref{[Equ. (166), Y20c]} \eqref{[Equ. (77), Y20c]}}{\leq} \lVert \Upsilon_{1,l}^{-\frac{1}{2}} \rVert_{C_{t}} \lVert \rho \rVert_{C_{t} L_{x}^{1}}^{\frac{1}{2}} \lVert \gamma_{\zeta} \rVert_{C(B_{\frac{1}{2}}(0))}   
\overset{ \eqref{[Equ. (15), Y20c]} \eqref{[Equ. (78) and (79), Y20c]} \eqref{[Equ. (153c), Y20c]}  }{\leq} \frac{c_{R}^{\frac{1}{4}} m_{L} M_{0}(t)^{\frac{1}{2}} \delta_{q+1}^{\frac{1}{2}}}{2 \lvert \Lambda \rvert} 
\end{equation} 
(see \cite[Equ. (167)-(168)]{Y20c}). On the other hand, relying on \eqref{[Equ. (84), Y20c]} that is still satisfied by $a_{\zeta}$ leads to for any $N \geq 0$ and $k \in \{0,1,2\}$, 
\begin{equation}\label{[Equ. (169), Y20c]}
\lVert \bar{a}_{\zeta} \rVert_{C_{t}C_{x}^{N}} \overset{\eqref{[Equ. (84), Y20c]}}{\leq} m_{L} c_{R}^{\frac{1}{4}} \delta_{q+1}^{\frac{1}{2}} M_{0}(t)^{\frac{1}{2}} l^{- \frac{3}{2} - 4N}, \hspace{1mm} \lVert \bar{a}_{\zeta} \rVert_{C_{t}^{1}C_{x}^{k}} \overset{\eqref{[Equ. (151) and (152), Y20c]} \eqref{[Equ. (161), Y20c]} }{\leq} m_{L} c_{R}^{\frac{1}{8}} \delta_{q+1}^{\frac{1}{2}} M_{0}(t)^{\frac{1}{2}} l^{- \frac{13}{2} - 4 k} 
\end{equation} 
(see \cite[Equ. (169)]{Y20c}). We define $w_{q+1}^{(p)}$ and $w_{q+1}^{(c)}$ identically to \eqref{[Equ. (86), Y20c]} with $a_{\zeta}$ replaced by $\bar{a}_{\zeta}$ from \eqref{[Equ. (166), Y20c]} and $M_{0}(t)$ from \eqref{[Equ. (150), Y20c]} within the definition of $\rho(\omega, t, x)$ and finally $w_{q+1}^{(t)}$ identically as in \eqref{[Equ. (86), Y20c]} with $a_{\zeta}$ still from \eqref{[Equ. (81), Y20c]} but with $M_{0}(t)$ from \eqref{[Equ. (150), Y20c]}. These choices allow us to define $w_{q+1}$ and $v_{q+1}$ identically to \eqref{[Equ. (85), Y20c]}. It follows that $w_{q+1}$ is divergence-free and mean-zero (see \cite[Equ. (171)]{Y20c}). For all $t \in [0, T_{L}]$ and $p \in (1,\infty)$, we can compute 
\begin{subequations}\label{estimate 345}
\begin{align}
&  \lVert w_{q+1}^{(p)} \rVert_{C_{t}L_{x}^{2}}  \overset{\eqref{[Equ. (166), Y20c]} \eqref{[Equ. (88) and (89a), Y20c]}}{\lesssim} m_{L} c_{R}^{\frac{1}{4}} \delta_{q+1}^{\frac{1}{2}} M_{0}(t)^{\frac{1}{2}},  \lVert w_{q+1}^{(p)} \rVert_{C_{t}L_{x}^{p}} \overset{\eqref{[Equ. (166), Y20c]} \eqref{[Equ. (88) and (89a), Y20c]}}{\lesssim} m_{L} \delta_{q+1}^{\frac{1}{2}} M_{0}(t)^{\frac{1}{2}} l^{-\frac{3}{2}} r^{1- \frac{2}{p}}, \label{[Equ. (172a) and (172b), Y20c]}\\
& \lVert w_{q+1}^{(c)} \rVert_{C_{t}L_{x}^{p}}  \overset{\eqref{[Equ. (166), Y20c]}\eqref{[Equ. (89b) and (89c), Y20c]} }{\lesssim} m_{L} \delta_{q+1}^{\frac{1}{2}} M_{0}(t)^{\frac{1}{2}} l^{-\frac{11}{2}} \sigma r^{2- \frac{2}{p}} \label{[Equ. (172c), Y20c]}
\end{align}
\end{subequations}
(see \cite[Equ. (172)]{Y20c}).  On the other hand, the estimate of $\lVert w_{q+1}^{(t)} \rVert_{C_{t}L_{x}^{p}}$ in \eqref{[Equ. (89b) and (89c), Y20c]} remains valid. It follows that for all $t \in [0, T_{L}]$, 
\begin{equation}\label{[Equ. (172), Y20c]}
\lVert w_{q+1} \rVert_{C_{t}L_{x}^{2}}  \overset{\eqref{[Equ. (85), Y20c]}}{\leq} \lVert w_{q+1}^{(p)} \rVert_{C_{t}L_{x}^{2}} + \lVert w_{q+1}^{(c)} \rVert_{C_{t}L_{x}^{2}} + \lVert w_{q+1}^{(t)} \rVert_{C_{t}L_{x}^{2}} \overset{\eqref{[Equ. (89b) and (89c), Y20c]} \eqref{estimate 345} \eqref{[Equ. (70), Y20c]}}{\leq}\frac{3m_{L} M_{0}(t)^{\frac{1}{2}} \delta_{q+1}^{\frac{1}{2}}}{4},  
\end{equation}
from which the first inequality in \eqref{[Equ. (153a) and (153b), Y20c]} at level $q+1$ and \eqref{[Equ. (160), Y20c]} can be verified using \eqref{[Equ. (85), Y20c]}, \eqref{[Equ. (165a) and (165b), Y20c]}, and \eqref{[Equ. (172), Y20c]} (see \cite[p. 31]{Y20c}). We can also compute for all $t \in [0, T_{L}]$, 
\begin{equation}\label{estimate 362}
\lVert w_{q+1}^{(p)} \rVert_{C_{t,x}^{1}} \overset{\eqref{[Equ. (169), Y20c]}}{\leq} m_{L} M_{0}(t)^{\frac{1}{2}} l^{-\frac{13}{2}} \lambda_{q+1} \sigma \mu r^{2}, \lVert w_{q+1}^{(c)} \rVert_{C_{t,x}^{1}} \overset{\eqref{[Equ. (13b), Y20c]} \eqref{[Equ. (24a) and (24b), Y20c]} }{\lesssim} m_{L} \delta_{q+1}^{\frac{1}{2}} M_{0}(t)^{\frac{1}{2}} \lambda_{q+1}^{3- 18 \eta} l^{-\frac{3}{2}} 
\end{equation} 
(see \cite[Equ. (174)]{Y20c}). This, along with the estimate \eqref{[Equ. (94), Y20c]} on $\lVert w_{q+1}^{(t)} \rVert_{C_{t,x}^{1}}$, and \eqref{[Equ. (85), Y20c]}, allows us to verify the second inequality of \eqref{[Equ. (153a) and (153b), Y20c]} at level $q+1$ (see \cite[Equ. (175)]{Y20c}). At last, with $v_{q+1}$ that we already constructed via \eqref{[Equ. (85), Y20c]}, identically to the proof of Proposition \ref{Proposition 5.6 for n=2}, we can deduce that $\Theta_{q+1}$ satisfies \eqref{estimate 70} while $\theta_{q+1} = \Upsilon_{2} \Theta_{q+1}$ satisfies \eqref{estimate 86} at level $q+1$. Concerning Cauchy property in \eqref{estimate 75}, we start from \eqref{estimate 72} to obtain 
\begin{align}
& \partial_{t} (\Theta_{q+1} - \Theta_{q}) + \frac{1}{2} (\Theta_{q+1} - \Theta_{q}) - \Delta (\Theta_{q+1} - \Theta_{q}) \nonumber\\
& + \Upsilon_{1} (v_{q+1} \cdot \nabla) (\Theta_{q+1} - \Theta_{q}) + \Upsilon_{1} (v_{q+1} - v_{q}) \cdot \nabla \Theta_{q} = 0, \label{estimate 80}
\end{align} 
on which $L^{2}$-inner products with $\Theta_{q+1} - \Theta_{q}$ leads us to, for all $t \in [0, T_{L}]$
\begin{align}
&\frac{1}{2} \partial_{t} \lVert \Theta_{q+1} - \Theta_{q} \rVert_{L_{x}^{2}}^{2} + \frac{1}{2} \lVert \Theta_{q+1} - \Theta_{q} \rVert_{L_{x}^{2}}^{2} + \lVert \Theta_{q+1} - \Theta_{q} \rVert_{\dot{H}_{x}^{1}}^{2} \nonumber \\
\overset{\eqref{[Equ. (151) and (152), Y20c]}}{\leq}& e^{L^{\frac{1}{4}}} \lVert v_{q+1} - v_{q} \rVert_{L_{x}^{2}} \lVert \Theta_{q+1} - \Theta_{q} \rVert_{\dot{H}_{x}^{1}} \lVert \Theta_{q} \rVert_{L_{x}^{\infty}}  \nonumber\\
\overset{\eqref{estimate 70}}{\leq}& \frac{1}{2} \lVert \Theta_{q+1} - \Theta_{q} \rVert_{\dot{H}_{x}^{1}}^{2} + \frac{1}{2} e^{2L^{\frac{1}{4}}} \lVert v_{q+1} - v_{q} \rVert_{L_{x}^{2}}^{2}e^{-t} \lVert \theta^{\text{in}} \rVert_{L_{x}^{\infty}}^{2} \label{estimate 23}
\end{align} 
by H$\ddot{\mathrm{o}}$lder's and Young's inequalities. Subtracting $\frac{1}{2} \lVert \Theta_{q+1} - \Theta_{q} \rVert_{\dot{H}_{x}^{1}}^{2}$ from both sides, integrating over $[0,t]$, applying \eqref{[Equ. (160), Y20c]}, and taking supremum over $[0,t]$ give \eqref{estimate 75}. 

\subsubsection{Reynolds stress} 
Similarly to \eqref{[Equ. (98), Y20c]}, we can write using \eqref{estimate 71}, \eqref{[Equ. (85), Y20c]}, and  \eqref{estimate 127}
\begin{align}
& \text{div}\mathring{R}_{q+1} - \nabla p_{q+1} \label{[Equ. (178), Y20c]}\\
=& \underbrace{\frac{1}{2} w_{q+1} + (-\Delta)^{m} w_{q+1} + \partial_{t} (w_{q+1}^{(p)} + w_{q+1}^{(c)}) + \Upsilon_{1,l} \text{div} (v_{l} \otimes w_{q+1} + w_{q+1} \otimes v_{l})}_{\text{Part of } \text{div} (R_{\text{lin}}) + \nabla p_{\text{lin}}} \nonumber\\
& \hspace{30mm} \underbrace{ + \mathcal{R} (((\Upsilon_{1}^{-1} \Upsilon_{2} \Theta_{q} e^{2}) \ast_{x} \phi_{l} )\ast_{t} \varphi_{l} - \Upsilon_{1}^{-1} \Upsilon_{2} \Theta_{q+1} e^{2} ) }_{\text{Another part of } \text{div}(R_{\text{lin}} ) + \nabla p_{\text{lin}}} \nonumber\\
&+ \underbrace{\Upsilon_{1, l} \text{div}((w_{q+1}^{(c)} + w_{q+1}^{(t)}) \otimes w_{q+1} + w_{q+1}^{(p)} \otimes (w_{q+1}^{(c)} + w_{q+1}^{(t)}))}_{\text{div} (R_{\text{cor}}) + \nabla p_{\text{cor}}} \nonumber\\
& + \underbrace{\text{div}(\Upsilon_{1, l} w_{q+1}^{(p)} \otimes w_{q+1}^{(p)} + \mathring{R}_{l}) + \partial_{t}w_{q+1}^{(t)}}_{\text{div} (R_{\text{osc}}) + \nabla p_{\text{osc}}} + \underbrace{(\Upsilon_{1} - \Upsilon_{1,l}) \text{div}(v_{q+1} \otimes v_{q+1})}_{\text{div} (R_{\text{com2}}) + \nabla p_{\text{com2}}} + \text{div}(R_{\text{com1}} ) - \nabla p_{l} \nonumber
\end{align} 
with 
\begin{subequations}
\begin{align}
R_{\text{lin}} \triangleq&  \mathcal{R} ( \frac{1}{2} w_{q+1} + (-\Delta)^{m} w_{q+1} + \partial_{t} (w_{q+1}^{(p)} + w_{q+1}^{(c)})) + \Upsilon_{1,l} (v_{l} \mathring{\otimes} w_{q+1} + w_{q+1} \mathring{\otimes} v_{l}) \nonumber\\
&\hspace{10mm} + \mathcal{R} (((\Upsilon_{1}^{-1} \Upsilon_{2} \Theta_{q} e^{2}) \ast_{x} \phi_{l} )\ast_{t} \varphi_{l} - \Upsilon_{1}^{-1} \Upsilon_{2} \Theta_{q+1} e^{2} ), \label{[Equ. (179a), Y20c]} \\
p_{\text{lin}} \triangleq& \Upsilon_{1,l} (v_{l} \cdot w_{q+1}), \label{[Equ. (179b), Y20c]} \\
R_{\text{cor}} \triangleq&  \Upsilon_{1,l} ((w_{q+1}^{(c)} + w_{q+1}^{(t)}) \mathring{\otimes} w_{q+1} + w_{q+1}^{(p)} \mathring{\otimes} (w_{q+1}^{(c)} + w_{q+1}^{(t)})), \label{[Equ. (179c), Y20c]} \\
p_{\text{cor}} \triangleq&  \frac{\Upsilon_{1,l}}{2} ((w_{q+1}^{(c)} + w_{q+1}^{(t)}) \cdot w_{q+1} + w_{q+1}^{(p)} \cdot (w_{q+1}^{(c)} + w_{q+1}^{(t)})), \label{[Equ. (179d), Y20c]} \\
R_{\text{com2}} \triangleq&  (\Upsilon_{1} - \Upsilon_{1,l}) (v_{q+1} \mathring{\otimes} v_{q+1}), \label{[Equ. (179e), Y20c]} \\
p_{\text{com2}} \triangleq&  \frac{\Upsilon_{1} - \Upsilon_{1,l}}{2} \lvert v_{q+1} \rvert^{2}, \label{[Equ. (179f), Y20c]}
\end{align}
\end{subequations}
where we refer to \cite[Equ. (181)]{Y20c} for specific form of $R_{\text{osc}}$ and $p_{\text{osc}}$. We define, along with $R_{\text{com1}}$ and $p_{l}$ in \eqref{estimate 128}, 
\begin{equation}\label{estimate 363}
\mathring{R}_{q+1} \triangleq R_{\text{lin}} + R_{\text{cor}} + R_{\text{osc}} + R_{\text{com2}} + R_{\text{com1}}, p_{q+1} \triangleq p_{l} - p_{\text{lin}} - p_{\text{cor}} - p_{\text{osc}} - p_{\text{com2}}, 
\end{equation}
fix the same $p^{\ast}$ as in \eqref{[Equ. (116), Y20c]} and first rewrite within $R_{\text{lin}}$, similarly to \eqref{estimate 63} 
\begin{equation}\label{estimate 77}
\lVert \mathcal{R}([( \Upsilon_{1}^{-1} \Upsilon_{2} \Theta_{q} e^{2}) \ast_{x} \phi_{l} ]\ast_{t} \varphi_{l} - \Upsilon_{1}^{-1} \Upsilon_{2} \Theta_{q+1} e^{2}) \rVert_{C_{t}L_{x}^{p^{\ast}}} \leq V + VI, 
\end{equation}
where 
\begin{subequations}\label{estimate 98}
\begin{align}
V \triangleq& \lVert \mathcal{R} (\Upsilon_{1}^{-1} \Upsilon_{2} \Theta_{q}e^{2} - \Upsilon_{1}^{-1} \Upsilon_{2}\Theta_{q+1}e^{2}) \rVert_{C_{t}L_{x}^{p^{\ast}}}, \\
VI \triangleq& \lVert \mathcal{R} (  
[(\Upsilon_{1}^{-1} \Upsilon_{2} \Theta_{q} e^{2})\ast_{x} \phi_{l} ] \ast_{t} \varphi_{l} - \Upsilon_{1}^{-1} \Upsilon_{2} \Theta_{q} e^{2} ) \rVert_{C_{t}L_{x}^{p^{\ast}}}.  
\end{align}
\end{subequations} 
We can estimate by Lemma \ref{[Def. 9, Lem. 10, CDS12]} for all $t \in [0, T_{L}]$ 
\begin{equation}\label{estimate 79} 
V \overset{\eqref{[Equ. (151) and (152), Y20c]}}{\lesssim} e^{2L^{\frac{1}{4}}} \lVert \Theta_{q+1} - \Theta_{q} \rVert_{C_{t}L_{x}^{p^{\ast}}}. 
\end{equation}  
To deal with $\lVert \Theta_{q+1} - \Theta_{q} \rVert_{C_{t}L_{x}^{p^{\ast}}}$ in \eqref{estimate 79}, we return to \eqref{estimate 80} and compute 
\begin{equation}\label{estimate 81} 
\lVert \Theta_{q+1} - \Theta_{q} \rVert_{C_{t}L_{x}^{p^{\ast}}} \overset{\eqref{[Equ. (151) and (152), Y20c]}}{\leq} e^{L^{\frac{1}{4}}} \lVert v_{q+1} - v_{q} \rVert_{C_{t}L_{x}^{p^{\ast}}} \int_{0}^{t} \lVert  \Theta_{q} \rVert_{\dot{W}_{x}^{1,\infty}} dr 
\end{equation} 
where Remark \ref{Remark 4.4} applies again as we emphasize that the way we formulated \eqref{estimate 80} with the difference of nonlinear terms as $(v_{q+1} \cdot \nabla) (\Theta_{q+1} - \Theta_{q}) + (v_{q+1} - v_{1}) \cdot \nabla \Theta_{q}$ instead of $(v_{q+1} - v_{q}) \cdot \nabla \Theta_{q+1} + (v_{q} \cdot \nabla) (\Theta_{q+1} - \Theta_{q})$ was crucial because if we have $\int_{0}^{t} \lVert \Theta_{q+1} \rVert_{\dot{W}_{x}^{1,\infty}} dr$ in \eqref{estimate 81} instead of $\int_{0}^{t} \lVert \Theta_{q} \rVert_{\dot{W}_{x}^{1,\infty}} dr$, then it would have been too large for us to handle. To deal with $\int_{0}^{t} \lVert \Theta_{q} \rVert_{\dot{W}_{x}^{1,\infty}} dr$ in \eqref{estimate 81}, we estimate from \eqref{estimate 72} as follows: 
\begin{align}
\frac{1}{2} \partial_{t} \lVert \Theta_{q} \rVert_{\dot{H}_{x}^{2}}^{2} + \frac{1}{2} \lVert \Theta_{q} \rVert_{\dot{H}_{x}^{2}}^{2} + \lVert \Theta_{q} \rVert_{\dot{H}_{x}^{3}}^{2} 
=& \Upsilon_{1} \int_{\mathbb{T}^{2}} \nabla v_{q} \cdot \nabla \Theta_{q} \cdot \nabla \Delta \Theta_{q} - (\nabla v_{q} \cdot \nabla) \nabla \Theta_{q} \Delta \Theta_{q} dx \nonumber \\
\overset{\eqref{[Equ. (151) and (152), Y20c]}}{\lesssim}& e^{L^{\frac{1}{4}}} \lVert \nabla v_{q} \rVert_{L_{x}^{\infty}} \lVert \nabla \Theta_{q} \rVert_{L_{x}^{2}} \lVert \nabla \Delta \Theta_{q} \rVert_{L_{x}^{2}} \label{estimate 99}
\end{align} 
by integration by parts and Gagliardo-Nirenberg's inequality. Relying on Young's inequality and \eqref{[Equ. (153a) and (153b), Y20c]}, and then integrating over $[0,t]$ give us for all $t \in [0, T_{L}]$   
\begin{equation}\label{estimate 82}
\lVert \Theta_{q}(t) \rVert_{\dot{H}_{x}^{2}}^{2} + \int_{0}^{t} \lVert \Theta_{q} \rVert_{\dot{H}_{x}^{3}}^{2} dr \leq \lVert \theta^{\text{in}} \rVert_{\dot{H}_{x}^{2}}^{2} +C e^{2L^{\frac{1}{4}}} m_{L}^{2} M_{0}(t) \lambda_{q}^{8} \int_{0}^{t} \lVert \Theta_{q} \rVert_{\dot{H}_{x}^{1}}^{2} dr; 
\end{equation} 
we point out for subsequent convenience that \eqref{estimate 82} holds for $n = 3$ by identical computations. Now relying on $H^{3}(\mathbb{T}^{2}) \hookrightarrow W^{1,\infty}(\mathbb{T}^{2})$ and the equality in \eqref{estimate 70} this time lead us to, for $a \in 10 \mathbb{N}$ sufficiently large  
\begin{align}
\int_{0}^{t} \lVert \Theta_{q} \rVert_{\dot{W}_{x}^{1,\infty}} dr \overset{\eqref{estimate 82}\eqref{estimate 70}}{\lesssim}& \sqrt{t} ( \lVert \theta^{\text{in}} \rVert_{\dot{H}_{x}^{2}}^{2} + e^{2L^{\frac{1}{4}}} m_{L}^{2} M_{0}(t) \lambda_{q}^{8} \lVert \theta^{\text{in}} \rVert_{L_{x}^{2}}^{2} )^{\frac{1}{2}}\nonumber\\
& \hspace{30mm} \lesssim \sqrt{t} \lVert \theta^{\text{in}} \rVert_{H_{x}^{2}} e^{L^{\frac{1}{4}}} m_{L} M_{0}(t)^{\frac{1}{2}} \lambda_{q}^{4}. \label{estimate 83}
\end{align} 
We will apply \eqref{estimate 83} to \eqref{estimate 81}. We will still have to estimate $\lVert v_{q+1} - v_{q} \rVert_{C_{t}L_{x}^{p^{\ast}}}$ in \eqref{estimate 81}; for that purpose,  we first split $\lVert v_{q+1} - v_{q} \rVert_{C_{t}L_{x}^{p^{\ast}}}$ to $V_{1} \triangleq \lVert v_{l} - v_{q} \rVert_{C_{t}L_{x}^{p^{\ast}}}$ and $V_{2}\triangleq  \lVert v_{q+1} - v_{l} \rVert_{C_{t}L_{x}^{p^{\ast}}}$ identically to \eqref{estimate 51} where the estimate that is similar to \eqref{estimate 52} applies to $V_{1}$ as follows: 
\begin{equation}\label{estimate 364}
V_{1} \lesssim \lVert v_{l} - v_{q} \rVert_{C_{t}L_{x}^{\infty}} \lesssim l \lVert v_{q} \rVert_{C_{t,x}^{1}}  \overset{\eqref{[Equ. (153a) and (153b), Y20c]}}{\lesssim} l m_{L}M_{0}(t)^{\frac{1}{2}} \lambda_{q}^{4} \overset{\eqref{[Equ. (70), Y20c]}}{\lesssim} \lambda_{q+1}^{-\alpha} m_{L}M_{0}(t)^{\frac{1}{2}}. 
\end{equation} 
The estimate on $V_{2} = \lVert v_{q+1} - v_{l} \rVert_{C_{t}L_{x}^{p^{\ast}}}$ is more subtle . We proceed as follows:
\begin{align}
V_{2} \overset{\eqref{[Equ. (85), Y20c]}}{\leq}& \lVert w_{q+1}^{(p)} \rVert_{C_{t}L_{x}^{p^{\ast}}} + \lVert w_{q+1}^{(c)} \rVert_{C_{t}L_{x}^{p^{\ast}}} + \lVert w_{q+1}^{(t)} \rVert_{C_{t}L_{x}^{p^{\ast}}} \nonumber\\
\overset{\eqref{[Equ. (89b) and (89c), Y20c]} \eqref{[Equ. (172a) and (172b), Y20c]} \eqref{[Equ. (172c), Y20c]}}{\lesssim}& m_{L} \delta_{q+1}^{\frac{1}{2}} M_{0}(t) l^{-\frac{3}{2}} r^{1- \frac{2}{p^{\ast}}}  + m_{L} \delta_{q+1}^{\frac{1}{2}} M_{0}(t)^{\frac{1}{2}} l^{-\frac{11}{2}} \sigma r^{2- \frac{2}{p^{\ast}}} + \mu^{-1} \delta_{q+1} M_{0}(t) l^{-3} r^{2- \frac{2}{p^{\ast}}} \nonumber \\
\overset{\eqref{[Equ. (70), Y20c]} \eqref{[Equ. (67), Y20c]} \eqref{[Equ. (116), Y20c]} }{\lesssim}& m_{L} M_{0}(t) \lambda_{q+1}^{- \frac{69 \alpha}{2} - 1 + 8 \eta} + m_{L} M_{0}(t)^{\frac{1}{2}} \lambda_{q+1}^{ - \frac{53\alpha}{2} + 4 \eta -1} + M_{0}(t) \lambda_{q+1}^{6 \eta - 1 - \frac{63 \alpha}{2}}  \label{estimate 129}
\end{align} 
where we used that 
\begin{align*}
&3 \alpha + (1- 6 \eta) (1 - \frac{2}{p^{\ast}}) \overset{\eqref{[Equ. (116), Y20c]}}{=} - \frac{69 \alpha}{2} - 1 + 8 \eta,\\
& 11 \alpha + 2 \eta - 1 + (1-6 \eta) (2 - \frac{2}{p^{\ast}})  \overset{\eqref{[Equ. (116), Y20c]}}{=}  - \frac{53\alpha}{2} + 4 \eta -1, \\
& 4 \eta - 1 + 6 \alpha + (1-6\eta) (2 - \frac{2}{p^{\ast}}) \overset{\eqref{[Equ. (116), Y20c]}}{=} 6 \eta - 1 - \frac{63 \alpha}{2}.
\end{align*}
By \eqref{[Equ. (64), Y20c]}-\eqref{[Equ. (66), Y20c]} it follows that 
\begin{equation}\label{estimate 84} 
V_{2} \overset{\eqref{estimate 129} \eqref{[Equ. (64), Y20c]} \eqref{[Equ. (65), Y20c]} \eqref{[Equ. (66), Y20c]}}{\lesssim} m_{L} M_{0}(t) \lambda_{q+1}^{ - \frac{69 \alpha}{2} - 1 + 8 \eta} \overset{\eqref{[Equ. (65), Y20c]}}{\lesssim} m_{L} M_{0}(t) \lambda_{q+1}^{ - \frac{69 \alpha}{2}}. 
\end{equation} 
Summing \eqref{estimate 364} and \eqref{estimate 84} gives for $a \in 10 \mathbb{N}$ sufficiently large for all $t \in [0, T_{L}]$ 
\begin{equation}\label{estimate 85}
\lVert v_{q+1} - v_{q} \rVert_{C_{t}L_{x}^{p^{\ast}}} \lesssim m_{L} M_{0}(t)^{\frac{1}{2}} \lambda_{q+1}^{-\alpha} + m_{L} M_{0}(t) \lambda_{q+1}^{- \frac{69\alpha}{2}} \lesssim m_{L} M_{0}(t) \lambda_{q+1}^{-\alpha}. 
\end{equation} 
We now apply \eqref{estimate 83} and \eqref{estimate 85} to \eqref{estimate 81} and obtain for all $t \in [0, T_{L}]$ 
\begin{equation}\label{estimate 90} 
\lVert \Theta_{q+1} - \Theta_{q} \rVert_{C_{t}L_{x}^{p^{\ast}}} 
\lesssim e^{2L^{\frac{1}{4}}} m_{L}^{2} M_{0}(t)^{\frac{3}{2}} \lambda_{q+1}^{-\alpha} \lambda_{q}^{4} \sqrt{t} \lVert \theta^{\text{in}} \rVert_{H_{x}^{2}}. 
\end{equation} 
Applying \eqref{estimate 90} to \eqref{estimate 79} and taking $a \in 10 \mathbb{N}$ sufficiently large lead to, for all $t \in [0, T_{L}]$,  
\begin{align}
V\overset{\eqref{estimate 79} \eqref{estimate 90}}{\lesssim} e^{4L^{\frac{1}{4}}} m_{L}^{2} M_{0}(t)^{\frac{3}{2}} \lambda_{q+1}^{-\alpha} \lambda_{q}^{4} \sqrt{t} \lVert \theta^{\text{in}} \rVert_{H_{x}^{2}} \ll c_{R}M_{0}(t) \delta_{q+2} \label{estimate 91}
\end{align} 
where we used \eqref{estimate 130}. Next, to handle $VI$ from \eqref{estimate 98}, we estimate as follows for any $\epsilon \in (0, 2- \frac{n}{2})$; for subsequent convenience, we compute for general $n \in \{2,3\}$, with the current case being $n = 2$. First,
\begin{equation}\label{estimate 353}
VI \overset{\eqref{[Equ. (151) and (152), Y20c]}}{\lesssim} l^{\frac{1}{2} - 2 \delta} e^{2L^{\frac{1}{4}}} (\lVert \Theta_{q} \rVert_{C_{t}^{\frac{1}{2} - 2 \delta} \dot{H}_{x}^{\frac{n}{2} - 1 + \epsilon}} + \lVert \Theta_{q} \rVert_{C_{t} \dot{H}_{x}^{\frac{n}{2} - \frac{1}{2} + \epsilon - 2 \delta}}). 
\end{equation} 
We apply $\nabla$ on \eqref{estimate 72}  and compute 
\begin{equation*}
\partial_{t} \nabla \Theta_{q} + \frac{1}{2} \nabla \Theta_{q} - \Delta \nabla \Theta_{q} + \Upsilon_{1} \nabla (v_{q} \cdot \nabla \Theta_{q}) = 0
\end{equation*} 
which leads to for all $t \in [0, T_{L}]$  
\begin{align}
\int_{0}^{t} \lVert \partial_{t} \nabla \Theta_{q} \rVert_{L_{x}^{2}}^{2} dr \lesssim& \int_{0}^{t} \lVert \Theta_{q} \rVert_{\dot{H}_{x}^{1}}^{2} dr + \int_{0}^{t} \lVert \Theta_{q} \rVert_{\dot{H}_{x}^{3}}^{2} dr + e^{2L^{\frac{1}{4}}} \lVert \nabla v_{q} \rVert_{C_{t,x}}^{2} \int_{0}^{t} \lVert \Theta_{q} \rVert_{\dot{H}_{x}^{2}}^{2} dr \nonumber \\
& \hspace{20mm} \overset{ \eqref{estimate 82}\eqref{estimate 70} \eqref{[Equ. (153a) and (153b), Y20c]} }{\lesssim} e^{4L^{\frac{1}{4}}} m_{L}^{4}M_{0}(t)^{2} \lambda_{q}^{16} \lVert \theta^{\text{in}} \rVert_{H_{x}^{2}}^{2} (t+1), \label{estimate 350}
\end{align} 
where we used \eqref{estimate 82}; recall that its computation was general in spatial dimension $n \in \{2,3\}$. This leads us to, for all $t \in [0, T_{L}]$, 
\begin{equation}\label{estimate 351}
 \lVert \Theta_{q} \rVert_{C_{t}^{\frac{1}{2} - 2 \delta} \dot{H}_{x}^{\frac{n}{2} - 1 + \epsilon}} \overset{\eqref{estimate 350}}{\lesssim} e^{2L^{\frac{1}{4}}} m_{L}^{2}M_{0}(t) \lambda_{q}^{8} \lVert \theta^{\text{in}} \rVert_{H_{x}^{2}} (\sqrt{t} + 1)
\end{equation} 
while 
\begin{equation}\label{estimate 352}
\lVert \Theta_{q} \rVert_{C_{t} \dot{H}_{x}^{\frac{n}{2} - \frac{1}{2} + \epsilon -  2 \delta}} \lesssim \lVert \Theta_{q} \rVert_{C_{t} \dot{H}_{x}^{2}}  \overset{\eqref{estimate 82}\eqref{estimate 70}}{\lesssim} \lVert \theta^{\text{in}} \rVert_{H_{x}^{2}} e^{L^{\frac{1}{4}}} m_{L} M_{0}(t)^{\frac{1}{2}} \lambda_{q}^{4}. 
\end{equation}
Applying \eqref{estimate 351}-\eqref{estimate 352} to \eqref{estimate 353} finally gives us for $\delta \in (0, \frac{1}{12})$ and all $t \in [0, T_{L}]$  
\begin{align}
VI\overset{\eqref{[Equ. (69), Y20c]}}{ \lesssim}& (\lambda_{q+1}^{-\frac{3\alpha}{2}} \lambda_{q}^{-2})^{\frac{1}{3}} e^{4L^{\frac{1}{4}}} m_{L}^{2} M_{0}(t) \lambda_{q}^{8} \lVert \theta^{\text{in}} \rVert_{H_{x}^{2}} (\sqrt{t} + 1)  \nonumber \\
\lesssim & c_{R} M_{0}(t) \delta_{q+2} a^{b^{q+1} [ - \frac{\alpha}{48}]} m_{L}^{2}e^{4L^{\frac{1}{4}}} \lVert \theta^{\text{in}} \rVert_{H_{x}^{2}} (\sqrt{t}+ 1) \ll c_{R} M_{0}(t) \delta_{q+2} \label{estimate 354}
\end{align} 
where we used that 
\begin{equation*}
-\frac{\alpha}{2} - \frac{2}{3b} + \frac{8}{b} + 2 \beta b \overset{\eqref{[Equ. (68), Y20c]}}{<} -\frac{\alpha}{2} + \frac{22}{3}(\frac{1}{b}) + \frac{\alpha}{48} \overset{\eqref{estimate 62}}{<}  - \frac{\alpha}{48}.
\end{equation*} 
Applying \eqref{estimate 91} and \eqref{estimate 354} to \eqref{estimate 77} gives us for all $t \in [0, T_{L}]$ 
\begin{equation}\label{estimate 355}
\lVert \mathcal{R} (((\Upsilon_{1}^{-1} \Upsilon_{2} \Theta_{q} e^{2}) \ast_{x} \phi_{l} )\ast_{t} \varphi_{l}  - \Upsilon_{1}^{-1} \Upsilon_{2} \Theta_{q+1} e^{2}  ) \rVert_{C_{t}L_{x}^{p^{\ast}}}  \ll c_{R} M_{0}(t) \delta_{q+2}. 
\end{equation} 
On the other hand, from \cite[Equ. (179a) and (188)]{Y20c} we can see that 
\begin{align}
&\lVert R_{\text{lin}} - \mathcal{R} (((\Upsilon_{1}^{-1} \Upsilon_{2} \Theta_{q} e^{2}) \ast_{x} \phi_{l} )\ast_{t} \varphi_{l}  - \Upsilon_{1}^{-1} \Upsilon_{2} \Theta_{q+1} e^{2}) \rVert_{C_{t}L_{x}^{p^{\ast}}} \nonumber \\
& \hspace{20mm} \lesssim M_{0}(t) \delta_{q+2} [m_{L} \lambda_{q+1}^{- \frac{275 \alpha}{8}} + m_{L}^{4} \lambda_{q+1}^{ \frac{ - 273 \alpha - 8 + 64 \eta}{8}}] 
\end{align} 
so that, because $\eta \leq \frac{1}{8}$ due to \eqref{[Equ. (65), Y20c]}, together with \eqref{estimate 355} we can conclude that   
\begin{equation}\label{estimate 133} 
\lVert R_{\text{lin}} \rVert_{C_{t}L_{x}^{p^{\ast}}}  \leq (2\pi)^{-2 ( \frac{p^{\ast} -1}{p^{\ast}})}c_{R} M_{0}(t) \delta_{q+2}
\end{equation} 
by taking $a \in 10 \mathbb{N}$ sufficiently large. By \cite[Equ. (131), (189)-(191)]{Y20c} we have \eqref{estimate 132}. Thus, together with \eqref{estimate 133}, we see that $\lVert \mathring{R}_{q+1} \rVert_{C_{t}L_{x}^{1}} \leq c_{R} M_{0}(t) \delta_{q+2}$ so that \eqref{[Equ. (153c), Y20c]} at level $q+1$ has been proven. The rest of the arguments are similar to those of proof of Proposition \ref{Proposition 4.8 for n=2}.  

Next, we consider the case $n = 3$ so that $m \in (\frac{13}{20}, \frac{5}{4})$ by \eqref{4}. For notations and preliminaries  throughout this subsection, we refer again to Subsection \ref{Preliminaries needed for convex integration in 3D case and more}. 
\begin{proposition}\label{Proposition 5.6 for n=3}
Fix $\theta^{\text{in}} \in H^{2}(\mathbb{T}^{3})$ that is deterministic and mean-zero. Let 
\begin{equation}\label{estimate 111}
v_{0}(t,x) \triangleq \frac{m_{L} e^{2L t + L}}{(2\pi)^{\frac{3}{2}} } 
\begin{pmatrix}
\sin(x^{3}) & 0 & 0 
\end{pmatrix}^{T}. 
\end{equation}
Then there exists a unique solution $\Theta_{0} \in L_{\omega}^{\infty} L_{t}^{\infty} H_{x}^{2} \cap L_{\omega}^{\infty} L_{t}^{2} H_{x}^{3}$ to \eqref{estimate 109}. It follows that together with 
\begin{align}
\mathring{R}_{0}(t,x) \triangleq& \frac{m_{L} (2L + \frac{1}{2}) e^{2L t + L}}{(2\pi)^{\frac{3}{2}}} 
\begin{pmatrix}
0 & 0 & -\cos(x^{3}) \\
0 & 0 & 0 \\
-\cos(x^{3}) & 0 & 0
\end{pmatrix} \nonumber\\
& \hspace{20mm} 
+ \mathcal{R} (-\Delta)^{m} v_{0} - \mathcal{R} (\Upsilon_{1}^{-1} \Upsilon_{2} \Theta_{0} e^{3}), \label{estimate 112}
\end{align} 
$(v_{0}, \Theta_{0})$ satisfy  \eqref{[Equ. (149), Y20c]} at level $q = 0$. Moreover, \eqref{[Equ. (153), Y20c]} at level $q = 0$ is satisfied provided 
\begin{subequations}\label{[Equ. (129), Y20a]} 
\begin{align} 
& \sqrt{3} \lVert \theta^{\text{in}} \rVert_{L_{x}^{2}} \leq L^{\frac{1}{4}} e^{L - (\frac{3}{2}) L^{\frac{1}{4}}}, \label{estimate 95} \\
&18 (2\pi)^{\frac{3}{2}} \sqrt{3} < 2 (2\pi)^{\frac{3}{2}} \sqrt{3} a^{2\beta b} \leq \frac{c_{R} e^{L}}{L^{\frac{1}{4}} (2L + 26) e^{\frac{1}{2} L^{\frac{1}{4}}}}, \text{ and } L \leq \frac{(2\pi)^{\frac{3}{2}} a^{4} -2}{2}, \label{estimate 136}
\end{align} 
\end{subequations} 
where the inequality $9 < a^{2 \beta b}$ in \eqref{estimate 136} is assumed for the justification of second inequality of \eqref{[Equ. (153a) and (153b), Y20c]}. Furthermore, $v_{0}(0,x)$ and $\mathring{R}_{0}(0,x)$ are both deterministic. 
\end{proposition} 

\begin{proof}[Proof of Proposition \ref{Proposition 5.6 for n=3}]
For all $t \in [0, T_{L}]$ we can deduce the same estimates in \eqref{[Equ. (157), Y20c]} by \eqref{[Equ. (129), Y20a]} (see \cite[Equ. (130)]{Y20a}). Identically to the proof of Proposition \ref{Proposition 5.6 for n=2}, we see that $\theta_{0}$ and $\Theta_{0} = \Upsilon_{2}^{-1} \theta_{0}$ satisfy \eqref{estimate 86}-\eqref{estimate 70}. Using \eqref{estimate 70}, \eqref{[Equ. (157), Y20c]}, and \eqref{[Equ. (151) and (152), Y20c]}, as well as the fact that $\Delta v_{0} = -v_{0}$, one can bootstrap from \eqref{estimate 70} to verify that $\Theta_{0} \in L_{\omega}^{\infty} L_{t}^{\infty} H_{x}^{2} \cap L_{\omega}^{\infty} L_{t}^{2} H_{x}^{3}$ again. Moreover, because $v_{0}(t)$ and $\Theta_{0}(t)$ are mean-zero for all $t \geq 0$, we see that $\mathring{R}_{0}$ is also trace-free and symmetric by Lemma \ref{divergence inverse operator}. It can be immediately verified that \eqref{estimate 71}  holds  if $p_{0} \equiv 0$. To verify the bound on $\lVert\mathring{R}_{0}\rVert_{L_{x}^{1}}$ in \eqref{[Equ. (153c), Y20c]}, we compute using Lemma \ref{divergence inverse operator} and the fact that $\Theta_{0}(t)$ is mean-zero for all $t \geq 0$, 
\begin{align}
\lVert \mathcal{R} ( \Upsilon_{1}^{-1} \Upsilon_{2} \Theta_{0} e^{3})(t) \rVert_{L_{x}^{1}} \overset{\eqref{[Equ. (151) and (152), Y20c]}}{\leq} 6 (2\pi)^{\frac{3}{2}} \lVert \theta^{\text{in}} \rVert_{L_{x}^{2}} e^{2L^{\frac{1}{4}}}  \overset{\eqref{[Equ. (129), Y20a]}}{\leq} (2\pi)^{\frac{3}{2}} 2 m_{L} e^{2L t + L}.  \label{estimate 134}
\end{align} 
It can be seen from the first inequality in \cite[Equ. (131)]{Y20a} that 
\begin{equation}\label{estimate 135}
\lVert \mathring{R}_{0}(t) + \mathcal{R} (\Upsilon_{1}^{-1} \Upsilon_{2} \Theta_{0} e^{3})(t) \rVert_{L_{x}^{1}} \leq m_{L} ( 2 L + \frac{1}{2}) e^{2L t + L} 8 (2\pi)^{\frac{1}{2}} + (2\pi)^{\frac{3}{2}} 48  \lVert v_{0} \rVert_{L_{x}^{2}}
\end{equation} 
and hence \eqref{[Equ. (157), Y20c]}, \eqref{estimate 134}-\eqref{estimate 135} verify that  
\begin{equation}
\lVert \mathring{R}_{0} \rVert_{C_{t}L_{x}^{1}} \leq m_{L} M_{0}(t)^{\frac{1}{2}}[ (2 L + \frac{1}{2}) 8 (2\pi)^{\frac{1}{2}} + (2\pi)^{\frac{3}{2}} 48 + 2 (2\pi)^{\frac{3}{2}}]  \overset{\eqref{[Equ. (129), Y20a]}}{\leq} M_{0}(t) c_{R} \delta_{1}.
\end{equation} 
\end{proof}

\begin{proposition}\label{Proposition 5.7 for n=3}
Fix $\theta^{\text{in}} \in H^{2}(\mathbb{T}^{3})$ that is deterministic and mean-zero from the hypothesis of Proposition \ref{Proposition 5.6 for n=3}. Let $L$ satisfy \eqref{estimate 95} and 
\begin{equation}\label{[Equ. (132), Y20a]} 
18 (2\pi)^{\frac{3}{2}} \sqrt{3} < \frac{c_{R} e^{L}}{L^{\frac{1}{4}} (2L + 26) e^{\frac{1}{2} L^{\frac{1}{4}}}},  
\end{equation} 
and suppose that $(v_{q}, \Theta_{q}, \mathring{R}_{q})$ are $(\mathcal{F}_{t})_{t\geq 0}$-adapted processes that solve \eqref{[Equ. (149), Y20c]} and satisfy \eqref{[Equ. (153), Y20c]}. Then there exist a choice of parameters $a, b,$ and $\beta$ such that \eqref{estimate 136} is fulfilled and an $(\mathcal{F}_{t})_{t\geq 0}$-adapted processes $(v_{q+1}, \Theta_{q+1}, \mathring{R}_{q+1})$ that satisfy \eqref{[Equ. (149), Y20c]}, \eqref{[Equ. (153), Y20c]} at level $q+1$, and \eqref{[Equ. (160), Y20c]}-\eqref{estimate 75}. Finally, if $v_{q}(0,x)$ and $\mathring{R}_{q}(0,x)$ are deterministic, then so are $v_{q+1}(0,x)$ and $\mathring{R}_{q+1}(0,x)$. 
\end{proposition}  

\begin{proof}[Proof of Theorem \ref{Theorem 2.3} if $n=3$ assuming Proposition \ref{Proposition 5.7 for n=3}]
Fix $\theta^{\text{in}} \in H^{2}(\mathbb{T}^{3})$ that is deterministic and mean-zero from the hypothesis of Proposition \ref{Proposition 5.6 for n=3}, any $T > 0, K > 1$, and $\kappa \in (0,1)$. Then we take $L$ that satisfies \eqref{[Equ. (132), Y20a]} and enlarge it if necessary to satisfy \eqref{[Equ. (208), Y20c]}. Starting from $(v_{0}, \Theta_{0}, \mathring{R}_{0})$ in Proposition \ref{Proposition 5.6 for n=3}, Proposition \ref{Proposition 5.7 for n=3} inductively gives us $(v_{q}, \Theta_{q}, \mathring{R}_{q})$ that satisfies \eqref{[Equ. (149), Y20c]}, \eqref{[Equ. (153), Y20c]}, and \eqref{[Equ. (160), Y20c]}-\eqref{estimate 75}. Identically to the proof of Theorem \ref{Theorem 2.3} in case $n=2$, we can deduce the limiting solution $\lim_{q\to \infty} v_{q} \triangleq v \in C_{T_{L}} \dot{H}_{x}^{\gamma}$ and $\lim_{q\to \infty} \Theta_{q} \triangleq \Theta \in \cap_{p \in [1,\infty)} C_{T_{L}} L_{x}^{p} \cap L_{T_{L}}^{2} \dot{H}_{x}^{1}$ which are both $(\mathcal{F}_{t})_{t\geq 0}$-adapted. It follows that $(v,\Theta)$ solves  \eqref{[Equ. (148), Y20c]}, that $\lVert v(t) - v_{0}(t) \rVert_{L_{x}^{2}} \leq \frac{m_{L}}{2} M_{0}(t)^{\frac{1}{2}}$ for all $t \in [0, T_{L}]$, and that \eqref{[Equ. (209), Y20c]} holds. Then \eqref{[Equ. (157), Y20c]}, \eqref{[Equ. (209), Y20c]}, and \eqref{[Equ. (208), Y20c]} lead to \eqref{[Equ. (210), Y20c]} on a set $\{T_{L} \geq T \}$. At last, identically to the proof of Theorem \ref{Theorem 2.3} in case $n = 2$, we can deduce \eqref{6}, that $u^{\text{in}}$ is deterministic, $(u,\theta)$ are both $(\mathcal{F}_{t})_{t\geq 0}$-adapted, and that  \eqref{[Equ. (6), Y20c]} holds. 
\end{proof} 

\subsection{Convex integration to prove Proposition \ref{Proposition 5.7 for n=3}}
\subsubsection{Choice of parameters}
We fix $L$ sufficiently large so that it satisfies \eqref{estimate 95} and \eqref{[Equ. (132), Y20a]}. We take same $\alpha$ from \eqref{[Equ. (52), Y20a]}, $l$ from \eqref{[Equ. (69), Y20c]}, $b$ from \eqref{estimate 62}, and $\beta$ sufficiently small so that at least \eqref{[Equ. (68), Y20c]} holds, Considering the requirement from \eqref{estimate 136}, $L \leq \frac{ (2\pi)^{\frac{3}{2}} a^{4} -2}{2}$ is satisfied by taking $a > 0$ sufficiently large while the other two inequalities in \eqref{estimate 136} can be achieved by taking $\beta > 0$ sufficiently small. We also use same definitions of $r_{\lVert}, r_{\bot}$, and $\mu$ in \eqref{[Equ. (62), Y20a]}. With such $\alpha, L$, and $b$ fixed, we take $a > 0$ such that $a^{\frac{25 - 20m}{24}} \in \mathbb{N}$, which will be needed in deriving \eqref{[Equ. (143), Y20a]}, as large and $\beta > 0$ as small as needed.

\subsubsection{Mollification}
Identical mollifications to \eqref{[Equ. (71), Y20c]} and \eqref{estimate 96} with the only exception being $\phi_{\epsilon} (\cdot) \triangleq \frac{1}{\epsilon^{3}} \phi(\frac{\cdot}{\epsilon})$ lead us to 
\begin{equation}\label{[Equ. (138), Y20a]}
 \partial_{t} v_{l} + \frac{1}{2} v_{l} + (-\Delta)^{m} v_{l} + \Upsilon_{1,l}\text{div} (v_{l} \otimes v_{l}) + \nabla p_{l}  =  \text{div}( \mathring{R}_{l}  + R_{\text{com1}}) + ((\Upsilon_{1}^{-1} \Upsilon_{2} \Theta_{q} e^{3}) \ast_{x} \phi_{l})\ast_{t} \varphi_{l} 
\end{equation}
with $R_{\text{com1}}$ identical to that in \eqref{estimate 147} while 
\begin{equation}\label{[Equ. (164a), Y20a]} 
p_{l} \triangleq (p_{q} \ast_{x} \phi_{l}) \ast_{t} \varphi_{l} - \frac{1}{3} (\Upsilon_{1,l} \lvert v_{l} \rvert^{2} - ((\Upsilon_{1} \lvert v_{q} \rvert^{2} )\ast_{x} \phi_{l} ) \ast_{t} \varphi_{l}).
\end{equation}
We can verify the same estimates in \eqref{[Equ. (165a) and (165b), Y20c]}-\eqref{[Equ. (165c), Y20c]} (see \cite[Equ. (140)]{Y20a}). 

\subsubsection{Perturbation}

We proceed with same definition of $\chi$ in \eqref{[Equ. (75), Y20c]} and $\rho$ in \eqref{[Equ. (76), Y20c]} so that  \eqref{[Equ. (77), Y20c]} follows, only with $M_{0}(t)$ defined now by \eqref{[Equ. (150), Y20c]}. We define the modified amplitude function $\bar{a}_{\zeta}$ identically to \eqref{[Equ. (166), Y20c]} so that it continues to satisfy \eqref{[Equ. (168), Y20c]}. Additionally, we can estimate for any $N \geq 0$ and $k \in \{0,1,2\}$, 
\begin{equation}\label{[Equ. (143), Y20a]}
\lVert \bar{a}_{\zeta} \rVert_{C_{t} C_{x}^{N}} \overset{\eqref{[Equ. (73), Y20a]}}{\leq} m_{L} c_{R}^{\frac{1}{4}} \delta_{q+1}^{\frac{1}{2}} M_{0}(t)^{\frac{1}{2}} l^{-2-5N},  \lVert \bar{a}_{\zeta} \rVert_{C_{t}^{1}C_{x}^{k}} \overset{\eqref{[Equ. (73), Y20a]}\eqref{[Equ. (151) and (152), Y20c]} }{\leq} m_{L} c_{R}^{\frac{1}{8}} \delta_{q+1}^{\frac{1}{2}} M_{0}(t)^{\frac{1}{2}} l^{-2 - (k+1) 5}  
\end{equation}
(see \cite[Equ. (143)]{Y20a}). We define $w_{q+1}^{(p)}$ and $w_{q+1}^{(c)}$ identically to \eqref{[Equ. (74a) and (74b), Y20a]}, only with $\bar{a}_{\zeta}$ instead of $a_{\zeta}$ while we continue to use the same definition of $w_{q+1}^{(t)}$ in \eqref{[Equ. (74c), Y20a]}. Identically to \eqref{[Equ. (85), Y20c]}, we define $w_{q+1}$ and $v_{q+1}$ which are both divergence-free. For all $t \in [0, T_{L}]$ and $p \in (1,\infty)$, by relying on \cite[Lem. 7.4]{BV19b} we can show that they satisfy  
\begin{subequations}\label{[Equ. (146), Y20a]}
\begin{align}
& \lVert w_{q+1}^{(p)} \rVert_{C_{t}L_{x}^{2}} \leq \frac{1}{2} m_{L} M_{0}(t)^{\frac{1}{2}} \delta_{q+1}^{\frac{1}{2}}, \hspace{3mm} \lVert w_{q+1}^{(p)} \rVert_{C_{t}L_{x}^{p}} \overset{\eqref{[Equ. (143), Y20a]} \eqref{[Equ. (174c), Y20a]} }{\lesssim} m_{L} M_{0}(t)^{\frac{1}{2}} \delta_{q+1}^{\frac{1}{2}} l^{-2} r_{\bot}^{\frac{2}{p} - 1} r_{\lVert}^{\frac{1}{p} - \frac{1}{2}}, \label{[Equ. (146a) and (146b), Y20a]}\\
& \lVert w_{q+1}^{(c)} \rVert_{C_{t}L_{x}^{p}} \overset{\eqref{[Equ. (143), Y20a]} \eqref{[Equ. (174c), Y20a]}}{\lesssim} m_{L} \delta_{q+1}^{\frac{1}{2}} M_{0}(t)^{\frac{1}{2}} l^{-12} r_{\bot}^{\frac{2}{p}} r_{\lVert}^{\frac{1}{p} - \frac{3}{2}}  \label{[Equ. (146c), Y20a]}
\end{align}
\end{subequations} 
(see \cite[Equ. (146)]{Y20a})
while the bound \eqref{[Equ. (78b) and (78c), Y20a]} on $\lVert w_{q+1}^{(t)} \rVert_{C_{t}L_{x}^{p}}$ remains valid. These lead to 
\begin{equation}\label{[Equ. (147), Y20a]}
 \lVert w_{q+1} \rVert_{C_{t}L_{x}^{2}} \overset{\eqref{[Equ. (85), Y20c]}}{\leq} \lVert w_{q+1}^{(p)} \rVert_{C_{t}L_{x}^{2}} + \lVert w_{q+1}^{(c)} \rVert_{C_{t}L_{x}^{2}} + \lVert w_{q+1}^{(t)} \rVert_{C_{t}L_{x}^{2}}   \overset{\eqref{[Equ. (70), Y20c]} \eqref{[Equ. (78b) and (78c), Y20a]} \eqref{[Equ. (146), Y20a]}  \eqref{[Equ. (151) and (152), Y20c]} }{\leq} \frac{3m_{L} M_{0}(t)^{\frac{1}{2}} \delta_{q+1}^{\frac{1}{2}}}{4}
\end{equation} 
which, together with \eqref{[Equ. (85), Y20c]}, \eqref{[Equ. (165a) and (165b), Y20c]}, and \eqref{[Equ. (147), Y20a]}, verifies the first inequality of \eqref{[Equ. (153a) and (153b), Y20c]} at level $q+1$ and \eqref{[Equ. (160), Y20c]}  (see \cite[Equ. (147)]{Y20a}). Moreover, we can show that for all $t \in [0, T_{L}]$ 
\begin{equation}\label{[Equ. (148), Y20a]}
\lVert w_{q+1}^{(p)} \rVert_{C_{t,x}^{1}} \overset{\eqref{[Equ. (143), Y20a]}}{\lesssim} m_{L} M_{0}(t)^{\frac{1}{2}} l^{-7} r_{\bot}^{-1} r_{\lVert}^{-\frac{1}{2}} \lambda_{q+1}^{2m}, \hspace{5mm}  \lVert w_{q+1}^{(c)} \rVert_{C_{t,x}^{1}}  \overset{\eqref{[Equ. (143), Y20a]}}{\lesssim} m_{L} M_{0}(t)^{\frac{1}{2}} l^{-17} r_{\lVert}^{-\frac{3}{2}} \lambda_{q+1}^{2m} 
\end{equation} 
(see \cite[Equ. (148)]{Y20a}). Along with the bound on $w_{q+1}^{(t)}$ in \eqref{[Equ. (85), Y20a]} that remains valid, one can now verify the second inequality in \eqref{[Equ. (153a) and (153b), Y20c]} at level $q+1$ via \eqref{[Equ. (85), Y20c]}.  At last, identically to the proof of Proposition \ref{Proposition 5.7 for n=2}, we can deduce that $\Theta_{q+1}$ satisfies \eqref{estimate 70} at level $q+1$, while $\theta_{q+1} = \Upsilon_{2}\Theta_{q+1}$ satisfies \eqref{estimate 86} at level $q+1$.  The proof of the Cauchy property \eqref{estimate 75} also follows from identical computations in \eqref{estimate 80}-\eqref{estimate 23}.  

\subsubsection{Reynolds stress} 
Identically to \eqref{estimate 363}, along with $p_{l}$ from \eqref{[Equ. (164a), Y20a]} and $R_{\text{com1}}$ from \eqref{estimate 147}, due to \eqref{estimate 71}, \eqref{[Equ. (85), Y20c]}, and \eqref{[Equ. (138), Y20a]}, we can define $\mathring{R}_{q+1}$ and $p_{q+1}$ with 
\begin{subequations}\label{[Equ. (152), Y20a]}
\begin{align}
& R_{\text{lin}} \triangleq \mathcal{R} ( \frac{1}{2} w_{q+1} + (-\Delta)^{m} w_{q+1} + \partial_{t} (w_{q+1}^{(p)} + w_{q+1}^{(c)})) + \Upsilon_{1,l} (v_{l} \mathring{\otimes} w_{q+1} + w_{q+1} \mathring{\otimes}v_{l}) \nonumber\\
& \hspace{30mm} + \mathcal{R} ( ((\Upsilon_{1}^{-1} \Upsilon_{2} \Theta_{q} e^{3} )\ast_{x} \phi_{l}) \ast_{t} \varphi_{l}  -\Upsilon_{1}^{-1} \Upsilon_{2} \Theta_{q+1} e^{3} ), \label{[Equ. (152a), Y20a]}\\
& p_{\text{lin}} \triangleq \Upsilon_{1,l} (\frac{2}{3}) (v_{l} \cdot w_{q+1}), \\
& R_{\text{cor}} \triangleq \Upsilon_{1,l}((w_{q+1}^{(c)} + w_{q+1}^{(t)}) \mathring{\otimes} w_{q+1} + w_{q+1}^{(p)} \mathring{\otimes} (w_{q+1}^{(c)} + w_{q+1}^{(t)})), \\
& p_{\text{cor}} \triangleq \frac{\Upsilon_{1,l}}{3} ((w_{q+1}^{(c)} + w_{q+1}^{(t)}) \cdot w_{q+1} + w_{q+1}^{(p)} \cdot (w_{q+1}^{(c)} + w_{q+1}^{(t)})) , \\
&R_{\text{osc}} \triangleq \sum_{\zeta \in \Lambda} \mathcal{R} (\nabla a_{\zeta}^{2} \mathbb{P}_{\neq 0} (W_{\zeta} \otimes W_{\zeta})) -  \mu^{-1} \sum_{\zeta \in \Lambda} \mathcal{R} (\partial_{t} a_{\zeta}^{2} \phi_{\zeta}^{2} \psi_{\zeta}^{2} \zeta), \\
&p_{\text{osc}} \triangleq \rho + \Delta^{-1} \text{div}[ \mu^{-1} \sum_{\zeta \in \Lambda } \mathbb{P}_{\neq 0} \partial_{t} (a_{\zeta}^{2} \phi_{\zeta}^{2} \psi_{\zeta}^{2} \zeta)], \\
& R_{\text{com2}} \triangleq (\Upsilon_{1} - \Upsilon_{1,l}) (v_{q+1} \mathring{\otimes} v_{q+1}), \\
& p_{\text{com2}}  \triangleq \frac{ \Upsilon_{1} - \Upsilon_{1,l}}{3} \lvert v_{q+1} \rvert^{2}, 
\end{align}
\end{subequations}
We use the same $p^{\ast}$ in \eqref{[Equ. (93), Y20a]}. We can split $\mathcal{R} (((\Upsilon_{1}^{-1} \Upsilon_{2} \Theta_{q} e^{3} )\ast_{x} \phi_{l}) \ast_{t} \varphi_{l} -\Upsilon_{1}^{-1} \Upsilon_{2} \Theta_{q+1} e^{3})$ within \eqref{[Equ. (152a), Y20a]} identically to \eqref{estimate 98} and retain the same estimates \eqref{estimate 79}-\eqref{estimate 82}, and even \eqref{estimate 83} which used $H^{3}(\mathbb{T}^{2}) \hookrightarrow W^{1,\infty} (\mathbb{T}^{2})$ because $H^{3}(\mathbb{T}^{3}) \hookrightarrow W^{1,\infty} (\mathbb{T}^{3})$ is valid; i.e., 
\begin{equation}\label{estimate 357}
\lVert \mathcal{R}([( \Upsilon_{1}^{-1} \Upsilon_{2} \Theta_{q} e^{3}) \ast_{x} \phi_{l} ]\ast_{t} \varphi_{l} - \Upsilon_{1}^{-1} \Upsilon_{2} \Theta_{q+1} e^{3}) \rVert_{C_{t}L_{x}^{p^{\ast}}} \leq VII + VIII
\end{equation}
where 
\begin{subequations}\label{estimate 137}
\begin{align}
&VII \triangleq  \lVert \mathcal{R} (\Upsilon_{1}^{-1} \Upsilon_{2} \Theta_{q}e^{3} - \Upsilon_{1}^{-1} \Upsilon_{2}\Theta_{q+1}e^{3}) \rVert_{C_{t}L_{x}^{p^{\ast}}}, \\
&VIII \triangleq \lVert \mathcal{R} (  
[(\Upsilon_{1}^{-1} \Upsilon_{2} \Theta_{q} e^{3})\ast_{x} \phi_{l} ] \ast_{t} \varphi_{l} - \Upsilon_{1}^{-1} \Upsilon_{2} \Theta_{q} e^{3} ) \rVert_{C_{t}L_{x}^{p^{\ast}}}.
\end{align}  
\end{subequations} 
To deal with $\lVert v_{q+1} - v_{q} \rVert_{C_{t}L_{x}^{p^{\ast}}}$ in \eqref{estimate 81}, we split it identically to \eqref{estimate 51}: $\lVert v_{q+1} - v_{q} \rVert_{C_{t}L_{x}^{p^{\ast}}} \leq  VII_{1} + VII_{2}$ where $VII_{1} \triangleq \lVert v_{l} - v_{q} \rVert_{C_{t}L_{x}^{p^{\ast}}}$ and $VII_{2} \triangleq \lVert v_{q+1} - v_{l} \rVert_{C_{t}L_{x}^{p^{\ast}}}$. Now $VII_{1}$ can be bounded identically to \eqref{estimate 364} while we carefully estimate $VII_{2}$ as follows: for all $t \in [0, T_{L}]$  
\begin{align}
VII_{2} \overset{\eqref{[Equ. (85), Y20c]}}{\leq}& \lVert w_{q+1}^{(p)}  \rVert_{C_{t}L_{x}^{p^{\ast}}} + \lVert w_{q+1}^{(c)}  \rVert_{C_{t}L_{x}^{p^{\ast}}} + \lVert w_{q+1}^{(t)} \rVert_{C_{t}L_{x}^{p^{\ast}}}  \label{estimate 100}\\
\overset{\eqref{[Equ. (146), Y20a]} \eqref{[Equ. (78b) and (78c), Y20a]} }{\lesssim}& m_{L} M_{0}(t)^{\frac{1}{2}} \delta_{q+1}^{\frac{1}{2}} l^{-2} r_{\bot}^{\frac{2}{p^{\ast}} - 1} r_{\lVert}^{\frac{1}{p^{\ast}} - \frac{1}{2}} + m_{L} \delta_{q+1}^{\frac{1}{2}} M_{0}(t)^{\frac{1}{2}} l^{-12} r_{\bot}^{\frac{2}{p^{\ast}}} r_{\lVert}^{\frac{1}{p^{\ast}} - \frac{3}{2}} \nonumber\\
&+ \delta_{q+1}M_{0}(t) l^{-4} r_{\bot}^{\frac{2}{p^{\ast}} - 1} r_{\lVert}^{\frac{1}{p^{\ast}} - 2} \lambda_{q+1}^{1-2m}  \nonumber\\
\overset{\eqref{[Equ. (62), Y20a]} \eqref{[Equ. (70), Y20c]}}{\lesssim}&  m_{L} M_{0}(t)^{\frac{1}{2}} \lambda_{q+1}^{\frac{ - 122 \alpha + 6 - 24m}{12}} + m_{L} M_{0}(t)^{\frac{1}{2}} \lambda_{q+1}^{\frac{236 \alpha - 1 - 28m}{24}} + M_{0}(t) \lambda_{q+1}^{\frac{ - 36 m + 9 - 148 \alpha}{24}} \overset{\eqref{[Equ. (161), Y20c]}}{\lesssim} M_{0}(t) \lambda_{q+1}^{\frac{-36 m + 9 - 148 \alpha}{24}} \nonumber 
\end{align} 
where we used that 
\begin{align*}
& 4 \alpha + (\frac{1-20m}{24}) (\frac{2}{p^{\ast}} - 1) + (\frac{13-20m}{12}) (\frac{1}{p^{\ast}} - \frac{1}{2}) \overset{\eqref{[Equ. (93), Y20a]}}{=}  \frac{ - 122 \alpha + 6 - 24m}{12}, \\
& 24 \alpha + (\frac{1-20m}{24}) (\frac{2}{p^{\ast}}) + (\frac{13-20m}{12}) (\frac{1}{p^{\ast}} - \frac{3}{2}) 
\overset{\eqref{[Equ. (93), Y20a]}}{=} \frac{236 \alpha - 1 - 28m}{24},  \\
& 8 \alpha + (\frac{1-20m}{24}) ( \frac{2}{p^{\ast}} - 1) + (\frac{13 - 20m}{12}) (\frac{1}{p^{\ast}} - 2) + 1 - 2m \overset{\eqref{[Equ. (93), Y20a]}}{=} \frac{ -36m + 9 - 148 \alpha}{24}.
\end{align*}
Therefore, \eqref{estimate 364} and \eqref{estimate 100} give
\begin{equation}\label{estimate 101} 
\lVert v_{q+1} - v_{q} \rVert_{C_{t}L_{x}^{p^{\ast}}} \lesssim m_{L} M_{0}(t)^{\frac{1}{2}} \lambda_{q+1}^{-\alpha} +  M_{0}(t) \lambda_{q+1}^{\frac{-36m + 9 - 148 \alpha}{24}} \lesssim m_{L} M_{0}(t) \lambda_{q+1}^{-\alpha}. 
\end{equation}
Applying \eqref{estimate 101} and \eqref{estimate 83} to \eqref{estimate 81} gives 
\begin{equation}\label{estimate 102} 
\lVert \Theta_{q+1} - \Theta_{q} \rVert_{C_{t}L_{x}^{p^{\ast}}} \overset{\eqref{estimate 81}\eqref{estimate 101}\eqref{estimate 83}}{\lesssim} e^{2L^{\frac{1}{4}}} m_{L}^{2} M_{0}(t)^{\frac{3}{2}} \lambda_{q+1}^{-\alpha} \sqrt{t} \lVert \theta^{\text{in}} \rVert_{H_{x}^{2}}  \lambda_{q}^{4}.
\end{equation} 
From \eqref{estimate 102} we can now deduce by taking $a > 0$ sufficiently large 
\begin{align}
VII \overset{\eqref{estimate 137}\eqref{[Equ. (151) and (152), Y20c]}}{\lesssim}& e^{2L^{\frac{1}{4}}}  \lVert \Theta_{q+1} - \Theta_{q} \rVert_{C_{t}L_{x}^{p^{\ast}}} \nonumber\\
\overset{\eqref{estimate 102}}{\lesssim}& c_{R} \delta_{q+2} M_{0}(t) [ M_{0}(t)^{\frac{1}{2}} e^{4L^{\frac{1}{4}}} m_{L}^{2} \sqrt{t} \lVert \theta^{\text{in}} \rVert_{H_{x}^{2}} a^{b^{q+1} [ - \frac{35\alpha}{48} ]}] \ll c_{R} M_{0}(t) \delta_{q+2}  \label{estimate 356}
\end{align}  
where we used \eqref{estimate 130}. Applying \eqref{estimate 356} and \eqref{estimate 354} that remains valid in case $n = 3$ to \eqref{estimate 357} gives us for all $t \in [0, T_{L}]$ 
\begin{equation}\label{estimate 358}
\lVert \mathcal{R}([( \Upsilon_{1}^{-1} \Upsilon_{2} \Theta_{q} e^{3}) \ast_{x} \phi_{l} ]\ast_{t} \varphi_{l} - \Upsilon_{1}^{-1} \Upsilon_{2} \Theta_{q+1} e^{3}) \rVert_{C_{t}L_{x}^{p^{\ast}}} \ll c_{R} M_{0}(t) \delta_{q+2}. 
\end{equation}
From \cite[Equ. (152a), (153)-(154)]{Y20a} we have 
\begin{align}
& \lVert R_{\text{lin}} -  \mathcal{R}([( \Upsilon_{1}^{-1} \Upsilon_{2} \Theta_{q} e^{3}) \ast_{x} \phi_{l} ]\ast_{t} \varphi_{l} - \Upsilon_{1}^{-1} \Upsilon_{2} \Theta_{q+1} e^{3})  \rVert_{C_{t}L_{x}^{p^{\ast}}} \nonumber\\
\lesssim& c_{R}\delta_{q+2} M_{0}(t) [ m_{L} \lambda_{q+1}^{- \frac{61 \alpha}{6} + \frac{\alpha}{48}} + \lambda_{q+1}^{\frac{12 m - 15 - 148 \alpha}{24} + \frac{\alpha}{48}} + m_{L} \lambda_{q+1}^{ - \frac{\alpha}{6} + \frac{\alpha}{48}} \nonumber\\
& \hspace{30mm} + m_{L} \lambda_{q+1}^{\frac{59 \alpha - 12m}{6} + \frac{\alpha}{48}} + m_{L}^{4} \lambda_{q+1}^{\frac{ - 119 \alpha - 24m + 12}{12} + \frac{\alpha}{48}}] \ll c_{R} \delta_{q+2} M_{0}(t). \label{estimate 104}
\end{align} 
Consequently,  due to \eqref{estimate 358}-\eqref{estimate 104} we obtain 
\begin{equation}\label{estimate 105}
\lVert R_{\text{lin}} \rVert_{C_{t}L_{x}^{p^{\ast}}} \leq  \frac{ (2\pi)^{-3 (\frac{p^{\ast} -1}{p^{\ast}})}}{5} c_{R} M_{0}(t) \delta_{q+2}. 
\end{equation}  
By \cite[Equ. (103)-(104), (157)-(159)]{Y20a} we have \eqref{estimate 381}, which, along with \eqref{estimate 105},  allows us to conclude that $\lVert \mathring{R}_{q+1} \rVert_{C_{t}L_{x}^{1}} \leq c_{R} M_{0}(t) \delta_{q+2}$, verifying \eqref{[Equ. (153c), Y20c]} at level $q+1$. The rest of the arguments are similar to that of proof of Proposition \ref{Proposition 4.8 for n=2}. 
 
\section{Appendix}
\subsection{Preliminaries needed for convex integration in 2D case and more}\label{Preliminaries needed for convex integration in 2D case and more} 
We describe the 2D intermittent stationary flows introduced in \cite{CDS12} and extended in \cite{LQ20}. We let 
\begin{equation}\label{[Equ. (10), Y20c]}
\Lambda^{+} \triangleq \{ \frac{1}{5} (3e^{1} \pm 4e^{2}), \frac{1}{5} (4e^{1} \pm 3e^{2}) \} \hspace{2mm}  \text{ and } \hspace{2mm}  \Lambda^{-} \triangleq \{\frac{1}{5} (-3e^{1}\mp 4e^{2}), \frac{1}{5} (-4 e^{1} \mp 3e^{2}) \}, 
\end{equation} 
i.e. $\Lambda^{-} = -\Lambda^{+}$, and $\Lambda \triangleq \Lambda^{+} \cup \Lambda^{-}$. It follows immediately that $\Lambda \subset \mathbb{S}^{1} \cap \mathbb{Q}^{2}$, $5 \Lambda \subset \mathbb{Z}^{2}$, and  
\begin{equation}\label{[Equ. (11), Y20c]}
\min_{\zeta, \zeta' \in \Lambda: \hspace{0.5mm} \zeta \neq - \zeta'} \lvert \zeta + \zeta' \rvert \geq \frac{\sqrt{2}}{5}.
\end{equation} 
For all $\zeta \in \Lambda$ and any $\lambda \in 5 \mathbb{N}$, we define $b_{\zeta}$ and its potential $\psi_{\zeta}$ as 
\begin{equation}\label{[Equ. (12), Y20c]}
b_{\zeta}(x) \triangleq b_{\zeta, \lambda} (x) \triangleq i \zeta^{\bot} e^{i \lambda \zeta \cdot x}, \hspace{3mm} \psi_{\zeta} (x) \triangleq \psi_{\zeta, \lambda} (x) \triangleq \frac{1}{\lambda} e^{i \lambda \zeta \cdot x}. 
\end{equation} 
It follows that for all $N \in \mathbb{N}_{0}$, 
\begin{subequations}
\begin{align}
& b_{\zeta}(x)  = \nabla^{\bot} \psi_{\zeta}(x), \hspace{3mm} \nabla\cdot b_{\zeta}(x) = 0, \hspace{3mm} \nabla^{\bot} \cdot b_{\zeta}(x) = \Delta \psi_{\zeta}(x) = -\lambda^{2} \psi_{\zeta}(x), \label{[Equ. (13a), Y20c]}\\
& \overline{b_{\zeta}}(x) = b_{-\zeta}(x), \hspace{1mm} \overline{\psi_{\zeta}}(x) = \psi_{-\zeta}(x), \hspace{1mm}  \lVert b_{\zeta} \rVert_{C_{x}^{N}} \overset{\eqref{C-t,x} }{\leq}(N+1) \lambda^{N}, \hspace{1mm} \lVert \psi_{\zeta} \rVert_{C_{x}^{N}} \overset{\eqref{C-t,x} }{\leq} (N+1) \lambda^{N-1}.   \label{[Equ. (13b), Y20c]} 
\end{align}
\end{subequations} 

\begin{lemma}\label{[Lemma 4.1, LQ20]}
\rm{(\cite[Lem. 4.1]{LQ20})} Denote by $\mathcal{M}$ the linear space of $2\times 2$ symmetric trace-free matrices. Then there exists a set of positive smooth functions $\{\gamma_{\zeta} \in C^{\infty} (\mathcal{M}): \hspace{0.5mm} \zeta \in \Lambda \}$ such that for each $\mathring{R} \in \mathcal{M}$, 
\begin{equation}
\gamma_{-\zeta} (\mathring{R}) = \gamma_{\zeta} (\mathring{R}), \hspace{3mm} \mathring{R} = \sum_{\zeta \in \Lambda} (\gamma_{\zeta} (\mathring{R} ))^{2} (\zeta \mathring{\otimes} \zeta), \hspace{3mm} \gamma_{\zeta} (\mathring{R}) \lesssim (1+ \lvert \mathring{R} \rvert)^{\frac{1}{2}}. 
\end{equation}  
\end{lemma} 
 
For convenience we set $\lvert \Lambda \rvert$ to be the cardinality of the set $\Lambda$ and 
\begin{equation}\label{[Equ. (15), Y20c]} 
C_{\Lambda} \triangleq 2 \sqrt{12} (4 \pi^{2} + 1)^{\frac{1}{2}} \lvert \Lambda \rvert \hspace{1mm}  \text{ and } \hspace{1mm}  M \triangleq C_{\Lambda} \sup_{\zeta \in \Lambda} ( \lVert \gamma_{\zeta} \rVert_{C(B_{\frac{1}{2}} (0))} + \lVert \nabla \gamma_{\zeta} \rVert_{C(B_{\frac{1}{2}} (0))}). 
\end{equation} 
We consider a 2D Dirichlet kernel for $r \in \mathbb{N}$
\begin{equation}\label{[Equ. (16), Y20c]} 
D_{r}(x) \triangleq \frac{1}{2r+1} \sum_{k \in \Omega_{r}} e^{ik\cdot x} \hspace{1mm} \text{ where } \hspace{1mm}
\Omega_{r} \triangleq \{k = 
\begin{pmatrix}
k^{1} & k^{2}
\end{pmatrix}^{T}: \hspace{0.5mm} k^{i} \in \mathbb{Z} \cap [-r, r] \text{ for } i = 1,2 \}
\end{equation}
that satisfies $\lVert D_{r} \rVert_{L_{x}^{p}} \lesssim r^{1- \frac{2}{p}}$ and $\lVert D_{r} \rVert_{L_{x}^{2}} = 2 \pi$ for all $p \in (1, \infty]$. We introduce $\sigma$ to parametrize the spacing between frequencies, and $\mu$ that measures the amount of temporal oscillation in the building blocks. These parameters must satisfy 
\begin{equation}\label{[Equ. (18), Y20c]}
1 \ll r \ll \mu \ll \sigma^{-1} \ll \lambda, \hspace{2mm} r \in \mathbb{N}, \hspace{2mm} \text{ and } \hspace{2mm} \lambda, \lambda \sigma \in 5 \mathbb{N}.
\end{equation}
Next, we define the directed-rescaled Dirichlet kernel by 
\begin{equation}\label{[Equ. (19), Y20c]}
\eta_{\zeta} (t,x) \triangleq \eta_{\zeta, \lambda, \sigma, r, \mu} (t,x) \triangleq 
\begin{cases}
D_{r} (\lambda \sigma (\zeta \cdot x + \mu t), \lambda \sigma \zeta^{\bot} \cdot x) & \text{ if } \zeta \in \Lambda^{+},\\
\eta_{-\zeta, \lambda, \sigma, r, \mu} (t,x) & \text{ if } \zeta \in \Lambda^{-}, 
\end{cases}
\end{equation} 
so that for all $\zeta \in \Lambda^{\pm}$ and $p \in (1, \infty]$, 
\begin{equation}\label{[Equ. (20a) and (20b), Y20c]}
\frac{1}{\mu} \partial_{t} \eta_{\zeta} (t,x) = \pm (\zeta\cdot\nabla) \eta_{\zeta} (t,x), \hspace{2mm}  \fint_{\mathbb{T}^{2}} \eta_{\zeta}^{2} (t,x) dx = 1, \hspace{2mm} \text{ and } \hspace{2mm}  \lVert \eta_{\zeta} \rVert_{L_{t}^{\infty} L_{x}^{p}} \lesssim r^{1- \frac{2}{p}}. 
\end{equation} 
Finally, we define the intermittent 2D stationary flow as 
\begin{equation}\label{[Equ. (21), Y20c]}
\mathbb{W}_{\zeta} (t,x) \triangleq \mathbb{W}_{\zeta, \lambda, \sigma, r, \mu} (t,x) \triangleq \eta_{\zeta,\lambda,\sigma,r,\mu} (t,x) b_{\zeta,\lambda}(x). 
\end{equation} 
\begin{lemma}\label{[Lem. 4.3, LQ20]}
\rm{ (\cite[Lem. 4.3]{LQ20}; cf. \cite[Pro. 3.5]{BV19a})} Define $\eta_{\zeta}$ and $\mathbb{W}_{\zeta}$ respectively by \eqref{[Equ. (19), Y20c]} and \eqref{[Equ. (21), Y20c]}, and assume \eqref{[Equ. (18), Y20c]}. Then for any $p \in (1,\infty]$, $k$, $N \in \{ 0, 1, 2, 3\}$, 
\begin{equation}\label{[Equ. (24a) and (24b), Y20c]}
 \lVert \nabla^{N} \partial_{t}^{k} \mathbb{W}_{\zeta} \rVert_{L_{t}^{\infty} L_{x}^{p}} \lesssim_{N, k, p} \lambda^{N} (\lambda \sigma r \mu)^{k} r^{1- \frac{2}{p}}, \hspace{3mm}  \lVert \nabla^{N} \partial_{t}^{k} \eta_{\zeta} \rVert_{L_{t}^{\infty} L_{x}^{p}} \lesssim_{N, k, p} (\lambda \sigma r)^{N} (\lambda \sigma r \mu)^{k} r^{1- \frac{2}{p}}. 
\end{equation} 
\end{lemma} 
We also used the following result often: 
\begin{lemma}\label{[Def. 9, Lem. 10, CDS12]}
\rm{ (\cite[Def. 9, Lem. 10]{CDS12}, also \cite[Def. 7.1, Lem. 7.2 and 7.3]{LQ20})} For $f \in C(\mathbb{T}^{2})$, set $\mathcal{R} f \triangleq \nabla g + (\nabla g)^{T} - (\nabla\cdot g) \text{Id}$, where $\Delta g = f - \fint_{\mathbb{T}^{2}} fdx$ and $\fint_{\mathbb{T}^{2}} g  dx= 0$. Then, for any $f \in C(\mathbb{T}^{2})$ such that $\fint_{\mathbb{T}^{2}} f dx = 0$, $\mathcal{R}f(x)$ is a trace-free symmetric matrix for all $x \in \mathbb{T}^{2}$.  Moreover, $\nabla\cdot \mathcal{R} f = f$ and $\fint_{\mathbb{T}^{2}} \mathcal{R} f(x) dx = 0$. Finally, for all $p \in (1,\infty)$, $\lVert \mathcal{R} \rVert_{L_{x}^{p} \mapsto W_{x}^{1,p}} \lesssim 1, \lVert \mathcal{R} \rVert_{C_{x} \mapsto C_{x}} \lesssim 1, \lVert \mathcal{R}  f \rVert_{L_{x}^{p}} \lesssim \lVert (-\Delta)^{-\frac{1}{2}} f \rVert_{L_{x}^{p}}$.  
\end{lemma}
 
 \subsection{Preliminaries needed for convex integration in 3D case and more}\label{Preliminaries needed for convex integration in 3D case and more} 
 We describe 3D intermittent jets from \cite[App. B]{HZZ19}, originally from \cite[Sec. 7.4]{BV19b} (also \cite[Sec. 4]{BCV18}). 
\begin{lemma}\label{[Lem. B.1, HZZ19]}
\rm{(\cite[Lem. 6.6]{BV19b})}
Let $\overline{B_{\frac{1}{2}}(\text{Id})}$ denote the closed ball of radius $\frac{1}{2}$ around an identity matrix in the space of $3\times 3$ symmetric matrices. Then there exists $\Lambda \subset \mathbb{S}^{2} \cap \mathbb{Q}^{3}$ such that for each $\zeta \in \Lambda$, there exist $C^{\infty}$ functions $\gamma_{\zeta}: \hspace{0.5mm}  B_{\frac{1}{2}} (\text{Id}) \mapsto \mathbb{R}$ which obey $R = \sum_{\zeta \in \Lambda} \gamma_{\zeta}^{2} (R) (\zeta \otimes \zeta)$ for every symmetric matrix $R$ that satisfies $\lvert R - \text{Id} \rvert \leq \frac{1}{2}$. 
\end{lemma} 
Define a constant 
\begin{equation}\label{[Equ. (163), Y20a]}
M \triangleq C_{\Lambda} \sup_{\zeta \in \Lambda}(\lVert \gamma_{\zeta} \rVert_{C^{0}} + \lVert \nabla \gamma_{\zeta} \rVert_{C^{0}}) \text{ where } C_{\Lambda} \triangleq 8 \lvert \Lambda \rvert (1+ 8\pi^{3})^{\frac{1}{2}}. 
\end{equation}
For every $\zeta \in \Lambda$, let $A_{\zeta} \in \mathbb{S}^{2} \cap \mathbb{Q}^{3}$ be an orthogonal vector to $\zeta$. It follows that for each $\zeta \in \Lambda$, $\{\zeta, A_{\zeta}, \zeta \times A_{\zeta} \} \subset \mathbb{S}^{2} \cap \mathbb{Q}^{3}$ forms an orthonormal basis for $\mathbb{R}^{3}$. Furthermore, we denote by $n_{\ast}$ the smallest natural number such that $\{n_{\ast}, \zeta, n_{\ast}A_{\zeta}, n_{\ast} \zeta \times A_{\zeta} \} \subset \mathbb{Z}^{3}$ for every $\zeta \in \Lambda$. Now let $\Phi: \hspace{0.5mm}  \mathbb{R}^{2} \mapsto \mathbb{R}^{2}$ be a smooth function with support contained in a ball of radius one. We normalize $\Phi$ so that $\phi \triangleq - \Delta \Phi$ obeys 
\begin{equation}\label{[Equ. (B.1), HZZ19]}
 \int_{\mathbb{R}^{2}} \phi^{2} ( x_{1}, x_{2}) dx_{1} dx_{2} = 4\pi^{2}. 
\end{equation} 
It follows that $\phi$ has mean zero. We define $\psi: \hspace{0.5mm}  \mathbb{R} \mapsto \mathbb{R}$ to be a smooth, mean-zero function with support in the ball of radius one such that $\int_{\mathbb{R}}\psi^{2} (x_{3}) dx_{3} = 2\pi$. Define 
\begin{equation}\label{[Equ. (B.2b), HZZ19]}
\phi_{r_{\bot}} (x_{1}, x_{2}) \triangleq \phi( \frac{x_{1}}{r_{\bot}}, \frac{x_{2}}{r_{\bot}})r_{\bot}^{-1}, \hspace{1mm} \Phi_{r_{\bot}} (x_{1}, x_{2}) \triangleq \Phi( \frac{x_{1}}{r_{\bot}}, \frac{x_{2}}{r_{\bot}})_{\bot}^{-1}\hspace{1mm}  \text{ and } \hspace{1mm} \psi_{r_{\lVert}} (x_{3}) \triangleq  \psi( \frac{x_{3}}{r_{\lVert }})r_{\lVert }^{-\frac{1}{2}}
\end{equation} 
so that $\phi_{r_{\bot}} = - r_{\bot}^{2} \Delta \Phi_{r_{\bot}}$ in which we will assume $r_{\bot}, r_{\lVert} > 0$ to satisfy 
\begin{equation}\label{[Equ. (B.2a), HZZ19]}
r_{\bot} \ll r_{\lVert} \ll 1 \text{ and } r_{\bot}^{-1} \ll \lambda_{q+1}.
\end{equation} 
By an abuse of notation, we periodize $\phi_{r_{\bot}}, \Phi_{r_{\bot}}$ and $\psi_{r_{\lVert}}$ so that they are treated as functions defined on $\mathbb{T}^{2}, \mathbb{T}^{2}$, and $\mathbb{T}$, respectively. For a large real number $\lambda$ such that $\lambda r_{\bot} \in \mathbb{N}$, and a large time oscillation parameter $\mu > 0$, for every $\zeta \in \Lambda$ we introduce 
\begin{subequations}\label{[Equ. (167), Y20a]}
\begin{align}
& \psi_{\zeta} (t,x) \triangleq \psi_{\zeta, r_{\bot}, r_{\lVert}, \lambda, \mu} (t,x) \triangleq \psi_{r_{\lVert}} (n_{\ast} r_{\bot} \lambda(x \cdot \zeta + \mu t)), \\
& \Phi_{\zeta} (x) \triangleq \Phi_{\zeta, r_{\bot}, \lambda} (x) \triangleq \Phi_{r_{\bot}} (n_{\ast} r_{\bot} \lambda (x - a_{\zeta}) \cdot A_{\zeta}, n_{\ast} r_{\bot} \lambda (x- a_{\zeta}) \cdot (\zeta \times A_{\zeta})), \\
& \phi_{\zeta} (x) \triangleq \phi_{\zeta, r_{\bot}, \lambda} (x) \triangleq \phi_{r_{\bot}} (n_{\ast} r_{\bot} \lambda (x- a_{\zeta}) \cdot A_{\zeta}, n_{\ast} r_{\bot} \lambda (x- a_{\zeta}) \cdot (\zeta \times A_{\zeta} )), 
\end{align} 
\end{subequations}
where $a_{\zeta} \in \mathbb{R}^{3}$ are shifts which ensure that the functions $\{ \Phi_{\zeta}\}_{\zeta \in \Lambda}$ have mutually disjoint support. We can now define intermittent jets $W_{\zeta}: \hspace{0.5mm}  \mathbb{T}^{3} \times \mathbb{R} \mapsto \mathbb{R}^{3}$ by 
\begin{equation}\label{[Equ. (168), Y20a]}
W_{\zeta} (t,x) \triangleq W_{\zeta, r_{\bot}, r_{\lVert}, \lambda, \mu} (t,x) \triangleq \zeta \psi_{\zeta} (t,x) \phi_{\zeta} (x). 
\end{equation} 
It follows that $W_{\zeta}$ is mean-zero, it is $(\mathbb{T}/r_{\bot}\lambda)^{3}$-periodic, and 
\begin{equation}\label{[Equ. (B.4), HZZ19]}
W_{\zeta} \otimes W_{\zeta'} = 0 \hspace{3mm} \forall \hspace{1mm} \zeta, \zeta' \in \Lambda \text{ such that } \zeta \neq \zeta'.
\end{equation} 
Due to \eqref{[Equ. (B.1), HZZ19]}-\eqref{[Equ. (B.2b), HZZ19]} we also have $\fint_{\mathbb{T}^{3}} W_{\zeta} (t,x) \otimes W_{\zeta} (t,x) dx = \zeta \otimes \zeta$. Lemma \ref{[Lem. B.1, HZZ19]} and \eqref{[Equ. (B.4), HZZ19]} imply $\sum_{\zeta \in \Lambda} \gamma_{\zeta}^{2}(R) \fint_{\mathbb{T}^{3}} W_{\zeta} (t,x) \otimes W_{\zeta} (t,x) dx = R$.  We also define 
\begin{equation}\label{[Equ. (173), Y20a]}
W_{\zeta}^{(c)} \triangleq \frac{ \nabla \psi_{\zeta} }{n_{\ast}^{2} \lambda^{2}} \times \text{curl} (\Phi_{\zeta} \zeta) = \text{curl curl} V_{\zeta} - W_{\zeta} \text{ with } V_{\zeta} (t,x) \triangleq \frac{ \zeta \psi_{\zeta} (t,x)}{n_{\ast}^{2} \lambda^{2}} \Phi_{\zeta} (x), 
\end{equation} 
from which it follows that $\text{div} (W_{\zeta} + W_{\zeta}^{(c)}) = 0$. Finally, for all $N, M \geq 0$ and $p \in [1, \infty]$, 
\begin{subequations}\label{[Equ. (174), Y20a]}
\begin{align}
& \lVert \nabla^{N} \partial_{t}^{M} \psi_{\zeta} \rVert_{L^{p}} \lesssim r_{\lVert}^{\frac{1}{p} - \frac{1}{2}} \left( \frac{r_{\bot} \lambda}{r_{\lVert}} \right)^{N} \left( \frac{ r_{\bot} \lambda \mu}{r_{\lVert}}\right)^{M}, \hspace{4mm} 
 \lVert \nabla^{N} \phi_{\zeta} \rVert_{L^{p}} + \lVert \nabla^{N} \Phi_{\zeta} \rVert_{L^{p}} \lesssim r_{\bot}^{\frac{2}{p} - 1} \lambda^{N}, \label{[Equ. (174a) and (174b), Y20a]}\\
& \lVert \nabla^{N} \partial_{t}^{M} W_{\zeta} \rVert_{L^{p}} + \frac{r_{\lVert}}{r_{\bot}} \lVert \nabla^{N} \partial_{t}^{M} W_{\zeta}^{(c)} \rVert_{L^{p}} + \lambda^{2} \lVert \nabla^{N} \partial_{t}^{M} V_{\zeta} \rVert_{L^{p}} \lesssim r_{\bot}^{\frac{2}{p} - 1} r_{\lVert}^{\frac{1}{p} - \frac{1}{2}} \lambda^{N} \left( \frac{r_{\bot} \lambda \mu}{r_{\lVert} } \right)^{M}, \label{[Equ. (174c), Y20a]}
\end{align} 
\end{subequations} 
where the implicit constants are independent of $\lambda, r_{\bot}, r_{\lVert}$, and $\mu$. We also relied on the following result. 
\begin{lemma}\label{divergence inverse operator}
\rm{(\cite[Equ. (5.34)]{BV19b})} For any $v \in C^{\infty}(\mathbb{T}^{3})$ that has mean zero, define 
\begin{equation}\label{estimate 336}
(\mathcal{R}v)_{kl} \triangleq ( \partial_{k}\Delta^{-1} v^{l} + \partial_{l} \Delta^{-1} v^{k}) - \frac{1}{2} (\delta_{kl} + \partial_{k} \partial_{l} \Delta^{-1}) \text{div} \Delta^{-1} v
\end{equation} 
for $k, l \in \{1,2,3\}$. Then $\mathcal{R} v(x)$ is a symmetric trace-free matrix for each $x \in \mathbb{T}^{3}$, that  satisfies $\text{div} (\mathcal{R} v) = v$. Moreover, $\mathcal{R}$ satisfies the classical Calder$\acute{\mathrm{o}}$n-Zygmund and Schauder estimates: $\lVert (-\Delta)^{\frac{1}{2}} \mathcal{R} \rVert_{L_{x}^{p} \mapsto L_{x}^{p}} + \lVert \mathcal{R} \rVert_{L_{x}^{p} \mapsto L_{x}^{p}}  + \lVert \mathcal{R} \rVert_{C_{x} \mapsto C_{x}} \lesssim 1$ for all $p \in (1, \infty)$. 
\end{lemma} 

\subsection{Proof of Proposition \ref{Proposition 4.1}}\label{Subsection 6.1}
The proof of Proposition \ref{Proposition 4.1} will rely on the following extension of \cite[Lem. A.1]{HZZ19}: 
\begin{proposition}\label{Proposition 6.10}
\rm{(cf. \cite[Lem. A.1]{HZZ19})} Let $\{ (s_{l}, \xi_{l,1}, \xi_{l,2} ) \}_{l \in \mathbb{N}} \subset [0,\infty) \times L_{\sigma}^{2} \times \mathring{L}^{2}$ be a family such that $\lim_{l\to\infty} \lVert (s_{l}, \xi_{l,1}, \xi_{l,2}) - (s, \xi_{1}^{\text{in}}, \xi_{2}^{\text{in}}) \rVert_{\mathbb{R} \times L_{\sigma}^{2} \times \mathring{L}^{2}} = 0$ and $\{P_{l}\}_{l \in \mathbb{N}}$ be a family of probability measures on $\Omega_{0}$ satisfying for all $l \in \mathbb{N}, P_{l} ( \{ (\xi_{1}, \xi_{2}) (t) = (\xi_{l,1}, \xi_{l,2}) \hspace{1mm} \forall \hspace{1mm} t \in [0, s_{l} ] \}) = 1$ and for some $\gamma > 0, \kappa > 0$, and any $T > 0$, 
\begin{align}
&\sup_{l\in\mathbb{N}} \mathbb{E}^{P_{l}}[ \lVert \xi_{1} \rVert_{C([0,T]; L_{x}^{2})} + \sup_{r, t \in [0,T]: \hspace{0.5mm} r \neq t} \frac{ \lVert \xi_{1}(t) - \xi_{1}(r) \rVert_{H_{x}^{-3}}}{\lvert t-r \rvert^{\kappa}} + \lVert \xi_{1} \rVert_{L^{2} ([s_{l}, T]; \dot{H}_{x}^{\gamma} )}^{2} \nonumber\\
& \hspace{8mm} + \lVert \xi_{2} \rVert_{C([0,T]; L_{x}^{2})} + \sup_{r, t \in [0,T]: \hspace{0.5mm} r \neq t} \frac{ \lVert \xi_{2} (t) - \xi_{2}(r) \rVert_{H_{x}^{-n}}}{\lvert t-r \rvert^{\kappa}} + \lVert \xi_{2} \rVert_{L^{2} ([s_{l}, T]; \dot{H}_{x}^{1})}^{2}] < \infty.  \label{[Equ. (A.2), HZZ19]}
\end{align} 
Then $\{P_{l} \}_{l\in\mathbb{N}}$ is tight in 
\begin{equation}\label{estimate 1}
\mathbb{S} \triangleq C_{\text{loc}} ([0,\infty); H^{-3} (\mathbb{T}^{n})) \cap L_{\text{loc}}^{2} (0,\infty; L_{\sigma}^{2}) \times C_{\text{loc}} ([0,\infty); H^{-n} (\mathbb{T}^{n})) \cap L_{\text{loc}}^{2} (0,\infty; \mathring{L}^{2}). 
\end{equation}  
\end{proposition} 

\begin{proof}[Proof of Proposition \ref{Proposition 6.10}]
We sketch its proof referring to \cite[Lem. A.1]{HZZ19} for details. We fix $\epsilon > 0$ and $k \in \mathbb{N}$ such that $k \geq k_{0} \triangleq \sup_{l \in \mathbb{N}} s_{l}$ and due to \eqref{[Equ. (A.2), HZZ19]} and Chebyshev's inequality we may choose $R_{k} > 0$ sufficiently large such that 
\begin{align}
&P_{l} ( \{ \xi \in \Omega_{0}: \hspace{0.5mm} \sup_{t \in [0,k]} \lVert \xi_{1}(t) \rVert_{L_{x}^{2}} + \sup_{r, t \in [0,k]: \hspace{0.5mm} r \neq t} \frac{ \lVert \xi_{1}(t) - \xi_{1} (r) \rVert_{H_{x}^{-3}}}{ \lvert t- r\rvert^{\kappa}} + \int_{s_{l}}^{k} \lVert \xi_{1}(r) \rVert_{\dot{H}_{x}^{\gamma}}^{2} dr\nonumber \\
&+ \sup_{t \in [0,k]} \lVert \xi_{2} (t) \rVert_{L_{x}^{2}} + \sup_{r, t \in [0, k]: \hspace{0.5mm} r \neq t} \frac{ \lVert \xi_{2}(t) - \xi_{2}(r) \rVert_{H_{x}^{-n}}}{\lvert t- r \rvert^{\kappa}} + \int_{s_{l}}^{k} \lVert \xi_{2}(r) \rVert_{\dot{H}_{x}^{1}}^{2} dr > R_{k} \}) \leq \frac{\epsilon}{2^{k}}. \label{[Equ. (A.2d), HZZ19]}
\end{align}
Then we define $\Omega_{l} \triangleq \{ \xi \in \Omega_{0}: \hspace{0.5mm} \xi(t) = \xi_{l} \hspace{1mm} \forall \hspace{1mm} t \in [0, s_{l}] \}$ and 
\begin{align}
K \triangleq& \cup_{q \in \mathbb{N}} \cap_{k \in \mathbb{N}: \hspace{0.5mm} k \geq k_{0}} \{ \xi \in \Omega_{q}: \hspace{0.5mm} \sup_{t \in [0,k]} \lVert \xi_{1}(t) \rVert_{L_{x}^{2}} + \sup_{r, t \in [0,k]: \hspace{0.5mm} r \neq t} \frac{ \lVert \xi_{1}(t) - \xi_{1}(r) \rVert_{H_{x}^{-3}}}{\lvert t-r \rvert^{\kappa}}  + \int_{s_{q}}^{k} \lVert \xi_{1}(r) \rVert_{\dot{H}_{x}^{\gamma}}^{2} dr \nonumber \\
&+ \sup_{t \in [0,k]} \lVert \xi_{2} (t) \rVert_{L_{x}^{2}} + \sup_{r, t \in [0,k]: \hspace{0.5mm} r \neq t} \frac{ \lVert \xi_{2}(t) - \xi_{2}(r) \rVert_{H_{x}^{-n}}}{\lvert t-r \rvert^{\kappa}} + \int_{s_{q}}^{k}\lVert \xi_{2}(r) \rVert_{\dot{H}_{x}^{1}}^{2} dr \leq R_{k} \}. \label{[Equ. (A.3), HZZ19]}
\end{align} 
We can compute $\sup_{l \in \mathbb{N}} P_{l} (\Omega_{0} \setminus \bar{K}) \leq \epsilon$ by relying on \eqref{[Equ. (A.2d), HZZ19]}. It now suffices to show that $\bar{K}$ is compact in $\mathbb{S}$. We take $\{\xi_{w} \}_{w \in \mathbb{N}} \subset K$ from \eqref{[Equ. (A.3), HZZ19]}. Suppose that for all $N \in \mathbb{N}, \xi_{w} \in \Omega_{N}$ for only finitely many $w \in \mathbb{N}$. Passing to a subsequence and relabeling, we can assume that $\xi_{w} \in \Omega_{w}$. Then, for all $k \geq k_{0}$, 
\begin{align*}
&\sup_{t \in [0,k]} \lVert \xi_{w,1} (t) \rVert_{L_{x}^{2}} + \sup_{r, t \in [0,k]: \hspace{0.5mm} r \neq t} \frac{ \lVert \xi_{w,1} (t) - \xi_{w,1} (r) \rVert_{H_{x}^{-3}}}{\lvert t-r \rvert^{\kappa}} \\
& + \sup_{t \in [0,k]} \lVert \xi_{w,2}(t) \rVert_{L_{x}^{2}} + \sup_{r, t \in [0,k]: \hspace{0.5mm} r \neq t} \frac{ \lVert \xi_{w,2} (t) - \xi_{w,2} (r)\rVert_{H_{x}^{-n}}}{\lvert t-r \rvert^{\kappa}} \leq R_{k}  
\end{align*} 
by \eqref{[Equ. (A.3), HZZ19]} and the assumption that $\xi_{w} \in \Omega_{w}$. Now for $n \in \{2,3\}$ 
\begin{align*}
L^{\infty} (0, k; L^{2}(\mathbb{T}^{n})) \cap C^{\kappa} ([0,k]; H^{-n} (\mathbb{T}^{n})) \hookrightarrow C([0,k]; H^{-n} (\mathbb{T}^{n})) 
\end{align*} 
is compact (cf. \cite[Cor. 2 on p. 82]{S86}, also \cite{BFH18}). Therefore, we can find a subsequence $\{\xi_{w_{l}}\} = \{\xi_{w_{l}, 1}, \xi_{w_{l}, 2} \}$ such that 
\begin{equation}\label{[Equ. (A.4), HZZ19]}
\lim_{l,q\to \infty} \sup_{t \in [0,k]} \lVert \xi_{w_{l}, 1} - \xi_{w_{q}, 1} \rVert_{H_{x}^{-3}} + \lVert \xi_{w_{l}, 2} - \xi_{w_{q}, 2} \rVert_{H_{x}^{-n}} = 0. 
\end{equation} 
It follows that for all $\delta > 0$, there exists $L \in \mathbb{N}$ such that $w_{l}, w_{q} \geq L$ implies 
\begin{align*}
\int_{0}^{k} \lVert \xi_{w_{l}, 1}(t) - \xi_{w_{q}, 1}(t) \rVert_{L_{x}^{2}}^{2} dt < \delta \text{ and } \int_{0}^{k} \lVert \xi_{w_{l}, 2}(t) - \xi_{w_{q}, 2}(t) \rVert_{L_{x}^{2}}^{2} dt < \delta 
\end{align*}  
identically to \cite[p. 45]{HZZ19}. The case in which there exists $N \in \mathbb{N}$ such that $\xi_{w} \in \Omega_{N}$ for infinitely many $w$ is simpler and we omit details. Therefore, we conclude that $\{ \xi_{w_{l}, 1}\}_{l}, \{\xi_{w_{l}, 2} \}_{l}$ are both Cauchy and thus $\bar{K}$ is compact. This completes the proof of Proposition \ref{Proposition 6.10}.  
\end{proof}
We now proceed with the proof of Proposition \ref{Proposition 4.1}.  
\begin{proof}[Proof of Proposition \ref{Proposition 4.1}]
The existence of a martingale solution can be deduced via analogous proofs to previous works (e.g., \cite{FR08, GRZ09, HZZ19, Y19}). Now we fix $\{P_{l} \}_{l \in \mathbb{N}} \subset \mathcal{C} ( s_{l}, \xi_{l}, \{C_{t,q} \}_{q \in \mathbb{N}, t \geq s_{l}} )$ where $\{(s_{l}, \xi_{l} ) \}_{l \in \mathbb{N}} \subset [0,\infty) \times L_{\sigma}^{2} \times \mathring{L}^{2}$ satisfies $\lim_{l\to\infty} \lVert (s_{l}, \xi_{l})- (s, \xi^{\text{in}}) \rVert_{\mathbb{R} \times L_{\sigma}^{2} \times \mathring{L}^{2}} = 0$ and show that it is tight in $\mathbb{S}$ from \eqref{estimate 1}. First, by (M1) of Definition \ref{Definition 4.1}, for all $l \in \mathbb{N}$, $P_{l} ( \{ \xi(t) = \xi_{l} \hspace{1mm} \forall \hspace{1mm} t \in [0, s_{l}] \}) = 1$. Second, we define 
\begin{equation}\label{estimate 300}
F_{1}(\xi) \triangleq - \mathbb{P} \text{div} (\xi_{1} \otimes \xi_{1}) - (-\Delta)^{m}\xi_{1} + \mathbb{P} \xi_{2} e^{n} \text{ and } F_{2}(\xi) \triangleq - \text{div} (\xi_{1}\xi_{2}) + \Delta \xi_{2}.
\end{equation} 
By (M2) of Definition \ref{Definition 4.1}, we know that for all $n \in \mathbb{N}$ and $t \in [s_{l}, \infty)$, $P_{l}$-a.s., 
\begin{equation}\label{[Equ. (176), Y20a]}
\xi_{1}(t) = \xi_{l,1} + \int_{s_{l}}^{t} F_{1} (\xi(\lambda)) d\lambda + M_{1, t, s_{l}}^{\xi} \text{ and } \xi_{2} (t) = \xi_{l,2} + \int_{s_{l}}^{t} F_{2} (\xi(\lambda)) d\lambda + M_{2, t, s_{l}}^{\xi},  
\end{equation}
where the mapping $t \mapsto M_{k, t, s_{l}}^{\xi, i} \triangleq \langle M_{k, t, s_{l}}^{\xi}, \psi_{i}^{k} \rangle$ for both $k \in \{1,2\}$, $\psi_{i} = (\psi_{i}^{1}, \psi_{i}^{2}) \in C^{\infty} (\mathbb{T}^{n}) \cap L_{\sigma}^{2} \times C^{\infty} (\mathbb{T}^{n}) \cap \mathring{L}^{2}$, and $\xi \in \Omega_{0}$ is a continuous, square-integrable $(\mathcal{B}_{t})_{t \geq s_{l}}$-martinalge under $P_{l}$ and 
\begin{equation}\label{estimate 301}
\langle \langle M_{k, t, s_{l}}^{\xi, i} \rangle \rangle = \int_{s_{l}}^{t} \lVert G_{k} (\xi_{k} (r))^{\ast} \psi_{i}^{k} \rVert_{U_{k}}^{2} dr.   
\end{equation} 
Similarly to \cite[Equ. (178)]{Y20a} we can deduce for any $\alpha \in (0, \frac{1}{2})$ by taking $p > \frac{1}{1-2\alpha}$,  
\begin{equation}\label{estimate 3} 
\mathbb{E}^{P_{l}} [ \sup_{r, t \in [s_{l},T]: \hspace{0.5mm} r \neq t} \frac{ \lVert M_{k, t, s_{l}}^{\xi} - M_{k, r, s_{l}}^{\xi} \rVert_{L_{x}^{2}}}{\lvert t- r \rvert^{\alpha}} ] \lesssim_{p} C_{T,p} (1+ \lVert \xi_{l,1} \rVert_{L_{x}^{2}}^{2p} + \lVert \xi_{l,2} \rVert_{L_{x}^{2}}^{2p}), \hspace{5mm} k \in \{1,2\}, 
\end{equation} 
by Kolmogorov's test (e.g., \cite[The. 3.3]{DZ14}) and consequently for all $\kappa \in (0, \frac{1}{2})$ 
\begin{equation}\label{[Equ. (179), Y20a]}
\sup_{l \in \mathbb{N}} \mathbb{E}^{P_{l}} [ \sup_{r, t \in [0,T]: \hspace{0.5mm} r \neq t} \frac{ \lVert \xi_{1}(t) - \xi_{1}(r) \rVert_{H_{x}^{-3}}}{\lvert t-r \rvert^{\kappa}} ] < \infty. 
\end{equation} 
Let us elaborate in the case of $\xi_{2}$. The case $n = 3$ can be handled  similarly as $\xi_{1}$: 
\begin{align}
& \mathbb{E}^{P_{l}} [ \sup_{r, t \in [s_{l}, T]: \hspace{0.5mm} r \neq t} \frac{ \lVert \int_{r}^{t} F_{2} (\xi(l)) dl \rVert_{H_{x}^{-3}}^{p}}{\lvert t-r \rvert^{p-1}} ] \lesssim \mathbb{E}^{P_{l}} [ \int_{s_{l}}^{T} (1+ \lVert \xi_{1} \rVert_{L_{x}^{2}}^{2} + \lVert \xi_{2} \rVert_{L_{x}^{2}}^{2})^{p} d \lambda ] \nonumber\\
\lesssim_{p}& T \mathbb{E}^{P_{l}} [ \sup_{\lambda \in [s_{l}, T]} 1+ \lVert \xi_{1} (\lambda) \rVert_{L_{x}^{2}}^{2p} + \lVert \xi_{2} (\lambda) \rVert_{L_{x}^{2}}^{2p}]  \lesssim_{p} T C_{T,p} (1+ \lVert \xi_{l} \rVert_{L_{x}^{2}}^{2p})  \label{estimate 116}
\end{align} 
by (M3) where the implicit constant is independent of $l$. In case $n= 2$, we compute 
\begin{align}
&\mathbb{E}^{P_{l}} [ \sup_{r, t \in [s_{l}, T]: \hspace{0.5mm} r \neq t} \frac{ \lVert \int_{r}^{t} F_{2} (\xi(\lambda)) d \lambda  \rVert_{H_{x}^{-2}}^{\frac{3}{2}}}{\lvert t-r \rvert^{\frac{1}{2}}}] \label{estimate 117} \\
\lesssim_{T}& \mathbb{E}^{P_{l}}[ \sup_{\lambda \in [s_{l}, T]} \lVert \xi_{2}(\lambda) \rVert_{L_{x}^{2}}^{6} + \int_{s_{l}}^{T} \lVert \xi_{1} \rVert_{\dot{H}_{x}^{\gamma}}^{2} d \lambda  + \sup_{\lambda \in [s_{l}, T]}\lVert \xi_{2}(\lambda) \rVert_{L_{x}^{2}}^{\frac{3}{2}}]  \lesssim_{T} C_{T,3} (1+ \lVert \xi_{l} \rVert_{L_{x}^{2}}^{6}).  \nonumber 
\end{align}  
Thus, we can first split 
\begin{align}
& \sup_{l \in \mathbb{N}} \mathbb{E}^{P_{l}}[ \sup_{r, t \in [0,T]: \hspace{0.5mm} r \neq t} \frac{ \lVert \xi_{2}(t) - \xi_{2} (r) \rVert_{H_{x}^{-n}}}{ \lvert t-r \rvert^{\kappa}} ] \nonumber\\
\overset{\eqref{[Equ. (176), Y20a]}}{\leq}& \sup_{l \in \mathbb{N}} \mathbb{E}^{P_{l}} [ \sup_{r, t \in [s_{l},T]: \hspace{0.5mm} r \neq t} \frac{ \lVert \int_{r}^{t} F_{2} (\xi(\lambda)) d \lambda \rVert_{H_{x}^{-n}}}{\lvert t-r \rvert^{\kappa}} + \frac{ \lVert M_{2, t, s_{l}}^{\xi} - M_{2, r, s_{l}}^{\xi} \rVert_{H_{x}^{-n}}}{\lvert t-r \rvert^{\kappa}} ] \label{estimate 118}
\end{align} 
and rely on \eqref{estimate 116}, \eqref{estimate 117}, and \eqref{estimate 3} to deduce for all $\kappa \in (0, \frac{1}{3})$ 
\begin{equation*}
\sup_{l \in \mathbb{N}} \mathbb{E}^{P_{l}} [ \sup_{r, t \in [0,T]: \hspace{0.5mm} r \neq t} \frac{ \lVert \xi_{2}(t) - \xi_{2}(r) \rVert_{H_{x}^{-n}}}{\lvert t-r \rvert^{\kappa}} ] < \infty. 
\end{equation*}
Together with \eqref{estimate 24} at $q=1$, we now conclude \eqref{[Equ. (A.2), HZZ19]}. Thus, by Proposition \ref{Proposition 6.10} we see that $\{P_{l}\}_{l \in \mathbb{N}}$ is tight in $\mathbb{S}$ of \eqref{estimate 1}. We deduce by Prokhorov's theorem (e.g., \cite[The. 2.3]{DZ14}) and Skorokhod's theorem (e.g., \cite[The. 2.4]{DZ14}) that there exists $(\tilde{\Omega}, \tilde{\mathcal{F}}, \tilde{P})$ and $\mathbb{S}$-valued random variables $\{ \tilde{\xi}_{l} \}_{l\in\mathbb{N}}$ and $\tilde{\xi}$ such that 
\begin{equation}\label{[Equ. (180), Y20a]}
\mathcal{L} (\tilde{\xi}_{l}) = P_{l} \hspace{1mm} \forall \hspace{1mm} l \in \mathbb{N}, \hspace{3mm} \tilde{\xi}_{l} \to \tilde{\xi} \text{ in } \mathbb{S} \hspace{1mm} \tilde{P}\text{-a.s. and } \mathcal{L} (\tilde{\xi}) = P. 
\end{equation} 
It follows that $P ( \{ \xi(t) = \xi^{\text{in}} \hspace{1mm} \forall \hspace{1mm} t \in [0, s] \})$ and for every $\psi_{i} = (\psi_{i}^{1}, \psi_{i}^{2}) \in C^{\infty} (\mathbb{T}^{n}) \cap L_{\sigma}^{2} \times C^{\infty} (\mathbb{T}^{n}) \cap \mathring{L}^{2}$, and $t \geq s$, $\tilde{P}$-a.s., for both $k \in \{1,2\}$, 
\begin{equation}\label{[Equ. (182), Y20a]}
\langle \tilde{\xi}_{l,k} (t), \psi_{i}^{k} \rangle \to \langle \tilde{\xi}_{k} (t), \psi_{i}^{k} \rangle, \hspace{3mm} \int_{s_{l}}^{t} \langle F_{k} (\tilde{\xi}_{l}(\lambda)), \psi_{i}^{k} \rangle d\lambda \to \int_{s}^{t} \langle F_{k} (\tilde{\xi}(\lambda)), \psi_{i}^{k} \rangle d\lambda  
\end{equation} 
as $l\to\infty$. Next, for every $t > r \geq s, p \in (1,\infty)$, and $g$ that is $\mathbb{R}$-valued, $\mathcal{B}_{r}$-measurable and continuous on $\mathbb{S}$, for both $k \in \{1,2\}$, one can verify using \eqref{[Equ. (180), Y20a]} and  \eqref{[Equ. (182), Y20a]} 
\begin{equation}
\sup_{l\in\mathbb{N}} \mathbb{E}^{\tilde{P}} [ \lvert M_{k, t, s_{l}}^{\tilde{\xi}_{l}, i} \rvert^{2p}] \lesssim_{p,t} 1, \hspace{1mm} \lim_{l\to\infty} \mathbb{E}^{\tilde{P}} [ \lvert M_{k, t, s_{l}}^{\tilde{\xi}_{l}, i} - M_{k, t, s}^{\tilde{\xi}, i} \rvert ] = 0, \hspace{1mm} \mathbb{E}^{P} [ ( M_{k, t, s}^{\xi, i} - M_{k, r, s}^{\xi, i} ) g(\xi) ] =0 
\end{equation} 
which implies that the mapping $t \mapsto M_{k, t, s}^{i}$ is a $(\mathcal{B}_{t})_{t\geq s}$-martingale under $P$, 
\begin{equation}
\lim_{l\to\infty} \mathbb{E}^{\tilde{P}} [ \lvert M_{k, t, s_{l}}^{\tilde{\xi}_{l}, i} - M_{k, t, s}^{\tilde{\xi}, i} \rvert^{2}] = 0 \text{ and } \langle \langle M_{k, t, s}^{\xi, i} \rangle \rangle \overset{\eqref{[Equ. (12), Y20a]} }{=} \int_{s}^{t} \lVert G_{k} (\xi_{k} (\lambda ))^{\ast} \psi_{i}^{k} \rVert_{U_{k}}^{2} d \lambda 
\end{equation} 
so that $M_{k, t, s}^{\xi, i}$ is square-integrable and (M2) is proven. Finally, the proof of (M3) follows from defining  
\begin{equation}\label{estimate 303} 
R(t,s, \xi) \triangleq \sup_{r \in [0,t]} \lVert \xi_{1}(r) \rVert_{L_{x}^{2}}^{2q} + \int_{s}^{t} \lVert \xi_{1}(r) \rVert_{\dot{H}_{x}^{\gamma}}^{2} dr+ \sup_{r \in [0,t]} \lVert \xi_{2}(r) \rVert_{L_{x}^{2}}^{2q} + \int_{s}^{t} \lVert \xi_{2} (r) \rVert_{\dot{H}_{x}^{1}}^{2} dr 
\end{equation} 
and relying on the fact that the mapping $\xi \mapsto R(t,s, \xi)$ is lower semicontinuous on $\mathbb{S}$. This completes the proof of Proposition \ref{Proposition 4.1}. 
\end{proof} 
 
\subsection{Proof of Proposition \ref{Proposition 4.5}}\label{Subsection 6.4}
For $C_{S} > 0$ from \eqref{[Equ. (31) and (32), Y20a]}, $L > 1$, and $\delta \in (0, \frac{1}{12})$, we define 
\begin{align}
T_{L} \triangleq& \inf\{t \geq 0: \hspace{0.5mm} C_{S} \max_{k=1,2}  \lVert z_{k} (t) \rVert_{\dot{H}_{x}^{\frac{n+2+\sigma}{2}}}  \geq L^{\frac{1}{4}} \} \nonumber \\
& \wedge \inf\{t \geq 0: \hspace{0.5mm} C_{S} \max_{k=1,2}  \lVert z_{k} \rVert_{C_{t}^{\frac{1}{2} - 2 \delta} \dot{H}_{x}^{\frac{n+\sigma}{2}}}  \geq L^{\frac{1}{2}} \} \wedge L. \label{[Equ. (33), Y20a]}
\end{align} 
Due to Proposition \ref{Proposition 4.4}, we see that $\textbf{P}$-a.s. $T_{L} > 0$ and $T_{L} \nearrow + \infty$ as $L \nearrow + \infty$. The stopping time $\mathfrak{t}$ in the statement of Theorem \ref{Theorem 2.1} is actually $T_{L}$ for $L > 0$ sufficiently large and thus by Theorem \ref{Theorem 2.1} there exist processes $(u,\theta)$ that is a weak solution on $[0, T_{L}]$ such that \eqref{estimate 17} holds. Hence, we see that $(u,\theta) (\cdot \wedge T_{L}) \in \Omega_{0}$, By \eqref{[Equ. (30), Y20a]}, \eqref{[Equ. (40a), Y20c]}, \eqref{3}, and \eqref{estimate 20}, we deduce 
\begin{equation}\label{[Equ. (211), Y20d]} 
Z_{1}^{(u,\theta)} (t) = z_{1}(t) \text{ and } Z_{2}^{(u,\theta)} (t) = z_{2}(t) \hspace{1mm} \forall \hspace{1mm} t \in [0, T_{L}] \hspace{1mm} \textbf{P}\text{-almost surely}. 
\end{equation}
By Proposition \ref{Proposition 4.4} we know that $z_{1}, z_{2} \in C_{T} \dot{H}_{x}^{\frac{n+2 + \sigma}{2}} \cap C_{\text{loc}}^{\frac{1}{2} - \delta} \dot{H}_{x}^{\frac{n+\sigma}{2}}$ $\textbf{P}$-a.s. and thus the trajectory 
\begin{equation*}
t \mapsto \lVert z_{k} (t) \rVert_{\dot{H}_{x}^{\frac{n+2+\sigma}{2}}} \text{ and } t \mapsto \lVert z_{k} \rVert_{C_{t}^{\frac{1}{2} - 2 \delta} \dot{H}_{x}^{\frac{n+\sigma}{2}}} \text{ for both } k \in \{1,2\} 
\end{equation*}  
is $\textbf{P}$-a.s. continuous. It follows from \eqref{[Equ. (31) and (32), Y20a]}-\eqref{[Equ. (211), Y20d]} that 
\begin{equation}\label{[Equ. (212), Y20d]}
\tau_{L} (u,\theta) = T_{L} \hspace{3mm} \textbf{P}\text{-almost surely}. 
\end{equation} 
Next, we verify that $P$ is a martingale solution to \eqref{3} on $[0, T_{L}]$. The verification of (M1) follows from \eqref{[Equ. (11), Y20a]} and \eqref{estimate 17}. The verification of (M3) follows from \eqref{[Equ. (31) and (32), Y20a]}, \eqref{[Equ. (211), Y20d]}, \eqref{[Equ. (59), Y20c]}, and \eqref{estimate 320}, and by choosing $C_{t,q}$ in Definitions \ref{Definition 4.1}-\ref{Definition 4.2} depending on $C_{L,1}$ and $C_{L,2}$ from \eqref{[Equ. (59), Y20c]} and \eqref{estimate 320}, respectively. Finally, in order to verify (M2), we let $s \leq t$ and $g$ be bounded, $\mathbb{R}$-valued, $\mathcal{B}_{s}$-measurable, and continuous on $\Omega_{0}$. By Theorem \ref{Theorem 2.1} we know that $(u,\theta)( \cdot \wedge T_{L})$ is $(\mathcal{F}_{t})_{t\geq 0}$-adapted so that $g((u,\theta) (\cdot\wedge \tau_{L} (u,\theta)))$ is $\mathcal{F}_{s}$-measurable by \eqref{[Equ. (212), Y20d]}. Then, for $\psi_{i} = (\psi_{i}^{1}, \psi_{i}^{2}) \in C^{\infty} (\mathbb{T}^{n}) \cap L_{\sigma}^{2} \times C^{\infty} (\mathbb{T}^{n}) \cap \mathring{L}^{2}$, $M_{k, t \wedge \tau_{L} (u,\theta), 0}^{(u,\theta), i}$ is an $(\mathcal{F}_{t})_{t \geq 0}$-martingale such that $\langle \langle M_{k, t \wedge \tau_{L}(u,\theta), 0}^{(u,\theta), i} \rangle \rangle = (t \wedge \tau_{L} (u,\theta)) \lVert G_{k} \psi_{i}^{k}\rVert_{L_{x}^{2}}^{2}$ under $\textbf{P}$ which implies that $M_{k, t \wedge \tau_{L}, 0}^{i}$ is a $(\mathcal{B}_{t})_{t \geq 0}$-martingale under $P$ and $(M_{k, t \wedge \tau_{L} (u,\theta), 0}^{(u,\theta), i})^{2} - (t \wedge \tau_{L} (u,\theta)) \lVert G_{k} \psi_{i}^{k} \rVert_{L_{x}^{2}}^{2}$ is a $(\mathcal{F}_{t})_{t\geq 0}$-martingale under $\textbf{P}$. This leads to $(M_{k, t \wedge \tau_{L}, 0}^{i})^{2} - (t \wedge \tau_{L}) \lVert G_{k} \psi_{i}^{k} \rVert_{L_{x}^{2}}^{2}$ being a $(\mathcal{B}_{t})_{t \geq 0}$-martingale under $P$ so that $\langle \langle M_{k, t \wedge \tau_{L}, 0}^{i} \rangle \rangle = (t \wedge \tau_{L}) \lVert G_{k} \psi_{i}^{k} \rVert_{L_{x}^{2}}^{2} = \int_{0}^{t \wedge \tau_{L}} \lVert G_{k} \psi_{i}^{k} \rVert_{L_{x}^{2}}^{2} dr$, successfully verifying (M2).
 
\subsection{Proof of Proposition \ref{Proposition 4.6}}\label{Subsection 6.5} 
Because $\tau_{L}$ is a $(\mathcal{B}_{t})_{t\geq 0}$-stopping time that is bounded by $L$ due to \eqref{[Equ. (31) and (32), Y20a]} while $P$ is a martingale solution to \eqref{3} on $[0, \tau_{L}]$ due to Proposition \ref{Proposition 4.5}, we see that Lemma \ref{Lemma 4.3} completes the proof once we verify \eqref{[Equ. (22), Y20a]}. First, it follows from \eqref{[Equ. (211), Y20d]}-\eqref{[Equ. (212), Y20d]} that there exists a $P$-measurable set $\mathcal{N}\subset \Omega_{0}$ such that $P(\mathcal{N}) = 0$ and for both $k \in \{1,2\}$, for any $T > 0$,
\begin{equation}\label{[Equ. (213), Y20d]}
Z_{k}^{\omega} ( \cdot \wedge \tau_{L}(\omega)) \in C_{T} \dot{H}_{x}^{\frac{n+2+\sigma}{2}} \cap C_{\text{loc}}^{\frac{1}{2} - \delta} \dot{H}_{x}^{\frac{n+\sigma}{2}} \hspace{3mm} \forall \hspace{1mm} \omega \in \Omega_{0} \setminus \mathcal{N}. 
\end{equation} 
For every $\omega' \in \Omega_{0}$ and $\omega \in\Omega_{0} \setminus \mathcal{N}$ we define 
\begin{subequations}\label{estimate 321}
\begin{align}
\mathbb{Z}_{1, \tau_{L}(\omega)}^{\omega'} (t) \triangleq& M_{1,t,0}^{\omega'} - e^{- (t - t \wedge \tau_{L}(\omega))(-\Delta)^{m}} M_{1, t \wedge \tau_{L}(\omega), 0}^{\omega'}  \nonumber\\
& \hspace{15mm}  - \int_{t \wedge \tau_{L} (\omega)}^{t} \mathbb{P} (-\Delta)^{m} e^{- (t-s) (-\Delta)^{m}} M_{1,s,0}^{\omega'} ds, \\
\mathbb{Z}_{2,\tau_{L}(\omega)}^{\omega'} (t) \triangleq& M_{2,t,0}^{\omega'} - e^{(t- t\wedge \tau_{L}(\omega)) \Delta} M_{2,t\wedge \tau_{L}(\omega), 0}^{\omega'} + \int_{t \wedge \tau_{L}(\omega)}^{t} \Delta e^{(t-s) \Delta} M_{2, s,0}^{\omega'} ds,
\end{align}
\end{subequations} 
so that due to $\nabla\cdot M_{1, t, 0}^{\omega} = 0$ from \eqref{estimate 141}, 
\begin{subequations}\label{[Equ. (214), Y20d]}
\begin{align}
\mathbb{Z}_{1, \tau_{L}(\omega)}^{\omega'} (t) =& M_{1, t,0}^{\omega'} - M_{1, t \wedge \tau_{L}(\omega), 0}^{\omega'} \nonumber\\
& \hspace{15mm}  - \int_{t \wedge \tau_{L}(\omega)}^{t} \mathbb{P} (-\Delta)^{m} e^{- (t-s) (-\Delta)^{m}} ( M_{1, s,0}^{\omega'} - M_{1, s \wedge \tau_{L}(\omega), 0}^{\omega'}) ds, \\
\mathbb{Z}_{2, \tau_{L}(\omega)}^{\omega'} (t) =& M_{2,t,0}^{\omega'} - M_{2, t\wedge \tau_{L}(\omega), 0}^{\omega'}  +\int_{t \wedge \tau_{L}(\omega)}^{t} \Delta e^{(t-s) \Delta}(M_{2,s,0}^{\omega'} - M_{2, s \wedge \tau_{L}(\omega), 0}^{\omega'}) ds. 
\end{align}
\end{subequations} 
Due to \eqref{[Equ. (40b), Y20c]}, this leads us to 
\begin{subequations}\label{[Equ. (215), Y20d]}
\begin{align}
Z_{1}^{\omega'}(t) - Z_{1}^{\omega'}(t \wedge \tau_{L} (\omega)) =& \mathbb{Z}_{1, \tau_{L} (\omega)}^{\omega' } (t) + (e^{ -(t- t \wedge \tau_{L} (\omega)) (-\Delta)^{m}} - \text{Id}) Z_{1}^{\omega'} (t \wedge \tau_{L}(\omega)), \\
Z_{2}^{\omega'}(t) - Z_{2}^{\omega'}(t  \wedge\tau_{L}(\omega)) =& \mathbb{Z}_{2, \tau_{L} (\omega)}^{\omega'}(t) + (e^{(t- t \wedge \tau_{L}(\omega))\Delta} - \text{Id}) Z_{2}^{\omega'} (t \wedge \tau_{L}(\omega)). 
\end{align}
\end{subequations} 
It follows from \eqref{[Equ. (214), Y20d]} that $\mathbb{Z}_{k, \tau_{L} (\omega)}^{\omega'}$ is $\mathcal{B}^{\tau_{L} (\omega)}$-measurable for both $k \in \{1,2\}$ and from \eqref{[Equ. (215), Y20d]} that 
\begin{align}
& Q_{\omega} ( \{ \omega' \in \Omega_{0}: \hspace{0.5mm} Z_{k}^{\omega'} (\cdot) \in C_{T} \dot{H}_{x}^{\frac{n+2+ \sigma}{2}} \cap C_{\text{loc}}^{\frac{1}{2} - \delta} \dot{H}_{x}^{\frac{n+\sigma}{2}} \text{ for both } k \in \{1,2\} ) \label{[Equ. (216), Y20d]}\\
=& \delta_{\omega} ( \{ \omega' \in \Omega_{0}: \hspace{0.5mm} Z_{k}^{\omega'} ( \cdot \wedge \tau_{L}(\omega)) \in C_{T} \dot{H}_{x}^{\frac{n+2+\sigma}{2}} \cap C_{\text{loc}}^{\frac{1}{2} - \delta} \dot{H}_{x}^{\frac{n+\sigma}{2}} \text{ for both } k \in \{1,2\} \}) \nonumber \\
& \otimes_{\tau_{L}(\omega)}  R_{\tau_{L}(\omega), \xi(\tau_{L}(\omega), \omega)} ( \{ \omega' \in \Omega_{0}: \hspace{0.5mm} \mathbb{Z}_{k, \tau_{L} (\omega)}^{\omega'} (\cdot) \in C_{T} \dot{H}_{x}^{\frac{n+2+\sigma}{2}} \cap C_{\text{loc}}^{\frac{1}{2} - \delta} \dot{H}_{x}^{\frac{n+\sigma}{2}} \text{ for both } k = 1, 2 \} ),  \nonumber 
\end{align} 
where for all $\omega \in \Omega \setminus\mathcal{N}$, 
\begin{align*}
\delta_{\omega} ( \{ \omega' \in \Omega_{0}: \hspace{0.5mm} Z_{k}^{\omega'} ( \cdot \wedge \tau_{L}(\omega)) \in C H_{x}^{\frac{n+2+\sigma}{2}} \cap C_{\text{loc}}^{\frac{1}{2} - \delta} H_{x}^{\frac{n+\sigma}{2}} \text{ for both } k \in \{1,2\} \}) \overset{\eqref{[Equ. (213), Y20d]}}{=} 1. 
\end{align*} 
We can also write  
\begin{subequations}\label{[Equ. (217), Y20d]}
\begin{align}
& \int_{0}^{t} \mathbb{P} e^{-(t-s) (-\Delta)^{m}} d(M_{1,s,0}^{\omega'} - M_{1, s\wedge \tau_{L}(\omega), 0}^{\omega'}) \overset{\eqref{[Equ. (214), Y20d]}}{=} \mathbb{Z}_{1, \tau_{L}(\omega)}^{\omega'} (t), \\
& \int_{0}^{t} e^{(t-s) \Delta} d(M_{2,s,0}^{\omega'} - M_{2, s \wedge \tau_{L}(\omega), 0}^{\omega'}) \overset{\eqref{[Equ. (214), Y20d]}}{=} \mathbb{Z}_{2, \tau_{L}(\omega)}^{\omega'} (t). 
\end{align}
\end{subequations} 
As we deduced \eqref{estimate 16} from \eqref{estimate 20}, \eqref{[Equ. (217), Y20d]} and the fact that the process $\omega' \mapsto M_{k, \cdot, 0}^{\omega'} - M_{k, \cdot \wedge \tau_{L}(\omega), 0}^{\omega'}$ is a $G_{k}G_{k}^{\ast}$-Wiener process for both $k \in \{1,2\}$  imply under our hypothesis \eqref{5} that 
\begin{equation*}
R_{\tau_{L} (\omega), \xi(\tau_{L}(\omega), \omega)} (\{ \omega' \in \Omega_{0}:\hspace{0.5mm}  \mathbb{Z}_{k, \tau_{L}(\omega)}^{\omega'} (\cdot) \in C_{T} \dot{H}_{x}^{\frac{n+2+\sigma}{2}} \cap C_{\text{loc}}^{\frac{1}{2} - \delta} \dot{H}_{x}^{\frac{n+\sigma}{2}} \text{ for both } k \in \{1,2\} \}) = 1. 
\end{equation*} 
Thus, by \eqref{[Equ. (216), Y20d]}, for all $\omega\in \Omega_{0} \setminus \mathcal{N}$, 
\begin{equation*}
Q_{\omega} ( \{ \omega' \in \Omega_{0}: \hspace{0.5mm} Z_{k}^{\omega'} (\cdot) \in C_{T}\dot{H}_{x}^{\frac{n+2+ \sigma}{2}} \cap C_{\text{loc}}^{\frac{1}{2} - \delta} \dot{H}_{x}^{\frac{n+\sigma}{2}} \text{ for both } k \in \{1,2\} ) = 1; 
\end{equation*} 
i.e., for all $\omega \in \Omega_{0} \setminus \mathcal{N}$, there exists a measurable set $N_{\omega}$ such that $Q_{\omega} (N_{\omega}) = 0$ and for all $\omega' \in \Omega_{0} \setminus N_{\omega}$, the mapping $t \mapsto Z_{k}^{\omega'} (t)$ for both $k \in \{1,2\}$ lies in $C_{T}\dot{H}_{x}^{\frac{n+2+\sigma}{2}} \cap C_{\text{loc}}^{\frac{1}{2} - \delta} \dot{H}_{x}^{\frac{n+\sigma}{2}}$. This implies by \eqref{[Equ. (31) and (32), Y20a]} that for all $\omega \in \Omega_{0} \setminus \mathcal{N}$ 
\begin{equation}\label{[Equ. (218), Y20d]}
\tau_{L} (\omega') = \bar{\tau}_{L} (\omega') \hspace{1mm} \forall \hspace{1mm} \omega' \in \Omega_{0} \setminus N_{\omega} 
\end{equation} 
if we define 
\begin{align}
\bar{\tau}_{L}(\omega') \triangleq& \inf\{t \geq 0: \hspace{0.5mm} C_{S} \max_{k=1,2} \lVert Z_{k}^{\omega'}(t) \rVert_{\dot{H}_{x}^{\frac{n+ 2 + \sigma}{2}}} \geq L^{\frac{1}{4}} \} \nonumber \\
&\wedge \inf\{t \geq 0: \hspace{0.5mm} C_{S} \max_{k=1,2} \lVert Z_{k}^{\omega'} \rVert_{C_{t}^{\frac{1}{2} - 2 \delta} \dot{H}_{x}^{\frac{n+\sigma}{2}}} \geq L^{\frac{1}{2}} \} \wedge L. \label{[Equ. (219), Y20d]}
\end{align}  
By identical arguments to \cite{HZZ19}, this gives for all $\omega \in \Omega_{0} \setminus \mathcal{N}$, 
\begin{equation}\label{[Equ. (221), Y20d]} 
Q_{\omega} (\{ \omega' \in \Omega_{0}: \hspace{0.5mm} \tau_{L}(\omega') = \tau_{L}(\omega) \}) = 1.
\end{equation}  
  
\section*{Acknowledgements}
The author expresses deep gratitude to Prof. David Ullrich for very valuable discussions concerning fractional Laplacian and the referees who gave valuable suggestions for revisions that improved this manuscript significantly.

\end{document}